\documentclass[11pt,a5paper,openright,twoside]{report}

\usepackage[utf8]{inputenc}
\usepackage[T1]{fontenc}
\usepackage{makeidx}
\makeindex
\usepackage[brazil]{babel}
\usepackage[left=2cm,right=1.5cm,top=2cm,bottom=2cm]{geometry}

\usepackage{comment}

\RequirePackage{etex}

\usepackage{caption}
\usepackage{enumerate}
\usepackage{lmodern}
\usepackage{float}
 \usepackage[pdftex,colorlinks]{hyperref}
\usepackage{xlop}

\hypersetup{	  
	citecolor=black, 
	filecolor=black, 
	linkcolor=black, 
	urlcolor=black,  
}
\usepackage{makeidx}
\makeindex
\usepackage{mathrsfs}
\usepackage{mathtools}
\usepackage{multicol}

\usepackage{tikz}
\usetikzlibrary{arrows}
\usetikzlibrary{angles, quotes}
\usetikzlibrary{shapes.geometric}
\usepackage{tikz-network}
\usetikzlibrary{graphs}
\usetikzlibrary{calc,math}
\usetikzlibrary{decorations.markings}
\usetikzlibrary{knots}
\usepgfmodule{decorations}
\usetikzlibrary{shapes.geometric}
\usetikzlibrary{graphs,graphs.standard}
\tikzgraphsset{declare={polygon_n}{[clique]\foreach\x in\tikzgraphV{\x/}}}
\usetikzlibrary{patterns}
\usetikzlibrary{positioning} 

\tikzset{
  my label/.style={font=\scriptsize,inner sep=0.6pt},
  a/.style={my label,above,node contents={\fontsize{4}{5}\selectfont{$ $}}},
  b/.style={my label,right,node contents={\fontsize{4}{4}\selectfont{$ $}}},
  a-1/.style={my label,above,node contents={\fontsize{4}{4}\selectfont{$ $}}},
  b-1/.style={my label,right,node contents={\fontsize{4}{4}\selectfont{$ $}}},
}
\newcommand\caley[6]{
  \ifthenelse{0<#1}{
    \pgfmathtruncatemacro\newlev{#1-1}
    \pgfmathtruncatemacro\len{#2}
    \draw[color=blue] (0,0) -- (\len pt,0) coordinate (O);
    \begin{scope}[shift={(O)}]
      \begin{scope}[rotate=90] \caley{\newlev}{\len/2}{#4}{#5}{#6}{#3} \end{scope}
      \begin{scope}[rotate=0]  \caley{\newlev}{\len/2}{#3}{#4}{#5}{#6} \end{scope}
      \begin{scope}[rotate=-90]\caley{\newlev}{\len/2}{#6}{#3}{#4}{#5} \end{scope}
    \end{scope}
  }{\fill[blue] circle(0.5pt);}
}

\usepackage{tikz-cd}
\usepackage{ulem}
\usepackage{url}
\usepackage{float}

\usepackage{xlop} 
\usepackage{xcolor}
\usepackage{amssymb} 
\def\acts{\curvearrowright}

\makeatletter
\newcommand{\chapterauthor}[1]{%
    {\parindent0pt\vspace*{-25pt}%
    \linespread{1.1}\large\scshape#1%
    \par\nobreak\vspace*{35pt}}
    \@afterheading%
}
\makeatother
\counterwithin*{chapter}{part}

\RequirePackage{etex}
\usepackage{amsthm}

\theoremstyle{plain}
\newtheorem{thm}{Teorema}[chapter] 
\newtheorem{conjecture}[thm]{Conjetura}
\newtheorem{probl}[thm]{Problema}
\newtheorem*{pergunta}{Pergunta}

\theoremstyle{definition}
\newtheorem{definition}[thm]{Definição} 
\newtheorem{example}[thm]{Exemplo}
\newtheorem{examples}[thm]{Exemplos}
\newtheorem{corollary}[thm]{Corolário}
\newtheorem{proposition}[thm]{Proposição}
\newtheorem{lemma}[thm]{Lema}
\newtheorem{remark}[thm]{Observação}

\newtheorem{exercise}[thm]{Exercício}
\newtheorem{claim}{Afirmação}

\counterwithin*{notacao}{chapter}

\AtEndEnvironment{solucao}{\QED}
\AtEndEnvironment{demo}{\qed}

\newcommand{\Z}{\mathbb{Z}}
\newcommand{\Q}{\mathbb{Q}}
\newcommand{\N}{\mathbb{N}}
\newcommand{\R}{\mathbb{R}}
\newcommand{\s}{\mathbb{S}}
\newcommand{\C}{\mathbb{C}}
\newcommand{\Fc}{\mathcal{F}}
\newcommand{\hip}{\mathbb{H}^2}
\newcommand{\hipn}{\mathbb{H}^n}

\newcommand{\SL}{\mathrm{SL}}
\newcommand{\PSL}{\mathrm{PSL}}
\newcommand{\GL}{\mathrm{GL}}

\newcommand{\CAT}{\mathrm{CAT}}
\newcommand{\dist}{\mathrm{dist}}

\renewcommand{\sin}{\mathrm{sen}}
\renewcommand{\sinh}{\mathrm{senh}}

\newcommand*{\QED}{\hfill\ensuremath{\square}}

\renewcommand{\ker}{\mathrm{Ker}}
\newcommand{\im}{\mathrm{Im}}
\newcommand{\re}{\mathrm{Re}}
\newcommand{\id}{\mathrm{Id}}
\newcommand{\cay}{\mathrm{Cay}}
\newcommand{\QI}{\mathrm{QI}}

\newcommand{\Bij}{\mathrm{Bij}}
\newcommand{\Isom}{\mathrm{Isom}}
\newcommand{\Homeo}{\mathrm{Homeo}}
\newcommand{\qi}{\stackrel{\mathclap{\mathrm{QI}}}{\sim}}
\newcommand{\vi}{\stackrel{\mathclap{\mathrm{VI}}}{\sim}}
\definecolor{light-gray}{gray}{0.75}
\newcommand{\vol}{\mathrm{vol}}

\newcommand{\card}{\mathrm{card}}
\newcommand{\limdir}{\underrightarrow{\lim}\,}
\newcommand{\liminv}{\underleftarrow{\lim}\,}
\newcommand{\tr}{\mathrm{tr}}

\newcommand{\Hyp}{\mathrm{Hyp}}
\newcommand{\Sol}{\mathrm{Sol}}

\newcommand\rightAngle[4]{
  \pgfmathanglebetweenpoints{\pgfpointanchor{#2}{center}}{\pgfpointanchor{#1}{center}}
  \coordinate (tmpRA) at ($(#2)+(\pgfmathresult+45:#4)$);
  \draw[blue!80!black,thick] ($(#2)!(tmpRA)!(#1)$) -- (tmpRA) -- ($(#2)!(tmpRA)!(#3)$);
}

\begin{document}
\thispagestyle{empty}
\begin{center}
	\textcolor{white}{.}\\[6cm]
	\textrm{{\Huge Introdução à Teoria Geométrica de Grupos}}\\[5mm]
	\textit{{\large Mikhail Belolipetsky e Gisele Teixeira Paula}}\\[7.6cm]
	{\Large 2026}\\
\end{center}

\chapter*{}
\begin{center}
\textit{Para María e Roberto}
\end{center}

\tableofcontents




\chapter*{Prefácio}
\addcontentsline{toc}{chapter}{Prefácio}
Grupos são estruturas matemáticas abstratas usadas para capturar a simetria de objetos aritméticos e geométricos. A teoria geométrica de grupos trata-os  como objetos geométricos por si próprios. Esta ideia remonta ao trabalho de Cayley, Klein e Dehn no final do século XIX e início do século XX. O estudo das propriedades geométricas dos grupos continuou nos trabalhos de Efremovich, Milnor, Schwarz, Stallings, Bass, Serre e outros. Nas últimas décadas do século XX, a área foi fortemente influenciada e ampliada por Gromov. Desde então, a teoria geométrica de grupos tornou-se um assunto matemático bem definido.

O tema central da teoria geométrica de grupos diz respeito às conexões entre as propriedades algébricas de grupos finitamente gerados e as propriedades geométricas dos espaços métricos sobre os quais um grupo atua de maneira agradável. Em particular, um grupo finitamente gerado pode ser visto como tal espaço métrico em si, considerando seu grafo de Cayley dotado da métrica das palavras. Este é o ponto de partida para a teoria que estudamos.

O presente livro baseia-se nos cursos ministrados no IMPA pelo primeiro autor durante vários anos. Já existe extensa literatura inglesa sobre o assunto, da qual podemos citar os livros ``Topics in geometric group theory'' por P.~de la Harpe, ``How groups grow'' por A.~Mann, ``Geometric group theory: An Introduction'' por C.~L\"oh, e ``Geometric group theory'' por C.~Dru\c{t}u e M.~Kapovich. Nosso texto tem diversas interseções com essas fontes, mas a escolha do material e da apresentação costuma ser diferente. O livro deve ser acessível aos alunos de pós-graduação e pode ser utilizado como uma introdução ao assunto.

Passamos agora a revisar brevemente o conteúdo. Os Capítulos 1 a 4 consistem de conceitos básicos e preliminares de todo o material contido no livro. Grupos hiperbólicos são discutidos nos Capítulos 7 e 8, enquanto grupos amenáveis aparecem no Capítulo 11.  Leitores interessados em crescimento de grupos podem fazer uma leitura dos Capítulos 1 a 4, seguindo ao Capítulo 6 para estudar as ferramentas necessárias para a prova dos principais teoremas sobre crescimento, e passar diretamente aos Capítulos 9 e 10.

No Capítulo~\ref{cap1} lembramos os conceitos básicos de teoria de grupos, grupos livres e apresentações de grupos. No final do capítulo apresentamos alguns exemplos e problemas bem conhecidos que desempenharam papel importante no desenvolvimento da teoria combinatória e da teoria geométrica de grupos.

O Capítulo~\ref{cap2} trata de ações de grupos. O principal resultado aqui é o lema do pingue-pongue, também conhecido como lema de Tits. Aplicações desse lema  a subgrupos de grupos livres são consideradas na última seção.

No Capítulo~\ref{cap3} estudamos os grafos de Cayley e grafos de Schreier. Os grafos de Schreier podem ser considerados uma generalização dos grafos de Cayley. Eles são muito úteis para aplicações, mas um pouco menos conhecidos.

No Capítulo~\ref{cap4} encontramos pela primeira vez os principais conceitos da teoria geométrica de grupos. Começamos com as definições e as propriedades básicas das quasi-isometrias. Seguindo Gromov, damos duas definições diferentes de quasi-isometria exibindo suas características distintas e depois provamos sua equivalência. Na sequência, provamos o teorema de Milnor--Schwarz e consideramos alguns corolários. O capítulo termina com um levantamento sobre rigidez quasi-isométrica.

O próximo capítulo é sobre invariantes topológicos de quasi-isometrias. Definimos fins de espaços topológicos e estudamos fins de grupos finitamente gerados, a partir de seus grafos de Cayley.

Como preparação para o material a seguir, no Capítulo~\ref{cap6} definimos ultrafiltros, ultralimites e cones assintóticos. Na última seção consideramos a ação de um grupo finitamente gerado em seu cone assintótico.

O Capítulo~\ref{cap7} começa com uma breve recapitulação da geometria hiperbólica e depois avança para a noção de hiperbolicidade no sentido de Gromov e Rips. Apresentamos um estudo  detalhado da geometria de espaços $\delta$-hiperbólicos e suas fronteiras ideais.

No Capítulo~\ref{cap8} estudamos a hiperbolicidade de grupos segundo Gromov. Além disso, trabalhamos com relações entre suas propriedades isoperimétricas e algorítmicas e apresentamos a alternativa de Tits para esta classe de grupos.

No Capítulo~\ref{cap:crescimento} consideramos outro tópico central da teoria geométrica de grupos, que é o crescimento dos grupos. O estudo de crescimento de grupos está relacionado com o crescimento de volume em variedades Riemannianas, devido ao Teorema de Milnor-Schwarz. Neste capítulo, revisamos a conjectura de Milnor e provamos os teoremas de Wolf, Bass--Guivarch e Milnor. O capítulo termina com uma prova do teorema de Gromov sobre grupos de crescimento polinomial. Esta parte segue essencialmente a abordagem original de Gromov  \cite{Gromov81}, embora existam pelo menos outras duas provas mais recentes desse resultado, a de Kleiner \cite{kleiner2010new} e a de Shalom--Tao \cite{shalomtao2010finitary}. Optamos por apresentar a prova original de Gromov devido à sua relação natural com o desenvolvimento de vários tópicos em teoria geométrica de grupos.

O Capítulo~\ref{cap:crescimento-interm} trata da célebre construção devida a Grigorchuk dos grupos com crescimento intermediário.

Grupos amenáveis são definidos e estudados no Capítulo~\ref{cap11}. Esses grupos foram definidos por John von Neumann em 1929 em seu trabalho sobre o paradoxo de Banach--Tarski. Nós começamos com a definição de von Neumann, discutimos a relação com decomposições paradoxais da esfera, depois continuamos com a condição de Følner e seus corolários.

Por fim, o Capítulo~\ref{cap12} coleciona alguns problemas em aberto bem conhecidos.

\medskip

\noindent \textbf{Agradecimentos.} Gostaríamos de agradecer aos alunos do IMPA e da UFPR que frequentaram o curso e leram a versão inicial do livro pela participação e comentários. Agradecimentos especiais vão para Alcides de Carvalho Junior, cujas notas digitadas foram utilizadas no início do nosso projeto. Agradecemos a Alexey Talambutsa pelas referências bibliográficas. Por fim, gostaríamos de agradecer ao revisor anônimo pela cuidadosa revisão e pelos valiosos comentários.

O trabalho de Belolipetsky foi apoiado pela Bolsa de Produtividade do CNPq e Bolsas da FAPERJ E-26/201.189/2021, E-26/204.250/2024, Teixeira Paula foi apoiada pela Bolsa do CNPq 408834/2023-4 e pelo Edital 19/2025-UFPR/R/PRPI/COFPI.

\chapter{Grupos e apresentações de grupos}
\label{cap1}
 \section{Teoria básica de grupos}
Esta seção consiste de uma revisão dos conceitos principais de teoria básica de grupos, os quais serão utilizados ao longo do texto. O leitor interessado pode se referir a \cite{artin} ou \cite{lang} para mais resultados e demonstrações. 

\begin{definition}
Um \textit{grupo}\index{grupo} é um conjunto $G$ munido de uma operação binária $\cdot: G\times G \to G$ que satisfaz as seguintes condições:
\begin{enumerate}[(i)]
    \item{\textit{Associatividade}:} Para todos $g_1, g_2, g_3 \in G$, tem-se $g_1 \cdot (g_2 \cdot g_3) = (g_1 \cdot g_2) \cdot g_3$;
    \item{\textit{Existência de elemento neutro}:} Existe um elemento $e \in G$ tal que $e\cdot g = g = g\cdot e$, para todo $g\in G$;
    \item{\textit{Existência de inversos}:} Para todo $ g\in G$, existe um elemento, denotado por $g^{-1}$, de $G$, tal que $g\cdot g^{-1} = e = g^{-1}\cdot g$.
    \end{enumerate}
    O grupo $G$ será dito \textit{abeliano} se a operação $\cdot$ for comutativa, isto é, se $g_1 \cdot g_2 = g_2 \cdot g_1$, para todos $g_1, g_2 \in G$.\index{grupo abeliano}
\end{definition}

Dado  um grupo $(G,\cdot)$, um subconjunto $H \subset G$ é dito um \textit{subgrupo}\index{subgrupo} de $G$ se $H$ é, ele próprio, um grupo com respeito à restrição de $\cdot$ ao subconjunto  $H\times H$ de $G\times G$. Usamos a notação $H\leq G$ para dizer que $H$ é subgrupo de $G$.

\begin{remark}
A menos que seja necessário explicitar a operação de $G$, denotaremos a partir de agora $g \cdot h$ simplesmente por $g h$.
\end{remark}

\begin{example}[Grupo simétrico]
Se $X$ é um conjunto qualquer, o conjunto $S_X$ de todas as bijeções
$f:X\to X$ é um grupo com respeito à composição de funções. Chamamos
$S_X$ de \textit{grupo simétrico} sobre $X$. Os elementos de $S_X$ são
chamados de \textit{permutações} de $X$. Se $n\in\N$, abreviamos
$S_{\{1,2,\ldots,n\}}=S_n$. Esse grupo, em geral, não é abeliano.

Uma \textit{transposição} é uma permutação que troca exatamente dois
elementos e fixa  os demais. Toda permutação de $S_n$ pode ser
escrita como um produto de transposições. Além disso, a paridade do
número de transposições nessa decomposição independe da
decomposição escolhida. Dizemos que uma permutação é \textit{par} se ela
pode ser escrita como produto de um número par de transposições e
\textit{ímpar} caso contrário. O conjunto das permutações pares
forma um subgrupo de $S_n$, chamado \textit{grupo alternado} e denotado
por $A_n$.

\end{example}

\begin{definition}
Sejam $(G,\cdot)$ e $(H, \times)$ dois grupos. 
\begin{itemize}
    \item Um mapa $\varphi: G \to H$ é um \textit{homomorfismo de grupos}\index{homomorfismo} se, para todos $g_1, g_2 \in G$, vale $\varphi(g_1 \cdot g_2) = \varphi(g_1)\times \varphi(g_2)$. Note que todo homomorfismo de $G$ em $H$ leva o elemento neutro de $G$ no elemento neutro de $H$, e vale $\varphi(g^{-1}) = (\varphi(g))^{-1}$, para todo $g\in G$.
    \item Um homomorfismo de grupos é dito um \textit{epimorfismo} \index{epimorfismo} se for sobrejetivo e um \textit{isomorfismo} \index{isomorfismo} se for bijetivo. Naturalmente, neste último caso, a inversa $\varphi^{-1}:H \to G$ também é um homomorfismo de grupos. Caso exista um isomorfismo entre $G$ e $H$, dizemos que eles são \textit{isomorfos} e escrevemos  $G \cong H$.
\end{itemize}
\end{definition}

\begin{example}[Grupos lineares] Um grupo $G$ é dito linear\index{grupo linear} se existe um corpo $K$, um inteiro $n$ e um homomorfismo injetivo de $G$ no grupo linear geral $\mathrm{GL}(n, K)$, das matrizes $n\times n$ invertíveis. Por exemplo, são lineares os grupos:
\begin{itemize}
    \item $\mathrm{SL}(n, K)$, das matrizes $n\times n$ com entradas em $K$ e determinante 1;
    \item O grupo das matrizes triangulares superiores  invertíveis com entradas no corpo $K$.
\end{itemize}
\end{example}

\begin{example}
Seja $\varphi:G\to H$ um homomorfismo de grupos. Os conjuntos $\ker(\varphi):= \{g \in G \mid \varphi(g) = e_H\}$ e $\im(\varphi):=\{\varphi(g) \mid  g \in G\}$ são subgrupos de $G$ e $H$, respectivamente. Chamamos $\ker(\varphi)$ de \textit{núcleo} de $\varphi$ e $\im(\varphi)$ de \textit{imagem} de $\varphi$.
\end{example} 

\begin{example}
Seja $G$ um grupo. O conjunto $\mathrm{Aut}(G)$ dos isomorfismos de $G$ sobre ele mesmo, também chamados \textit{automorfismos}\index{automorfismo} de $G$, é um grupo com respeito à composição de funções. Evidentemente, $ \mathrm{Aut}(G)$ é um subgrupo de $S_G$.
\end{example}

\begin{definition}
Seja $H\leq G$. 
\begin{itemize}
    \item A cardinalidade de $H$ será chamada de \textit{ordem} de $H$ e denotada por $|H|$.
     \item A \textit{ordem} de um elemento $g\in G$ é a ordem do subgrupo $\langle g \rangle :=\{g^k\mid k\in \Z\}$, isto é, o menor $n\in \N$ para o qual $g^n = e$, caso exista tal inteiro. Se não existir, dizemos que $g$ tem ordem infinita. Neste caso, $\langle g \rangle \cong \Z$.
    \item As relações $g\sim g' \Leftrightarrow g^{-1}g' \in H$  e $g\equiv g' \Leftrightarrow g(g')^{-1} \in H$ são relações de equivalência em $G$. A classe de um elemento $g \in G$ por $\sim$ é o conjunto $gH :=\{gh; h \in H\}$, chamado \textit{classe lateral à direita} de $g$. Já a classe de $g$ por $\equiv$ é o conjunto $Hg :=\{hg; h \in H\}$, chamado \textit{classe lateral à esquerda} de $g$.\index{classe lateral}
    \item O \textit{quociente}\index{quociente de grupos} $G/ H$ (respectivamente $H\backslash G$) de $G$ por $H$ é o conjunto de classes laterais à direita (respect., à esquerda) de $G$ módulo $\sim$ (respect., módulo  $\equiv$). Esse conjunto nem sempre terá estrutura de grupo, a menos que $H$ seja normal, como veremos a seguir.  
    \item A cardinalidade do quociente $G/H$, a qual coincide com a do quociente $H\backslash G$, é chamada \textit{índice} \index{indice de um subgrupo@índice de um subgrupo} de $H$ em $G$, e denotada por $|G:H|$.
   
\end{itemize}
\end{definition}

Um subgrupo $N$ de um grupo $G$ é dito \textit{normal} em $G$ se for invariante por conjugações, isto é, se $gng^{-1} \in N$, sempre que $n\in N$ e $g\in G$. Neste caso, escrevemos $N \triangleleft G$. Um grupo que não possui nenhum subgrupo normal próprio não trivial é chamado de \textit{grupo simples}\index{grupo simples}.

Uma das principais propriedades de subgrupos normais é que o quociente de $G$ por um subgrupo normal $N$ herda naturalmente uma estrutura de grupo, como vemos na seguinte proposição, cuja prova pode ser encontrada em \cite[Chapter~2]{artin}.

\begin{proposition}
Sejam $G$ um grupo e $N$ um subgrupo normal de $G$. Considere o quociente $G/N$. 
Então:
\begin{enumerate}
    \item  O mapa 
    \begin{eqnarray*}
       \cdot:  G/N \times G/N &\to & G/N \\
(g_1N, g_2N) & \mapsto & (g_1N)\cdot (g_2N) :=(g_1  g_2)N
    \end{eqnarray*}
está bem-definido se, e somente se,  $N$ é normal em $G$. Neste caso, $G/N$ é um grupo com respeito à operação acima, chamado grupo quociente de $G$ por $N$.
\item A projeção canônica 
 \begin{eqnarray*}
      \pi: G &\to & G/N \\
g & \mapsto & gN
    \end{eqnarray*}
é um homomorfismo de  grupos, e o grupo quociente $G/N$, munido da operação $\cdot$, possui a seguinte propriedade universal: Para todo grupo $H$ e todo homomorfismo de grupos $\varphi:G \to H$ com $N\subset \ker(\varphi)$, existe exatamente um homomorfismo de grupos $\tilde{\varphi}:G/N \to H$ tal que $\tilde{\varphi} \circ \pi =\varphi$.

\begin{figure}[!ht]
	\centering
	\resizebox{3cm}{3cm}{
	\begin{tikzpicture}
\draw (0,0)node{$ G $}; 
\draw (0,-2)node{$G/N $}; 
\draw[->] (0.2,0) -- (1.7,0);
\draw[->] (0,-0.2) -- (0,-1.7);
\draw (-0.3,-1)node{$ \pi $}; 
\draw (1,0.2)node{$ \varphi $}; 
\draw (1.5,-1)node{$ \tilde{\varphi}$}; 
\draw (2,0)node{$ H$}; 
\draw[->] (0.3, -1.7) -- (1.9,-0.2);

 	\end{tikzpicture}}\\
 
	\end{figure}

\end{enumerate}
\end{proposition}

A relação entre homomorfismos de grupos e subgrupos normais é descrita pelo seguinte resultado (veja ibid.):

\begin{thm}\label{1th} Se $\varphi :G \to H$ é um homomorfismo de grupos, então  $\ker(\varphi)$ é normal em $G$ e $G/\ker(\varphi) \cong \varphi(G)$.

\end{thm}

\begin{example}
O \textit{centro} de um grupo $G$ é o subgrupo $Z(G)$ formado pelos elementos que comutam com todos os outros elementos de $G$, isto é, $Z(G) = \{g \in G\mid gh=hg, \mbox{ para todo } h \in G\}$. É imediato verificar que $Z(G)\triangleleft G$.
\end{example}

\begin{definition}
Seja $I$ um conjunto e considere uma família de grupos $(G_i)_{i \in I}$. O \textit{produto direto} \index{produto direto de grupos} dos grupos dessa família é o grupo $\displaystyle\prod_{i \in I}G_i$, produto cartesiano dos $(G_i)_{i \in I}$, com a operação feita pontualmente: 
$$(g_i)_{i\in I}\cdot (h_i)_{i\in I} = (g_i\cdot h_i)_{i\in I}.$$
\end{definition}

O produto direto de grupos tem a propriedade universal de que homomorfismos desse produto estão em bijeção com as famílias de homomorfismos em cada fator.

\begin{definition}
Sejam $N$ e $H$ grupos e $\varphi: H \to \mathrm{Aut}(N)$ um homomorfismo de grupos. O \textit{produto semidireto} \index{produto semidireto de grupos} de $H$ por $N$ com respeito a $\varphi$ é o grupo $N \rtimes_{\varphi} H$, definido como o produto cartesiano $N \times H$, munido da operação:
$$(n_1,h_1)*(n_2, h_2) = (n_1\cdot\varphi(h_1)(n_2), h_1h_2).$$
Se $H = \Z$, $\varphi: \Z \to \mathrm{Aut}(N)$ é determinado por $\varphi(1) = \psi$, por isso também podemos escrever $N \rtimes_{\psi} \Z$ com $\psi \in \mathrm{Aut}(N)$.  
\end{definition}

\begin{example}
\begin{enumerate}
    \item Se $N$ e $H$ são grupos e $\varphi: H \to \mathrm{Aut}(N)$ é o automorfismo identidade, então
$N \rtimes_{\varphi} H = N \times H$.
    \item Seja $n\geq 3$ um inteiro. O grupo diedral (veja Exemplo~\ref{diedral}) $D_{n}$ é isomorfo ao produto semidireto $\Z / n\Z \rtimes_{\varphi} \Z / 2\Z$, onde $\varphi: \Z/ 2\Z \to \mathrm{Aut} (\Z/ n\Z)$ é dado por multiplicação por $-1$.
\end{enumerate}
\end{example}
Equivalentemente, poderíamos definir o produto semidireto da seguinte maneira, usando o grupo ambiente: se $G$ é um grupo e $N\triangleleft G$, $H \leq G$ são subgrupos, dizemos que $G = N \rtimes H $ (é o produto semidireto de $H$ e $N$) se, e somente se, vale um dos itens equivalentes a seguir:
\begin{enumerate}[(1)]
    \item $G=NH$ e $H \cap N = \emptyset$;
    \item $G=HN$ e $H \cap N = \emptyset$;
    \item Para todo $g \in G$, existem únicos $n\in N$ e $h \in H$ tais que $g=nh$;
    \item Para todo $g \in G$, existem únicos $n\in N$ e $h \in H$ tais que $g=hn$;
    \item Existe o homomorfismo de retração $f:G\to H$ cujo núcleo é $N$.
\end{enumerate}

\begin{exercise}
Se $G = N \rtimes H $ segundo a definição acima, então $G = N \rtimes_{\varphi} H $ com $\varphi(h): n \mapsto hnh^{-1}$. 
\end{exercise}

Considere uma ação natural de $H$ na soma direta $\displaystyle \bigoplus_{h\in H} G$ dada por
$$ \phi: H \to \mathrm{Aut}\Big(\bigoplus_{h\in H} G\Big),\ \phi(h)f(x) = f(h^{-1}x),\ \forall x\in H.$$
Assim, podemos definir o produto semidireto
\begin{equation}\label{eq:wr prod}
\Big(\bigoplus_{h\in H} G\Big) \rtimes_{\phi} H.    
\end{equation}
\begin{definition}\label{def:wr prod}
O produto semidireto \eqref{eq:wr prod} é chamado de \textit{produto entrelaçado}\index{produto entrelaçado de grupos} de $G$ com $H$ e é denotado por $G\wr H$.
\end{definition}
Esta construção é fonte de muitos exemplos interessantes na teoria dos grupos.

\begin{example}
    \label{def:lamplighter} Um grupo que pode ser descrito como produto entrelaçado é o \textit{grupo acendedor de lâmpadas} $L$. A ideia por trás desse grupo é a seguinte: considere uma rua infinita (pensada como a reta real), com uma origem $\ell_0$ bem definida e postes de iluminação $\ldots \ell_{-2}, \ell_{-1}, \ell_{0}, \ell_{1}, \ell_{2}, \ldots$, situados sobre cada número inteiro e podendo cada um estar aceso ou apagado. Considere também um acendedor de lâmpadas, cuja função é caminhar pela rua, acendendo ou apagando alguma lâmpada, e então terminando sua jornada em algum dos postes, digamos $\ell_k$. 

    Uma boa analogia para descrever cada elemento do grupo é como um conjunto de instruções. Por exemplo, se estamos numa configuração onde o acendedor acendeu as lâmpadas nas posições -2, 1, 3 e 4, e terminou sua jornada na posição -4, o elemento correspondente pode ser descrito como o seguinte conjunto de instruções, assumindo que a posição inicial do acendedor é a origem: 
\begin{enumerate}[(1)]
     \item Vá uma vez para a direita; 
     \item Acenda a lâmpada;
     \item Ande duas vezes para a direita; 
     \item Acenda a lâmpada; 
     \item Ande uma vez para a direita; 
     \item Acenda a lâmpada;
     \item Ande seis vezes para a esquerda;
     \item Acenda a lâmpada;
     \item Ande duas vezes para a esquerda.
\end{enumerate}
Para compor dois elementos, basta aplicar os dois conjuntos de instruções correspondentes, um após o outro, notando que a posição inicial do segundo grupo de comandos não será necessariamente a origem, e sim a posição final correspondente ao elemento anterior. 

Para dar uma descrição mais algébrica deste grupo, observe que cada elemento é caracterizado por dois objetos:
\begin{enumerate}
    \item Um inteiro $k \in \Z$, correspondente à posição do acendedor.
    \item Uma sequência com suporte finito $(a_n)_{n\in\Z}$, onde $a_n \in  \Z_2$ representa o estado da lâmpada na posição $n$.
\end{enumerate}

Existem dois geradores para o grupo: um gerador, digamos $t$, incrementa $k$, de modo que o acendedor se move para a próxima lâmpada ($t^{-1}$ decrementa $k$), enquanto o outro gerador, $a$, age de modo que o estado da lâmpada $\ell_k$ é alterado (de apagado para aceso ou vice-versa).

Podemos assumir que apenas um número finito de lâmpadas estão acesas em qualquer momento, uma vez que a ação de qualquer elemento de $L$ altera no máximo um número finito de lâmpadas. O número de lâmpadas acesas é, no entanto, ilimitado.
\end{example}

\begin{exercise}
    Mostre que o grupo acendedor de lâmpadas pode ser escrito como o produto entrelaçado $\Z_2 \wr \Z$.
\end{exercise}

\begin{definition}
Um grupo $G$ é dito  \textit{grupo de torção}\index{torção} se todos os seus elementos tem ordem finita, e \textit{sem torção}, e é dito \textit{livre de torção}, se todo elemento não-trivial tem ordem infinita. 
\end{definition}

Note que o subconjunto $ \mathrm{Tor}\, G = \{g \in G \mid g \mbox{ tem ordem finita}\}$ em geral não é um subgrupo de $G$.

\begin{definition}
Dizemos que um grupo tem \textit{virtualmente} uma certa propriedade $\mathcal{P}$  se algum subgrupo $H$ de índice finito em $G$ tem tal propriedade. 
\end{definition}

Por exemplo, um grupo é virtualmente livre de torção se ele contém algum  subgrupo de índice finito livre de torção; é virtualmente  abeliano se contém subgrupo de índice finito que é abeliano, etc.

\section{Grupos livres}
\label{gruposlivres}

Definiremos agora um tipo especial de grupos, chamados grupos livres, que possuem um papel importante no estudo de geometria hiperbólica e topologia, por exemplo. Como veremos, esses grupos tem as mais simples apresentações possíveis.


\begin{definition}[Propriedade universal dos grupos livres]
\index{grupo livre} Seja $X$ um subconjunto de um grupo $F$. Dizemos que $F$ é \textit{livremente gerado} por $X$ se toda função $\varphi: X \to G,$ do conjunto $X$ para um grupo $G$ pode ser estendida a um único homomorfismo $\Phi : F \to G.$
\end{definition}

Diremos que $F$ é um grupo \textit{livre} se ele for livremente gerado por algum conjunto $X\subset F$. 

\begin{lemma}
Se $F$ é livremente gerado por $X$, então $F$ é o menor subgrupo de $F$ que contém $X$.
\end{lemma}
\begin{proof}
Denote por $\Gamma$ o menor subgrupo de $F$ que contém $X$. Claramente, $\Gamma \subset F$. Para verificar a inclusão contrária, note que a inclusão $i_1:X \to \Gamma$ se estende, pela propriedade universal acima, a um único homomorfismo $f : F \to \Gamma$. Compondo $f$ com a inclusão $i_2: \Gamma \to F$, obtemos um homomorfismo $\theta: F \to F$. Mas note que os mapas $\theta$ e $\id_F$ são homomorfismos de $F$ em $F$ que estendem a inclusão $i:X \to F$. Pela unicidade desse homomorfismo, devemos ter $\theta = \id_F$, e disso segue que $F\subset \Gamma$.
\end{proof}

\begin{lemma}
\label{cardS}
Se $F$ e $F'$ são grupos livremente gerados por $X$ e $X'$, respectivamente, e $|X|= |X'|$, então $F\cong F'$.
\end{lemma}
\begin{proof}
Da hipótese de que $X$ e $X'$ tem a mesma cardinalidade, segue que existe uma bijeção $f:X\to X'$, com inversa $g= f^{-1}:X'\to X$. Essas bijeções, quando compostas com as inclusões de $X'$ em $F'$ e de $X$ em $F$, respectivamente, se estendem a homomorfismos $\Bar{f}: F \to F'$ e $\Bar{g}:F'\to F$. Mas $\Bar{g} \circ \Bar{f}$ é um automorfismo de $F$ que estende a inclusão de $X$ em $F$, bem como a aplicação $\id_F$, e portanto esses mapas são iguais pela unicidade da extensão. O mesmo argumento mostra que $\Bar{f} \circ \Bar{g} = \id_{F'}$. Concluímos portanto que $\Bar{f}$ e $\Bar{g}$ são, na verdade, isomorfismos.
\end{proof}

\begin{definition}
Se $F$ é livremente gerado por $X$ e $|X|= n < \infty$,  denotamos $F$ por $F_n$ e esse grupo é dito \textit{grupo livre de posto $n$}\index{posto de um grupo livre}. 
\end{definition}

O grupo $F_n$ está bem definido, a menos de isomorfismo, devido ao  Lema~\ref{cardS}. Veremos no Capítulo~\ref{cap2} que $F_m \cong F_n$ se, e somente se, $m=n$.

\subsection{Construção de grupos livres}

Até então, vimos uma caracterização de grupos livres, mas não exibimos exemplos de que esses grupos de fato existem. O objetivo desta seção é fornecer uma construção do grupo livre gerado por um conjunto qualquer.

Seja $X$ um conjunto. Na construção que faremos a seguir, seus elementos serão chamados {\it letras} ou {\it símbolos}. Definimos o conjunto de {\it letras inversas} $X^{-1}=\{x^{-1}; x \in X\}$ e vamos pensar em $X \cup X^{-1}$ como um {\it alfabeto}.  Uma {\it palavra}  em $X \cup X^{-1}$ é uma sequência finita (possivelmente vazia) de letras em  $X \cup X^{-1}$, isto é, uma expressão da forma $x_{i_1}^{\varepsilon_1}\ldots x_{i_k}^{\varepsilon_k}$ onde  $x_{i_j} \in X$ e $\varepsilon_j=\pm 1.$ Vamos chamar de $1$ a palavra vazia, isto é, uma palavra que não usa nenhuma letra do alfabeto $X \cup X^{-1}$.

Denotamos por $W(X)$ o conjunto de todas as palavras, incluindo a palavra vazia,  no alfabeto $X \cup X^{-1}$. Por exemplo, se $x_1, x_2 \in X$, então $x_1x_2x_2x_1x_1^{-1}\in W(X)$.

\begin{remark}
No caso em que $X$ é subconjunto de um grupo $(G, \cdot)$, teremos uma ambiguidade na notação, já que $x_1 \ldots x_n$ poderá representar tanto uma palavra formal no alfabeto $X \cup X^{-1}$, quanto um produto de elementos do grupo, isto é, $x_1\cdot \ldots \cdot x_n$. É preciso ser claro sobre em que sentido estamos usando essa notação neste caso, pois uma palavra determina um elemento do grupo, mas um dado elemento do grupo pode ser representado por diferentes palavras. Por exemplo, $ab$ e $ba$ são palavras diferentes, mas se o grupo que as contém for abeliano, elas representam ali o mesmo elemento.
\end{remark}

 O {\it comprimento} de uma palavra $w$ no alfabeto $X \cup X^{-1}$ é definido como  o número de letras usadas para escrever $w$ neste alfabeto. O comprimento da palavra vazia é igual a $0$. Uma palavra é dita {\it reduzida} se não contiver nenhum  par de letras consecutivas da forma $x^{-1}x$ ou $xx^{-1}$, com $x \in X.$ Uma {\it redução} \index{redução de uma palavra} de uma palavra $w \in W(X)$ é uma operação que  elimina da escrita de $w$ algum par de letras consecutivas de uma das formas acima.
 
 Por exemplo, se $X = \{x_1, x_2\}$, as palavras $1$, $x_1$, $x_{1}^{-1}$, $x_{1}x_2$ e $x_{1}x_2x_1^{-1}$ são reduzidas, enquanto $w = x_1x_2x_2x_1x_1^{-1}$ não é reduzida por conter o termo $x_1x_1^{-1}$. Neste último caso, uma redução de $w$ é a eliminação deste termo, levando à palavra $ w'= x_1x_2x_2$.
 
 Definimos uma relação de equivalência em $W(X)$ dizendo que $w \sim w'$ se, e somente se, $w$ pode ser obtida de $w'$ por uma sequência finita de reduções ou de reduções inversas (o acréscimo de termos do tipo $x^{-1}x$ ou $xx^{-1}$, com $x \in X$). Por exemplo, $ux^{-1}xv \sim uv$, onde $u, v \in W(X)$ e $x \in X$.

 \begin{proposition}\label{proposition1}
 Cada palavra  $w \in W(X)$ é equivalente a uma única palavra reduzida.
\end{proposition}

\begin{proof}
\textit{Existência:} Vamos provar a afirmação por indução no comprimento de $w$. Para palavras de comprimento 0 ou 1, o resultado é direto, já que ambas são reduzidas. Assuma que a afirmação é verdadeira para todas palavras de comprimento $n$ e considere uma palavra de comprimento $n+1$, $w = a_1 \ldots a_n a_{n+1}$, onde $a_i \in X \cup X^{-1}$. De acordo com a hipótese de indução, existe uma palavra reduzida $ u = b_1 \cdots b_k$ com $b_j\in X \cup X^{-1}$, tal que $ a_2\ldots a_{n+1} \sim u$. Assim, $w \sim a_1u$. Se $a_1 \neq b_1^{-1}$, então $a_1u$ é reduzida, e a afirmação está provada. Caso   $a_1 = b_1^{-1}$, teremos $a_1u \sim b_2 \ldots b_k$ e esta última palavra é reduzida.

\textit{Unicidade:} Seja $F(X)$ o conjunto de todas as  palavras reduzidas em $ X \cup X^{-1}$. Para cada $a\in  X \cup X^{-1}$, definimos a função de translação  $L_a : F(X) \to  F(X)$ por 
\begin{equation*}
    L_a(b_1\ldots b_k) = \left\{\begin{array}{cc}
       ab_1 \ldots b_k, & \mbox{ se } a \neq b_1^{-1};\\
        b_2 \ldots b_k, & \mbox{ se } a = b_1^{-1}.
    \end{array} \right.
\end{equation*}

Para cada palavra $w= a_1 \ldots a_n$, defina $L_w = L_{a_1}\circ \ldots\circ L_{a_n}$. Se $w =1 $ é a palavra vazia, defina  $L_1 = \id$. Note que $ L_a \circ L_{a^{-1}} = \id$, para todo $a \in X \cup X^{-1}$, e disso podemos concluir que $v \sim w $ implica $ L_v = L_w$. 

Vamos também provar por indução que se $w$ é reduzida, então $w = L_w(1)$. De fato, se $w$ tem comprimento 0 ou 1, não há o que provar. Assuma que a afirmação seja válida para palavras de comprimento $n$, e seja $w$ uma palavra reduzida de comprimento $n+1$. Então $w= au$, onde $a\in X \cup X^{-1}$ e $u$ é uma palavra reduzida em $X \cup X^{-1}$, de comprimento $n$, que não começa com $a^{-1}$, isto é, tal que $L_a(u) = au$. Então $L_w(1) = L_a \circ L_u(1) = L_a(u) = au = w$. 

Para provar a unicidade, é suficiente mostrar que se $v\sim w$, com $v$ e $w$ reduzidas, então $ v = w$. Mas como $v \sim w$ implica que $L_v = L_w$, temos $v = L_v(1) = L_w(1) = w$.
\end{proof}

Se $F(X) $ é como na prova da Proposição~\ref{proposition1}, $W(X)/{\sim}$ pode ser identificado com $F(X)$. A classe de uma palavra $w \in W(X)/{\sim}$ é denotada por $[w]$.

O conjunto $F(X)$ é um grupo, se equipado com a seguinte operação: para $[w],[w'] \in F(X)$ o produto $[w]\cdot [w']$ é a única palavra reduzida equivalente a $ww'$. O elemento neutro desse grupo é a palavra vazia. A cardinalidade de $X$ é chamada de {\it posto} de $F(X).$\index{posto do grupo livre}

 \begin{proposition} O grupo $F(X)$ definido acima é livremente gerado por $X$.
 
\end{proposition}

\begin{proof}
Seja $\varphi: X \to G$ uma função qualquer de $X$ em um grupo $G$. Podemos definir uma extensão  $\Phi: F(X)\to G$, começando de maneira natural em $X$ e estendendo-a, inicialmente a $X\cup X^{-1}$ pondo  $\varphi (x^{-1}) = 
(\varphi(x))^{-1}$. Em seguida, estendemos para $F(X)$, escrevendo, para cada palavra reduzida $w =x_1\ldots x_k$ em $F(X)$: $$\Phi(w) := \varphi(x_1)\ldots \varphi (x_k).$$

Para verificar que a extensão $\Phi$ é única, suponha que exista um homomorfismo $ \Psi: F(X) \to  G$  com $\Psi(x) =\varphi(x)$, para todo $x \in X$. Então para cada palavra reduzida $ w = a_1 \ldots a_n$ em $F(X)$, vale:
$$	\Psi(w) = \Psi(a_1)\ldots \Psi(a_n) = \varphi(a_1) \ldots \varphi(a_n) = \Phi(w).$$
Pela propriedade universal dos grupos livres, segue que $F(X)$ é livremente gerado por $X$.
\end{proof}

\begin{remark} Um grupo livre com posto $\geq 2$ é não-abeliano. 
\end{remark}

\section{Geradores e relações de grupos}

Alguns objetos matemáticos são definidos de forma abstrata, como um conjunto munido de uma ou mais operações que satisfazem certas condições. Esse é o caso de espaços vetoriais, anéis, corpos, grupos, entre outros. Muitas vezes, é difícil visualizar ou trabalhar com esses objetos através da definição, e buscamos ``modelos'' para facilitar o trabalho. No caso de espaços vetoriais, por exemplo, escolhemos trabalhar com bases de vetores linearmente independentes que geram o espaço, tendo como referência o espaço euclidiano $\R^n$.  

Neste sentido, trabalharemos com grupos a partir de suas apresentações, isto é, listas de elementos geradores, junto a uma descrição das relações entre esses elementos. A escolha de um conjunto de geradores para um grupo funciona como a escolha de uma base em um espaço vetorial. No entanto, em geral esses geradores podem não ser independentes. Grupos que admitem apresentações onde os geradores são independentes são chamados de grupos livres.

Seja $S$ um subconjunto de um grupo $G$ e $H\leq G$ um subgrupo. As seguintes afirmações são equivalentes: 
\begin{enumerate}[(1)]
    \item $H$ é o menor subgrupo de $G$ contendo $S$;
    \item $H = \displaystyle \bigcap_{S \subset \Gamma \leq G} \Gamma $;
    \item $H = \{s_1s_2 \cdots s_n \mid s_i  \mbox{ ou } s_i^{-1}\in S, \, \forall i= 1, 2, \ldots, n \}$.
\end{enumerate}

O subgrupo $H$ satisfazendo qualquer uma das condições acima é chamado \textit{subgrupo gerado por $S$}\index{subgrupo gerado por $S$}, denotado por $\langle S \rangle$. Neste caso, $S$ é dito um \textit{conjunto de geradores} de $H$. \index{geradores de um grupo}

Quando $S$ consiste de um único elemento, digamos  $S = \{g\}$, o grupo $\langle S \rangle$ é usualmente denotado por $\langle g \rangle$. Este é o grupo \textit{cíclico} gerado por $g$, e consiste de todas as potências de $g$ com respeito à operação de $G$, isto é, $\langle g \rangle = \{g^n \mid n \in \Z\}$. Um grupo cíclico é isomorfo a $\Z$ ou a $\Z_n$, para algum $n\in \N$. Caso seja infinito, este grupo (na verdade, esta classe de grupos, a menos de isomorfismos) também será referido como grupo livre abeliano de posto 1.

Mais geralmente, nos referimos à classe de isomorfismo do grupo $\Z^n$ como o ``grupo livre abeliano'' de posto $n$.\index{grupo livre abeliano} Este grupo é gerado pelos elementos $e_1,\ldots, e_n$, da base canônica de $\R^n$. O grupo livre abeliano de posto 0 é o grupo trivial. 

Note que o termo ``livre abeliano'' deve ser pensado como uma só palavra. Um grupo livre abeliano não é ``livre'' no sentido da Seção~\ref{gruposlivres}, a menos que seja trivial ou cíclico.

Dizemos que um subgrupo normal $K\triangleleft G$ é \textit{normalmente gerado} \index{subgrupo normalmente gerado} por um subconjunto $R\subset K$ se $K$ é o menor subgrupo normal de $G$ que contém $R$, isto é $$ K = \bigcap_{R \subset N \triangleleft \, G}N. $$
Usaremos a notação $K = \langle \langle R\rangle \rangle$ para representar este subgrupo. Dado qualquer subconjunto $A\subset G$, o subgrupo $\langle\langle A\rangle\rangle$ é chamado de \textit{fecho normal} de $A$ em $G$.\index{fecho normal}

\begin{exercise}
Dado um grupo $G$ e um subconjunto $A\subset G$, então $\langle\langle A\rangle\rangle$ é gerado por $S=\{gag^{-1} \mid a \in A, g\in G\}$.
\end{exercise}

\begin{definition}
Um grupo $G$ é dito \index{grupo finitamente gerado}{\it finitamente gerado} se tem um conjunto finito de geradores, ou seja, se existe $S \subset G$ finito tal que $\langle S \rangle =G$.
\end{definition}

\begin{example}
O conjuto $S=\{\pm 1\}$ gera $(\mathbb{Z},+)$. Já $(\mathbb{Q},+)$ não é
finitamente gerado.
\end{example}

\begin{example}
Os grupos lineares $\mathrm{GL}(n,\Z)$, $\mathrm{SL}(n,\Z)$ e $\mathrm{PSL}(n,\Z)$ são finitamente gerados, mas $\mathrm{GL}(n,\R)$ não o é.
\end{example}

Para $\mathrm{GL}(n,\Z)$, os geradores são: 
\begin{eqnarray*}
 s_1 = \left(\begin{array}{ccccc}
     0 & 0&\cdots &0 &1  \\
    1 &0 &\cdots & 0 &0 \\
    0 &1 &\ddots &0 &0 \\
    \vdots & \ddots&\ddots &\ddots &\vdots \\
    0 &0 &\cdots  & 1&0  \end{array}\right) &
     s_2 = \left(\begin{array}{ccccc}
    0 &1& 0&\cdots &0 \\
    1 &0 & 0&\cdots & 0 \\
    0 &0 & 1& \ddots &0 \\
    \vdots & \ddots& \ddots &\ddots &\vdots \\
    0 &0 & 0&\cdots  &1  \end{array}\right)\\
     s_3 = \left(\begin{array}{ccccc}
    1 &1& 0&\cdots &0 \\
    0 &1 & 0&\cdots & 0 \\
    0 &0 & 1& \ddots &0 \\
    \vdots & \ddots& \ddots &\ddots &\vdots \\
    0 &0 & 0&\cdots  &1  \end{array}\right) & 
     s_4 = \left(\begin{array}{ccccc}
    -1 &0& 0&\cdots &0 \\
    0 &1 & 0&\cdots & 0 \\
    0 &0 & 1& \ddots &0 \\
    \vdots & \ddots& \ddots &\ddots &\vdots \\
    0 &0 & 0&\cdots  &1  \end{array}\right)\\
\end{eqnarray*}
Uma demonstração pode ser encontrada em \cite[Seção~7.1]{kd}.

\begin{proposition}\label{prop:fggroups} Sejam $G$ um grupo e $N\triangleleft G$ um subgrupo normal. 
\begin{enumerate}
\item Se $N$ e $G/N$ são
finitamente gerados, então $G$ também o é.
\item Se $G$ é finitamente gerado, então $G/N$ também o é.
\end{enumerate}
\end{proposition}

\begin{proof}

\begin{enumerate}
\item Considere um conjunto finito $\{n_1,\ldots,n_k\}$ de geradores  de $N$ e um conjunto finito $\{g_1N,\ldots,g_mN\}$  de geradores de $G/N$. É fácil verificar que o conjunto finito abaixo gera $G$: $$\{n_i,g_j \mid 1\leq i\leq k, 1\leq j\leq m \},$$
\item Basta notar que o conjunto formado pelas imagens dos geradores de $G$ pela projeção canônica $\pi :G \to G/N$  gera $G/N$.
\end{enumerate}
\end{proof}

\begin{remark} Mesmo que tenhamos $G$ e $G/N$ finitamente gerados, isto não implica  que $N$ seja finitamente gerado.
\end{remark}
\begin{example}
Seja $H$ o subgrupo das permutações de $\mathbb{Z}$ gerado por $ t=(0,1)$ e $s(i)=i+1$. Denote por $H_{n}$, para cada $n\in \N$, o subgrupo de $H$ das permutações de $\Z$ suportadas em $[-n,n],$ e seja $H_{\omega}$ o subgrupo das permutações de $\Z$ com suporte finito, i.e., o grupo das bijeções $f: \Z \to \Z$ tais que $f$ é a identidade fora de um conjunto finito de $\Z$. Note que  $$H_{\omega} =  \bigcup_{n=0}^{\infty}H_n.$$ 

Então $H_{\omega} \triangleleft H$ e $H/H_{\omega} \cong \Z$. De fato, a relação $s^{k}ts^{-k} = (k,k+1)$, para todo $k \in \Z$, implica que $H_{\omega}$ é subgrupo de $H,$ pois  toda permutação é gerada por transposições. Para ver que $H_{\omega}$ é subgrupo normal, basta mostrar que $t^{k}H_{\omega}t^{-k} \subset H_{\omega}$ e $s^{k}H_{\omega}s^{-k} \subset H_{\omega}$. Dada $f \in H_{n}$, temos  $s^{k}fs^{-k}(i) = s^{k}f(i-k) = f(i-k) + k$. Mas caso $i-k\geq n$ temos $f(i-k) = i-k$ e, portanto, $s^{k}fs^{-k}$ é a identidade fora de $[-(k+n),k+n]$, ou seja, $s^{k}fs^{-k}\in H_{k+n}$, donde segue que $s^{k}H_{\omega}s^{-k} \subset H_{\omega}$. A inclusão $t^{k}H_{\omega}t^{-k} \subset H_{\omega}$ é imediata, já que $t \in H_{1}$. Como $t \in H_{\omega}$, segue que $H/H_{\omega}$ é cíclico gerado pela classe de $s$.

Os grupos $H$ e $H/H_{\omega}$ são finitamente gerados, mas $H_{\omega}$ não o é. Note que, se  $g_1, \ldots, g_k$ fosse um conjunto finito gerando $H_{\omega}$, então existiria  $n_0  = \max \{n_j \in \N, \mathrm{supp}(g_j) \subset [-n_j,n_j]\} \in \N$ com $g_j \in H_{n_0}$ para todo $j$. Portanto teríamos $H_{\omega} =  H_{n_0}$, o que é um absurdo.
\end{example}

\begin{proposition}\label{prop:finindexfg}
Se $H\leq G$ possui índice finito, então $G$ é finitamente gerado se, e somente se, $H$ é finitamente gerado.
\end{proposition}

\begin{proof} 
Sejam inicialmente $G = \langle S \rangle$, onde $S$ é finito e $H$ um subgrupo de $G$ tal que $|G : H| = n$, e tome $\mathcal{T} = \{a_1 = 1, \ldots , a_n\}$  uma \textit{transversal} para $H$ em $G$, i.e. um conjunto de representantes para as classes laterais à direita de $H$. Tome $x = s_1 \cdots s_n$ qualquer elemento de $G$, onde cada $s_i \in S\cup S^{-1}$. Seja $a_{i_1}$ o representante de $Hs_1$, e escreva $x = s_1 a_{i_1}^{-1}a_{i_1} s_2 \cdots s_n$. Agora seja $a_{i_2}$ o representante de $Ha_{i_1}s_2$, e escreva 
$$x = s_1 a_{i_1}^{-1} \cdot a_{i_1} s_2 a_{i_2}^{-1} \cdot a_{i_2} s_3\cdots y_n.$$ 

Procedendo da mesma maneira sucessivamente, obtem-se:
$$x = s_1 a_{i_1}^{-1} \cdot a_{i_1} s_2 a_{i_2}^{-1} \cdots a_{i_{n-1}} s_n a_{i_n}^{-1}\cdot a_{i_n}.$$ 

Aqui, os termos da forma $a_j s_k a_l^{-1}$ estão em $H$. Logo, se $x \in H$, temos $a_{i_n} = 1$ e $x$ é escrito como um produto de um número finito de produtos triplos $a_j s_k a_l^{-1}$. Portanto, o conjunto finito $$\{a_j s_k a_l^{-1} \mid a_j, a_l \in \mathcal{T}, s_k \in S \cup S^{-1}\}$$ gera $H$. 

 Suponha agora que $H=\langle T\rangle \leq G$, onde $[G:H] = n$ e $T$ é finito. Escolha representantes
$g_i,\, i = 1, \ldots,n$ das classes laterais de $H$ em $G$ e considere o conjunto finito
$$X = T \cup \{g_1,\dots,g_n\}.
$$
Vejamos que $G= \langle X\rangle$. De fato, dado $g\in G$, existem $i \in \{1, \ldots, n\}$ e $h \in H$ tais que $g=g_ih$. Mas como $H$ é gerado por $T$, segue que $g$ se escreve em termos dos elementos de $X$. 
\end{proof}

\begin{proposition}\label{prop:fin_many_subgrps}
Um grupo finitamente gerado contém apenas um número finito de subgrupos de um índice finito fixo. Além disso, $H$ contém um subgrupo de índice finito $K$, que é normal em $G$
\end{proposition}

\begin{proof}
Sejam dados $G$ um grupo finitamente gerado, $H\leq G$ tal que $|G : H| = n < \infty$, e escreva $G/H = \{g_iH \mid i = 1,\ldots,n\}$. Podemos construir um homomorfismo natural entre $G$ e o grupo simétrico $S_n$, vendo este último como o grupo $\Sigma(G/H)$, das permutações em $G/H$: 
\begin{align*}
\varphi_H :\; & G \to \Sigma(G/H) \\
              & g \mapsto (g_iH \mapsto (gg_i)H).
\end{align*}

Note que $H=\{g\in G \mid gH=H\} =\{g\in G \mid \varphi_H(g)H=H\}$. Isto implica que $H$ pode ser completamente recuperado se conhecemos apenas o homomorfismo entre $G$ e $S_n$. Sendo $G$ finitamente gerado, existem finitos homomorfismos possíveis, logo existem finitos subgrupos $H$ de índice $n$ em $G$.

O núcleo deste homomorfismo é $K := \ker(\varphi_H)=\displaystyle \bigcap_{x \in G} xHx^{-1},$
um subgrupo normal de $G$, contido em $H$ e que possui índice finito em $G$ (e em $H)$, pelo Teorema do Isomorfismo. 
\end{proof}


 
\begin{proposition}
\label{quocientelivre}
Todo grupo é quociente de um grupo livre.
 \end{proposition} 
\begin{proof}
Seja $G$ um grupo e escolha um subconjunto $S\subset G$ tal que $G = \langle S\rangle$ (por exemplo, $S = G$). A inclusão $i: S \to G$ se estende a um homomorfismo $\pi_s: F(S) \to G$. Como  $G = \langle S\rangle$, segue que $\pi_S$ é sobrejetivo. Assim, pelo Teorema~\ref{1th}, segue que $G \cong  F(S)/\ker(\pi_S)$. 
\end{proof}

\subsection{Apresentações} 
 Sejam $G$ um grupo e $S$ um conjunto de geradores de $G$. Como vimos na Proposição~\ref{quocientelivre}, a inclusão $i: S \to G$ pode ser estendida a um único homomorfismo $\pi_S : F(S) \to G$, e  $G \cong  F(S)/\ker(\pi_S)$.   Os elementos de $\ker (\pi_S)$ são chamados {\it relações} do grupo $G = \langle S \rangle$.
 
 Se $R=\{r_i\mid i \in I\}$ é tal que $\ker(\pi_S) = \langle\langle R \rangle \rangle $, o par denotado por $\langle S\mid R \rangle$ é dito uma \textit{apresentação}\index{apresentação de um grupo} de $G$.  Escrevemos $G = \langle S\mid R \rangle$ para indicar que esta é uma apresentação de $G$.

 \begin{definition}
 Um grupo $G$ é dito {\it finitamente apresentado}\index{grupo finitamente apresentado}  se existir uma apresentação finita de $G$, isto é, se existirem $S$ e $R$ finitos tais que $\langle S\mid R \rangle= F(S)/\langle\langle R \rangle \rangle $ é isomorfo a $G$.
 \end{definition}
 
\begin{example}
Se $R=\emptyset$, então $\langle\langle R \rangle \rangle= \{e\}$, e portanto $\langle S \mid R \rangle = F(S)/\{e\}$ é isomorfo a $F(S)$. Assim, $\langle a \mid \emptyset\rangle$ é uma apresentação de $\Z$, $\langle a,b \mid \emptyset\rangle$ é uma apresentação de $F_2$, etc. Em geral, um grupo livre é um grupo sem relações.
\end{example}

\begin{example}
Os grupos $\Z^n$ e $\Z_n$ são finitamente apresentados. De fato:
\begin{enumerate}

\item $\langle x_1,\ldots ,x_n \mid [x_i,x_j], i,j \in \{1,\ldots,n\} \rangle$ é uma apresentação de $\mathbb{Z}^n$, onde $[x_i,x_j] = x_ix_jx_i^{-1}x_j^{-1}$. 

\item  $\langle x \mid x^n \rangle$ é uma apresentação do grupo cíclico $\mathbb{Z}_{n}.$

\item Apresentações finitas não são únicas: por exemplo, $\langle x,y \mid x^2,y^3, xyx^{-1}y^{-1} \rangle$ é outra apresentação do grupo cíclico $\mathbb{Z}_{6}.$
\end{enumerate}
\end{example} 
 
 Note que dizer que $xyx^{-1}y^{-1}$ é uma relação no grupo $G$ significa dizer que $x$ e $y$ comutam. É comum escrevermos $\langle x,y \mid xy=yx  \rangle$ em vez de $\langle x,y \mid xyx^{-1}y^{-1} \rangle$.
 
\begin{example}[Grupo Diedral Finito]\label{diedral} Sejam $n\in\N$, com $n\geq 3$ e $X_n \subset \R^2$ um polígono regular de $n$ lados, com métrica induzida da métrica Euclidiana em $\R^2$. Então o grupo de isometrias de $X_n$ é chamado \textit{grupo diedral}:
$$ D_{n} :=  \langle s, t\mid s^n, t^2, tst^{-1} = s^{-1} \rangle = \langle s, t\mid s^n, t^2, (ts)^2 \rangle.$$

Geometricamente, $s$ corresponde à rotação de $\frac{2\pi}{n}$ em torno do centro do polígono regular $X_n$, e $t$  corresponde à reflexão sobre um diâmetro passando por um dos vértices de $X_n$, como na Figura~\ref{Dn}. Além disso, podemos mostrar que o grupo $D_{n}$ é finito de ordem $2n$.
\end{example}
\begin{figure}[H]
     \centering
 \begin{tikzpicture}
\tikzset{VertexStyle/.style = {shape = circle,fill = black,minimum size = 4pt,inner sep=0pt}}
\Vertex[x=1,y=0]{1}
\Vertex[x=0.5,y=1]{2} 
\Vertex[x=-0.5,y=1]{3}
\Vertex[x=-1,y=0]{4}
\Vertex[x=-0.5,y=-1]{5}
\Vertex[x=0.5,y=-1]{6}
\Vertex[x=0,y=0]{7}

\Edge[lw=1pt](1)(2)
\Edge[lw=1pt](2)(3)
\Edge[lw=1pt](3)(4)
\Edge[lw=1pt](4)(5)
\Edge[lw=1pt](5)(6)
\Edge[lw=1pt](6)(1)

\draw [color=blue, dashed] (-2,0) -- (2,0);
\draw [color=red, dashed] (0,0) -- (0.5,1);

\draw [color=blue, <->] (1.7,0.2) -- (1.7,-0.2);
\draw  [color=blue] (1.8,0.4)node{$t$};
\draw  [color=red] (0.45,0.2)node{$s$};

 \draw[red,->] (0.3,0) arc (0:60:0.3);
  
\end{tikzpicture}
         \caption{O grupo diedral $D_{6}$.}
     \label{Dn}   
\end{figure}

\begin{example}
Um exemplo de grupo finitamente gerado que não é finitamente apresentável é o grupo \textit{acendedor de lâmpadas} \index{grupo acendedor de lâmpadas} , que tem apresentação 
$ \langle s,t \mid \{[t^nst^{-n},  t^mst^{-m}],\, m,n \in \Z\} \rangle$
(veja o Exemplo \ref{def:lamplighter} e \cite{Baumslag}).
\end{example}


Em geral, é um problema difícil reconhecer a classe de isomorfismos de um grupo a partir de uma apresentação dada. Por exemplo, não existe algoritmo que reconheça se uma dada apresentação fornece o grupo trivial. De fato, existem muitas apresentações ``complicadas'' para o grupo trivial. Vejamos um exemplo:

\begin{example}
O grupo $G := \langle x, y \mid xyx^{-1}= y^2, yxy^{-1}= x^2\rangle$ é trivial. 
De fato, sejam $ \Bar{x} \in G$ e $\Bar{y} \in  G$ as imagens de $x$ e $y$, 
respectivamente, pela projeção canônica
\begin{equation*}
    F(\{x, y\}) \to F(\{x, y\})/ \langle\langle\{xyx^{-1}y^{-2}, yxy^{-1}x^{-2}\}\rangle\rangle = G.
\end{equation*}

Por definição, em $G$ obtemos 
$$\Bar{x}= \Bar{x} \Bar{y} \Bar{x}^{-1}\Bar{x} \Bar{y}^{-1}= \Bar{y}^2 \Bar{x} \Bar{y}^{-1} = \Bar{y} \Bar{y} \Bar{x} \Bar{y}^{-1}= \Bar{y} \Bar{x}^2;$$
e assim,  $\Bar{x} = \Bar{y}^{-1}$. Portanto,
$$\Bar{y}^{-2} = \Bar{x}^2 = \Bar{y} \Bar{x} \Bar{y}^{-1}= \Bar{y}\Bar{y}^{-1}y^{-1} = \Bar{y}^{-1}$$
e então $\Bar{x} = \Bar{y}^{-1} = e$. Como $\Bar{x}$ e $\Bar{y}$ geram $G$, concluímos que $G$ é trivial.
\end{example}

Uma conjectura ainda em aberto é a Conjectura de Andrews--Curtis\index{conjectura de Andrews--Curtis}, que afirma que qualquer apresentação do grupo trivial com o mesmo número de geradores e relações pode ser reduzida a uma apresentação trivial por uma sequência de movimentos simples, dados pela conjugação de relatores ou por transformações de Nielsen elementares. 

As \textit{transformações de Nielsen}\index{transformações de Nielsen} são definidas da  seguinte forma: dados um grupo $G$ e $n > 1$, uma transformação de Nielsen elementar em uma $n$-upla $T = (g_1, \ldots , g_n)$ de elementos de $G$ é um dos seguintes movimentos:
\begin{enumerate}
    \item Para algum $i \in  \{1,\ldots , n\}$ substitua  $g_i$ por  $g_i^{-1}$ em $T$.
    \item  Para $i \neq j$, com $i, j \in \{1,\ldots, n\}$, permute $g_i$ e $g_j$ em $T$.
    \item Para $i \neq j$, com $i, j \in \{1,\ldots, n\}$, troque $g_i$ por $g_ig_j$ em $T$. 
\end{enumerate}

Dizemos que duas $n$-uplas $T$ e  $T'$ são Nielsen equivalentes se existe uma sequência finita de $n$-uplas $T = T_0, T_1, \ldots , T_{k-1}, T_k = T'$, tal que $T_i$ é obtido de $T_{i-1}$ por uma transformação de  Nielsen elementar para $1 \leq i \leq k$.

\begin{exercise}
   O grupo de Thompson \index{grupo de Thompson} é definido como $$F := \langle x_0, x_1, \ldots \mid \{x_k^{-1}x_nx_k = x_{n+1} \mid  k, n \in \N, k < n\}\rangle.$$
   Mostre que F admite uma apresentação finita, a saber:
$$F \cong \langle A,B \mid\ [AB^{-1},A^{-1}BA],\ [AB^{-1},A^{-2}BA^{2}]\rangle.$$
\end{exercise}

\medskip

Agora vamos discutir as propriedades básicas das apresentações dos grupos.

Seja $\langle S\mid R \rangle$ uma apresentação de um grupo $G$. Dados um grupo $H$ e uma função $\varphi: S \to H$  que preserva as relações, isto é, tal que  $\varphi(r)=1$ para cada $r \in R$ (rigorosamente, estamos identificando $\varphi$ com sua extensão definida no grupo livre com base em $S$ e  pedindo que sua extensão preserve as relações). Então vale a seguinte:
\begin{proposition}[Propriedade Universal] 
A função $\varphi$ pode ser estendida a um homomorfismo $\tilde{\varphi} : G \to H$.
\end{proposition}
\begin{proof} 
Seja $H$ um grupo e $\varphi: S \to H$ uma função que preserva as relações. Pela propriedade universal dos grupos livres $\varphi$ se estende a um único homomorfismo $\Bar{\varphi} : F(S)\to H$. Como, por hipótese, $R \subset \ker (\varphi)$ temos que $\langle \langle R \rangle\rangle\subset \ker (\varphi).$ Portanto, $\bar{\varphi}$ desce ao quociente $F(S)/\langle\langle R \rangle \rangle$, definindo um homomorfismo $\tilde{\varphi} : G \to H$ dado por $\tilde{\varphi} \circ \pi_S=\bar{\varphi}$.

\begin{figure}[H]
	\centering
	\resizebox{2.5cm}{2.5cm}{
	\begin{tikzpicture}
\draw (0,0)node{$ F(S) $}; 
\draw (0,-2)node{$G$}; 
\draw[->] (0.5,0) -- (1.7,0);
\draw[->] (0,-0.3) -- (0,-1.7);
\draw (-0.3,-1)node{$ \pi_S $}; 
\draw (1,0.2)node{$ \bar{\varphi} $}; 
\draw (1.5,-1)node{$ \tilde{\varphi}$}; 
\draw (2,0)node{$ H$}; 
\draw[->] (0.3, -1.7) -- (1.9,-0.2);
 	\end{tikzpicture}}
	\end{figure}
	\vspace{-0.8cm}
\end{proof}

A proposição a seguir mostra que a propriedade de ser finitamente apresentado não depende da escolha de geradores:

\begin{proposition} Assuma que $G$ tem apresentação finita $\langle S\mid R\rangle$ e seja $\langle X\mid T\rangle$ uma apresentação arbitrária de $G$ com $X$ finito. Então existe um subconjunto finito $T_0 \subset T$ tal que $\langle X\mid T_0 \rangle$ é uma apresentação finita de $G$. 
\end{proposition}

\begin{proof}
Cada elemento $s \in S$ pode ser escrito como uma palavra reduzida $s=a_s(X)$ no alfabeto $X \cup X^{-1}$. Considere a função $i_{SX}: S \to F(X)$ definida por $i_{SX}(s) = a_s(X)$. Pela propriedade universal de grupos livres, $i_{SX}$ se estende a um único homomorfismo $p : F(S) \to F(X)$. 

Lembramos que se $G$ tem apresentação $\langle S \mid R \rangle$, denotamos por $\pi_S : F(S) \to G $ o epimorfismo que estende a inclusão $i:S \to G$, e neste caso $\ker(\pi_S)$ é normalmente gerado por $R$. 

Além disso, como $\pi_X \circ  i_{SX}$ é a inclusão de $S$ em  $G$ (já que $\pi_X(a_s(X)) = a_s(X)$, por essa palavra ser reduzida em $X\cup X^{-1}$), e ambos $\pi_S$ e  $\pi_X \circ p$ são homomorfismos de $F(S)$ para $G$ estendendo esse mapa, pela unicidade da extensão, temos $\pi_S = \pi_X \circ p$.

Analogamente, todo $x \in X$ pode ser expresso como  $b_x(S)$, uma palavra em $S \cup S^{-1}$, e isso define um mapa $i_{XS} : X \to F(S)$, $i_{XS}(x) = bx(S)$, o qual se estende a um homomorfismo 
$q : F(X) \to F(S)$. Um argumento análogo ao anterior mostra que $\pi_S \circ q = \pi_X$.

Para todo $x \in X$, vale $\pi_X(p(q(x))) = \pi_S(q(x)) = \pi_X(x)$, o que implica $x^{-1}p(q(x)) \in  \ker(\pi_X)$. Seja $N$ o subgrupo de $ F(X)$ normalmente gerado por $\{p(r) \mid r \in R\} \cup \{x^{-1}p(q(x)) \mid x \in X\}$.
Note que $N \leq \ker(\pi_X)$, já que $N$ é normalmente gerado por um conjunto contido em $\ker(\pi_X)$. 

Portanto, existe uma projeção natural
$$proj : F(X)/N \to F(X)/ \ker(\pi_X).$$
Seja $\tilde{p} : F(S) \to  F(X)/N$ o homomorfismo natural induzido por $p$. Como $\tilde{p}(r) = 1$ para todo $r \in R$, segue que $\tilde{p}(\ker (\pi_S)) = 1$, e portanto, $\tilde{p}$ induz um  homomorfismo 
$$\varphi : F(S)/\ker(\pi_S) \to  F(X)/N.$$
Vejamos que $\varphi$ é sobrejetivo. De fato, $ F(X)/N$ é gerado pelos elementos da forma $xN = p(q(x))N$, e estes são a imagem por $\varphi$ de $q(x) \ker(\pi_S)$.

Considere o  homomorfismo 
$$proj \circ \varphi : F(S)/\ker(\pi_S) \to  F(X)/\ker(\pi_X).$$
Tanto o domínio quanto o contradomínio de $proj \circ \varphi $ são isomorfos a $G$. Cada  $x \in X$ é enviado por $G \cong  F(S)/ \ker(\pi_S)$ em $q(x) \ker(\pi_S)$, e pelo isomorfismo  $G \cong  F(X)/ \ker(\pi_X)$ em $x \ker(\pi_X)$.

Note que $$proj \circ \varphi(q(x) \ker(\pi_S)) = proj(xN) = x \ker(\pi_X).$$
Isto significa que, módulo os dois isomorfismos mencionados acima, o mapa $proj \circ\varphi$ é a identidade $\id_G$. Daí, segue que $\varphi$ é injetivo, portanto uma bijeção. Assim, $ proj $ também é uma bijeção entre $F(X)/N$ e $F(X)/ \ker(\pi_X)$, o que pode ocorrer se, e somente se $ N = \ker(\pi_X)$. Em particular, $\ker(\pi_X)$ é normalmente gerado pelo conjunto finito de  relatores 
$$ \mathcal{R} := \{p(r) \mid r \in R\} \cup \{x^{-1}p(q(x)) \mid x \in X\}.$$
Como  $ \mathcal{R} \subset \langle \langle T\rangle \rangle$, todo relator $\rho \in  \mathcal{R}$ pode ser escrito como um produto
$$\rho = \displaystyle \prod_{i \in I_{\rho}} (\nu_i t_i \nu _i^{-1}), \,\, \nu_i \in F(X), t_i \in T \mbox{ e } I_{\rho} \mbox{ finito}.$$
Segue que $\ker(\pi_X)$ é  normalmente gerado pelo conjunto finito $$T_0 = \displaystyle\bigcup _{\rho \in \mathcal{R}} \{t_i, i \in I_{\rho}\} \subset T.$$
\end{proof}

\begin{proposition}
Sejam $G$ um grupo e $N\triangleleft G$. Se $N$ e $G/N$ são finitamente apresentados, então $G$ também o é.
\end{proposition}
\begin{proof}
Sejam $X$ um conjunto finito que gera $N$ e $Y$ um subconjunto finito de $G$ tal que $\Bar{Y} = \{ yN \mid y \in Y \} $ gera $G/N$, de modo que  $\langle X \mid  r_1, \ldots r_k \rangle$ e $\langle \Bar{Y} \mid \rho_1, \ldots \rho_n \rangle$ sejam apresentações finitas de $N$ e de $G/N$, respectivamente.

Na Proposição~\ref{prop:fggroups}, vimos que $G$ é gerado por  $S = X \cup Y $. Esse conjunto de geradores satisfaz a seguinte lista de relações:
\begin{eqnarray}
     r_i(X) = 1, \forall 1\leq i \leq k ; \:
     \rho_j(Y) = u_j(X),  \forall 1 \leq j \leq n; \label{rel1}\\ 
     yxy^{-1} = v_{xy}(X) ; y^{-1}xy = w_{xy}(X), \forall x \in X, y\in Y \label{rel2}
\end{eqnarray}para algumas palavras $u_j,  v_{xy}, w_{xy}$ de $S$, escritas com elementos de $X$, já que todos esses elementos estão em $N$. 

Afirmamos que este é um conjunto completo de relações de $G$.
Todas as relações acima podem ser reescritas como $t(X, Y ) = 1$, para um conjunto finito $T$ de palavras  em $S$. Se $K = \langle\langle T \rangle\rangle \triangleleft F(S)$, então o epimorfismo $\pi_S : F(S)\to  G$ define um epimorfismo  $\varphi : F(S)/K \to  G$.
Dada $wK \in \ker(\varphi)$, onde $w$ é uma palavra em $S$, graças às relações 
\eqref{rel2}, existem $ w_1(X)$ em $X$ e $w_2(Y)$ em $Y$ tais que $wK = w_1(X)w_2(Y )K$.
Aplicando a projeção $\pi : G \to G/N$, vemos que $\pi \varphi(wK))=1$, isto é, $\pi(\varphi (w_2(Y)K))= 1$. Isto implica que $w_2(Y)$ é um produto de uma quantidade finita de conjugados de $\rho_i(Y)$, e portanto $w_2(Y )K$ é o produto de uma quantidade finita de conjugados de $u_j(X)K$, pelo conjunto de relações \eqref{rel1}. Esse fato, junto às relações \eqref{rel2} implica que $w_1(X)w_2(Y)K = v(X)K$, para alguma palavra $v(X)$ em $X$. Então $\varphi(wK) = \varphi(v(X)K)$ está em $N$. Logo, $v(X)$  é o produto de uma quantidade finita de relatores $ r_i(X)$, o que implica que $v(X)K = K$.
Provamos assim que $\ker(\varphi)$ é trivial, e portanto $\varphi$ é um isomorfismo. Equivalentemente, $K = \ker(\pi_S)$. Isto implica que $\ker(\pi_S)$ é normalmente gerado pelo conjunto finito de relatores listado em \eqref{rel1} e \eqref{rel2}. Com isso, $G \cong \langle S \mid T \rangle$ é finitamente apresentado.
\end{proof}

\begin{exercise}
Sejam $F$ um grupo livre não abeliano e $\varphi : F \to \Z$ um homomorfismo não trivial. Mostre que $\ker (\varphi)$ não é finitamente gerado.
\end{exercise}

\begin{exercise}
Um elemento $g$ de um grupo $G$ é chamado \textit{não-gerador} se para cada
conjunto de geradores $S$ contendo $g$, o conjunto  $S\setminus\{g\}$ continua gerando $G$.
\begin{enumerate}[(a)]
\item Mostre que o conjunto de todos os não-geradores formam um subgrupo de $G$, chamado \textit{subgrupo de Frattini}.
\item Determine o subgrupo de Frattini de $(\Z,+).$
\item Determine o subgrupo de Frattini de $(\Z^{n},+).$
\end{enumerate}
\end{exercise}

\subsection{Abelianizações}
 
Para terminar esta seção, faremos um comentário sobre abelianizações. Seja $G$ um grupo. O \textit{comutador} \index{comutador de dois elementos} de $x,y \in G$ é o  elemento $[x, y] = xyx^{-1}y^{-1} \in G$. Sejam $H,K$ dois subgrupos de $G$. Denotamos por $[H,K]$ o subgrupo de $G$ gerado por todos os comutadores $ [h,k]$, com $h\in H, k\in K$.

Note que  para  $g,x,y \in X$, tem-se $ g[x, y]g^{-1} =
[gxg^{-1}, gyg^{-1}]$. 
Segue daí que $[G,G]$ é normal em $G$. Este grupo é usualmente chamado de \textit{grupo derivado de G} \index{grupo derivado} denotado  por $G'$. O grupo quociente  $G/G'$ é abeliano, e será chamado  \textit{abelianização} \index{abelianização} de $G$.

\begin{exercise}
Mostre que $G/[G,G]$ é abeliano. 
\end{exercise}

\begin{exercise}
     Mostre que $F_n/[F_n,F_n] \cong \Z^n$.
\end{exercise}

\begin{exercise}
Encontre o grupo derivado e a abelianização do grupo diedral finito $D_{n}$ e do grupo diedral infinito $D_{\infty}$. Observamos que $D_{\infty}$ é o grupo de isometrias de $\Z$, o qual é gerado pela translação $t(x) = x+1$ e pela simetria $s(x) = -x$.
\end{exercise}

Observe que não é necessariamente verdade que o grupo derivado $G'$ seja formado somente de comutadores $[x, y]$, com $ x, y\in G$. No entanto, eventualmente isso pode acontecer. Por exemplo, todo elemento do grupo alternado $A_n \leq S_n$ é um comutador de elementos de $S_n$ (veja \cite{Ore51}).

Isto leva a um invariante geométrico interessante, chamado \textit{comprimento do comutador} (cf. \cite{Bav91, Cal08}),\index{comprimento do comutador}  e denotado por $\ell_c(g)$, para $g \in G'$, o qual é definido como o menor número inteiro $k$ de modo que $g$ possa ser expresso como um produto de comutadores
$$g=[x_1, y_1] \ldots [x_k, y_k].$$

Também podemos definir o \textit{comprimento estável do comutador} de $g\in G'$:
$$\limsup{\frac{\ell_c(g^n)}{n}}.$$

Por exemplo, se $G$ é livre de posto 2, então todo elemento não trivial de $G'$ tem comprimento estável do comutador maior que 1.

Uma importante aplicação de abelianizações é que, se $X$ é um espaço topológico conexo por caminhos e $x\in X$, então o primeiro grupo de homologia $H_1(X)$ é isomorfo à abelianização do grupo fundamental $\pi_1(X,x)$.

\subsection{Produtos com amalgamação e extensões HNN}\label{sec:amalgam e HNN}

Amálgamas (produtos livres com amalgamação e extensões HNN) são ferramentas que permitem construir grupos mais complicados a partir de um par de grupos dado ou de um grupo junto com um par de seus subgrupos que são isomorfos entre si.

Primeiramente, definiremos o \textit{produto livre} \index{produto livre de grupos} de dois grupos com apresentações  $G_1 = \langle X_1\mid R_1\rangle$ e $G_2 = \langle X_2\mid R_2\rangle$. Este produto é dado pela apresentação 
$$G_1 * G_2 = \langle G_1, G_2\mid \quad \rangle,$$
que denota, de forma abreviada, o grupo:
$$\langle X_1 \sqcup X_2\mid R_1  \sqcup R_2\rangle.$$
Por exemplo, o grupo livre $F_2$ é isomorfo a $\mathbb{Z} * \mathbb{Z}$. É interessante observar aqui que o produto livre de dois grupos não é necessariamente um grupo livre. O exemplo anterior é livre, mas, por exemplo, $\mathbb{Z}/2\mathbb{Z} * \mathbb{Z}/2\mathbb{Z}$ não é. Um grupo livre não tem relações escritas nos geradores, nem \textit{entre} os geradores.

\begin{example}
    O grupo diedral infinito $D_\infty$ é o grupo das isometrias de $\mathbb{Z}$ (com a métrica induzida de $\mathbb{R}$). Este grupo é isomorfo
ao produto livre $\mathbb{Z}/2\mathbb{Z} * \mathbb{Z}/2\mathbb{Z}$. Por exemplo, a reflexão em \(0\)
e a reflexão em \(1/2\) fornecem geradores de \(D_\infty\) correspondentes aos
geradores óbvios de $\mathbb{Z}/2\mathbb{Z} * \mathbb{Z}/2\mathbb{Z}$.
\end{example}

Pode acontecer que, embora os geradores em si sejam mantidos livres de relações, queiramos impor condições que permitam que as palavras se misturem de alguma forma controlada. Este é o objetivo da próxima construção. Suponha que sejam dados grupos $H$, $G_{\alpha} $, $\alpha \in L$ e, para cada $\alpha \in L$,  um monomorfismo
$$\varphi_{\alpha} : H \to G_{\alpha}.$$

Definimos o \textit{produto livre com amalgamação:} \index{produto livre com amalgamação}
$$*_{\alpha} G_{\alpha} \;=\;
\langle \{G_{\alpha}\}_{\alpha\in L} \mid \varphi_{\alpha}(h)\varphi_{\beta}(h)^{-1}, \; h \in H, \, \alpha, \beta \in L \rangle.$$

Em outras palavras, além dos relatores em cada \(G_{\alpha}\), identificamos \(\varphi_{\alpha}(h)\) com \(\varphi_{\beta}(h)\) para cada \(h \in H\). Em comparação com o que foi feito antes, estamos identificando dentro do produto livre $* G_{\alpha}$ todas as cópias de um subgrupo em comum.

Uma notação usual para o produto livre com amalgamação, no caso de dois grupos $G_1$ e $G_2$, e quando os mapas $\varphi_1$ e $\varphi_2$ estão claros pelo contexto, é $$G_1 *_H G_2,$$
onde $H$ é o subgrupo comum (a menos de isomorfismo) de $G_1$ e $G_2$.

Caso $H = \{e\}$ seja o grupo trivial, então o produto amalgamado coincide com o produto livre dos grupos $G_{\alpha}$.

\begin{example}
O grupo $\mathrm{PSL}_2(\mathbb{Z}) = \mathrm{SL}_2(\mathbb{Z}) / \{\pm 1\}$ é isomorfo ao produto livre 
$$
\mathbb{Z}/2\mathbb{Z} * \mathbb{Z}/3\mathbb{Z}.
$$
Já  $\mathrm{SL}_2(\mathbb{Z}) $ é isomorfo ao produto livre amalgamado
$$
\mathbb{Z}/4\mathbb{Z} *_{\mathbb{Z}/2\mathbb{Z}} \mathbb{Z}/6\mathbb{Z}.$$
Neste último exemplo, os grupos cíclicos de ordens 4 e 6 dentro de $\mathrm{SL}_2(\mathbb{Z}) $ são gerados, respectivamente, por

$$\begin{pmatrix}
0 & 1 \\
-1 & 0
\end{pmatrix}
\quad \text{e} \quad
\begin{pmatrix}
0 & -1 \\
1 & 1
\end{pmatrix}.$$
\end{example}

Produtos livres com amalgamação surgem naturalmente em topologia. Por exemplo, pelo Teorema de Seifert–van Kampen, enunciado abaixo, o grupo fundamental de um espaço pontuado obtido pela colagem de vários componentes é um produto livre com amalgamação dos grupos fundamentais dos componentes sobre o grupo fundamental das interseções desses componentes (os subespaços envolvidos e suas interseções devem ser não vazios e conexos por caminhos).

\begin{thm}[Seifert-van Kampen]
Se um espaço topológico $X$ pode ser escrito como a união de subconjuntos abertos e conexos por caminhos $A_{\alpha}$, cada um deles contendo o ponto base $x_0\in  X$, e se cada interseção $A_{\alpha} \cap A_{\beta}$ é também conexa por caminhos, então o 
homomorfismo de grupos $\Phi : \displaystyle *_{\alpha} \pi_1(A_{\alpha})\to \pi_1
(X)$ é sobrejetivo. Se supomos ainda que cada interseção $A_{\alpha} \cap A_{\beta} \cap A_{\gamma}$ é conexa por caminhos, então  $\mathrm{Ker}(\Phi)$ é o subgrupo normal $N$ 
gerado por todos os elementos da forma $i_{\alpha \beta}(w)i_{\beta \alpha}(w)^{-1}$, com  $w \in \pi_1(A_{\alpha} \cap A_{\beta})$, onde  $i_{\alpha \beta} : \pi_1
(\pi_1(A_{\alpha} \cap A_{\beta})
)\to \pi_1
(A_{\alpha})$ é o homomorfismo induzido pela inclusão $\pi_1(A_{\alpha} \cap A_{\beta}) \to A_{\alpha}$. Portanto $\Phi$ induz um isomorfismo entre $\pi_1
(X)$ e $ \displaystyle*_{\alpha} \pi_1(A_{\alpha})/N$.    
\end{thm}

\begin{example}\label{ex:bitoro}
Considere a decomposição do bitoro $ X = T^2 \# T^2$, representada na Figura \ref{bitoro}:
$$X = A_1 \cup A_2,$$
onde $A_1$ e $A_2$ são toros com um disco removido, e $A_1 \cap A_2$ é homeomorfo a um anel.

Escolhendo um ponto base em $A_1 \cap A_2$, temos:
$$\pi_1(A_1) \cong F_2 = \langle a_1, b_1 \rangle, \quad
\pi_1(A_2) \cong F_2 = \langle a_2, b_2 \rangle
\mbox{ e} \quad
\pi_1(A_1 \cap A_2) \cong \mathbb{Z}.
$$

O gerador de $\pi_1(A_1 \cap A_2)$ corresponde ao laço da borda, que em cada $A_i$ é homotópico ao comutador
$[a_i,b_i] = a_i b_i a_i^{-1} b_i^{-1}$.

Pelo Teorema de Seifert-van Kampen,
$$
\pi_1(T^2 \# T^2)
\cong
\pi_1(A_1) \ast_{\pi_1(A_1 \cap A_2)} \pi_1(A_2),
$$
onde a amalgamação identifica os elementos
$[a_1,b_1] = [a_2,b_2].$

Portanto,
$$
\pi_1(T^2 \# T^2)
\cong
\langle a_1, b_1, a_2, b_2 \mid [a_1,b_1] = [a_2,b_2] \rangle.
$$
Equivalentemente,
$$
\pi_1(T^2 \# T^2)
\cong
\langle a_1, b_1, a_2, b_2 \mid [a_1,b_1][a_2,b_2] = 1 \rangle.
$$
\end{example}

\begin{figure}[H]
\begin{center}
\resizebox{4cm}{2cm}{
\begin{tikzpicture}
\draw[smooth] (0,1) to[out=30,in=150] (2,1) to[out=-30,in=210] (3,1) to[out=30,in=150] (5,1) to[out=-30,in=30] (5,-1) to[out=210,in=-30] (3,-1) to[out=150,in=30] (2,-1) to[out=210,in=-30] (0,-1) to[out=150,in=-150] (0,1);

 \fill[blue!40, opacity=0.5](0,1)
to[out=30,in=150] (2,1)to[out=-30,in=210] (3,1)to[out=90,in=270]  (3,-1) to[out=150,in=30] (2,-1) to[out=210,in=-30] (0,-1) to[out=150,in=-150] (0,1);

\fill[red!40, opacity=0.5] (2,1) to[out=-30,in=210] (3,1) to[out=30,in=150] (5,1) to[out=-30,in=30] (5,-1) to[out=210,in=-30] (3,-1) to[out=150,in=30] (2,-1) to[out=90,in=-90] (2,1);

\fill[white]
(0.4,0.1)
.. controls (0.8,-0.25) and (1.2,-0.25) .. (1.6,0.1)
.. controls (1.2,0.45) and (0.8,0.45) .. (0.4,0.1)
-- cycle;

\draw
(0.4,0.1)
.. controls (0.8,-0.25) and (1.2,-0.25) .. (1.6,0.1)
.. controls (1.2,0.45) and (0.8,0.45) .. (0.4,0.1);

\fill[white]
(3.4,0.1)
.. controls (3.8,-0.25) and (4.2,-0.25) .. (4.6,0.1)
.. controls (4.2,0.45) and (3.8,0.45) .. (3.4,0.1)
-- cycle;

\draw
(3.4,0.1)
.. controls (3.8,-0.25) and (4.2,-0.25) .. (4.6,0.1)
.. controls (4.2,0.45) and (3.8,0.45) .. (3.4,0.1);

\node at (2.5,0) {\tiny{$A_1 \cap A_2$}};
\node at (1,-1) {$A_1$};
\node at (4,-1) {$A_2$};
\end{tikzpicture}}
\end{center}
    \caption{Bitoro.}
    \label{bitoro}
\end{figure}

Por fim, definiremos o que é uma \textit{extensão HNN} \index{extensão HNN} de um grupo, com respeito a algum subgrupo. Esta é uma variação do
produto livre amalgamado, onde $G_1 = G_2$. Suponha que nos sejam dados um grupo $G$, um subgrupo $H\leq G$ e um monomorfismo $\phi: H \to G$. Então a extensão HNN de $G$ via $\phi$ é definida como
$$G *_{H, \phi}= \langle G , t \mid tht^{-1} = \phi(h), \, \forall h \in  H
\rangle. $$
É comum aparecer a notação $G *_{H}$ quando o monomorfismo $\phi$ for claro pelo contexto. As extensões HNN receberam esse nome em homenagem a G.~Higman, B.H.~Neumann e H.~Neumann, que foram os primeiros a estudar sistematicamente esses grupos.

\begin{example}
    A apresentação  $\langle a,t \mid tat^{-1}=a\rangle $  define uma extensão HNN do grupo 
$\langle a\rangle\cong\mathbb Z$ associada ao automorfismo identidade. 

Como a relação equivale a $at=ta$, obtemos $$\langle a,t \mid tat^{-1} = a\rangle  \cong \mathbb Z^2.$$ 
\end{example}

\begin{example}

Considere o grupo cíclico infinito $G= \langle a\rangle\cong\mathbb Z$. Os subgrupos $H=\langle a\rangle$ e $K=\langle a^2\rangle$ são isomorfos, via $\phi(a)=a^2$.

A extensão HNN associada a $\phi$ é
$$ \mathrm{BS}(1,2)= \langle a,t \mid tat^{-1}=a^2\rangle.$$
Este grupo é um caso particular da classe de grupos de Baumslag--Solitar $\mathrm{BS}(p,q)$. Mais geralmente, temos 
$$\mathrm{BS}(p,q)= \langle a,t \mid ta^pt^{-1}=a^q\rangle$$ 
é a extensão HNN de $G = \langle a\rangle$ com $H = \langle a^p\rangle$ e $\phi(a^p)=a^q$.
\end{example}
\begin{exercise}
    Se $H = \{e\}$, então $G *_{H} \simeq G*\mathbb{Z}$.
\end{exercise}

De maneira geral,  pode-se definir a extensão HNN simultânea de $G$ ao longo de uma coleção
de subgrupos isomorfos: dados $H_j \leq G,$ $j\in  J$ e imersões isomórficas $\phi_j : H_j \to G$, podemos definir o grupo
$$ G *_{\{\phi_j : H_j \to G, j \in J\}}= \langle G , t_j, j \in J \mid t_jht_j^{-1} = \phi_j(h), \, \forall h \in  H_j, j \in J
\rangle.$$

\section{Exemplos e problemas na teoria combinatória dos grupos}\label{exemplosteoriacombgrupos}
  
Esta seção é dedicada aos exemplos interessantes da teoria combinatória dos grupos e aos problemas algorítmicos na teoria.

\subsection{Exemplos}

\begin{enumerate}
\item \textit{Grupos de superfícies}\index{grupos de superfícies}: Generalizando o Exemplo \ref{ex:bitoro}, temos os grupos do tipo
$$G = \langle a_1, b_1,\ldots, a_g, b_g \mid \prod_{i=1}^{g}[a_i,b_i]\rangle, $$ que são grupos fundamentais de superfícies compactas, conexas, orientadas de gênero $g\geq 1$. Uma tal superfície é obtida identificando lados de um polígono com $4g$ lados, e cada par de lados identificado origina uma classe diferente no grupo fundamental da superfície. Na versão planar da superfície fica mais evidente a relação que aparece entre os lados, com o produto dos comutadores sendo trivial no grupo fundamental.

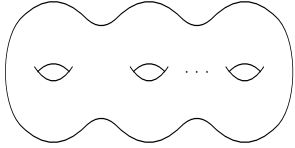
\begin{figure}[H]
\begin{center}
\resizebox{4cm}{2cm}{
\begin{tikzpicture}
\draw[smooth] (0,1) to[out=30,in=150] (2,1) to[out=-30,in=210] (3,1) to[out=30,in=150] (5,1) to[out=-30,in=210] (6,1) to[out=30,in=150] (8,1) to[out=-30,in=30] (8,-1) to[out=210,in=-30] (6,-1) to[out=150,in=30] (5,-1) to[out=210,in=-30] (3,-1) to[out=150,in=30] (2,-1) to[out=210,in=-30] (0,-1) to[out=150,in=-150] (0,1);
\node at (5.5,0) {$. \; \; . \; \; .$};
\draw[smooth] (0.4,0.1) .. controls (0.8,-0.25) and (1.2,-0.25) .. (1.6,0.1);
\draw[smooth] (0.5,0) .. controls (0.8,0.2) and (1.2,0.2) .. (1.5,0);
\draw[smooth] (3.4,0.1) .. controls (3.8,-0.25) and (4.2,-0.25) .. (4.6,0.1);
\draw[smooth] (3.5,0) .. controls (3.8,0.2) and (4.2,0.2) .. (4.5,0);
\draw[smooth] (6.4,0.1) .. controls (6.8,-0.25) and (7.2,-0.25) .. (7.6,0.1);
\draw[smooth] (6.5,0) .. controls (6.8,0.2) and (7.2,0.2) .. (7.5,0);

\end{tikzpicture}}
\end{center}
    \caption{Superfície compacta orientável de gênero $g$.}
    \label{genusg}
\end{figure}


\item \textit{Grupos de Artin de ângulos retos} (RAAG)\index{grupos de Artin de ângulos retos}:
Seja $\Gamma $ um grafo finito com conjunto de vértices $V=\{v_1, \ldots ,v_n\}$ e com conjunto de arestas $E = \{(v_i, v_j)\}_{ij}$ (cf. Seção~\ref{grafos}). 

O grupo de Artin de  ângulos retos correspondente a $\Gamma$, denotado por $A_{\Gamma}$, é o grupo com a seguinte apresentação: $$A_{\Gamma} : =  \left\langle V \mid [v_i,v_j], \mbox{ se }\ (v_i, v_j)\in E \right\rangle .$$ 

Por exemplo, se $\Gamma = (V,A)$ não tem arestas ($A=\emptyset$), então o grupo de Artin de ângulos retos $A_{\Gamma}$ é o grupo livre $F(V)$, de posto  $|V|$.
Se $\Gamma = (V,A)$ é um grafo completo, então o grupo de Artin de ângulos retos $A_{\Gamma}$ é o grupo $\Z^{n}$, onde $n = |V|$.

\item \textit{Grupos de Coxeter}:
\index{grupos de Coxeter} Seja $\Gamma = (V,E)$ um grafo finito simples. Associamos a cada aresta $e=(v_i, v_j) \in E$ um número natural $ m(e)  \neq 1$, chamado \textit{marcação} de $e$. Chamamos de \textit{grafo de Coxeter} \index{grafo de Coxeter} o par
$(\Gamma, m: E \to \N\setminus \{1\} )$. Definimos, a partir desse grafo o grupo de Coxeter $C_{\Gamma}$ com apresentação 
$$C_{\Gamma} = \left\langle V\mid v_i^{2}, (v_iv_j)^{m(e)} \mbox{ se existe uma aresta } e = (v_i,v_j)\right\rangle .$$

Por exemplo, se $\Gamma$ é um grafo com dois vértices $v_1$ e $v_2,$ e uma única aresta $e =  (v_1,v_2)$, com marcação $m$, então $C_{\Gamma}$ tem apresentação $\langle v_1, v_2 \mid v_1^{2},v_2^{2}, (v_1v_2)^{m} \rangle$, e portanto é isomorfo ao grupo diedral $D_{m}$.

\item \textit{Grupos de Artin}: 
\index{grupos de Artin}
Seja $(\Gamma, m: E \to \N\setminus \{1\} )$ um grafo de Coxeter. Definimos o grupo de Artin associado a $\Gamma$ por: 
$$A_{\Gamma} =  \left\langle V \mid \underbrace{v_iv_j\ldots }_{\text{m(e) termos}}= \underbrace{ v_jv_i\ldots}_{\text{m(e) termos}} \quad \mbox{se}\quad e =(v_i,v_j) \in E \right\rangle .$$ Note que  $A_{\Gamma}$ é um grupo de Artin de ângulos reto se, e somente se, $m (e) = 2$ para cada $e \in A$. 
Em geral, $C_{\Gamma}$ é o quociente de $A_\Gamma$ pelo subgrupo normal gerado pelo conjunto $\{v_i^{2}; v_i \in V \}$.

\item \textit{Grupo de Heisenberg discreto}\index{grupo de Heisenberg discreto}:
\begin{equation*} 
\begin{split}
  H_{2n+1} = \langle  x_1,\ldots, x_n,\, & y_1, \ldots, y_n, z \mid [x_i,z], [y_i,z], [x_i,x_j], \\
      &  [y_i,y_j], [x_i,y_j]z^{-\delta_{ij}}, \forall 1\leq i,j \leq n\rangle.\\
\end{split}
\end{equation*}

É possível verificar que 
\begin{equation*} 
  H_{2n+1} \cong \left\{  \begin{pmatrix}
1 & x_1 &  \cdots & x_n &  z \\
0 & 1 & \cdots&0 & y_n \\
\vdots  & \vdots  & \ddots & \vdots& \vdots \\
0 & 0 & \cdots &1 & y_1 \\
0 & 0 & \cdots &0 & 1 
\end{pmatrix} \mid x_i, y_j, z \in \Z \right\}.\\
\end{equation*}
Os grupos de Heisenberg são exemplos de grupos nilpotentes não-abelianos (cf. Seção~\ref{sec:sol-nilp-pol}).

\item \textit{Grupos de Baumslag--Solitar}:
\index{grupos de Baumslag--Solitar}
Dados $p,q\in\Z$, o grupo de Baumslag--Solitar associado a $p$ e $q$ é o grupo com apresentação
$$\mathrm{BS}(p,q): = \langle a, b \mid ab^{p}a^{-1}b^{-q} \rangle.$$

Todos os grupos de Baumslag--Solitar podem ser obtidos como extensões HNN de $\Z$ (veja Seção~\ref{sec:amalgam e HNN}). 
Observamos que $\mathrm{BS}(1,1) \cong \Z^2$, e $\mathrm{BS}(1,-1) \cong \pi_1(K)$, onde $K$ é a garrafa de Klein. 

Essa é uma classe interessante de grupos finitamente gerados com uma única relação que foi surpreendentemente útil para a construção de exemplos e contra-exemplos em combinatória e teoria geométrica de grupos, exemplos esses que marcam fronteiras entre diferentes classes de grupos (cf. \cite{BSgroups}). 

\item \textit{Grupos de Burnside}:
\index{grupos de Burnside}
Para $m,n\in \N$, definimos o grupo de Burnside por 
$$B(m,n) = \langle a_1,\ldots, a_n \mid \omega^{m}\rangle,$$
para qualquer palavra  $\omega$ nos geradores $ a_1,\ldots, a_n $.

O problema de Burnside consiste em determinar para quais inteiros positivos $m$, $n$ o grupo $B(m,n)$ é finito. A solução completa para este problema não é conhecida.

O próprio Burnside considerou alguns casos mais simples em seu artigo original e mostrou que $B(1, n)$ é o grupo cíclico de ordem $n$, e que $B(m, 2)$ é o produto direto de $m$ cópias do grupo cíclico de ordem $2$ e, portanto, esses casos são finitos. Os seguintes resultados adicionais são conhecidos (devido a Burnside, Sanov, M.~Hall): $B(m, 3)$, $B(m, 4)$ e $B(m, 6)$ são finitos para todos os valores de $m$. O caso particular de $B(2, 5)$ permanece em aberto.

Um grande avanço na resolução do problema de Burnside foi alcançado por P.~S.~Novikov e S.~I.~Adian em 1968. Usando um argumento combinatório complicado, eles demonstraram que, para cada número ímpar suficientemente grande $n$, o grupo $B(2,n)$ é infinito.
\end{enumerate}

\subsection{Problemas algorítmicos na teoria combinatória dos grupos}
No início do século XX,  Max Dehn introduziu a ideia de apresentações de grupos como uma ferramenta compreensível para estudar grupos infinitos, até então pouco compreendidos. Dehn percebeu que grupos infinitos frequentemente apareciam no estudo de topologia. Por exemplo, os \textit{grupos de superfície}, bem como os grupos fundamentais de uma \textit{esfera de homologia} (uma variedade que tem os mesmos grupos de homologia de uma esfera $S^n$), são  infinitos. Em seu artigo \cite{Dehn1911}, Dehn apresenta os seguintes problemas para grupos que possuem apresentação finita:

\begin{itemize}
\item {\bf Problema da Palavra.} Seja $G=\langle S \mid R\rangle$ um grupo finitamente apresentado. Construir ou provar a existência de  uma máquina de Turing (um algoritmo) tal que, dada uma palavra $\omega$, esse algoritmo determina se $\omega$ representa o elemento trivial de $G$, isto é, se $\omega \in \langle \langle R \rangle \rangle.$

\item {\bf Problema da Conjugação.} Seja $G=\langle S \mid R\rangle$ um grupo finitamente apresentado. Construir um algoritmo que, tendo como entrada um par de palavras $v,w$ no alfabeto $ S\cup S^{-1}$, determina se $v$ e $w$ representam elementos conjugados de $G$, isto é, se existe $g \in G$ tal que $[w] = g^{-1} [v] g $.

\item {\bf Problema do Isomorfismo.} Dados dois grupos finitamente apresentados $G_i= \langle S_i \mid R_i\rangle$, com $ i=1,2 $, determinar se $G_1$ é isomorfo a $G_2$.
\end{itemize}

 Os mais importantes resultados no artigo de Dehn de 1911 foram as soluções dos problemas da palavra e da conjugação para a classe de grupos de superfície. Há também outros artigos de Dehn com provas alternativas. Na primeira prova, ele usa conceitos de distância e área em grafos de Cayley associados aos grupos e relaciona os problemas algébricos listados com a geometria de tesselações do plano hiperbólico, enquanto a última é quase completamente combinatorial. Para mais sobre a história de seu trabalho, veja \cite{Peifer}.

Um pouco depois da morte de Dehn, nos anos 1950s, Novi\-kov \cite{Novikov} e Boone \cite{Boone} provaram de forma independente que o problema da palavra para grupos finitamente apresentados em geral é algoritmicamente insolúvel. Eles mostram que existem grupos finitamente apresentados para os quais nenhum algoritmo resolve o problema da palavra.

Outros problemas que também se mostraram insolúveis para grupos finitamente apresentados são: 
\begin{itemize}
    \item {\bf Problema de Trivialidade.} Dada uma apresentação finita $G= \langle S \mid R\rangle $, determinar algoritmicamente se $G$ é trivial.
    
    \item \textbf{Problema de Conjugação Simultânea}. 
    Dadas n-uplas $(v_1,\ldots v_n)$ e $(w_1,\ldots, w_n)$ de palavras em $X\cup  X^{-1}$ e uma apresentação finita $G = \langle X \mid R \rangle$, determinar se existe $g \in G$ tal que $[w_i] = g^{-1}[v_i]g, \, \forall i = 1, \ldots n$.
    \item \textbf{Problema do Mergulho.} Dados dois grupos finitamente apresentados $G_i = \langle X_i \mid R_i\rangle, \, i =1, 2$, determinar se $G_1$ é ou não isomorfo a algum subgrupo de $G_2$.
    
    \item \textbf{Problema da Associação (Membership Problem).}  Sejam $G$ um grupo finitamente apresentado, $h_1\ldots,  h_k \in G$ e $H$ o subgrupo de $G$ gerado pelos elementos $h_i$. Dado um elemento $g \in G$, determinar se $g\in H$.
\end{itemize}

Isso decorre de um teorema geral que foi provado independentemente por Adian \cite{Ad57}, na União Soviética e Rabin \cite{Rabin}, nos Estados Unidos. \index{teorema de Adian--Rabin} Eventualmente, uma situação semelhante a respeito de provas independentes pode acontecer com outros resultados que aparecem neste livro.

Fridman \cite{Fridman} provou que existem grupos para os quais o problema da palavra é solúvel, mas o da conjugação não é. Há também exemplos de grupos hiperbólicos com problema da palavra solúvel, mas problema da associação insolúvel. Mais ainda, há exemplos de grupos finitamente apresentados onde o problema da conjugação é solúvel, mas o problema da conjugação simultânea é insolúvel para $n>2$.
Referimo-nos aos artigos de pesquisa \cite{AdDu00} e \cite{Miller92} para informações mais detalhadas sobre problemas de decisão na teoria dos grupos. Mais recentemente, há estudos que discutem a decidibilidade desses problemas em tempo polinomial.

\medskip

A área de Teoria Geométrica de Grupos tem como um dos objetivos estudar sob quais hipóteses geométricas nos grupos e em seus subgrupos, um problema de um dos tipos acima é algoritmicamente solúvel. Alguns exemplos de solubilidade do problema da palavra são dados pelos seguintes resultados:

\begin{proposition}
Em todo grupo livre de posto finito, o problema da palavra é solúvel.
\end{proposition}

\begin{proof}
Dada uma palavra $w$ nos geradores $x_i$ de $F$, podemos cancelar recursivamente os possíveis pares $x_ix_i^{-1}$ ou $x_i^{-1}x_i$ em $W$. Eventualmente, em finitos passos, este processo resulta numa palavra reduzida $w_0$ equivalente a $w$. Se $w_0$ é não vazia, então $w$ representa um elemento não trivial de $F$. Caso $w_0$ seja vazia, então $[w] = 1$ em $F$.
\end{proof}

Uma ferramenta importante para definir solubilidade do problema da palavra é a função de Dehn, cuja definição veremos  no Capítulo \ref{cap8} (veja a Definição \ref{dehnfunction}). O seguinte resultado, devido a Gersten \cite{gersten}, caracteriza os grupos para os quais o problema da palavra é solúvel.
\begin{thm}
O problema da palavra é solúvel em um grupo $G$ se, e somente se, sua \textit{função de Dehn} for recursiva.

\end{thm}

\begin{exercise}
Mostre que um grupo para o qual o problema da conjugação ou o problema da associação é solúvel também terá solução para o problema da palavra.
\end{exercise}

\chapter{Geometria das ações de grupos}
\label{cap2}

Neste capítulo, revisamos o conceito de ação de um grupo sobre um conjunto, uma ferramenta fundamental para compreender como estruturas algébricas se manifestam como simetrias de objetos matemáticos. As noções aqui apresentadas serão importantes para a compreensão do Teorema de Milnor–Schwarz, no Capítulo \ref{cap4}, que relaciona grupos finitamente gerados com espaços métricos nos quais esses grupos agem de forma adequada.

\begin{definition}
Seja $G$ um grupo e $ X$ um conjunto qualquer. Uma \textit{ação à esquerda de $G$ em $X$} \index{ações de grupos}
é uma função 
\begin{eqnarray}
	\label{action}
\rho & : & G \times X   \rightarrow  X \nonumber \\
& & (g, x)  \longmapsto  gx =: \rho_g(x)
\end{eqnarray}
 tal que:
 \begin{enumerate}[(i)]
     \item $\rho_e(x) = x$ para todo $x\in X$;
     \item $\rho_{g_1g_2}(x) = \rho_{g_1}(\rho_{g_2}(x))$ para todos $g_1, g_2 \in G$ e $x \in X$.
 \end{enumerate}

\end{definition}
Podemos definir analogamente ações à direita de $G$ em $X$. A única diferença estará na ordem em que o produto de dois elementos agem. Para ações à direita, escrevemos $\rho_g(x)=:xg$ e substituímos a segunda condição por  $\rho_{g_1g_2}(x) = \rho_{g_2}( \rho_{g_1}(x))$.

Em geral, quando não especificarmos com que tipo de ação de $G$ em $X$ estamos lidando, esta será uma ação à esquerda, denotada por $G \acts X$. Toda ação à esquerda pode ser vista como um homomorfismo de $G$ no grupo $\Bij(X)$, das bijeções de $X$ sobre ele mesmo. A ação será chamada \textit{efetiva}\index{ação efetiva} ou \textit{fiel} se esse homomorfismo for injetivo. 

\section{Noções básicas em ações de grupos} 

Sejam $G$ um grupo e $X$ um conjunto onde $G$ age. Apresentamos a seguir algumas propriedades de ações de grupos que surgem naturalmente em diversos contextos e que poderão ser usadas ao longo do texto.

Se $G$ e $X$ são espaços topológicos, dizemos que a ação é \textit{contínua} se o mapa \eqref{action} for contínuo. Nesse caso, para cada $g \in G$, a aplicação $\rho_g : X \to X$, dada por $\rho_g(x)=gx$, é um homeomorfismo, com inversa $\rho_{g^{-1}}$. Um exemplo básico é a ação de $(\mathbb{R},+)$ em si mesmo por translações, dada por $\rho_t(x) = x+t$.

Quando $X$ é um espaço métrico, é comum considerar ações que preservam a estrutura métrica. Dizemos que a ação é \textit{isométrica} se cada $\rho_g$ é uma isometria de $X$. Ainda nesse contexto, a ação é dita \textit{colimitada} se existem $R \in \mathbb{R}$ e $x \in X$ tais que
$$
X=\bigcup_{g \in G} g(B(x,R)),
$$
onde $B(x,R)=\{y\in X \mid \dist(x,y)<R\}$ é a bola de centro $x$ e raio $R$. Por exemplo, a ação de $\mathbb{Z}$ em $\mathbb{R}$ por translações é tanto isométrica quanto colimitada.

No caso em que $X=M$ é uma variedade diferenciável e $G$ é um grupo de Lie, dizemos que a ação é \textit{suave} (ou \textit{diferenciável}) se o mapa \eqref{action} for diferenciável. Nesse caso, cada $\rho_g$ é um difeomorfismo, com inversa $\rho_{g^{-1}}$. Um exemplo é a ação  em $\mathbb{R}^n$ do grupo ortogonal $\mathrm{O}(n)$, formado por matrizes $A_{n\times n}$ reais que cumprem $A^{t}A = AA^{t} = I_n$, dada por $A\cdot x = Ax$.

Dada uma ação $G \acts X$ e um ponto $x \in X$, a \textit{órbita} de $x$ é o conjunto $Gx=\{gx \mid g \in G\}$, enquanto o \textit{estabilizador} (ou \textit{grupo de isotropia}) de $x$ é o grupo $G_x=\{g \in G \mid gx=x\}\leq G$. 

Nesse contexto, podemos definir uma relação de equivalência em $ X $, a saber,
$$x \sim y \mbox{ se, e somente se, existe } g \in G \mbox{ tal que } y = gx.$$
\begin{definition}
 O conjunto das classes de equivalência da relação acima será denotado por $ X/G $ e chamado de \textit{quociente} \index{quociente pela ação de um grupo} de $ X $ pelo grupo  $ G $. Esse quociente corresponde ao conjunto de órbitas da ação: $ X/G = \{Gx \mid x\in X\} $.
 \end{definition}

Dizemos que a ação é \textit{livre} se $G_x=\{e\}$ para todo $x \in X$ e  \textit{transitiva} se $Gx=X$ para todo $x \in X$, isto é, se há apenas uma órbita. Ações livres aparecem usualmente em contextos onde se quer passar alguma estrutura boa vinda de $X$ para o quociente $X/G$. Já as ações transitivas são importantes, por exemplo, em contextos geométricos, pois propriedades locais em $X$ se propagam globalmente, permitindo que se entenda a geometria do espaço a partir da geometria em um ponto.
Por exemplo, o grupo $\mathbb{Z}^n$ age livremente em $\R^n$ por translações inteiras.
Os grupos de isometrias de $\R^2, \mathbb{H}^2$ e $\mathbb{S}^2$ agem transitivamente nesses espaços. 


Voltando ao contexto de ações contínuas de grupos topológicos, uma noção importante é a de \textit{ação própria}, que ocorre quando o mapa
$$ G \times X \to X \times X, \quad (g,x) \mapsto (gx,x),
$$
é próprio, ou seja, a pré-imagem de compactos é compacta. 

No caso em que $G$ é um grupo discreto agindo continuamente em um espaço topológico $X$,
uma noção frequentemente mais adequada  e, em geral, mais fraca é a de
\textit{ação propriamente descontínua}. Dizemos que a ação de $G$ em $X$ é propriamente descontínua se, para todo compacto
$K \subset X$, o conjunto
$$
\{g \in G \mid gK \cap K \neq \emptyset\}
$$
é finito.
Por exemplo, a ação de $\mathbb{Z}$ em $\mathbb{R}$ por translações inteiras é
propriamente descontínua, enquanto a ação de $\mathbb{R}$ em si mesmo
por translações não satisfaz essa propriedade.

Ações propriamente descontínuas desempenham um papel importante na geometria hiperbólica. Em geral, quando um grupo discreto $\Gamma$ age por isometrias no espaço hiperbólico $\mathbb{H}^n$
 de maneira propriamente descontínua, o quociente $\mathbb{H}^n /\Gamma$ herda uma estrutura geométrica bem comportada, a de \textit{orbifold} hiperbólica. Se a ação é livre, esse quociente terá estrutura de variedade hiperbólica. Mais geralmente, se um grupo $G$ age livremente e propriamente descontinuamente em um espaço topológico simplesmente conexo $X$, então o espaço topológico
$X/G$, com a topologia quociente, tem grupo fundamental isomorfo a $G$.

\begin{exercise}
    Mostre que, se $X$ é um espaço métrico próprio (onde bolas fechadas são compactas) e
$G$ é um grupo discreto que age por isometrias em $X$, então a ação é própria
se, e somente se, é propriamente descontínua.
\end{exercise} 

Em particular, nos principais
exemplos provenientes da geometria, como as ações por isometrias em espaços
hiperbólicos mencionadas acima, essas duas noções coincidem. Deixamos como referência o artigo recente de Kapovich \cite{kapovich2024note}, que fornece relações entre outras noções semelhantes de ações próprias.

Por fim, ainda no contexto de ações contínuas, dizemos que uma ação é \textit{discreta} se o conjunto $\{\rho_g \mid g \in G\}$ é discreto no espaço dos homeomorfismos de $X$, munido da topologia compacto-aberta.
 
Caso $X$ seja um espaço topológico, essa ação será dita \textit{cocompacta} \index{ação cocompacta} se o quociente $X/G$ for compacto quando munido da topologia quociente. Além disso, se a ação for isométrica e propriamente descontínua, dizemos que é uma \textit{ação geométrica}\index{ação geométrica}.

\section{Lema do pingue-pongue}
O lema que apresentaremos a seguir remonta ao trabalho de Klein. Na teoria combinatória de grupos, é chamado de critério de Klein ou teorema de combinação de Klein. Mais tarde, foi usado por Tits na prova de seu famoso resultado, agora conhecido como a alternativa de Tits \cite{Tits72}. Desde então, o lema é frequentemente chamado de lema de Tits ou lema do pingue-pongue. Este resultado básico tem várias aplicações importantes em geometria e teoria dos grupos.

\begin{lemma}[Pingue-pongue]
\label{pingponglemma}
Sejam $X$ um conjunto e $g,h : X \to X$ bijeções. Se $A, B$ são subconjuntos não vazios  de $X$, tais que $A \nsubseteq B$ e $$g^{n}(A) \subset B \quad\forall n \in \Z\setminus\{0\},$$
 $$h^{m}(B) \subset A \quad\forall m \in \Z\setminus\{0\},$$ então $g$ e $h$ geram um subgrupo livre de posto $2$ de $\Bij(X)$.
\end{lemma}

\begin{proof}
Seja $\omega$ uma palavra reduzida em $\{g,h,g^{-1},h^{-1}\}$, com $\omega \neq 1$. Queremos provar que $\omega$ não é igual a $1$, como elemento de $\Bij(X)$. É suficiente considerar $\omega$ da seguinte forma:
\begin{equation}\label{pingpong}
\omega = g^{n_1}h^{m_1}g^{n_2}h^{m_2}\ldots g^{n_k}, \mbox{ com }  n_i, m_i \in \Z\setminus\{0\}.
\end{equation} 

 De fato: \begin{itemize}
 \item Se $\omega = h^{m_1}g^{n_1}h^{m_2}\ldots g^{n_k}h^{m_{k+1}},$ então $g\omega g^{-1}$ é da forma \eqref{pingpong}, e $g\omega g^{-1}\neq 1$ implica $\omega \neq 1$.
 \item Se $\omega =  g^{n_1}h^{m_1}g^{n_2}h^{m_2}\ldots g^{n_k}h^{m_k},$ então para $m \neq -n_1$, tem-se que $g^{m}\omega g^{-m}$ é da forma \eqref{pingpong}, e $g^m\omega g^{-m}\neq 1$ implica $\omega \neq 1$.
 \item Se $\omega = h^{m_1} g^{n_1}h^{m_2}\ldots g^{n_k},$ então para $m \neq n_k$, $g^{m}\omega g^{-m}$ é da forma \eqref{pingpong}, e novamente  $g^m\omega g^{-m}\neq 1$ implica $\omega \neq 1$.
\end{itemize} 
Agora, se $\omega$ é da forma \eqref{pingpong}, então
\begin{equation*}
    \begin{split}
        \omega(A) &= g^{n_1}h^{m_1}g^{n_2}h^{m_2}\ldots g^{n_k}(A) \\
        &\subset g^{n_1}h^{m_1}g^{n_2}h^{m_2}\ldots h^{n_{k-1}}(B) \\
        &\subset \ldots \subset  g^{n_1}(A) \subset B.
    \end{split}
\end{equation*}
Caso tivéssemos $\omega=1$, concluiríamos que $A \subset B$, contradizendo as hipóteses.
\end{proof}

\begin{exercise}
Suponha que $g : X\to X$ seja uma bijeção tal que $g(A) \subsetneq A$, para algum conjunto $A \subset X$. Mostre que $g$ tem ordem infinita.
\end{exercise}

\begin{example}[Subgrupos livres do grupo modular]
\label{exemplogrupomodular}

Para cada inteiro $k\geq 2$, as matrizes 
$$
g = \left( \begin{array}{cc}
1 & k \\ 
0 & 1
\end{array} \right)\quad\mbox{e}\quad
h = \left( \begin{array}{cc}
1 & 0 \\ 
k & 1
\end{array} \right)$$
geram um subgrupo livre de $\mathrm{SL}(2,\Z)$. 

Com efeito, o grupo $\mathrm{SL}(2,\Z)$ age em $\mathbb{H}^{2} =\{z \in \mathbb{C} \mid \im{(z)} > 0\}$ como transformações de Möbius: 
$$z \mapsto \dfrac{az+b}{cz+d}.$$
A matriz $g$ age como translação horizontal $z\mapsto z+k,$ enquanto $$
h = \left( \begin{array}{cc}
0 & 1 \\ 
-1 & 0
\end{array} \right)\left( \begin{array}{cc}
1 & -k \\ 
0 & 1
\end{array} \right)\left( \begin{array}{cc}
0 & -1 \\ 
1 & 0
\end{array} \right)$$ 
age 
 levando o interior de um disco limitado $C$, de centro $-\frac{1}{k}$ e raio $\frac{1}{k}$ no exterior de um disco limitado $C'$, de centro $\frac{1}{k}$ e raio $\frac{1}{k}$ . Aplicamos o lema do pingue-pongue  para os elementos $g, h\in \mathrm{SL}(2,\Z)$ e aos conjuntos $$A=\left\{z\in\mathbb{H}^2; -\frac{k}{2}<\Re(z)<\frac{k}{2}\right\} \mbox{ e } B= \mathbb{H}^2\setminus \overline{A}.$$
Note que $g^{n}(A) \subset B$ e $h^{n}(B) \subset A$. Segue portanto que $g$ e $h$ geram um subgrupo livre de $\mathrm{SL}(2,\Z)$.
\end{example}

\noindent\textit{Segunda prova.}
Outra forma de demonstrar a afirmação feita no Exemplo~\ref{exemplogrupomodular} é considerando a ação linear de $\mathrm{SL}(2,\Z)$ em $\R^2$ dada por 
$$\left( \begin{array}{cc}
a & b \\ 
c & d
\end{array} \right) \left( \begin{array}{c}
x\\ 
y
\end{array} \right) = \left( \begin{array}{c}
ax+by\\ 
cx+dy
\end{array} \right),$$
aplicando o Lema~\ref{pingponglemma} para as matrizes $g$ e $h$, e para os conjuntos 
$A:= \left\{ \left( \begin{array}{c}
x\\ 
y
\end{array} \right)\Big|\ |x| < |y| \right\}$ e $B:= \left\{\left( \begin{array}{c}
x\\ 
y
\end{array} \right)\Big|\ |x| > |y|\right\}$.

\begin{remark}
O Exemplo~\ref{exemplogrupomodular} não funciona para $k=1$, já que $$g^{-1}hg^{-1} = \left( \begin{array}{cc}
1 & -1 \\ 
0 & 1
\end{array} \right)\left( \begin{array}{cc}
1 & 0 \\ 
1 & 1
\end{array} \right)\left( \begin{array}{cc}
1 & -1 \\ 
0 & 1
\end{array} \right) = \left( \begin{array}{cc}
0 & 1 \\ 
-1 & 0
\end{array} \right).$$
Assim, $(g^{-1}hg^{-1})^4 = \id$, e o grupo gerado por $g$ e $h$ não é livre.
\end{remark}

O seguinte enunciado é uma outra versão do lema do pingue-pongue que é frequentemente usado.
\begin{lemma}
Sejam $g$ e $h$ bijeções de um conjunto $X$ e sejam $A^+$, $A^-$, $B^+$ e $B^-$ subconjuntos não vazios disjuntos de $X$ tais que 
$$h(A^+ \cup B^- \cup B^+) \subset A^+,\ g(B^+ \cup A^- \cup A^+) \subset B^+,$$ $$h^{-1}(A^- \cup B^- \cup B^+) \subset A^- \mbox{ e }\ g^{-1} (B^- \cup A^- \cup A^+) \subset B^-.$$ 
Então $g$ e $h$ geram um subgrupo livre de posto $2$ de $\Bij(X)$.
\end{lemma}

\begin{proof}
Esta afirmação segue como corolário do lema anterior se assumimos $A = A^-\cup A^+$ e $B = B^- \cup B^+$.
\end{proof}

Concluímos esta seção com o lema do pingue-pongue generalizado para posto maior que dois.

\begin{lemma} \label{pingponglemmagen} Sejam $X$ um conjunto e $g_i : X \to X$ bijeções, onde  $i \in \{1,\ldots, k\}$. Suponha que $A_1, \ldots, A_k$ são subconjuntos não vazios de $X$, tais que $\displaystyle \bigcup_{i=2}^{k}A_i \nsubseteq A_1$ e que para todo $i \in \{1,\ldots , k \}$, 
$$g_{i}^{n}\left(\bigcup_{j\neq i}A_j\right) \subset A_i \mbox{ para cada }  n \in \Z\setminus\{0\}.$$
Então $g_1, \ldots, g_k$ geram um subgrupo livre de posto $k$ de $\Bij(X)$.
\end{lemma}
\begin{proof}
Seja $\omega$ uma palavra  reduzida não trivial. Como na prova do Lema~\ref{pingponglemma} podemos supor sem perda de generalidade que $\omega$ começa com $g_1^{n}$ e termina com $g_1^{m},$ onde $n,m \in \Z$ são não nulos. Observamos que  $\omega \left(\displaystyle\bigcup_{i=2}^{k}A_i\right) \subset A_1.$ Se $\omega=1$ teríamos $\displaystyle\bigcup_{i=2}^{k}A_i \subset A_1,$ uma contradição.
\end{proof}

\section{Subgrupos de grupos livres}

\begin{proposition}Dois grupos livres $F(X)$ e $F(Y)$ são isomorfos se, e somente se,  $X$ e $Y$ tem a mesma cardinalidade.
\end{proposition}

\begin{proof} Já vimos no Lema~\ref{cardS}
 que dois grupos livres $F(X)$ e $F(Y)$ são isomorfos se $X$ e $Y$ tem a mesma cardinalidade. 
 
 Reciprocamente, seja $\Phi : F(X)\to F(Y)$ um isomorfismo. Denote por $N(X)$ o subgrupo de $F(X)$ gerado por $\{g^2 \mid g \in F(X)\}$. Claramente, $N(X)$ é normal em $F(X)$, já que $hg^{2}h^{-1} = (hgh^{-1})^{2}$, para todos $g,h \in F(X)$.  Analogamente,  $\Phi (N(X)) = N(Y)$ é o subgrupo normal gerado por $\{h^{2}\mid h\in F(Y)\}$. Assim, $\Phi$ induz um isomorfismo $\Psi:  F(X)/N(X)\to F(Y)/N(Y)$. 
 
 Vamos mostrar que $F(X)/N(X) \cong \Z_2^{\oplus X}$ e,  do isomorfismo entre  $F(X)/N(X)$ e $F(Y)/N(Y)$ obtemos $\Z_2^{\oplus X} \cong \Z_2^{\oplus Y}$, como espaços vetoriais sobre $\Z_2$.  Pela unicidade da dimensão de espaço vetoriais, concluiremos que $X$ e $Y$ tem a mesma cardinalidade. 
 
 Seja $\pi : F(X)\to F(X)/N(X)$ a projeção canônica. Como $ g^2 \in N(X)$, para todo $g \in F(X)$, temos  $\pi (g^2) =1,$ ou seja, $\pi(g) =  \pi(g^{-1})$ para todo $g \in F(X).$ Concluímos então que 
 $$\arraycolsep=0pt\def\arraystretch{1.2}
\begin{array}{rcl}
\pi([g,h])\*&\;=\;&\* \pi(gh)\pi((hg)^{-1}) \\
\*&\;=\;&\* \pi(gh)\pi(hg)  \\
\*&\;=\;&\*\pi(g)\pi(h)^2\pi(g)   \\
\*&\;=\;&\* \pi(g)^2=1.
\end{array}$$

Ou seja, $F(X)/N(X)$ é abeliano. Note que $\Z_2^{\oplus X}$ tem a seguinte apresentação $\langle X\mid \{ x^2,[x,y] \mid  x, y \in X\} \rangle$. Além disso, a projeção $\pi: F(X)\to F(X)/N(X)$ preserva as relações de $\Z_2^{\oplus X}$. Pela propriedade universal, $\pi$ induz um homomorfismo $\varphi :\Z_2^{\oplus X} \to F(X)/N(X).$ Por outro lado, considere um homomorfismo $\tilde{\eta}: F(X)\to \Z_2^{\oplus X}$ que envia  geradores em geradores. É claro que $N(X)\subset \ker(\tilde{\eta})$, logo $\tilde{\eta}$ induz um homomorfismo $\eta : F(X)/N(X)\to A$ dado por $\tilde{\eta} = \eta \circ \pi$. Não é difícil verificar que $\eta$ e $\varphi$ são inversos um do outro.
\end{proof}

\begin{proposition}[Nielson--Schreier] \label{prop:NS}
Todo subgrupo de um grupo livre é um grupo livre.
\end{proposition}
Veremos mais adiante que um grupo é livre se, e somente se, age livremente numa árvore, donde o Teorema de Nielson--Schreier vai seguir como corolário.

\begin{proposition}
O grupo livre de posto $2$ contém um subgrupo isomorfo a $F_k$, para cada $k\in \N$ ou para $k=\aleph_{0}.$
\end{proposition}

\begin{proof} 
Se $x$ e $y$ são os geradores livres de $F_2$, considere o conjunto $S = \{x_k := y^kxy^{-k} \mid k \in \N \}$. Vamos mostrar que $\langle S \rangle = F_{\aleph_{0}}$.

De fato, seja $$A_k:=\{ \mbox{palavras reduzidas que começam com } y^kx\}.$$ Consideramos a translação a esquerda de $a = x^\epsilon$ ou $y^\epsilon$, $\epsilon = \pm 1$, dada por  
$$L_a(a_1\ldots a_n) = \left\{\begin{array}{l}
aa_1\ldots a_n \quad \mbox{se}\quad a \neq a_1^{-1},\\
a_2\ldots a_n \quad \mbox{se}\quad a = a_1^{-1}. 
\end{array}\right.$$ 
Para $w = b_1\ldots b_m$ definimos $L_w = L_{b_1}\circ\ldots\circ L_{b_m}$. Para a palavra vazia $1$ defina $L_1 = \id$. Temos $L_a \circ L_{a^{-1}} = \id$ para todos geradores $a$ e assim $v \sim w$ implica $L_v = L_w$.

Agora temos que $L_{x_k}(A_j) \subset A_k$ para todo $j \neq k$. Então, pelo Lema~\ref{pingponglemmagen}, as translações $\{L_{x_k}\mid k = 1,\ldots m\}$ gera um subgrupo livre de posto $m$ de $\Bij(F_2)$, o que implica que  $\{x_k\mid k=1,\ldots,m\}$ gera um subgrupo livre de posto $m$ de $F_2$.
\end{proof}

\chapter{Grafos de Cayley e grafos de Schreier}
\label{cap3}
\section{Preliminares}\label{grafos}
    
Antes de definir o grafo de Cayley associado à apresentação de um grupo, relembramos algumas definições relacionadas a grafos.

\begin{definition}
Um \textit{grafo} \index{grafo} é um par $\Gamma=(V,A)$ de conjuntos, onde A é um conjunto de pares de elementos de V.  
\end{definition}

Chamamos os elementos de $V$ de \textit{vértices} e os elementos de $A$ de \textit{arestas} do grafo $\Gamma$.  Dizemos que dois vértices $v, v' \in V$ são \textit{adjacentes} ou \textit{vizinhos} se $\{v,v'\} \in A$, isto é, se há uma aresta que os contenha. Um grafo é dito \textit{$k$-regular} \index{grafo $k$-regular} se de cada vértice saem exatamente $k$ arestas. O \textit{grau} (ou \textit{valência}) de um vértice de um grafo é o número de  vértices adjacentes a ele.

\begin{example}
\label{grafo1}
 O par $\Gamma =(V,A)$, onde  $V=\{a,b,c,d,e\}$ e  $A=\{\{a,c\},\{a,d\},\{b,e\},\{c,e\}\}$,  é um grafo que podemos desenhar como segue: 
 
 
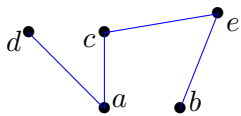
\begin{center}
      \begin{tikzpicture}[scale=0.5]
 \draw (0.4,0.2)node{$a$};
  \draw (2.4,0.2)node{$b$};
  \draw (-0.4,1.8)node{$c$};
  \draw (-2.4,1.8)node{$d$};
  \draw (3.4,2.2)node{$e$};
  \draw (0,0)node{$\bullet$};  
  \draw (3,2.5)node{$\bullet$};
  \draw (2,0)node{$\bullet$};
  \draw (-2,2)node{$\bullet$};
  \draw (0,2)node{$\bullet$};
\draw [color=blue] (2,0) -- (3,2.5);
\draw [color=blue] (0,0) -- (-2,2);   
\draw [color=blue] (0,0) -- (0,2);
\draw [color=blue] (0,2) -- (3,2.5);      
      
 \end{tikzpicture}
 \end{center}
\end{example}
 
\begin{example}
\label{grafo2}
  Um grafo é um objeto combinatório. Mais de uma figura pode representar o mesmo grafo geometricamente. Por exemplo, se $V = \{1, 2, 3, 4\}$ e $A=\{ \{1, 2\},\{2, 3\}, \{3, 1\}\}$, então as duas figuras abaixo representam bem o grafo $G=(V,A)$:
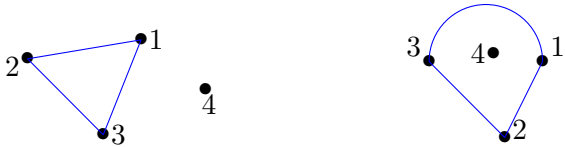
 \begin{center}
  \begin{tikzpicture}[scale=0.5]
  \draw (2.4,0)node{$3$};
  \draw (-0.4,1.8)node{$2$};
  \draw (4.8,0.8)node{$4$};
  \draw (3.4,2.5)node{$1$};
  \draw (3,2.5)node{$\bullet$};
  \draw (2,0)node{$\bullet$};
  \draw (0,2)node{$\bullet$};
  \draw (4.7,1.2)node{$\bullet$};
\draw [color=blue] (2,0) -- (3,2.5);
\draw [color=blue] (2,0) -- (0,2);
\draw [color=blue] (0,2) -- (3,2.5); 
 \end{tikzpicture}
 \hspace{2cm}
   \begin{tikzpicture}[scale=0.5]
  \draw (2.4,0.2)node{$2$};
  \draw (-0.4,2.4)node{$3$};
  \draw (1.3,2.2)node{$4$};
  \draw (3.4,2.4)node{$1$};
  \draw (3,2)node{$\bullet$};
  \draw (2,0)node{$\bullet$};
 \draw (1.7,2.2)node{$\bullet$};
  \draw (0,2)node{$\bullet$};
\draw [color=blue] (2,0) -- (3,2);
\draw [color=blue] (2,0) -- (0,2);
\draw[color=blue]  (3,2) arc[start angle=0, end angle=180,radius=1.5cm] -- (0,2);
 \end{tikzpicture}
 \end{center}
\end{example}
 
 \begin{definition}
 (Isomorfismo de grafos). Sejam $G = (V,A)$ e $G' = (V',A')$ dois grafos. Diremos que eles são \textit{isomorfos} se existir uma bijeção $f:V \to V'$ tal que para todos $v,w \in V$ temos $\{v,w\} \in A$ se e somente se $\{f(v), f(w)\} \in A'$.
 \end{definition} 
 O problema de decidir quando dois grafos dados são isomorfos ou não é bastante difícil. Mesmo no caso de grafos finitos, esse problema tem uma complexidade algorítmica alta, embora sua classe de complexidade não seja ainda conhecida (veja \cite{graphs}).

\medskip

Trabalharemos com métricas em grafos de Cayley,  que são grafos associados a um grupo através de apresentações desse grupo. Para isso, introduzimos agora  alguns conceitos geométricos para grafos.
\begin{definition}
Seja $G = (V,A)$ é um grafo.
\begin{itemize}
    \item Para $n\in \N \cup \{\infty\}$, um \textit{caminho}\index{caminho em um grafo} de comprimento $n$ em $G$  é uma sequência $v_0, \ldots, v_n$, de  vértices distintos $v_i \in V$, com a propriedade de que $\{v_j , v_{j+1}\} \in A$ vale para todo $j \in \{0,\ldots, n-1\}$. Se $ n < \infty$, dizemos que esse caminho conecta os vértices $ v_0$ e $v_n$.
    \item Um \textit{ciclo} \index{ciclo em um grafo} em $G$ de tamanho $n\geq2$ é um caminho $v_0, \ldots, v_{n-1}$, com a propriedade adicional de que $\{v_{n-1}, v_0\} \in A$. Um ciclo é chamado \textit{simples} se todos os vértices $v_0, \ldots, v_{n-1}$ são distintos.
    \item O grafo $G$ é dito \textit{conexo} se qualquer par de vértices podem ser conectados por um caminho em $G$.
    \item Uma \textit{árvore} \index{arvore@árvore} é um grafo conexo que não contém ciclos simples.
    \item Uma \textit{árvore de extensão} \index{arvore de extensao@árvore de extensão} é o subconjunto de arestas de um grafo que forma uma árvore contendo todos os vértices.
\end{itemize}
\end{definition}

O grafo do Exemplo~\ref{grafo1} é conexo e não contém ciclos, logo é uma árvore. Já o grafo do Exemplo~\ref{grafo2} não é conexo, pois o vértice 4 não pode ser conectado a nenhum dos outros vértices, Além disso, esse último contém um ciclo 1,2,3.

 \section{Grafos de Cayley}
 
Seja $G$ um grupo finitamente gerado com um conjunto de geradores $S$. Suponha que $1 \notin S$ e que $S= S^{-1}$, isto é, se $x \in S$, então $x^{-1} \in S$.

\begin{definition} O \textit{grafo de Cayley}\index{grafo de Cayley}  de $G$ com respeito a $S$, denotado por $\cay (G,S)$, é o grafo $(V,A)$, onde 
\begin{enumerate}[(1)]
\item O conjunto de vértices é $V=G$;
\item O par de vértices $\{g,h\}$ é uma aresta se, e somente se, existe $s \in S$ tal que $h=gs.$
\end{enumerate}
\end{definition}

\begin{example}
 O grafo de Cayley de $\mathbb{Z}$, com respeito aos geradores $\{\pm 1\}$ é uma árvore $2$-regular infinita:
\begin{figure}[!ht]
\begin{center}
    \begin{tikzpicture}[scale=0.9]
  \draw (-3,0)node{$\bullet$}; 
  \draw (-2,0)node{$\bullet$};
  \draw (-1,0)node{$\bullet$};
  \draw (0,0)node{$\bullet$};
  \draw (1,0)node{$\bullet$};
  \draw (2,0)node{$\bullet$};
  \draw (3,0)node{$\bullet$};
  \draw (4,0)node{$\bullet$};
  \draw (-3.1,0.3)node{\fontsize{4}{4}\selectfont{$-3$}}; 
  \draw (-2.1,0.3)node{\fontsize{4}{4}\selectfont{$-2$}};
  \draw (-1.1,0.3)node{\fontsize{4}{4}\selectfont{$-1$}};
  \draw (0,0.3)node{\fontsize{4}{4}\selectfont{$0$}};
  \draw (1,0.3)node{\fontsize{4}{4}\selectfont{$1$}};
  \draw (2,0.3)node{\fontsize{4}{4}\selectfont{$2$}};
  \draw (3,0.3)node{\fontsize{4}{4}\selectfont{$3$}};
  \draw (4,0.3)node{\fontsize{4}{4}\selectfont{$4$}};
  
   \draw [color=blue] (0.5, -0.2)node{\fontsize{7}{7}\selectfont{$+1$}};
   \draw [color=blue] (-0.5, -0.2)node{\fontsize{7}{7}\selectfont{$-1$}};
\draw [color=blue] (-4,0) -- (-3,0);
\draw [color=blue] (-3,0) -- (-2,0);
\draw [color=blue] (-2,0) -- (-1,0);
\draw [color=blue] (-1,0) -- (0,0);
\draw [color=blue] (0,0) -- (1,0);
\draw [color=blue] (1,0) -- (2,0);
\draw [color=blue] (2,0) -- (3,0);
\draw [color=blue] (3,0) -- (4,0);
\draw [color=blue] (4,0) -- (5,0);
\end{tikzpicture}
\end{center}
 \caption{$\cay (\mathbb{Z},\{\pm 1\})$}
    \label{cayZpm1}
\end{figure}
\end{example}

\begin{example}
Alterar o conjunto de geradores altera também os grafos de  Cayley. Por exemplo, $\cay (\Z,\{\pm 2, \pm 3\})$, mostrado na Figura~\ref{cayZ23}, não é isomorfo ao grafo da Figura~\ref{cayZpm1}:
\end{example}
\begin{figure}[H]
\begin{center}
    \begin{tikzpicture}[scale=0.7]
    \draw (-4,0)node{$\ldots$}; 
    \draw (5,0)node{$\ldots$}; 
  \draw (-3,0)node{$\bullet$}; 
  \draw (-2,0)node{$\bullet$};
  \draw (-1,0)node{$\bullet$};
  \draw (0,0)node{$\bullet$};
  \draw (1,0)node{$\bullet$};
  \draw (2,0)node{$\bullet$};
  \draw (3,0)node{$\bullet$};
  \draw (4,0)node{$\bullet$};
    \draw (-0.3,0)node{\fontsize{4}{4}\selectfont{$0$}};
   \draw  [color=blue](1,1.2)node{\fontsize{4}{4}\selectfont{$+2$}};
   \draw  [color=red](-1.5,-1.2)node{\fontsize{4}{4}\selectfont{$-3$}};
   \draw  [color=red](1.5,-1.2)node{\fontsize{4}{4}\selectfont{$+3$}};
  \draw  [color=blue](-1,1.2)node{\fontsize{4}{4}\selectfont{$-2$}};
\draw [color=blue] (-2,0) arc[start angle=0, end angle=90,radius=1cm] -- (-3,1);
\draw [color=blue] (0,0) arc[start angle=0, end angle=180,radius=1cm] -- (-2,0);
\draw [color=blue] (2,0) arc[start angle=0, end angle=180,radius=1cm] -- (0,0);
\draw[color=blue]  (4,0) arc[start angle=0, end angle=180,radius=1cm] -- (2,0);
\draw [color=blue]  (-1,0) arc[start angle=0, end angle=180,radius=1cm] -- (-3,0);
\draw [color=blue] (1,0) arc[start angle=0, end angle=180,radius=1cm] -- (-1,0);
\draw [color=blue] (3,0) arc[start angle=0, end angle=180,radius=1cm] -- (1,0);
\draw [color=blue]  (4,1) arc[start angle=90, end angle=180,radius=1cm] -- (3,0);
\draw [color=red] (-3,0) arc[start angle=180, end angle=360, x radius=1.5cm, y radius =1cm] -- (0,0);
\draw [color=red](0,0) arc[start angle=180, end angle=360, x radius=1.5cm, y radius =1cm] -- (3,0);
\draw [color=red](3,0) arc[start angle=180, end angle=270, x radius=1.5cm, y radius =1cm] -- (4.5,-1);
\draw [color=red] (-2,0) arc[start angle=180, end angle=360, x radius=1.5cm, y radius =1cm] -- (1,0);
\draw [color=red](-1,0) arc[start angle=180, end angle=360, x radius=1.5cm, y radius =1cm] -- (2,0);
\draw [color=red] (1,0) arc[start angle=180, end angle=360, x radius=1.5cm, y radius =1cm] -- (4,0);
\draw [color=red](-3.5,-1) arc[start angle=270, end angle=360, x radius=1.5cm, y radius =1cm] -- (-2,0);
\draw [color=red](-3,-0.94) arc[start angle=252.6, end angle=360, x radius=1.5cm, y radius =1cm] -- (-1,0);
\draw [color=red](2,0) arc[start angle=180, end angle=287.4, x radius=1.5cm, y radius =1cm] -- (4,-0.94);
\end{tikzpicture}
\end{center}
    \caption{$\cay (\Z,\{\pm 2, \pm 3\})$}
    \label{cayZ23}
\end{figure}

\begin{example}
 O grafo de Cayley de $\Z^2$, com respeito ao conjunto de geradores $S=\{\pm a = (\pm 1,0), \pm b = (0,\pm 1)\}$ está na Figura~\ref{cayZ2}.
 \end{example}
\begin{figure}[ht]
    \centering
    \begin{center}
    \begin{tikzpicture}[scale=0.7]
      \draw[step=1.0,blue,thin,xshift=0cm,yshift=0cm] (-2.5,-1.5) grid (2.5,2.5);
 \draw (-2,0)node{$\bullet$};
  \draw (-2,1)node{$\bullet$};
  \draw (-2,2)node{$\bullet$}; 
  \draw (-2,-1)node{$\bullet$}; 
   \draw (-1,0)node{$\bullet$};
  \draw (-1,1)node{$\bullet$};
  \draw (-1,2)node{$\bullet$}; 
  \draw (-1,-1)node{$\bullet$}; 
  
   \draw (0,0)node{$\bullet$};
  \draw (0,1)node{$\bullet$};
  \draw (0,2)node{$\bullet$}; 
  \draw (0,-1)node{$\bullet$}; 
  
  \draw (1,0)node{$\bullet$};
  \draw (1,1)node{$\bullet$};
  \draw (1,2)node{$\bullet$}; 
  \draw (1,-1)node{$\bullet$}; 

  \draw (2,0)node{$\bullet$};
  \draw (2,1)node{$\bullet$};
  \draw (2,2)node{$\bullet$}; 
  \draw (2,-1)node{$\bullet$}; 
  
  \draw [color=blue] (0.6,0.15)node{\fontsize{6}{6}\selectfont{$a$}};
   \draw [color=blue] (0.15,0.6)node{\fontsize{6}{6}\selectfont{$b$}};
   \draw [color=blue] (0.2,-0.6)node{\fontsize{5}{5}\selectfont{$-b$}};
   \draw [color=blue] (-0.6,0.15)node{\fontsize{6}{6}\selectfont{$-a$}};
  


  \draw (-0.2,-0.2)node{\fontsize{4}{4}\selectfont{$0$}};
  

    \end{tikzpicture}
\end{center}
\caption{$\cay ( \Z^2, \{(\pm 1,0), (0,\pm 1)\})$}
\label{cayZ2}
\end{figure}

\begin{example}
 $\cay (F_2, \{a^{\pm1}, b^{\pm1}\})$ é uma árvore  $4$-regular:
 \end{example}
\begin{figure}[H]
   \begin{center}
   \begin{tikzpicture}[scale=0.5]
   \begin{scope}[rotate=-90] \caley{4}{4cm}{b-1}{a}{b}{a-1} \end{scope}
   \begin{scope}[rotate=0]   \caley{4}{4cm}{a}{b}{a-1}{b-1} \end{scope}
   \begin{scope}[rotate=90]  \caley{4}{4cm}{b}{a-1}{b-1}{a} \end{scope}
   \begin{scope}[rotate=180] \caley{4}{4cm}{a-1}{b-1}{a}{b} \end{scope}
   \draw [color=blue] (2,0.2)node{\fontsize{8}{8}\selectfont{$a$}};
   \draw [color=blue] (-2,0.3)node{\fontsize{8}{8}\selectfont{$a^{-1}$}};
   \draw [color=blue] (0.2,2)node{\fontsize{8}{8}\selectfont{$b$}};
   \draw [color=blue] (0.55,-2)node{\fontsize{7}{7}\selectfont{$b^{-1}$}};
   \draw (0.2,0.25)node{\fontsize{6}{6}\selectfont{$1$}};
   \draw (0,0)node{\fontsize{7}{7}\selectfont{$\bullet$}};
\end{tikzpicture}
\end{center}
\caption{$\cay (F_2, \{a^{\pm1}, b^{\pm1}\})$}
\label{cayF2}
\end{figure}

O seguinte teorema será utilizado na prova do Teorema~\ref{caracterizaçãogplivre}, que fornece uma caracterização de grupos livres. 

\begin{thm}
\label{serretrees}
Seja $\Gamma = \cay (G,S)$ o grafo de Cayley de um grupo $G$ com respeito ao conjunto de geradores $S$. São equivalentes: 
\begin{enumerate}
    \item O grafo $\Gamma$ é uma árvore;
    \item Supondo que $S$ satisfaz que $s \neq t^{-1}$ para todos $s, t\in S$, o grupo $G$ é livremente gerado por $S$.
\end{enumerate}
\end{thm}

\begin{exercise}
    Dar uma prova do Teorema \ref{serretrees} (cf. \cite[Proposição 15]{Serre}).
\end{exercise}

O fato de que $S$ gera $G$ garante que $\cay (G,S)$ é conexo, já que todo vértice $g$ pode ser conectado ao vértice $1$ pelo caminho que passa pelas arestas correspondentes à escrita mínima de $g$ em termos dos geradores. Um problema em teoria dos grafos é saber se um dado grafo $\Gamma = (V,A)$ é $k$-{\it colorível}, isto é,  se podemos atribuir $k$ cores para colorir os vértices $V$ de modo que vértices adjacentes possuam cores diferentes. Note que, como $\cay (G,S)$ é $|S|$-regular, ou seja, de cada vértice do grafo de Cayley saem no máximo $|S|$ arestas, podemos sempre colorir $\cay (G,S)$ com $|S|+1$ cores.

Podemos colocar uma métrica em $\cay (G,S)$,  supondo que cada aresta seja isométrica ao intervalo $[0,1]$.  Assim, temos uma maneira natural de definir o comprimento de um caminho $p$ em $\cay (G,S)$, onde esse comprimento é $\ell_S(p) = n$ se ele é obtido da concatenação de $n$ arestas. Definimos assim, a distância $\dist_S (x, y)$ entre dois vértices $x, y \in \cay (G,S)$ como o ínfimo dos comprimentos de caminhos ligando $x$ a $y$ e estendemos de forma natural essa definição a pontos das arestas usando o fato de que cada aresta é isométrica ao intervalo $[0,1]$. 
\begin{remark}

O espaço métrico  $\cay (G,S)$ é próprio, isto é, bolas fechadas são compactas. 
Com efeito, como todo subconjunto fechado de um compacto é compacto, basta considerar bolas fechadas com raio $n\in \N$. Observe que em cada bola fechada centrada em $1 \in G$ de raio $n$ existem, no máximo, $\displaystyle \sum_{j=1}^{n} |S|(|S|- 1)^{j-1}$ arestas. Assim, cada bola de raio $n$ é a união de uma quantidade finita de arestas. O resultado segue, já que cada aresta é compacta.
\end{remark}

\begin{definition}
A restrição desta métrica a  $G$ é chamada de {\it métrica das palavras}\index{métrica das palavras}. Denotamos, para cada $g \in G$, $\mathrm{dist}_S(1,g) = |g|_S$, e assim, $\mathrm{dist}_S(g,h) = |g^{-1}h|_S = |h^{-1}g|_S$.
\end{definition}

O grupo $G$ age isometricamente em $\cay (G,S).$ Tal ação estende a ação isométrica de $G$ em si mesmo por translação à esquerda $L_g(h) = gh$.  
Mais formalmente, a fim de definir uma ação por isometrias de $G$ em $\cay  (G, S)$ temos que atribuir a cada $g \in G$ uma isometria $\varphi_g$. Tal isometria pode ser escrita como segue: seja $x \in \cay (G,S)$. Caso $x \in G$, definimos $\varphi_g (x)$ como $L_g(x) = gx$. Caso $x \notin G$, isto é, $x$ está no interior de alguma aresta $\{h_1,h_2\} $, $\varphi_g (x)$ será definido como o único ponto na aresta $\{gh_1,gh_2\}$ que  satisfaz $\dist_S (gh_1, \varphi_g (x)) = \dist_S (h_1, x)$.

\begin{exercise}
\label{açãolivre} Prove que a ação de $G$ em $\cay (G, S)$ é livre se, e somente se, $S$ não possui elementos de ordem 2. Além disso, a ação é própria, discreta e cocompacta.
\end{exercise}

\begin{example}
    A ação de $\Z_4$ em $\cay(\Z_4, \{\Bar{1}, \Bar{3}\})$ é livre, enquanto sua ação em $\cay(\Z_4, \{\Bar{1},\Bar{2}, \Bar{3}\})$ não o é.
\end{example}

\begin{exercise}
Mostre que, supondo que $1 \notin S$ e que $S= S^{-1}$ e $S$ não possui elementos de ordem $2$, então o quociente $\cay(G,S)/G$ é homeomorfo ao buquê de $n$ círculos, onde $n= \frac{|S|}{2}$.
\end{exercise}

\begin{remark} 
A ação de $G$ em $\cay (G, S)$ pela direita, $R_g(h) = hg$  não é necessariamente isométrica, já que podemos ter 
$$ \{x,xs\} \stackrel{R_{x^{-1}}}{\longrightarrow} \{xx^{-1}, xsx^{-1}\} = \{1, xsx^{-1}\} \notin A.$$
No entanto, ainda vale a seguinte propriedade: 
$$ \dist_S(\id(h), R_g(h)) = |g|_S.$$
\end{remark}

 O seguinte resultado é uma caracterização clássica de grupos livres em termos de ações desse grupo.
\begin{thm}
\label{caracterizaçãogplivre}
Um grupo $G$ finitamente gerado é livre se, e somente se, $G$ age  por automorfismos livremente numa árvore não vazia.
\end{thm}

\noindent
\textit{Demonstração do Teorema~\ref{caracterizaçãogplivre} -- Parte 1:}
Caso $G$ seja um grupo livre finitamente gerado, seu grafo de Cayley será uma árvore, e podemos portanto considerar a ação de $G$ à esquerda no grafo de Cayley. Observe que, neste caso, $S$ não pode possuir elementos de ordem $2$. Usando o  Exercício~\ref{açãolivre}, concluímos que essa ação é livre.\hfill \qedsymbol

Para demonstrar a recíproca do Teorema~\ref{caracterizaçãogplivre}, vamos mostrar que, supondo que existe uma ação livre de $G$ em uma árvore não vazia $T$, a partir de $T$ e dessa ação, podemos construir, contraindo certas subárvores, uma árvore que é o grafo de Cayley de $G$ com respeito a um conjunto de geradores adequado. Pelo Teorema~\ref{serretrees}, seguirá que $G$ é livre. As subárvores que serão contraídas são conhecidas como ``\textit{árvores de extensão}'', 
e serão discutidas a seguir.

\begin{definition}
Considere a ação de um grupo $G$ em um grafo conexo $X$ via automorfismos de grafos. Uma \textit{árvore de extensão para essa ação} \index{arvore de extensao@árvore de extensão} é uma subárvore de $X$ que contém exatamente um vértice de cada órbita da ação induzida de $G$ no conjunto de vértices de $X$.
\end{definition}

\begin{remark}
Um subgrafo  de um grafo $(V,A)$ é um grafo $(V',A')$ com  $V'\subset V$ e $A'\subset A$. Um subgrafo é uma subárvore se, além disso, for uma árvore.
\end{remark}

\begin{example}
\label{spanningtree}
Seja $\Gamma$ é o grafo representado pela Figura~\ref{spannintreeexample}, e consideramos a ação de $\Z$ em $\Gamma$ por translação horizontal de vértices e arestas, então o conjunto destacado é um exemplo de  árvore de extensão para a ação de $\Z$ em $\Gamma$.
\end{example}
\begin{figure}[!ht]
\begin{center}
    \begin{tikzpicture}[scale=0.6]
  \draw (-3,0)node{$\bullet$}; 
  \draw [color=red] (-2,0)node{$\bullet$};
  \draw (-1,0)node{$\bullet$};
  \draw (0,0)node{$\bullet$};
  \draw (1,0)node{$\bullet$};
  \draw (2,0)node{$\bullet$};
  \draw (3,0)node{$\bullet$};
  \draw (4,0)node{$\bullet$};
  
  \draw (-3,1)node{$\bullet$}; 
      \draw  [color=red](-2,1)node{$\bullet$};
  \draw (-1,1)node{$\bullet$};
  \draw (0,1)node{$\bullet$};
  \draw (1,1)node{$\bullet$};
  \draw (2,1)node{$\bullet$};
  \draw (3,1)node{$\bullet$};
  \draw (4,1)node{$\bullet$};
  
  \draw (-3,-1)node{$\bullet$}; 
  \draw  [color=red] (-2,-1)node{$\bullet$};
  \draw (-1,-1)node{$\bullet$};
  \draw (0,-1)node{$\bullet$};
  \draw (1,-1)node{$\bullet$};
  \draw (2,-1)node{$\bullet$};
  \draw (3,-1)node{$\bullet$};
  \draw (4,-1)node{$\bullet$};

\draw[color=gray, opacity=0.7]   (-4,0) -- (-3,0);
\draw[color=gray, opacity=0.7]   (-3,0) -- (-2,0);
\draw[color=gray, opacity=0.7]   (-2,0) -- (-1,0);
\draw[color=gray, opacity=0.7]   (-1,0) -- (0,0);
\draw[color=gray, opacity=0.7]   (0,0) -- (1,0);
\draw[color=gray, opacity=0.7]   (1,0) -- (2,0);
\draw[color=gray, opacity=0.7]   (2,0) -- (3,0);
\draw[color=gray, opacity=0.7]   (3,0) -- (4,0);
\draw[color=gray, opacity=0.7]   (4,0) -- (5,0);

\draw[color=gray, opacity=0.7]   (-3,-1) -- (-3,1);
\draw   [color=red, thick](-2,-1) -- (-2,1);
\draw[color=gray, opacity=0.7]  (-1,-1) -- (-1,1);
\draw[color=gray, opacity=0.7]  (0,-1) -- (0,1);
\draw[color=gray, opacity=0.7]  (1,-1) -- (1,1);
\draw[color=gray, opacity=0.7]  (2,-1) -- (2,1);
\draw[color=gray, opacity=0.7]  (3,-1) -- (3,1);
\draw[color=gray, opacity=0.7]  (4,-1) -- (4,1);
 \draw [color=red] (-2,0)node{$\bullet$};
\end{tikzpicture}
\end{center}
 \caption{Árvore de extensão}
    \label{spannintreeexample}
\end{figure}

\begin{thm} \label{spanningtreethm}
Toda ação, por automorfismos de grafos, de um grupo $G$ em um grafo conexo $X\neq \emptyset$ admite uma árvore de extensão.
\end{thm}

\begin{proof}
Considere o conjunto $T_G$ de todas as subárvores de $X$ que contém no máximo um vértice de cada $G$-órbita. Mostraremos que $T_G$ possui um elemento $T$ que é maximal com respeito à relação de subárvores. O conjunto $T_G$ é não vazio, pois qualquer árvore com um único vértice de $X$ é um elemento de $T_G$. Além disso, $T_G$ é parcialmente ordenado pela relação de subgrafos, e qualquer cadeia totalmente ordenada de $T_G$ possui uma cota superior em $T_G$ (a saber, a união de todas as árvores nessa cadeia). Pelo Lema de Zorn, existe então algum elemento maximal $T$ em $T_G$. Como $X$ é não vazio, $T$ também é não vazia. 

Assuma, por contradição, que $T$ não é uma  árvore de extensão para a ação de $G$ em $X$. Então deve existir algum vértice $v$ de $X$ para o qual nenhum dos vértices na órbita $Gv$ está em $T$. Mostraremos agora que existe algum vértice $v'$ com essa propriedade e, mais ainda, tal que algum vértice adjacente a $v'$ é um vértice de $T$.

Como $X$ é conexo, existe um caminho $p$ ligando algum vértice $u \in T$ a $v$. Seja $v'$ o primeiro vértice de $p$ que não está em $T$. Podem ocorrer dois casos:
\begin{enumerate}
    \item Nenhum dos vértices da órbita de $v'$ está em $T$. Então $v'$ terá a propriedade desejada.
    \item Existe $g\in G$ tal que $gv' \in T$. Neste caso, se $p'$ denota o subcaminho de $p$ começando em $v'$ e terminando em $v$, então $gp'$ é um caminho começando em $gv'$, que está em $T$, e terminando em $gv$, o qual é um vértice tal que nenhum dos vértices em $G(gv)  = Gv$ está em $T$. Note que o caminho $p'$ é estritamente mais curto que $p$, de forma que  podemos repetir o processo finitas vezes e obter um vértice com a propriedade desejada.
\end{enumerate}
Seja $v$  um vértice de $X$ para o qual nenhum dos vértices na órbita $Gv$ está em $T$, mas algum vértice $u$ de $T$ é adjacente a $v'$. Com isso, podemos adicionar o vértice $v'$ e a aresta $\{u,v'\}$ à arvore $T$, obtendo um elemento de $T_G$ que contém propriamente $T$ como subgrafo. Isso contradiz a maximalidade de $T$. Portanto, $T$ deve ser uma árvore de extensão para a ação de $G$ em $X$.
\end{proof}

Vamos agora finalizar a demonstração do Teorema~\ref{caracterizaçãogplivre}. A  ideia da prova é usar o grafo obtido de $T$ após contrair cada cópia $gT_0$, com $g\in G$, de uma árvore de extensão $T_0$ para essa ação a um único vértice (observe que as árvores são contratíveis). Os  geradores livres de $G$ virão das arestas em $T_0$ que ligam esses novos vértices.

Um exemplo do que será feito, pensando no caso do Exemplo~\ref{spannintreeexample}, está na Figura~\ref{spannintreeexample2}. Vamos trabalhar com as arestas que ligam os novos vértices após todas as contrações.

\begin{figure}[!ht]
\begin{center}
 \begin{tikzpicture}[scale=0.6]
  \draw (-3,0)node{$\bullet$}; 
  \draw [color=red] (-2,0)node{$\bullet$};
  \draw (-1,0)node{$\bullet$};
  \draw (0,0)node{$\bullet$};
  \draw (1,0)node{$\bullet$};
  \draw [color=red](2,0)node{$\bullet$};
  \draw (3,0)node{$\bullet$};
  \draw (4,0)node{$\bullet$};
  
  \draw (-3,1)node{$\bullet$}; 
  \draw  [color=red](-2,1)node{$\bullet$};
  \draw (-1,1)node{$\bullet$};
  \draw (0,1)node{$\bullet$};
  \draw (1,1)node{$\bullet$};
  \draw [color=red](2,1)node{$\bullet$};
  \draw (3,1)node{$\bullet$};
  \draw (4,1)node{$\bullet$};
  \draw  [color=red](-1.9,-1.5)node{$T_0$};
  \draw  [color=red](2,-1.5)node{$gT_0$};
  
  \draw (-3,-1)node{$\bullet$}; 
  \draw  [color=red] (-2,-1)node{$\bullet$};
  \draw (-1,-1)node{$\bullet$};
  \draw (0,-1)node{$\bullet$};
  \draw (1,-1)node{$\bullet$};
  \draw [color=red](2,-1)node{$\bullet$};
  \draw (3,-1)node{$\bullet$};
  \draw (4,-1)node{$\bullet$};

\draw  (-4,0) -- (-3,0);
\draw  (-3,0) -- (-2,0);
\draw  (-2,0) -- (-1,0);
\draw  (-1,0) -- (0,0);
\draw  (0,0) -- (1,0);
\draw  (1,0) -- (2,0);
\draw  (2,0) -- (3,0);
\draw  (3,0) -- (4,0);
\draw  (4,0) -- (5,0);

\draw  (-3,-1) -- (-3,1);
\draw   [color=red](-2,-1) -- (-2,1);
\draw  (-1,-1) -- (-1,1);
\draw  (0,-1) -- (0,1);
\draw  (1,-1) -- (1,1);
\draw  [color=red](2,-1) -- (2,1);
\draw  (3,-1) -- (3,1);
\draw  (4,-1) -- (4,1);
 \draw [color=red] (-2,0)node{$\bullet$};
 \draw [color=red](2,0)node{$\bullet$};
 
 \draw [->] (0,-1.5) to (0,-2.5);
\end{tikzpicture}

\begin{tikzpicture}[scale=0.6]
  \draw (-3,0)node{$\bullet$}; 
  \draw[color=lightgray](-2,0)node{$\bullet$};
  \draw (-1,0)node{$\bullet$};
  \draw (0,0)node{$\bullet$};
  \draw (1,0)node{$\bullet$};
  \draw [color=lightgray](2,0)node{$\bullet$};
  \draw (3,0)node{$\bullet$};
  \draw (4,0)node{$\bullet$};
  
  \draw (-3,1)node{$\bullet$}; 
  \draw  [color=lightgray](-2,1)node{$\bullet$};
  \draw (-1,1)node{$\bullet$};
  \draw (0,1)node{$\bullet$};
  \draw (1,1)node{$\bullet$};
  \draw [color=lightgray](2,1)node{$\bullet$};
  \draw (3,1)node{$\bullet$};
  \draw (4,1)node{$\bullet$};
  
  \draw (-3,-1)node{$\bullet$}; 
  \draw  [color=lightgray] (-2,-1)node{$\bullet$};
  \draw (-1,-1)node{$\bullet$};
  \draw (0,-1)node{$\bullet$};
  \draw (1,-1)node{$\bullet$};
  \draw[color=lightgray](2,-1)node{$\bullet$};
  \draw (3,-1)node{$\bullet$};
  \draw (4,-1)node{$\bullet$};

\draw  (-4,0) -- (-3,0);
\draw  (-3,0) -- (-2,0);
\draw  (-2,0) -- (-1,0);
\draw  (-1,0) -- (0,0);
\draw  (0,0) -- (1,0);
\draw  (1,0) -- (2,0);
\draw  (2,0) -- (3,0);
\draw  (3,0) -- (4,0);
\draw  (4,0) -- (5,0);

\draw  (-3,-1) -- (-3,1);
\draw   [color=lightgray](-2,-1) -- (-2,1);
\draw  (-1,-1) -- (-1,1);
\draw  (0,-1) -- (0,1);
\draw  (1,-1) -- (1,1);
\draw  [color=lightgray] (2,-1) -- (2,1);
\draw  (3,-1) -- (3,1);
\draw  (4,-1) -- (4,1);
 \draw [color=red] (-2,0)node{$\bullet$};
 \draw [color=red](2,0)node{$\bullet$};
\draw (2.6,-1.8)node{\tiny{\mbox{Contraindo todas as}}};
\draw (2.6,-2.3 )node{\tiny{\mbox{cópias de} $T_0$}};

 \draw [->] (0,-1.5) to (0,-2.5);
\end{tikzpicture}

    \begin{tikzpicture}[scale=0.6]

  \draw [color=lightgray](-3,1)node{$\bullet$}; 
  \draw  [color=lightgray](-2,1)node{$\bullet$};
  \draw [color=lightgray](-1,1)node{$\bullet$};
  \draw [color=lightgray](0,1)node{$\bullet$};
  \draw [color=lightgray](1,1)node{$\bullet$};
  \draw [color=lightgray](2,1)node{$\bullet$};
  \draw [color=lightgray](3,1)node{$\bullet$};
  \draw [color=lightgray](4,1)node{$\bullet$};
  
  \draw [color=lightgray](-3,-1)node{$\bullet$}; 
  \draw  [color=lightgray] (-2,-1)node{$\bullet$};
  \draw [color=lightgray](-1,-1)node{$\bullet$};
  \draw [color=lightgray](0,-1)node{$\bullet$};
  \draw [color=lightgray](1,-1)node{$\bullet$};
  \draw [color=lightgray](2,-1)node{$\bullet$};
  \draw [color=lightgray](3,-1)node{$\bullet$};
  \draw [color=lightgray] (4,-1)node{$\bullet$};

\draw  (-4,0) -- (-3,0);
\draw  (-3,0) -- (-2,0);
\draw  (-2,0) -- (-1,0);
\draw  (-1,0) -- (0,0);
\draw  (0,0) -- (1,0);
\draw  (1,0) -- (2,0);
\draw  (2,0) -- (3,0);
\draw  (3,0) -- (4,0);
\draw  (4,0) -- (5,0);

\draw  [color=lightgray](-3,-1) -- (-3,1);
\draw  [color=lightgray](-2,-1) -- (-2,1);
\draw  [color=lightgray](-1,-1) -- (-1,1);
\draw  [color=lightgray](0,-1) -- (0,1);
\draw  [color=lightgray](1,-1) -- (1,1);
\draw [color=lightgray] (2,-1) -- (2,1);
\draw  [color=lightgray](3,-1) -- (3,1);
\draw  [color=lightgray](4,-1) -- (4,1);

  \draw [color=red] (-3,0)node{$\bullet$}; 
  \draw [color=red] (-2,0)node{$\bullet$};
  \draw [color=red] (-1,0)node{$\bullet$};
  \draw [color=red] (0,0)node{$\bullet$};
  \draw [color=red] (1,0)node{$\bullet$};
  \draw [color=red] (2,0)node{$\bullet$};
  \draw [color=red] (3,0)node{$\bullet$};
  \draw [color=red] (4,0)node{$\bullet$};
\end{tikzpicture}
\end{center}
 \caption{Contração das cópias de uma árvore de extensão.}
    \label{spannintreeexample2}
\end{figure}
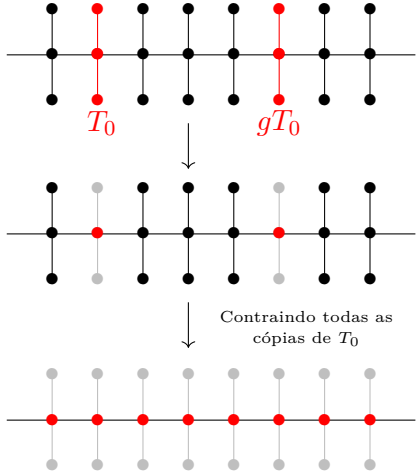

\noindent
\textit{Demonstração do Teorema~\ref{caracterizaçãogplivre} -- Parte 2:} Suponhamos agora que $G$ age por automorfismos livremente em uma árvore  $T$ não vazia. Pelo Teorema~\ref{spanningtreethm}, existe uma árvore de extensão $T_0$ para essa ação. Fazendo a contração de todas as cópias de $T_0$ pela ação de $G$ como no exemplo da Figura~\ref{spannintreeexample2}, diremos que uma aresta de $T$ é essencial se ela não estiver contida em $T_0$, mas algum dos seus vértices estiver em $T_0$ (portanto o outro vértice não estará em $T_0$, pela unicidade de caminhos ligando vértices em árvores).


Seja $a=\{u,v\}$ uma aresta essencial de $T$, com $u \in T_0$ e $v\notin T_0$. Como $T_0$ é uma árvore de extensão, existe um elemento $g_a \in G$ tal que $g_a^{-1}\cdot v \in T_0$.  Equivalentemente, $v \in g_a\cdot T_0$. O elemento $g_a$ é unicamente determinado por essa propriedade, dado que a órbita $G\cdot v$ intersecta $T_0$ em um único ponto e $G$ age livremente em $T$.

Definimos agora $$\Tilde{S} := \{g_a \in G \mid a \mbox{ é uma aresta essencial de } T\}.$$
O conjunto $\tilde{S}$ possui as seguintes propriedades:

\begin{enumerate}[(1)]
    \item Por definição, $1 \in G$ não está em $\tilde{S}$.
    \item  $\tilde{S}$ não contém elementos de ordem 2, pois $G$ age livremente na árvore não vazia $T$ (vide Exercício~\ref{açãolivre}).
    \item Se $a$ e $a'$ são arestas essenciais com $g_a = g_{a'}$, então $a = a'$ (como $T$ é uma árvore, não podem existir diferentes arestas conectando os subgrafos $T_0$ e $g_a \cdot T_0 = g_{a'}\cdot T_0$).
    \item Se $ g = g_a$ para uma aresta essencial $a$, então $g^{-1} = g_{g^{-1}\cdot a} \in \tilde{S}$, já que $g^{-1}\cdot a$ também é uma aresta essencial de $T$.
\end{enumerate}

Em particular, existe $S\subset \tilde{S}$ tal que $ S\cap S^{-1} = \emptyset $ e $$|S| = \frac{|\tilde{S}|}{2} =  \frac{\mbox{número de arestas essenciais de T}}{2}.$$

Vejamos que $\tilde{S}$, e portanto $S$, gera $G$: sejam dados $g\in G$ e um vértice $v\in T_0$. Como $T$ é conexa, existe algum caminho $p$ em $T$ ligando $v$ a $g\cdot v$. Esse caminho passa por algumas cópias de $T_0$, digamos $g_0\cdot T_0$, \ldots, $g_n\cdot T_0$, nessa ordem, com $g_0= 1$, $g_n=g$ e $g_{j+1} \neq g_j$ para todo $j=0, 1, \ldots, n-1$.

Seja $j \in \{0, 1, \ldots, n-1\}$. Como $T_0$ é árvore de extensão e $g_j \neq g_{j+1}$, as cópias $g_j\cdot T_0$ e $g_{j+1}\cdot T_0$ são conectadas por uma aresta $a_j$. Por definição, $g_j^{-1}\cdot a_j$ é uma aresta essencial, e o elemento 
$s_j := g^{-1}_j  g_{j+1}$ está em $\tilde{S}$. Assim, obtemos que
$$g = g_n = g^{-1}_0  g_n = g^{-1}_0 g_1 g_1^{-1}  g_2 \ldots g^{-1}_{n-1}g_n
= s_0 \ldots s_{n-1}$$
está no subgrupo de $G$ gerado por $\tilde{S}$.

\begin{figure}[!ht]
\begin{center}
 \begin{tikzpicture}[scale=0.6]
  \draw [color=lightgray] (-3,0)node{$\bullet$}; 
  \draw [color=red] (-2,0)node{$\bullet$};
  \draw [color=lightgray](-1,0)node{$\bullet$};
  \draw [color=red](0,0)node{$\bullet$};
  \draw [color=red](1,0)node{$\bullet$};
  \draw [color=lightgray](2,0)node{$\bullet$};
  \draw [color=red](3,0)node{$\bullet$};
  \draw [color=lightgray](4,0)node{$\bullet$};
  
\draw [color=red] (-2.3,0.2)node{$v$};
  
  \draw [color=lightgray](-3,1)node{$\bullet$}; 
  \draw  [color=red](-2,1)node{$\bullet$};
  \draw [color=lightgray](-1,1)node{$\bullet$};
  \draw [color=lightgray](0,1)node{$\bullet$};
  \draw [color=lightgray](1,1)node{$\bullet$};
  \draw [color=lightgray](2,1)node{$\bullet$};
  \draw [color=red](3,1)node{$\bullet$};
  \draw [color=lightgray](4,1)node{$\bullet$};
  \draw  [color=red](-1.9,-1.5)node{\tiny$T_0$};
  \draw  [color=red](3,-1.5)node{\tiny$gT_0$};
  \draw  [color=lightgray](0,-1.5)node{\tiny$g_jT_0$};
  \draw  [color=lightgray](1,1.5)node{\tiny$g_{j+1}T_0$};
   \draw [color=red](0.5,0.3)node{\tiny $a_j$};

  \draw [color=red] (3.5,0.3)node{\tiny $g\cdot v$};
  
  \draw [color=lightgray](-3,-1)node{$\bullet$}; 
  \draw  [color=red] (-2,-1)node{$\bullet$};
  \draw [color=lightgray](-1,-1)node{$\bullet$};
  \draw [color=lightgray](0,-1)node{$\bullet$};
  \draw [color=lightgray](1,-1)node{$\bullet$};
  \draw [color=lightgray](2,-1)node{$\bullet$};
  \draw [color=red](3,-1)node{$\bullet$};
  \draw [color=lightgray](4,-1)node{$\bullet$};

\draw  [color=lightgray](-4,0) -- (-3,0);
\draw  [color=lightgray](-3,0) -- (-2,0);
\draw  [color=lightgray](-2,0) -- (-1,0);
\draw  [color=lightgray](-1,0) -- (0,0);
\draw  [color=red, thick](0,0) -- (1,0);
\draw  [color=lightgray](1,0) -- (2,0);
\draw  [color=lightgray] (2,0) -- (3,0);
\draw [color=lightgray] (3,0) -- (4,0);
\draw [color=lightgray] (4,0) -- (5,0);
\draw [color=red,style={dashed,thick}] (-2,0) -- (3,0);

\draw  [color=lightgray](-3,-1) -- (-3,1);
\draw   [color=red](-2,-1) -- (-2,1);
\draw  [color=lightgray](-1,-1) -- (-1,1);
\draw [color=lightgray] (0,-1) -- (0,1);
\draw  [color=lightgray](1,-1) -- (1,1);
\draw [color=lightgray] (2,-1) -- (2,1);
\draw  [color=red] (3,-1) -- (3,1);
\draw  [color=lightgray](4,-1) -- (4,1);
 \draw [color=red] (-2,0)node{$\bullet$};
 \draw [color=red](3,0)node{$\bullet$};
 
\end{tikzpicture}
 \caption{$\tilde{S}$ gera $G$.}
    \label{spannintreeexample3}
    \end{center}
\end{figure}

Mais ainda, vemos que o grafo obtido ao colapsar cada transladado de $T_0$ em $T$ é o grafo de Cayley de $G$ com respeito a $\tilde{S}$.

Por fim, afirmamos que o conjunto $S$ gera $G$ livremente. Pelo Teorema~\ref{serretrees}, é suficiente mostrar que o grafo de Cayley de $G$ com respeito a $S$ não possui ciclos. Suponha que existe algum $n\in \N_{\geq 3}$ e um ciclo $ g_0,\ldots ,g_{n-1}$ em $\cay(G, S) = \cay(G, \tilde{S})$. Por  definição, para todo $j\in\{0,\ldots,n-1\}$, os elementos $s_{j}= g_j^{-1}  g_{j+1}$ e $s_n = g^{-1}_{n-1}g_0$ estão em $\tilde{S}$. Para cada $j =1\ldots n$, seja $a_j$ uma aresta essencial ligando $T_0$ e $s_j\cdot T_0$. 

Como cada um dos transladados de $T_0$ é um subgrafo conexo, podemos conectar os vértices das arestas $g_ja_j$ e $g_js_ja_{j+1}= g_{j+1}a_{j+1}$ que estão em $g_{j+1}\cdot T_0$  por um caminho em $g_{j+1}\cdot T_0$.

\begin{figure}[!ht]
\begin{center}
 \begin{tikzpicture}[scale=0.8]
  \draw [color=lightgray] (-3,0)node{$\bullet$}; 
  \draw [color=lightgray] (-3,1)node{$\bullet$}; 
  \draw [color=lightgray] (-3,-2)node{$\bullet$}; 
  \draw [color=lightgray]
  (-3,-1)node{$\bullet$}; 
  \draw [color=lightgray]
  (-3,2)node{$\bullet$}; 
    
  \draw [color=lightgray] (-2,0)node{$\bullet$}; 
  \draw [color=lightgray] (-2,1)node{$\bullet$}; 
  \draw [color=lightgray] (-2,-2)node{$\bullet$}; 
  \draw [color=lightgray]
  (-2,-1)node{$\bullet$}; 
  \draw [color=lightgray]
  (-2,2)node{$\bullet$}; 
  
    \draw [color=lightgray] (-1,0)node{$\bullet$}; 
  \draw [color=lightgray] (-1,1)node{$\bullet$}; 
  \draw [color=lightgray] (-1,-2)node{$\bullet$}; 
  \draw [color=lightgray]
  (-1,-1)node{$\bullet$}; 
  \draw [color=lightgray]
  (-1,2)node{$\bullet$};

\draw  [color=lightgray](-3.5,2)node{\tiny$g_jT_0$};
\draw  [color=lightgray](-2.2,2.3)node{\tiny$g_{j+1}T_0$};
\draw  [color=lightgray](-0.1,2)node{\tiny$g_{j+2}T_0$};
\draw  (-3.4,1)node{\tiny$g_{j}a_j$};
\draw  (0.1,-1)node{\tiny$g_{j+1}a_{j+1}$};
 
\draw  [color=lightgray](-3,2) -- (-3,-2);
\draw  [color=lightgray](-2,2) -- (-2,-2);
\draw  [color=lightgray](-1,2) -- (-1,-2);
\draw  [color=red,thick, dashed](-2,1) -- (-2,-1);
\draw  [thick](-3,1) -- (-2,1);
\draw  [thick] (-2,-1) -- (-1,-1);

\draw [color=red] (-2,1)node{$\bullet$}; 
  \draw [color=red] (-2,-0)node{$\bullet$}; 
  \draw [color=red]
  (-2,-1)node{$\bullet$}; 
  \draw  (-3,1)node{$\bullet$}; 
  \draw  (-1,-1)node{$\bullet$}; 

\end{tikzpicture}
 \caption{$\cay(G,\tilde{S})$ não contém ciclos.}
    \end{center}
\end{figure}

Usando o fato de que $ g_0,\ldots, g_{n-1}$ é um ciclo em $\cay(G,S)$, vemos que ao concatenar esses caminhos, obtemos um ciclo em $T$, uma contradição com o fato de que $T$ é árvore. 

\hfill \qedsymbol

Observamos que existe uma prova bem mais curta desse teorema, a qual segue da teoria de recobrimentos. De fato, se $G$ age por automorfismos livremente em uma árvore não vazia $T$, então $G \cong \pi_1(F)$, onde $F = T/G$. Como $F$ é um grafo conexo, segue do Teorema de Seifert--Van Kampen que $\pi_1(F)$ é livre. No entanto, optamos por escrever a demonstração acima pelo fato das técnicas usadas dependerem puramente da teoria de grupos e das propriedades da ação de grupo na árvore $T$. Além disso, a noção de árvore de extensão usada na prova tem várias outras aplicações. 

\begin{exercise}
 Mostre que a métrica das palavras em $G$, com conjunto de geradores $S$ simétrico, também pode ser definida por: 
\begin{enumerate}[(1)]
    \item Única métrica maximal invariante à esquerda, com $\dist(1,s)=\dist(1,s^{-1})=1$ para todo $s \in S$; ou
    \item $\dist_S(h,g) $ é o comprimento minimal de uma palavra $w$ escrita em $S$ tal que $w= g^{-1}h$ como elemento de $G$.
\end{enumerate}
\end{exercise}

\begin{exercise}\label{metricadaspalavras}
  Sejam $S$ e $\Bar{S}$ conjuntos finitos que geram $G$. Mostre que $\dist_S$ e $\dist_{\Bar{S}}$ são \textit{bi-Lipschitz equivalentes}, \index{equivalência bi-Lipschitz} isto é, existe $L>0$ tal que 
  $$\frac{1}{L} \dist_S(g,g') \leq \dist_{\Bar{S}}(g,g') \leq L\,\dist_S(g,g'). $$
  Conclua que cada isomorfismo $f : G\to H$ de grupos finitamente gerados é bi-Lipschitz.
\end{exercise}

\begin{exercise}\label{exerc:geodesicabiinfinita}
Mostre que o grafo de Cayley de um grupo infinito, mas finitamente gerado,  contém uma cópia isométrica de $\R$, ou seja, uma geodésica bi-infinita.
\end{exercise}

\section{Grafos de Schreier e o método de Reide\-meis\-ter--Schreier}

Sejam $G=\langle X \mid R\rangle$ um grupo finitamente apresentado e $H \leq G$ um subgrupo. O método de Reidemeister--Schreier permite recuperar uma apresentação do grupo $H$ usando informações sobre o conjunto quociente $G/H$. 
Para apresentar o subgrupo $H$ precisamos de palavras em $X$, que geram $H$. Mas, além desses geradores de $H$, precisaremos de um processo para ``reescrever'' uma palavra em $X$ que defina um elemento de $H$, como uma palavra nos geradores de $H$. O processo que vamos definir nesta seção é chamado de processo de reescrita de Reidemeister--Schreier (para mais sobre esse processo, o leitor pode se referir a \cite{magnus2004combinatorial} ou \cite{bogopolski}).

\subsection{Transversal de Schreier}

\begin{definition}
Sejam $G = F(X)$ um grupo livre e $H \leq G.$ Um conjunto $\mathcal{T} \subset G$ é dito um  \textit{transversal de Schreier} \index{transversal de Schreier} para $H$ quando:
\begin{itemize}
\item[(i)] $Ht\neq Hs$, para $t \neq s \in \mathcal{T}$;
\item[(ii)] $\displaystyle \bigcup_{t \in \mathcal{T}}Ht=G$;
\item[(iii)] se $t=x_1\ldots x_n \in \mathcal{T}$, com $x_i \in X$, então $x_1\ldots x_j \in \mathcal{T}$, para  todo $1\leq j< n$.
\end{itemize}
\end{definition}
Note que sempre vale $e\in \mathcal{T}$. Para cada $g \in G$, denote por $\overline{g}$ o elemento de $\mathcal{T}$ com a propriedade $Hg = H\overline{g}$.

\begin{lemma}[Existência de transversais de Schreier]\label{existST} Seja $G$ um grupo livre gerado por $X$. Para cada subgrupo $H \leq G$, existe um transversal de Schreier para $H$ em $G$. 
\end{lemma}

\begin{proof}
Definimos o comprimento de uma classe de $G$ modulo $H$ como o comprimento da palavra mais curta nesta classe. Os representantes de Schreier serão definidos indutivamente, usando o comprimento de classes.

Escolha a palavra vazia como representante de $H$, a classe de comprimento zero. Se $S_1$ for uma classe de comprimento um, escolha qualquer palavra $x_1$ de comprimento um em $S_1$ como seu representante. Se $S_2$ tiver comprimento dois, selecione uma palavra $x_1x_2$ de comprimento dois em $S_2$ (onde $x_1$ e $x_2$ são geradores de $G$ ou seus inversos). Agora $\overline{x}_1x_2$ está bem definido e, como $Hx_1 = H\overline{x}_1$, ele pertence a $S_2$. Como $\overline{x}_1$ tem comprimento igual a um, $\overline{x}_1x_2$ deve ter comprimento dois. Escolhemos $\overline{x}_1x_2$ como o representante de $S_2$. Em geral, assumindo que escolhemos representantes para todas as classes de comprimento menor que $k$, se $S_k$ for uma classe de comprimento $k$ e $x_1\ldots x_{k-1}x_k$ for um elemento em $S_k$, escolhemos $\overline{x_1\ldots x_{k-1}}x_k$ (que tem comprimento $k$) como o representante de $S_k$. Claramente, por construção, se o último gerador for excluído de um representante, outro representante será obtido. Assim, qualquer segmento inicial de um representante é um representante, e todas as propriedades de uma transversal de Schreier estão satisfeitas.
\end{proof}

\begin{proposition}
Seja $\mathcal{T}$ um  transversal de Schreier para $H$ em $G$. Então $H$ tem base de geradores 
$$\{tx(\overline{tx})^{-1} \mid t \in  \mathcal{T},\ x \in X \mbox{ e } tx(\overline{tx})^{-1} \neq  1\}.$$
\end{proposition}

\begin{proof}
Nós vamos usar as seguentes propriedades básicas de representantes de classes:
\begin{itemize}
    \item[(i)] $\overline{w} = 1$ se e somente se $w\in H$;
    \item[(ii)] $\overline{\overline{w}} = \overline{w}$;
    \item[(iii)] $\overline{wv} = \overline{\overline{w}v}$.
\end{itemize}

Claramente, $tx(\overline{tx})^{-1}$ é um elemento de $H$ pois $tx$ e $\overline{tx}$ pertencem à mesma classe. Falta mostrar que todos elementos de $H$ pode ser escritos como produtos dos elementos $tx(\overline{tx})^{-1}$ e de seus inversos.

Observe que 
$$v = tx^{-1}(\overline{tx^{-1}})^{-1}$$
é inverso de um elemento da base. De fato, por (ii) e  (iii), obtemos
$$\overline{\overline{tx^{-1}}x} = \overline{tx^{-1}x} = \overline{t} = t.$$ 
Então $tx^{-1}(\overline{tx^{-1}})^{-1}$ é inverso de $ux(\overline{ux})^{-1}$ com $u = \overline{tx^{-1}} \in \mathcal{T}$.

Seja $w = x_1x_2\ldots x_r$ com $x_i \in X\cup X^{-1}$ um elemento de $H$. Temos que escrever $w$ como um produto de elementos da base e seus inversos. Para isso vamos inserir, antes e depois cada $x_i$ em $w$, palavras $\overline{v_i}$ e $(\overline{v_ix_i})^{-1}$, respectivamente, e tentar escolher as palavras $v_i$ tais que o produto $$\overline{v_1}x_1(\overline{v_1x_1})^{-1} \overline{v_2}x_2(\overline{v_2x_2})^{-1}\ldots \overline{v_r}x_r(\overline{v_rx_r})^{-1}$$
defina o mesmo elemento $w \in H$. Isso será verdadeiro se escolhermos
$$v_1 = 1,\ v_2 = v_1x_1, \ldots, v_r = v_{r-1}x_{r-1}.$$
De fato, neste caso temos 
$$\overline{v_1}x_1(\overline{v_1x_1})^{-1} \overline{v_2}x_2(\overline{v_2x_2})^{-1}\ldots \overline{v_r}x_r(\overline{v_rx_r})^{-1} = $$
$$
\overline{v_1}w(\overline{v_rx_r})^{-1} = \overline{1}w\overline{w}^{-1} = 1 w 1^{-1} = w.$$
(aqui $\overline{w} = 1$ pois $w \in H$). 

A propriedade (iii) implica que as palavras $\overline{v_i}x_i(\overline{v_ix_i})^{-1}$ definem elementos da base ou seus inversos (veja também a observação sobre $v$ no início da prova). Portanto, cada $w\in H$ é um produto de elementos da base ou seus inversos. 
\end{proof}

\begin{example}
O conjunto $\{a^nb^m \mid m,n\in\Z\}$ é um  transversal de Schreier para o grupo derivado do grupo livre $F(a,b)$. Além disso, $\overline{a^n b^m\cdot a}=a^{n+1}b^m$
e $\overline{a^nb^m \cdot b} =  a^nb^{m+1}$. Portanto, esse subgrupo tem como base o conjunto $$ \{a^nb^mab^{-m}a^{-(n+1)} \mid n,m \in \Z;\ m \neq 0\}.$$
\end{example}

\begin{exercise}
Represente o grupo  diedral $D_n$ como o quociente do grupo $F(a,c)$ pelo fecho normal do conjunto $\{a^2,c^2,(ac)^n\}$. Seja $H$ o núcleo do epimorfismo canônico $\varphi: F(a,c) \to D_n$. Prove que os seguintes conjuntos são  transversais de Schreier para $H$ em $F(a,c)$:
\begin{itemize}
    \item[(i)] o conjunto de todos os segmentos iniciais das palavras $(ac)^k$ e $(ca)^{k-1}c$, se  $n=2k$ é par;
    \item[(ii)]  o conjunto de todos os segmentos iniciais das palavras $(ac)^ka$ e $(ca)^k$, se $n=2k+1$ é ímpar.
\end{itemize}
Em seguida, encontre uma base de geradores para o grupo $H$.
\end{exercise}

\begin{exercise}
O grupo diedral infinito $D_{\infty}$ é o grupo das isometrias de $\Z$ (com a métrica induzida de $\R$), o qual é gerado pela translação $t(x) = x + 1$ e pela simetria $s(x) = -x$ e tem a seguinte apresentação:  $D_{\infty}= \langle s,t \mid s^2,s^{-1}tst \rangle$.
Represente esse grupo como o quociente de $F(a,c)$ pelo fecho normal do conjunto $\{a^2,c^2\}$. Encontre uma base de geradores para o núcleo do epimorfismo canônico $\varphi: F(a,c) \to D_{\infty}$.
\end{exercise}

\subsection{Processo de reescrita de Reidemeister--Schreier}

Sejam $F$ um grupo livremente gerado por um conjunto $X$, $H$ um subgrupo de $F$ e $\mathcal{T}$ um transversal de Schreier para $H$ em $F$. Para cada $t\in \mathcal{T}$ e cada $x \in X \cup X^{-1}$, defina $\gamma(t,x)= tx(\overline{tx})^{-1}$. Já vimos que os elementos não triviais  $\gamma(t,x)$, com $t\in \mathcal{T}$ e  $x\in X$, formam uma  base do grupo livre $H$, a qual denotaremos a partir de agora por $Y$. O processo de reescrever $w$ como uma palavra na base $Y$ é chamado de \textit{processo  de reescrita de Reidemeister--Schreier}. \index{processo de Reidemeister--Schreier}

Seja $H^*$ o grupo livre gerado por uma cópia $Y^*:=\{y^* \mid y \in Y\}$ de $Y$. A correspondência $y \mapsto y^*$ estende-se unicamente a um isomorfismo
$$
\tau: H \to H^*.
$$

Mais explicitamente, o mapa $\tau$ é definido da seguinte forma: dado $w \in H$, escrevemos $w$ como uma palavra nos geradores de Schreier $Y$ por meio do processo de reescrita de Reidemeister--Schreier. Então $\tau(w)$ é obtido substituindo cada gerador $y \in Y$ pela sua cópia $y^* \in Y^*$.

Em particular, se $w = x_1 \cdots x_n \in H$, com $x_i \in X \cup X^{-1}$, então
\[
w = \gamma(1, x_1)\cdot
\gamma (\overline{x_1}, x_2)\cdots  
\gamma(\overline{x_1\ldots x_{n-1}},  x_n),
\]
e, portanto,
\[
\tau(w) = \gamma(1, x_1)^* \cdot
\gamma (\overline{x_1}, x_2)^* \cdots  
\gamma(\overline{x_1\ldots x_{n-1}},  x_n)^*.
\]



\begin{proposition} \label{red-sch}
Sejam $G$ um grupo com apresentação $\langle X \mid  R\rangle$ e $\varphi: F(X) \to  G$ o epimorfismo correspondente a essa apresentação. Considere $G_1 \leq G$ e defina $H = \varphi^{-1}(G_1)$.  Então, de acordo com a notação acima, $G_1$ tem a seguinte apresentação:
$$G_1= \langle Y^* \mid R^*\rangle, \mbox{ onde } R^*= \{\tau(trt^{-1}) \mid  t\in \mathcal{T},  r\in  R\}.$$
\end{proposition}

\begin{exercise}
Prove a Proposição~\ref{red-sch}.
\end{exercise}

Como um corolário imediato, obtemos:
\begin{corollary}\label{cor:reid-sc} Seja $H\leq G$ um subgrupo de índice finito.
\begin{enumerate}
\item Se $G$ é finitamente gerado, então $H$ é finitamente gerado.
\item Se $G$ é finitamente apresentado, então $H$ é finitamente apresentado.
\end{enumerate}
\end{corollary}

Note que o item 1 no Corolário \ref{cor:reid-sc} já havia sido apresentado na Proposição  \ref{prop:finindexfg} com outra abordagem, no entanto o item 2 ainda não havia sido provado. 

Uma outra consequência é o Teorema de Nielsen--Schreier:
\begin{corollary}[Nielsen--Schreier]
Um subgrupo de um grupo livre é livre.
\end{corollary}

\begin{proof}
Se $G$ é um grupo livre, ele pode ser apresentado com um conjunto vazio de relatores $R$. A proposição acima implica que, neste caso, temos $R^* = \emptyset$, e então o subgrupo $H$ é livre.
\end{proof}

\begin{example}
Sejam  $G$ o grupo fundamental do nó de trevo (Figura~\ref{trevo}), dado por $G = \langle a, b\mid  a^2= b^3 \rangle$, e $S_3$ o grupo simétrico das permutações de 3 elementos. Considere $\theta: G \to S_3$, um  homomorfismo entre esses grupos definido por  $a \mapsto(12), b \mapsto (123)$. Vamos obter uma apresentação do seu núcleo $G_1$.

\begin{figure}[!ht]
    \centering
   \begin{tikzpicture}
    \begin{knot}[
        clip width=5,
        consider self intersections
    ]
        \strand[blue,thick] (90:1)
            \foreach \x in {1,2,3} {
                to [bend left=117,looseness=1.9] ({90+120*\x}:1)
            }
        ;
        \flipcrossings{1,3}
    \end{knot}
\end{tikzpicture}
    \caption{Nó de trevo}
    \label{trevo}
\end{figure}

Seja $\varphi:F(a, b) \to  G$ o epimorfismo canônico e $H = \varphi^{-1}(G_1)$. Como representantes de Schreier das classes laterais à direita de $H$ em $F(a, b)$, escolhemos $1, b, b^2, a, ab, ab^2$. Então os seguintes elementos geram $H$:
\begin{equation*} 
\begin{array}{ll}
     1 \cdot a \cdot\overline{a}^{-1}=1  & \,\, 1 \cdot b \cdot\overline{b}^{-1} =1 \\
     x = b\cdot a \cdot (\overline{ba})^{-1} = bab^{-2}a^{-1} & \, \, b \cdot b \cdot(\overline{b^2})^{-1}=1 \\
     y = b^2 \cdot a \cdot (\overline{b^2a})^{-1} = b^2ab^{-1}a^{-1}  & \,\, w = b^2 \cdot b \cdot(\overline{b^3})^{-1} = b^3 \\
     z= a \cdot a \cdot(\overline{a^2})^{-1} = a^2  &\, \, a \cdot b \cdot(\overline{ab})^{-1} = 1 \\
      u = ab \cdot a \cdot (\overline{aba})^{-1} =abab^{-2} &\, \, ab\cdot b \cdot(\overline{ab^2})^{-1} = 1 \\
      v = ab^2 \cdot a \cdot (\overline{ab^2a})^{-1} =ab^2ab^{-1}  & \,\,s = ab^2 \cdot b \cdot (\overline{ab^3})^{-1} =ab^3a^{-1}  \\
\end{array}
\end{equation*}
Podemos assumir que esses elementos geram $G_1$. Para obter as relações que definem $G_1$, precisamos reescrever as relações $t rt^{-1}$, onde $t \in \{1, b, b^2, a, ab, ab^2\}$ e $r = b^3a^{-2}$, como palavras nos geradores $x, y, z, u, v, w$ e $s$. Temos assim:
\begin{equation*} 
\begin{array}{lll}
     r=wv^{-1}  & \, brb^{-1}= wv^{-1}x^{-1} & \, b^2rb^{-2}= wu^{-1}y^{-1} \\
     ara^{-1} =sz^{-1} & \, abr(ab)^{-1} = sy^{-1}u^{-1} &\, ab^2r(ab^2)^{-1} = sx^{-1}v^{-1}.
\end{array}
\end{equation*}

Agora  eliminamos geradores $w, v, u, s$, substituindo-os por todas as relações nas palavras $z, x^{-1}z, y^{-1}z, z$. Obtemos como resultado a apresentação $\langle x, y, z \mid yz = zy, xz= zx\rangle$ do grupo $ G_1$. Segue que $ G_1 \cong F(z)\times F(x,y)$.
\end{example}

\subsection{Grafos de Schreier}
\begin{definition} Seja $H$ um subgrupo do grupo finitamente gerado $G=\langle X\mid R\rangle$. O \textit{grafo de classes de Schreier} \index{grafo de Schreier}  é o grafo $\Sigma(G,X,H)$, cujos vértices são as classes laterais à direita $Hg$ e o par de vértices 
$(Hg_1, Hg_2)$ define uma aresta se, e somente se, existe $x \in X$ tal que $g_1=g_2x$. 
Note que $\Sigma(G,S,\{e\})$ é o grafo de Cayley de $G$ com respeito a $S$.
\end{definition}

Observe que o grupo $G$ age transitivamente no conjunto $\Omega$ de todas as classes laterais de $H$ por permutações com $H= \mathrm{Stab}(p)$, para todo $p\in\Omega$.

\begin{example}\label{sch1}
Sejam $G=F(\{x,y,x^{-1},y^{-1}\})$ e $H \leq G$ um subgrupo de índice $5$ tal que a ação de $G$ em $G/H \cong \{1,2,3,4,5\} $ é dada por $x  \mapsto  (12)(45)$ e $y \to  (2354)$.
Podemos fazer uma tabela de classes, que auxilia na construção do grafo de classes de Schreier: 
\begin{center}
\begin{tabular}{| l | c | c | c | c|}
\hline
 & $x$ & $y$ & $x^{-1}$ & $y^{-1}$\\
\hline
1 & 2 & 1 & 2 & 1\\ \hline
2 & 1 & 3 & 1 & 4\\ \hline
3 & 3 & 5 & 3 & 2\\ \hline
4 & 5 & 2 & 5 & 5\\ \hline
5 & 4 & 4 & 4 & 3\\ 
\hline
\end{tabular}
\end{center}
Com isso, o grafo de Schreier obtido para esse exemplo é:
\end{example}

\begin{figure}[H]
     \centering
 \begin{tikzpicture}[scale=1]
  \tikzset{VertexStyle/.style = {shape = circle,fill = black,minimum size = 4pt,inner sep=0pt}}
\Vertex[x=0,y=1]{1}
\Vertex[x=0,y=0]{2} 
\Vertex[x=0,y=-1.5]{3}
\Vertex[x=1.5,y=0]{4}
\Vertex[x=1.5,y=-1.5]{5}

\draw (-0.2,1)node{\tiny $1$};
\draw (-0.2,0)node{\tiny $2$};
\draw (0,-1.7)node{\tiny $3$};
\draw (1.5,0.2)node{\tiny $4$};
\draw (1.5,-1.7)node{\tiny $5$};

\draw (0.75,1.5)node{\tiny $y,y^{-1}$};
\draw (-0.4,0.5)node{\tiny $x$};
\draw (0.5,0.6)node{\tiny $x^{-1}$};
\draw (1,-0.7)node{\tiny $x^{-1}$};
\draw (2.3,-0.7)node{\tiny $y^{-1}$};
\draw (-1.15,-1.5)node{\tiny $x,x^{-1}$};

\Edge[lw=1pt,bend=-45,fontscale=.9](1)(2)
\Edge[lw=1pt,loopsize=1.3cm,loopposition=90](1)(1)
\Edge[lw=1pt,label=$y$,fontscale=.9](2)(3)
\Edge[lw=1pt,label=$y^{-1}$,fontscale=.8](2)(4)
\Edge[lw=1pt,label=$y$,fontscale=.9](3)(5)
\Edge[lw=1pt,fontscale=.9,bend=-20](4)(5)
\Edge[lw=1pt, bend=45,fontscale=.9](1)(2)
\Edge[lw=1pt,label=$x$,fontscale=1,bend=20](4)(5)
\Edge[lw=1pt,bend=50,fontscale=.8](4)(5)
\Edge[lw=1pt,loopsize=1.3cm,loopposition=180](3)(3)

\end{tikzpicture}
     \caption{Grafo de Schreier do Exemplo~\ref{sch1}}  
\end{figure}

\begin{remark}
 Os elementos de $H$ correspondem aos circuitos de $\Sigma(G,S,H)$ com base em $p= \{H\}$. Observe que cada caminho no grafo de Schreier representa uma palavra $w=w(X)$. Esse caminho será fechado se, e somente se, $Hw = H$, ou seja, se $w \in H$. Além disso, os geradores de $H$ correspondem aos circuitos simples no grafo  $\Sigma(G,S,H)$.
\end{remark} 
   
    Um transversal de Schreier de $H$ corresponde a uma árvore de extensão $T$ de $\Sigma(G,S,H)$. Um caminho na árvore de extensão, baseado em $\{H\}$, corresponde a uma palavra $w=w(X)$ com subpalavras iniciais correspondentes a subcaminhos desse caminho original.
   
    Um conjunto de geradores de Schreier de $H$ em $G$ corresponde a um conjunto de  arestas de $\Sigma(G,S,H)$ não contidas em $T$. De fato, cada aresta fora de $T$ produz um circuito simples no grafo.

\begin{figure}[H]
     \centering
 \begin{tikzpicture}
\tikzset{VertexStyle/.style = {shape = circle,fill = black, minimum size = 4pt,inner sep=0pt}}
\Vertex[x=0,y=0]{1}
\Vertex[x=1,y=2]{2} 
\Vertex[x=1,y=0.1]{3}
\Vertex[x=1.9,y=1]{4}
\Vertex[x=2.3,y=-0.7]{5}
\Vertex[x=3,y=2]{6}
\Vertex[x=4, y=0.1]{7}

\draw (-0.2,0)node{\tiny $H$};
\draw (0.6,2)node{\tiny $Hx_i$};
\draw (1,-0.2)node{\tiny $Hx_k$};
\draw (1.9,0.7)node{\tiny $H(x_ix_j)$};
\draw (2.3,-0.9)node{\tiny $H(x_kx_l)$};
\draw (3.3,2)node{\tiny $Hu$};
\draw (4.5,-0.2)node{\tiny $Hv = H(ux_t)$};

\Edge[lw=1pt,opacity=.2](1)(2)
\Edge[lw=1pt, opacity=.2](1)(3)
\Edge[lw=1pt, opacity=.2](2)(4)
\Edge[lw=1pt, opacity=.2](3)(5)
\Edge[lw=1pt,style={dashed}, opacity=.2](4)(6)
\Edge[lw=1pt,style={dashed}, opacity=.2](5)(7)
\Edge[lw=1pt,color=blue](6)(7)

\end{tikzpicture}
     \caption{Geradores de Schreier}    
\end{figure}

Ainda no Exemplo~\ref{sch1}, obtemos os seguintes geradores de Schreier:
\begin{figure}[H]
     \centering
 \begin{tikzpicture}[scale=1]
  \tikzset{VertexStyle/.style = {shape = circle,fill = black,minimum size = 4pt,inner sep=0pt}}
\Vertex[x=0,y=1]{1}
\Vertex[x=0,y=0]{2} 
\Vertex[x=0,y=-1.5]{3}
\Vertex[x=1.5,y=0]{4}
\Vertex[x=1.5,y=-1.5]{5}

\draw (-0.2,1)node{\tiny $1$};
\draw (-0.2,0)node{\tiny $2$};
\draw (-0.2,-1.5)node{\tiny $3$};
\draw (1.5,0.2)node{\tiny $4$};
\draw (1.5,-1.7)node{\tiny $5$};


\Edge[lw=1pt,opacity=.2,bend=-45](1)(2)
\Edge[lw=1pt,color=blue,loopsize=1.3cm,loopposition=90](1)(1)
\Edge[lw=1pt,opacity=.2](2)(3)
\Edge[lw=1pt,opacity=.2](2)(4)
\Edge[lw=1pt,opacity=.2](3)(5)
\Edge[lw=1pt,color=blue](4)(5)
\Edge[lw=1pt,color=blue, bend=45](1)(2)
\Edge[lw=1pt,color=blue,bend=35](4)(5)
\Edge[lw=1pt,color=blue,bend=90](4)(5)
\Edge[lw=1pt,color=blue,loopsize=1.3cm,loopposition=180](3)(3)

\end{tikzpicture}
     \caption{Geradores de Schreier do Exemplo~\ref{sch1}}    
\end{figure}

\begin{example}[Grupo modular] \label{sch2}
Considere  $G=\langle x,y,x^{-1},y^{-1} \mid$ $x^2,y^3\rangle$ e tome $H=\mathrm{Stab}(1)$, um subgrupo de índice $3$ em $G$. Se considerarmos a ação de $G$ em $G/H \cong \{1,2,3\} $ dada por $x \to (23)$ e $y \to (123)$, a tabela de classes para esse exemplo será 
\begin{center}
\begin{tabular}{| l | c | c | c | c|}
\hline
 & $x$ & $y$ & $x^{-1}$ &  $y^{-1}$\\
\hline
1 & 1 & 2 & 1 & 3\\ \hline
2 & 3 & 3 & 3 & 1\\ \hline
3 & 2 & 1 & 2 & 2\\ 
\hline
\end{tabular}
\end{center}

Os geradores de Schreier serão $A= x$  ($=x^{-1}$), $B=y^3$ ($=1$), $C=yxy$ e $D=y^{-1}xy^{-1}$.

A relação $x^2 = 1$ fornece $A^2 = 1$ e $CD = 1$. Já $y^3 = 1$ gera a  nova relação $B = 1$. Portanto, $H$ tem apresentação $\langle A,C \mid A^2\rangle$ com $A=x$ e $C=yxy$.

\begin{figure}[H]
     \centering
 \begin{tikzpicture}
\tikzset{VertexStyle/.style = {shape = circle,fill = black,minimum size = 4pt,inner sep=0pt}}
\Vertex[x=0,y=0]{1}
\Vertex[x=1.5,y=1]{2} 
\Vertex[x=1.5,y=-1]{3}

\draw (0,-0.3)node{\tiny $H$};
\draw (1.1,1)node{\tiny $Hy$};
\draw (1.4,-1.3)node{\tiny $Hy^{-1}$};

\Edge[lw=1pt,color=blue,loopsize=1.5cm,label=$x$,fontscale=1,loopposition=180](1)(1)
\Edge[lw=1pt](1)(2)
\Edge[lw=1pt](1)(3)

\Edge[lw=1pt,color=blue,label=$x$,fontscale=1,bend=-25](2)(3)
\Edge[lw=1pt,color=blue,label=$x^{-1}$,fontscale=.9,bend=25](2)(3)
\Edge[lw=1pt,color=blue,bend=100,label=$y$,fontscale=1](2)(3)
\end{tikzpicture}
     \caption{Grafo e geradores de Schreier do Exemplo~\ref{sch2} }    
\end{figure}
\end{example}


\begin{exercise}
Seja $G=\langle x, y \mid x^2, y^3\rangle$, e considere em $G$ os seguintes
subgrupos: 
\begin{align*}
H_1 &= G' = \langle \{[a,b]; a,b \in G \}\rangle,\\  
H_2 &= G^{(2)} =\langle\{a^2; a \in G\}\rangle.
\end{align*}
Usando o processo de Reidmeister--Schreier e os grafos de Schreier, mostre que $H_1$ é livre de posto $2$ e dê uma apresentação para $H_2$.
\end{exercise}

\chapter{Geometria grosseira e quasi-isometrias}
\label{cap4}
 \section{Quasi-isometrias}

Queremos definir uma noção geométrica a uma larga escala de similaridade. Para isto, formalizaremos a ideia de que dois espaços métricos são quasi-isométricos se eles ``se parecem o mesmo quando olhados de longe''. Um exemplo para se ter em mente é que queremos que a reta real e o conjunto dos inteiros, com a métrica induzida da reta real, sejam equivalentes via quasi-isometrias. Um outro bom exemplo é o par $\R^2$ e $\Z^2$, com a métrica induzida do plano euclidiano. Essa noção será particularmente útil na teoria geométrica de grupos, pois permitirá remover a dependência do modelo geométrico da escolha da apresentação do grupo.

No que segue, definiremos duas noções equivalentes  de quasi-isome\-tri\-as. Uma é fácil de visualizar, e a outra facilita a compreensão de por que se trata de uma relação de equivalência. Começamos com a primeira definição na Subseção~\ref{subsec:QI1}, continuamos com a segunda na Subseção~\ref{subsec:QI2} e demonstramos a sua equivalência na Subseção~\ref{subsec:QI3}.

\subsection{Discretização da métrica}\label{subsec:QI1}
Inicialmente, relembramos a definição de espaço métrico.

\begin{definition}
Um \textit{espaço métrico} \index{espaço métrico} é um par $(X,d)$, onde $X$ é um conjunto e $ d: X \times X \to \R_{\geq 0}$ é uma função que satisfaz:
\begin{enumerate}[(i)]
    \item para $x, y \in X$, temos $d(x, y) = 0$ se e somente se $x=y$;
    \item para todos $x,y \in X$, temos $ d(x, y) = d(y, x)$;
    \item vale a \textit{desigualdade triangular}\index{desigualdade triangular}: $d(x, z) \leq d(x, y) + d(y, z)$, para todos $x,y,z \in X$.
\end{enumerate}
\end{definition}

Há duas definições equivalentes de quasi-isometrias, sendo uma delas mais geométrica e a outra mais analítica. Como veremos, cada uma delas tem suas vantagens. Falaremos sobre as duas e mostraremos em seguida que elas são equivalentes.

\begin{definition}
Seja $(X,d)$ espaço métrico. Dado $\varepsilon>0$, dizemos que um conjunto  $A \subset X$ é \textit{$\varepsilon$-separado} \index{conjunto $\varepsilon$-separado} se $d(a_1,a_2)\geq \varepsilon$ para todos $a_1 \neq a_2 \in A$.
\end{definition}

Relembramos que a \textit{distância de Hausdorff}\index{distância de Hausdorff} entre dois subconjuntos $A$ e $B$ de um espaço métrico $X$ é dada por 
$$d_{Haus}(A,B):=\inf\{r\ge 0 \mid A \subset \overline{\mathcal{N}_r(B)} \mbox{ e } B \subset \overline{\mathcal{N}_r(A) \}},$$ 
onde $\mathcal{N}_r(B) =  \{x \in X \mid d(x, B)< r\}$ e $\overline{\mathcal{N}_r(B)}$ denota o fecho de $\mathcal{N}_r(B)$. Observamos que, embora o nome ``distância'' possa sugerir que $d_{Haus}$ é uma métrica em $\mathcal{P}(X)$, esta é na verdade uma \textit{pseudo-métrica}, já que $d_{Haus}(A,B)=0$ não  implica necessariamente que $A$ seja igual a $B$.

Um subconjunto $S$ de um espaço métrico $X$ é dito \textit{$\delta$-denso} em $X$ se a distância de Hausdorff entre $S$ e $X$ é no máximo $\delta$. Em outras palavras, $S$ é \textit{$\delta$-denso} em $X$ se, para cada $x \in  X$, vale a desigualdade $d(x, S)\leq \delta$.\index{conjunto $\delta$-denso}

\begin{definition}
Seja $(X,d_X)$ um espaço métrico. Dizemos que $A \subset X$ é uma \textit{$\delta$-rede $\varepsilon$-separada} se $A$ é $\varepsilon$-separado e $\delta$-denso em $X$. 
Caso $\delta= \varepsilon$, diremos que $A$
 é uma \textit{rede $\varepsilon$-separada}. \index{rede $\varepsilon$-separada}
 \end{definition}

\begin{proposition} \label{redeseparada}
Um conjunto $\varepsilon$-separado maximal em $X$ é uma rede $\varepsilon$-separada.
\end{proposition}

\begin{proof}
Seja $A\subset X$ um conjunto $\varepsilon$-separado maximal. Se $x \in X\setminus A$, note que $A \cup \{x\}$ não é $\varepsilon$-separado, pela maximalidade de $A$. Então existe $y \in A$ tal que $d(x,y) < \varepsilon$. Portanto,   $A$
 é uma \textit{rede $\varepsilon$-separada}.
\end{proof}
Pelo Lema de Zorn, sempre existe um subconjunto  $\varepsilon$-separado maximal em um dado espaço métrico. Então, para cada $\varepsilon>0$, todo espaço métrico possui uma rede $\varepsilon$-separada.
\begin{example}
Se $(X,d)$ é compacto, então cada rede $\varepsilon$-separada é finita e, portanto,  todo subconjunto $\varepsilon$-separado é finito. 
Com efeito, se $A \subset X$ é uma rede $\varepsilon$-separada, então $A$ deve ser um conjunto discreto, por ser $\varepsilon$-separado, e $\displaystyle \bigcup_{x \in A}B(x,\varepsilon)=X$, já que $A$ é  $\varepsilon$-denso, portanto A é finito pela compacidade de $X$. Agora, dado qualquer conjunto $\varepsilon$-separado $B\subset X$, ele está contido num subconjunto $\varepsilon$-separado maximal. Pela Proposição~\ref{redeseparada}, $B$ está contido na verdade em uma rede $\varepsilon$-separada, a qual deve ser finita pelo argumento acima, logo o conjunto $B$ é finito. 
\end{example}

Lembramos que um mapa $f:X \to Y$ entre dois espaços métricos $(X, d_X)$ e $(Y, d_Y)$ é dito \textit{bi-Lipschtiz} \index{função bi-Lipschtiz} se for uma bijeção e se existe $L\geq 1$ tal que, para quaisquer $x,y \in X$, 
$$\frac{1}{L}d_X(x,y) \leq d_Y(f(x),f(y)) \leq Ld_X(x,y).$$

\begin{definition}\label{qi1}
Sejam $(X,d_X), (Y,d_Y)$ espaços métricos. Dizemos que $X$ e $Y$ são \textit{quasi-isométricos} \index{espaços quasi-isométricos} se existem $A \subset X$ e $B \subset Y$ redes separadas tais que $(A,d_X)$ e $(B,d_Y)$ são bi-Lipschitz equivalentes, isto é, existe um mapa bi-Lipschitz $f:A \to B$, onde $A$ e $B$ tem métricas induzidas das métricas de $X$ e $Y$, respectivamente.
\end{definition}

\begin{example}
 Se  $X$ é um espaço métrico com diâmetro finito, então $X$ é quasi-isométrico a um conjunto unitário.
\end{example}
\begin{example}
O espaço $\R^n$ é quasi-isométrico a $\Z^n.$
\end{example}

É direto ver que a relação quasi-isometria definida acima é uma relação reflexiva e simétrica, mas ao tentar provar a transitividade, a seguinte pergunta surge naturalmente:

\medskip

\begin{pergunta}[\cite{gromov1993asymptotic}] Pode um espaço métrico ter duas redes separadas que não são bi-Lipschitz?
\end{pergunta}

A pergunta de Gromov foi respondida positivamente por Burago--Kleiner \cite{burago1998separated}, através do seguinte teorema:
\begin{thm}\label{thmbk}
Existe uma rede separada no plano Euclidiano que não é bi-Lipschitz a $\Z^2.$
\end{thm}

De maneira independente, McMullen \cite{mcmullen1998lipschitz} dá uma resposta negativa à pergunta de Gromov, no caso em que relaxamos a condição de redes bi-Lipschitz para redes bi-Hölder equivalentes.

A demonstração do Teorema~\ref{thmbk} é uma construção implícita da existência dessa rede separada. A partir desse trabalho, em \cite{burago2002rectifying}, Burago e Kleiner propuseram uma candidata explícita a rede separada que não é bi-Lipschitz a $\Z^2$, a saber, o conjunto  $X:=\{x\in \R^2 \mid x \mbox{ é o baricentro de uma telha de Penrose}\}$. No entanto, Solomon provou em \cite{solomon2011substitution} que $X$ é, na verdade, bi-Lipschitz a $\Z^2$. Há vários outros trabalhos relacionados a este problema. Assim, a pergunta sobre propriedades básicas de quasi-isometrias provocou um interessante campo de pesquisa.

\subsection{Funções Lipschitz grosseiras} \label{subsec:QI2}
Felizmente, contornamos o problema de mostrar a transitividade da relação de quasi-isometrias a partir de uma segunda definição, equivalente à anterior.

\begin{definition}
Sejam $(X, d_X)$ e $(Y,d_Y)$ dois espaços métricos, $L\geq 1$ e $C\geq 0$. Uma aplicação $f : X \to Y$ é chamada \textit{$(L,C)$-Lipschitz grosseira} \index{função Lipschitz grosseira} se $$d_Y(f(x),f(y)) \leq Ld_X(x,y)+ C,$$ para todos $x, y \in X.$  Dizemos que $f$ é um \textit{$(L,C)$-mergulho quasi-isométrico} \index{mergulho quasi-isométrico} se $$\frac{1}{L} d_X(x,y)- C\leq d_Y(f(x),f(y)) \leq Ld_X(x,y)+ C,$$  para todos $x, y \in X.$
\end{definition}

Note que um mergulho quasi-isométrico não precisa ser um mergulho no sentido topológico. No entanto, pontos suficientemente distantes tem imagens distintas.

\begin{example}
\begin{enumerate}
\item A função piso, $f : \R \to \Z$, $f(x) =  \lfloor x\rfloor $, que tem como imagem de $x \in \R$ o maior inteiro $\leq x$, é um $(1,1)$-mergulho quasi-isométrico de $\R$ em $\Z.$
\item A função $f : \R \to \R$, $f(x) = x^2$ não é Lipschitz grosseira.
\end{enumerate}
\end{example}

Se $X=[a,b]$ é um intervalo compacto então um mergulho $(L,C)$-Lipschitz grosseiro $q : X \to Y$, bem como sua imagem, são ditos um \textit{segmento $(L,C)$-quasi-geodésico} ou simplesmente uma \textit{quasi-geodésica} em $Y$. Se um dos extremos $a=-\infty$ ou $b=\infty$ então $q$ é chamado de \textit{raio $(L,C)$-quasi-geodésico}. Se ambos $a=-\infty$ e $b=\infty$ então $q$ é chamado de linha \textit{$(L,C)$-quasi-geodésica}\index{quasi-geodésica}. Observamos que, muitas vezes, os valores das constantes $L$ e $C$ não serão importantes, mas sim a sua existência. Nesse caso, podemos escrever apenas ``mergulho quasi-isométrico'' ou ``segmento quasi-geodésico'', entre outros, sem necessariamente mencionar as constantes.

\begin{definition}
Um espaço métrico $X$ é dito \textit{quasi-geodésico} \index{espaço quasi-geodésico} se existirem constantes positivas $L,C$ tais que para todos $x,y \in X$ existe uma $(L,C)$-quasi-geodésica ligando $x$ a $y$.
\end{definition} 
\begin{definition}
Dado $C\geq 0$, duas aplicações entre espaços métricos $f: X \to Y$ e $\overline{f}: Y\to X$ são ditas \textit{$C$-quasi inversas} \index{quasi inversa} se, para todos $x\in X$ e $y\in Y$,  
$$d_X(\overline{f}\circ f(x), x) \leq C \mbox{ e } d_Y (f\circ \overline{f}(y), y) \leq C.$$
\end{definition}

Note que, em particular, um mapa $0$-quasi inverso de $f$ é o mapa inverso $f^{-1}$, no sentido usual.
Finalmente, podemos definir quasi-isometrias e falar sobre a segunda definição de espaços quasi-isométricos:

\begin{definition}
\label{qi2}
Um mapa $f: X \to Y$ entre espaços métricos é chamado de \textit{quasi-isometria} \index{quasi-isometria}
se for um  mergulho quasi-isométrico e admitir uma quasi-inversa que, por sua vez, também é um mergulho quasi-isométrico.
Dois espaços  métricos $X$ e $Y$ são ditos \textit{quasi-isométricos}\index{espaços quasi-isométricos} se existir uma quasi-isometria $f: X \to Y$.
\end{definition}

\begin{example}
 Considere as métricas usuais em $\R$ e $\R^2$. Se $h: \R \to \R$ é uma função $L$-Lipschitz, então a aplicação $$f: \R \to \R^2, \quad f(x)= (x,h(x))$$ é um mergulho quasi-isométrico. Com efeito, $$\arraycolsep=0pt\def\arraystretch{1.2}
\begin{array}{rcl}
\*& d(x,y )\;\leq\;&\*  d(f(x),f(y))\\
\*&\;=\;&\*   \sqrt{d(x,y)^2+d(h(x),h(y))^2} \\
\*&\;\leq\;&\* \sqrt{d(x,y)^2+L^2d(x,y)^2} \\
\*&\;=\;&\* \sqrt{1+L^2}d(x,y)
\end{array}.$$ Portanto, $$ \dfrac{1}{\sqrt{1+L^2}} d(x,y ) \leq  d(f(x),f(y)) \leq \sqrt{1+L^2}d(x,y ) $$ e $f$ é um $(\sqrt{1+L^2},0)$-mergulho quasi-isométrico.
\end{example} 

\begin{example} 

Seja $\varphi : [1,\infty)\to \R$ uma função  diferenciável tal que $\displaystyle\lim_{r\to \infty}\varphi(r) = \infty ,$ e existe $C \in \R $ para o qual $|r\varphi'(r)|\leq C$ para todo $r \in [1, \infty)$ (por exemplo, tome $\varphi(r) = \log(r)$). 

 Considere a função $F : \R^2\setminus B(0,1) \to \R^2\setminus B(0,1)$ dada em coordenadas polares por $F(r,\theta) = (r, \theta + \varphi(r))$, como na Figura \ref{fig:distlog}.

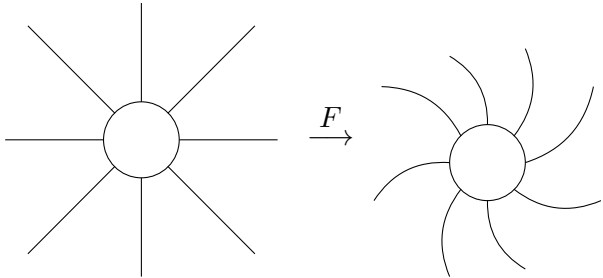
\begin{figure}[!ht]
 \centering
 \begin{tikzpicture}
    \draw  (0,0) circle (0.5);
    \draw   (-1.8,0) -- (-0.5,0);
    \draw   (0.5,0) -- (1.8,0);
    \draw   (0,0.5) -- (0,1.8);
    \draw   (0,-0.5) -- (0,-1.8);
   \draw    (0.35,0.35) -- (1.5,1.5);
   \draw    (-0.35,-0.35) -- (-1.5,-1.5);
   \draw    (-0.35,0.35) -- (-1.5,1.5);
   \draw    (0.35,-0.35) -- (1.5,-1.5);
   \draw (2.5,0) node{$\longrightarrow $};
  \draw (2.5,0.3) node{$F$};
\end{tikzpicture}
 \begin{tikzpicture}
    \draw  (0,0) circle (0.5);
     \draw   (-1.4,1) to [bend left] (-0.35,0.35);
    \draw   (0.5,0) to [bend right] (1.4,1);
    \draw   (0,0.5) to [bend right]  (-0.5,1.4);
    \draw   (0,-0.5) to [bend right]  (0.5,-1.4);
   \draw    (0.35,0.35) to  [bend right]  (0.5,1.5);
   \draw    (-0.5,0) to [bend right] (-1.5,-0.5);
   \draw    (-0.35,-0.35) to [bend right] (-0.5,-1.5);
   \draw    (0.35,-0.35) to [bend right](1.5,-0.5);
\end{tikzpicture}
\caption{Quasi-isometria não é sensível a distorção logarítmica.}
 \label{fig:distlog}
 \end{figure}
 
Note que $F$ é L-bi-Lipschitz com $L = \sqrt{1+C^2}$. De fato, a métrica euclidiana em coordenadas polares é dada por $$ds^2=dr^2+r^2d\theta ^2 .$$ Então, 
$$F^{\ast}(ds^2)=[(r\varphi'(r))^2+1]dr^2 + 2r^2\varphi'(r) dr d\theta + r^2d\theta ^2 $$ 
e a afirmação segue porque a alteração máxima da distância ocorre nas direções radiais.
Podemos estender $F$ como identidade em $B(0,1).$ Então, $\tilde{F} : \R^2 \to \R^2$ é um mergulho quasi-isométrico sobrejetivo, o que implica que $\tilde{F}$ é uma quasi-isometria.

Este exemplo mostra uma propriedade importante de que espirais logarítmicas podem ser retificadas por mapas bi-Lipschitz.
\end{example}

\begin{exercise}
 Mostre, usando a Definição~\ref{qi2}, que a relação de quasi-isometria de espaços métricos é uma
relação de equivalência.
\end{exercise} 

\subsection{Equivalência das definições de quasi-iso\-met\-ria} \label{subsec:QI3}

Começamos esta seção com um fato básico que será utilizado na demonstração. Esse resultado não é difícil e o deixamos como um bom exercício para o leitor.

\begin{exercise} \label{imagemrdensa} Se $f: X\to Y$ é um mergulho quasi-isométrico e $f(X)$ é
$r$-denso em $Y$, mostre que $f$ é uma quasi-isometria.
\end{exercise}

\begin{proposition}
Dois espaços métricos $(X,d_X)$ e $(Y,d_Y)$ são qua\-si-isométricos no sentido da Definição~\ref{qi1} se, e somente se, existe uma quasi-isometria $f: X \to Y$ no sentido da Definição~\ref{qi2}. 
\end{proposition}

\begin{proof}
Assuma inicialmente que, para $L\geq 1$ e $C\geq 0$, existe uma $(L,C)$-quasi-isometria $f : X \to Y$, com uma $C$-quasi-inversa $\overline{f}$. Escrevendo $\delta = L(C+1)$, e dado $\varepsilon >0$, sabemos pelo Lema de Zorn que existe uma $\varepsilon$-rede $\delta$-separada $A\subset X$. Afirmamos que  $B=f(A)$ é uma $(L\varepsilon+2C)$-rede $1$-separada. 

De fato, como $A$ é $\delta$-separado, dados $b_1,b_2 \in B=f(A)$ temos $d(b_1,b_2)=d(f(a_1),f(a_2))\geq \dfrac{1}{L}d(a_1,a_2)-C \geq \dfrac{1}{L}\delta-C =1,$ logo $B$ é $1$-separado.

Agora, dado $y \in Y$, deve existir $x \in X$ tal que $d(f(x),y)\leq C$ (tome, por exemplo, $x=\overline{f}(y)$). Como $A$ é $\varepsilon$-denso em $X$, existe \mbox{$a \in A$} tal que $d(a,x)\leq \varepsilon$, assim $d(f(a),y)\leq d(f(a),f(x))+d(f(x),y)\leq Ld(a,x)+C+C\leq L\varepsilon +2C$ e portanto $B$ é $(L\varepsilon+2C)$-rede $1$-separada.

 Além disso, dados $a, a' \in A$, valem as desigualdades:
$$d(f(a),f(a'))\leq L d(a,a')+C \leq \left(L+\dfrac{C}{\delta}\right)d(a,a');$$
\begin{eqnarray*}
d(f(a),f(a'))\geq \dfrac{1}{L} d(a,a')-C &\geq& \left(\dfrac{1}{L}-\dfrac{C}{\delta}\right)d(a,a')\\
&=& \frac{1}{L(C+1)}d(a,a').
\end{eqnarray*} Portanto, $A$ e $B$ são bi-Lipschitz equivalentes via $f$ e então os espaços métricos $(X,d_X)$ e $(Y,d_Y)$ são quasi-isométricos no sentido da Definição~\ref{qi1}.

Reciprocamente, assuma que $A \subset X$ e $B \subset Y$ são duas $\delta$-redes $\varepsilon$-separadas, e que existe um aplicação $L$-bi-Lipschitz $g: A \to B$. Podemos definir $f: X \to Y$ por $x \in X \mapsto g(a_x) \in B=g(A),$ onde $a_x$ é escolhido de modo que $d(x,a_x) \leq \delta$. Como $f(X)=g(A)=B$, a imagem de $f$ é $\delta$-densa. Além disso, para cada par $x,y \in X$, $$d(f(x),f(y))=d(g(a_x),g(a_y))\leq L d(a_x,a_y) \leq L(d(x,y)+2\delta)$$
e
$$d(f(x),f(y))=d(g(a_x),g(a_y))\geq \dfrac{1}{L} d(a_x,a_y) \geq \dfrac{1}{L}(d(x,y)-2\delta).$$ Portanto, $f$ é um $(L, 2\delta L )$-mergulho quasi-isométrico com imagem $\delta$-densa, e o resultado segue do Exercício~\ref{imagemrdensa}.
\end{proof}

\section{Grupo de quasi-isometrias}

Dado um espaço métrico $(X,d_X)$, denotamos por $\QI(X)$ o quociente do conjunto $\{f :  X \to X \mid f \mbox{ é quasi-isometria}\}$ pela relação de equivalência $f\sim g$ se, e somente se, $d(f,g)<\infty $, onde $$d(f,g) = \sup_{x\in X}\{d_X(f(x),g(x))\}.$$
A classe de equivalência de uma quasi-isometria $f:X\to X$ será denotada por $[f]$. O conjunto $\QI(X)$ é um grupo com a operação $[f]\cdot[g]:= [f\circ g]$, chamado {\it grupo de quasi-isometrias} do espaço métrico $X$ \index{grupo de quasi-isometrias}. Quando $G$ é um grupo finitamente gerado, denotamos por $\QI(G)$ o grupo de quasi-isometrias de $G$, visto como espaço métrico, equipado com a métrica das palavras.

\begin{exercise}
Mostre que a operação de grupo em $\QI(X)$ está bem-definida e que a quasi-inversa de uma quasi-isometria $f:X\to X$ define o inverso  de $[f]$ em $\QI(X)$, concluindo assim que este quociente é um grupo.
\end{exercise}

\begin{exercise}
Mostre que se $f:X \to Y$ é uma $(L,C)$-quasi-isometria, e $g:X\to Y$ é uma função tal que $d(f,g) \leq a < \infty$, então  $g$ é uma $(L,C+2a)$-quasi-isometria.
\end{exercise}

\begin{lemma}\label{isomQI(X)QI(Y)}
Se $h: X\to Y$ é uma quasi-isometria entre dois espaços métricos, então os grupos $\QI(X)$ e $\QI(Y)$ são isomorfos, sendo um isomorfismo entre eles dado pela aplicação $$[f]\mapsto [h\circ f \circ \overline{h}],$$
onde $\overline{h}$ é alguma quasi-inversa de $h.$
\end{lemma}

\begin{exercise}
Prove o Lema~\ref{isomQI(X)QI(Y)}.
\end{exercise}

\begin{corollary} Se $G$ é um grupo finitamente gerado, então $\QI(G)$ não depende da escolha de geradores.
\end{corollary}
\begin{proof}
Se $S$ e $S'$ são dois conjuntos finitos de geradores de $G$, então a aplicação identidade $id:(G,d_S) \to (G,d_{S'})$ é uma quasi-isometria, pelo Exercício~\ref{metricadaspalavras}, e o resultado segue do Lema~\ref{isomQI(X)QI(Y)}. 
\end{proof}

Para cada espaço métrico $X$, existe um homomorfismo natural $q_X : \Isom(X)\to \QI(X),$ dado por $f\mapsto [f]$. Em geral, este homomorfismo não é injetivo. Por exemplo, se $X=\R^n$ então o núcleo de $q_X$ é o grupo de todas as  translações, isomorfo a $\R^n$.

\begin{exercise}
Dê um exemplo de um espaço métrico não compacto $X$ tal que:
\begin{enumerate}[(a)]
    \item $q_X: \Isom(X) \to \QI(X)$ é isomorfismo;
    \item $\Isom(X)$ é trivial mas $\QI(X)$ é infinito;
    \item $\QI(X)$ é trivial mas $\Isom(X)$ é infinito.
\end{enumerate}
\end{exercise}

Considere agora um grupo finitamente gerado $G$, com conjunto de geradores $S$. Como $G$ age isometricamente em $(G,d_S)$, obtemos um homomorfismo natural $q_G: G \to \QI(G)$ definido por $$q_G(g):=q_{\cay(G,S)}( L_{g} )=[L_g],$$ onde $L_g$ é a aplicação de translação à esquerda definida na demonstração da Proposição~\ref{proposition1}.

Se $H$ é um grupo $(L,C)$-quasi-isométrico a $G$, sendo $q: H\to G$ uma quasi-isometria com quasi-inversa $\overline{q}$, então podemos definir um homomorfismo $\varphi: H \to \QI(G)$ dado por $ \varphi (h) = q_h := q\circ L_h \circ \overline{q} \in \QI(G)$. O homomorfismo $\varphi$
tem as seguintes propriedades:
\begin{enumerate}[(i)]
\item Existe uma constante $D = D(L,C)$ tal que $d(q_h\circ q_k, q_{hk})\leq D.$ 
\item O núcleo $\ker (\varphi)$ é quasi-finito, isto é, para todo $r>0$, o conjunto $\{h\in H \mid d(q_h, 1)\leq r\}$ é finito. Mais ainda, se $G$  age geometricamente em algum espaço de curvatura negativa, então $\ker (\varphi)$ é finito.
\item A imagem $\varphi (H)$ age quasi-transitivamente em $G$, ou seja,  para todos $g, g' \in G$ existe $h \in H$ tal que $d(q_h(g),g')\leq D_1$, onde $D_1 = D_1(L,C)$.
\end{enumerate}

\begin{exercise}
Verifique as propriedades (i) a (iii).
\end{exercise}

\section{O teorema de Milnor--Schwarz}
Um dos principais exemplos de quasi-isometrias (e de sua importância), 
que parcialmente justifica o interesse nesses mapas, vem do teorema de Milnor--Schwarz. Ele foi demonstrado inicialmente por Schwarz \cite{vsvarc1955volume} no contexto de variedades Riemannianas, e 13 anos depois de modo independente e na forma mais geral por Milnor \cite{milnor1968note}. Ambos tinham interesse em relacionar o crescimento de volume em recobrimentos universais de variedades Riemannianas compactas e o crescimento da cardinalidade de bolas no grupo fundamental dessas variedades, munido da métrica das palavras.  Efremovich foi o primeiro a observar, em \cite{efremovic1953proximity}, que essas funções cresciam na mesma taxa.
Inclusive, é comum que este teorema apareça na literatura enunciado em sua versão que captura essa relação. 

Antes de enunciar o teorema, relembramos algumas definições a respeito de espaços métricos que serão usadas nesta seção.

Um espaço métrico $(X,d)$ é dito \textit{próprio}\index{espaço métrico próprio} se, para todos $x \in X$ e $R>0$, a bola fechada $\overline{B}(x,R)$ é compacta.
Uma \textit{geodésica}\index{geodésica} em X é uma isometria $\alpha : [a,b] \to X $. Dizemos que o espaço métrico $X$ é \textit{geodésico}\index{espaço métrico geodésico} se cada par de pontos em $X$ pode ser conectado por uma geodésica.

Um \textit{grupo topológico} é um grupo $G$ munido de uma topologia tal que as operações de multiplicação e inversão sejam contínuas. Dizemos que uma ação $\mu : G\curvearrowright X$ é uma \textit{ação topológica}\index{ação topológica} se $G$ é um grupo topológico e $\mu : G \times X \to X$ é contínua.

Uma ação topológica $\mu : G\curvearrowright X$ é própria se, e somente se, para todo par de compactos $K_1$ e $K_2$ de $X$, o conjunto $G_{K_1,K_2}:=\{g \in G \mid gK_1\cap K_2 \neq \emptyset\}$ é compacto, como subespaço de $G$. Se $G$ é discreto, uma ação própria também é \textit{propriamente descontínua}\index{ação propriamente descontínua}. Note que, neste último caso, os conjuntos $G_{K_1,K_2}$ são  finitos. Além disso,  $\mu$ é \textit{cocompacta}\index{ação cocompacta} se, e somente se, existe um compacto $K \subset X$ tal que $G\cdot K:= \displaystyle \bigcup_{g\in G}gK =X.$ 

\begin{example}
A ação de um grupo $G$ finitamente gerado no seu  grafo de Cayley $\cay(G,S)$ é cocompacta (note que $G$, com a topologia dada pela métrica das palavras, é discreto e que a ação de $G$ em $\cay(G,S)$ é topológica). Basta tomar, segundo a definição acima, $K$ igual à bola fechada de raio $1$ e centro em $1_G$.
\end{example}

Relembramos que uma ação topológica $\mu : G\curvearrowright X$ sobre um espaço métrico $(X,d)$ é \textit{colimitada} \index{ação colimitada} se existirem uma constante $r>0$ e um ponto $x \in X$ tal que $G\cdot B(x,r):= \displaystyle \bigcup_{g\in G}gB(x,r) =X$.  

Claramente, se $X$ é um espaço métrico próprio então uma ação $\mu : G\curvearrowright X$ é colimitada se e somente se ela é cocompacta.

Uma ação isométrica, cocompacta e pro\-pria\-mente descontínua é chamada \textit{ação geométrica}\index{ação geométrica}. 

\begin{exercise}
Mostre que todo grupo finitamente gerado age propriamente e de forma colimitada por isometrias em um espaço métrico geodésico e próprio: seu grafo de Cayley.
\end{exercise}

Estamos agora em condições de enunciar o principal resultado desta seção:

\begin{thm}[Milnor--Schwarz] \label{MS} \index{teorema de Milnor--Schwarz}
Sejam $(X,d)$ um espaço métrico próprio e geodésico e $G$ um grupo que age geometricamente em $X$. Então
\begin{enumerate}
\item O grupo $G$ é finitamente gerado;
\item Para qualquer métrica das palavras $d_S$ em $G$ e qualquer ponto $x\in X$, a aplicação $g\in G\mapsto gx \in X$ é uma quasi-isometria.
\end{enumerate}
\end{thm}
\begin{proof}

 \textbf{Passo 1 (os geradores de $G$):}
Lembramos que uma ação é geométrica se for cocompacta, isométrica e pro\-pria\-mente
descontínua. Cada ação geométrica $G\curvearrowright X$ é colimitada, já que $X$ é próprio. Logo, existe alguma bola fechada $\overline{B}$ de raio $D>0$ e centro $x_0 \in X$, tal que $G\cdot \overline{B} = X$ (observe que, com isso, para qualquer $x\in X$, vale que $G\cdot \overline{B}(x,2D) = X$). Como $X$ é próprio, segue que $\overline{B}$ é compacta. Defina 
$$ S = \{s\in G \mid s\overline{B}\cap \overline{B}\neq \emptyset\}. $$
Note que $S$ é finito, pois a ação é própria, e que $1 \in S = S^{-1}$. Se $G=S $, então claramente $S$ gera $G$. Assumimos assim que $G \neq S$.

 \textbf{Passo 2 (o que ocorre fora do conjunto de geradores):}
Considere, conforme a Figura \ref{fig:generatorsofG},
$$2d := \inf\{d(\overline{B},g\overline{B}) \mid g \in G\setminus S\}.$$ 

\begin{figure}[!ht]
    \centering
    \begin{tikzpicture}[scale=0.8]
         \draw[thick, fill=lightgray, opacity=0.6]  (0,0) circle (1);
         \draw[thick, fill=lightgray,  opacity=0.6] (3,-1) circle (1);
         \draw[thick] (-0.7,0) node{$\overline{B}$};
         \draw (0,1.2) circle (1);
         \draw (0,-1.1) circle (1);
         \draw (0.707,0.707) circle (1);
         \draw (-1.4,1.2) node{\tiny{$s_1\overline{B}$}};
         \draw (1.7,1.6) node{\tiny{$s_2\overline{B}$}};
         \draw (-1.3,-1.4) node{\tiny{$s_3\overline{B}$}};
         \draw (3,0.2) node{$g\overline{B}$};
         \draw[dashed] (0.95, -0.32) -- (2.05,-0.68);
         \draw (1.7, -0.3) node{\tiny{$\geq 2d$}};
    \end{tikzpicture}
    \caption{definição de $d$.}
    \label{fig:generatorsofG}
\end{figure}

Fixe $ g_0 \in G\setminus S$. Da definição de $S$, segue que a distância $R = d(\overline{B},g_0\overline{B})$ será positiva. Pondo $H  = \{h \in G \mid d(\overline{B},h\overline{B})\leq R\}$, temos
$$H \subset \{g\in G \mid g (\overline{B}(x_0,D+R))\cap \overline{B}(x_0,D+R)\neq \emptyset\}.$$ 
Logo $H$ é finito. Mais ainda, 
$$\inf\{d(\overline{B},g\overline{B}) \mid g \in G\setminus S\}= \inf\{d(\overline{B},g\overline{B}) \mid g \in H\setminus S\}.$$ 
Note que o último ínfimo é tomado sobre um conjunto finito de números positivos. Portanto, existe $h_0 \in H \setminus S$ tal que $d(\overline{B} , h_0\overline{B})$ atinge esse ínfimo, que é, portanto, positivo. Concluímos que se $d(\overline{B}, g\overline{B})<2d$, então $g \in S$.

  \textbf{Passo 3 ($G$ é finitamente gerado por $S$):}
Considere uma geodésica $[x, gx] \subset X$, onde $g \in G$  e $x\in X$, e defina $$k =\Big\lfloor\dfrac{d(x,gx)}{d}\Big\rfloor.$$ 
Então existe uma sequência contendo $k+2$ pontos sobre  $[x, gx]$, digamos $y_0=x$, $y_1, \ldots , y_k$, $y_{k+1}=gx$, tal que $d(y_i,y_{i+1})\leq d$ para $i=0,\ldots, k.$ Para cada $i$,  seja $h_i \in G$ tal que $y_i \in h_i \overline{B}$ e tome $h_0=1$, $h_{k+1}=g$. Como $$d(\overline{B},h_i^{-1} h_{i+1} \overline{B}) = d(h_i\overline{B}, h_{i+1} \overline{B}) \leq d(y_i, y_{i+1})\leq d,$$ segue do passo 2 que $h_i^{-1} h_{i+1}= s_i \in S$. Assim, $g$ pode ser escrito como $g = s_0s_1\ldots s_k$ com $s_i\in S$.

\begin{figure}[!ht]
    \centering
    \begin{tikzpicture}[scale=0.8]
         \draw  (0,0) circle (0.9);
         \draw  (6,2) circle (0.9);
         \draw (-0.6,0) node{$\overline{B}$};
         \draw (0,-0.2) node{$x$};
         \draw (0,0) node{\tiny{$\bullet$}};
          \draw (6.4,2) node{$gx$};
         \draw (6,2) node{\tiny{$\bullet$}};
          \draw (0,0) .. controls (2,1) and (3,1) .. (6,2); 
          
         \draw (0.9,0.7) node{\tiny{$y_1$}};
         \draw (1,0.45) node{\tiny{$\bullet$}}; 
         \draw (1,0.45) circle (0.9);
         \draw (0.7,1.6) node{\tiny{$h_1(\overline{B})$}};

         \draw (2.1,1) node{\tiny{$y_2$}};
         \draw (2,0.8) node{\tiny{$\bullet$}}; 
         \draw (2,0.8) circle (0.9);
         \draw (2.7,-0.3) node{\tiny{$h_2(\overline{B})$}};

        \draw (3.9,1.2) node{$\cdot$};
        \draw (4,1.23) node{$\cdot$};
        \draw (4.1,1.26) node{$\cdot$};

         \draw (6,0.8) node{\tiny{$h_{k+1}(\overline{B})$}};
    \end{tikzpicture}
    \caption{O grupo $G$ é gerado por $S$.}
    \label{fig:SgeneratesG}
\end{figure}
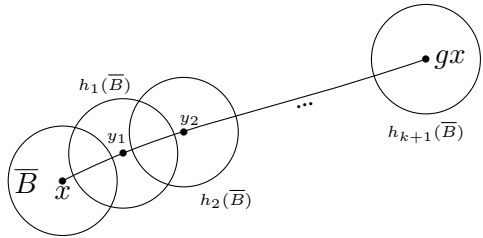

  \textbf{Passo 4 (a quasi-isometria):}
Como todas as métricas das palavras em $G$ são bi-Lipschitz equivalentes, basta provar a parte $(2)$ do teorema para a métrica das palavras $d_S$ em $G$, onde $S$ é o conjunto gerador finito encontrado acima a partir do ponto escolhido~$x$. O espaço $X$ está contido, pelo fato da ação ser colimitada, numa vizinhança $2D$-tubular da órbita $Gx$. Portanto, resta provar que o mapa órbita $f:G \to X$, dado por $g \mapsto gx$ é um mergulho quasi-isométrico. Vimos no passo anterior que $g = s_0s_1\ldots s_k$, onde $s_i \in S$. Logo, pela definição da métrica das palavras e de $k$, obtemos:
$$|g|_S\leq k+1 \leq \dfrac{1}{d}d(x,gx)+1.$$ 
Por outro lado, se $|g|_S = m \leq k+1$ e $t_1\ldots t_m =w$ é uma palavra reduzida em $S$ que representa $g$, obtemos pela desigualdade triangular: $$d(x,gx)=d(x,t_0t_1\ldots t_m x)$$ $$ \leq d(x,t_1x)+d(t_1x,t_1t_2 x)+\ldots +d(t_1t_2\ldots t_{m-1}x,t_1t_2\ldots t_{m}x)$$ $$= d(x,t_1x)+d(x,t_2 x)+\ldots +d(x,t_{m}x)\leq (2D)m= 2D|g|_S,$$ onde $2D> \displaystyle \max_{s \in S} d(x,sx)$, já que se $s\in S$, temos $s\overline{B}\cap \overline{B} \neq \emptyset$. 

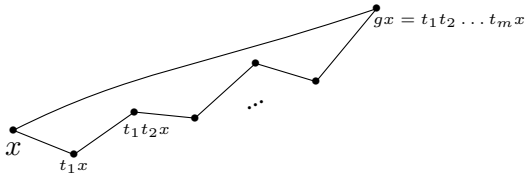
\begin{figure}[!ht]
    \centering
    \begin{tikzpicture}[scale=0.8]
         \draw (0,-0.3) node{$x$};
         \draw (0,0) node{\tiny{$\bullet$}};
          \draw (7.2,1.8) node{\tiny{$gx = t_1t_2\ldots t_m x$}};
         \draw (6,2) node{\tiny{$\bullet$}};
          \draw (0,0) .. controls (2,1) and (3,1) .. (6,2); 

        \draw (1,-0.6) node{\tiny{$t_1x$}};
        \draw (1,-0.4) node{\tiny{$\bullet$}};

        \draw (2.2,0) node{\tiny{$t_1t_2x$}};
        \draw (2,0.3) node{\tiny{$\bullet$}};

        \draw (3,0.2) node{\tiny{$\bullet$}};
        \draw (4,1.1) node{\tiny{$\bullet$}};
        \draw (5,0.8) node{\tiny{$\bullet$}};

        \draw (0,0)-- (1,-0.4) -- (2,0.3) -- (3,0.2) -- (4,1.1) -- (5,0.8) -- (6,2);

        \draw (3.9,0.35) node{$\cdot$};
        \draw (4,0.4) node{$\cdot$};
        \draw (4.1,0.45) node{$\cdot$};
    
    \end{tikzpicture}
\end{figure}

Assim, para todo $g \in G$, vale 
$$d\,d_S(1,g)-d\leq d(x,gx)\leq (2D) d_S(1,g).$$
Como tanto a métrica das palavras $|\cdot|_S=d_S(1,\cdot)$ como a métrica $d(\cdot,\cdot)$ em $X$ são invariantes  em relação à ação de $G$ à esquerda, na desigualdade acima, $1$ pode ser substituído por qualquer elemento $h \in G$.

Assim, tomando $L=\max\{2D,d\}$, concluímos que o mapa $f$ é um $(L,d)$-mergulho quasi-isométrico, cuja imagem é $2D$-densa em $X$. Portanto, $f$ é uma quasi-isometria.

\end{proof}

\begin{corollary} \label{fundgroup}
Dada uma variedade Riemaniana $M$ compacta, considere seu recobrimento universal  $\tilde{M}$, dotado da métrica Riemanniana de pull-back. Então o grupo fundamental $\pi_1(M)$ age geometricamente em $\tilde{M}$.
Pelo teorema de Milnor--Schwarz, o grupo $\pi_1(M)$ é finitamente gerado, e o espaço métrico $\tilde{M}$ é quasi-isométrico a $\pi_1 (M)$, munido com alguma métrica das palavras.
Em particular, podemos usar o Teorema~\ref{MS} para obter restrições a respeito dos tipos de métricas Riemannianas que uma certa variedade pode ter.
\end{corollary}

\begin{corollary}\label{Mil} Seja $G$ um grupo finitamente gerado.
\begin{enumerate}
\item Se $H$ é um subgrupo  de índice finito de $G$, então $H$ também é finitamente gerado. Além disso, $H$ é quasi-isométrico a $G$.
\item Dado um subgrupo normal e finito $N\triangleleft G$, então os grupos $G$ e $G/N$ são quasi-isométricos.
\end{enumerate}
\end{corollary}

\begin{proof}
É claro que a ação de $H$ no espaço métrico próprio e geodésico $\cay(G,S)$ é isométrica e propriamente descontínua. Agora, como $H$ tem índice finito, segue que $G =\displaystyle \bigcup_{i=1}^ng_iH$, para finitos elementos $g_1 = 1, g_2, \ldots, g_n \in G$. Escrevendo $R=\max{|g_i|_S}$, obtemos que  $\cay(G,S)= H\cdot \overline{B}(1,R)$ e portanto a ação de $H$ em $\cay(G,S)$ é colimitada. Assim, o item 1 segue do teorema de Milnor--Schwarz.

Como $G$ é finitamente gerado, então $G/N$ também é gerado por um conjunto finito, digamos $S'$, pela Proposição \ref{prop:fggroups}. Sejam $\mu : G/N \curvearrowright \cay(G/N,S')$ a ação de $G/N$ por translação à esquerda e  $\pi: G \to G/N$ a projeção canônica. Podemos, usando $\mu$ e $\pi$, definir uma ação induzida $\mu': G \curvearrowright \cay(G/N, S')$, a qual é uma ação geométrica. Com efeito, da definição da distância em $G/N$  segue que esta é uma ação por isometrias. Seja $B=B(N,1)$ é uma bola em $\cay(G/N,S')$ então  $G\cdot B = \cay(G/N,S')$ e portanto a ação é colimitada. Por fim,  sendo $N$ finito, podemos mostrar que esta ação é própria. Daí, pelo teorema de Milnor--Schwarz concluímos que $G$ e $G/N$ são quasi-isométricos.
\end{proof}

\begin{corollary}
Sejam $(X,d_i)$, $i=1,2$ espaços métricos próprios e geodésicos. Suponha que exista uma ação de um grupo $G $ em $ X$, a qual é geométrica com respeito a ambas as distâncias $d_1$ e $d_2$. Então a aplicação identidade $$ id: (X,d_1)\to (X,d_2)$$ é uma quasi-isometria. Em particular, cada geodésica em $X$ com relação a $d_1$ é uma quasi-geodésica com relação a $d_2$.
\end{corollary}

\begin{proof} 
Pelo teorema de Milnor--Schwarz, $G$ é finitamente gerado e, fixado $x_0\in X$ e escolhendo uma métrica das palavras $d_S$ com respeito a algum conjunto de geradores $S$ de $G$, para cada $i=1,2$, o mapa $f_i:G \to (X,d_i)$,  dado por $g \mapsto g x_0$ é uma quasi-isometria, com quasi-inversa denotada por $\Bar{f}_i$.
Afirmamos que $d_2(x,f_2\circ \overline{f}_1(x))\leq K$, para alguma constante $K>0$. Para ver isso, basta mostrarmos a desigualdade em $x$ num subconjunto $A$-denso em $(X,d_1)$. Como $f_1$ é mergulho quasi-isométrico, $f_1(G)$ é $A$-denso em $(X,d_1)$. Assim, temos:
\begin{eqnarray*}
d_2(f_1(g),f_2\circ \overline{f}_1(f_1(g)) &=&d_2(f_2(g),f_2\circ \overline{f}_1(f_1(g))\\
& \leq & L d_S(g, \overline{f}_1(f_1(g)))+C \\
& \leq & LD+C, 
\end{eqnarray*}
onde usamos que $f_2$ é $(L,C)$-quasi-isometria e $d_S(g, \overline{f}_1(f_1(g))) \leq D$, para todo $g\in G$.
\end{proof}

Apresentaremos agora algumas consequências do Teorema de Milnor--Schwarz com respeito a noções de comensurabilidade e isomorfismo virtual de grupos.

\begin{definition}\label{def:comm}
Dois grupos $G_1$ e $G_2$ são ditos \textit{comensuráveis} \index{grupos comensuráveis} se existem subgrupos de índice finito $H_i\leq G_i, \, i=1,2$, tais que $H_1$ é isomorfo $H_2$.
\end{definition}

\begin{definition}\label{def:vi}
Dois grupos $G_1$ e $G_2$ são ditos \textit{virtualmente isomorfos} \index{grupos virtualmente isomorfos} (denotamos por $G_1 \vi G_2$) se existirem subgrupos de índice finito $H_i \leq G_i$ e subgrupos finitos e normais $F_i \triangleleft H_i$, para cada $i=1,2$, tais que  os quocientes $H_1/F_1$ e $H_2/F_2$ são isomorfos. Tal isomorfismo $\varphi: H_1/F_1 \to H_2/F_2$, caso exista, é chamado \textit{isomorfismo virtual} entre $G_1$ e $G_2$. 
\end{definition}

\begin{exercise}
    Verifique que ``ser virtualmente isomorfo'' é uma relação de equivalência.
\end{exercise}

\begin{exercise}
Mostre que um homomorfismo $f:G\to H$, entre dois grupos finitamente gerados $G$ e $H$, é uma quasi-isometria se, e somente se, seu núcleo $\mathrm{Ker}(f)$ é finito e sua imagem $\mathrm{Im}(f)$ tem índice finito em $H$.
\end{exercise}

\begin{corollary}
Se $G_1$ e $G_2$ são dois grupos finitamente gerados virtualmente isomorfos então $G_1$ e $G_2$ são quasi-isométricos.
\end{corollary}
\begin{proof}
Como $G_1$ e $G_2$ são virtualmente isomorfos, existem  subgrupos de índice finito $H_i\leq G_i$, $i=1,2$  e subgrupos normais finitos $F_i \triangleleft H_i$, $i=1,2$ tais que os quocientes $H_1/F_1$ e $H_2/F_2$ são isomorfos, portanto quasi-isométricos. Pelo item $2$ do Corolário~\ref{Mil} temos que $H_i/F_i$ é quasi-isométrico a $H_i$ e pelo item $1$ desse mesmo corolário, temos $G_i$ quasi-isométrico a $H_i$. Assim, concluímos que  $G_1$ é quasi-isométrico a $G_2.$
\end{proof}

Em geral, a existência de uma quasi-isometria entre dois grupos não implica que eles sejam virtualmente isomorfos. A seguir, apresentaremos alguns exemplos desse fenômeno, cada um exigindo uma técnica diferente. Não abordaremos todos os detalhes, mas mostraremos as ideias envolvidas nas construções.

\begin{example}
Seja $\Hyp (2,\mathbb{Z})$ o conjunto das matrizes  $A \in \SL_2(\mathbb{Z})$ diagonalizáveis sobre $\mathbb{R}$ tais que $A^2\neq I$. Para cada $A \in \Hyp (2,\mathbb{Z})$, considere a ação de $\mathbb{Z}\curvearrowright \mathbb{Z}^2$ tal que $1 \in \mathbb{Z}$ age como $A$ em $\mathbb{Z}^2$. Seja $G_A=\mathbb{Z}^2\rtimes_A \mathbb{Z}$ o produto semidireto associado. Podemos ver $G_A$ como grupo fundamental da 3-variedade 
$$T_A = \mathbb{T}^2\times [0,1]/ \big((p,0) \sim (f_A(p), 1)\big),$$ 
onde  $\mathbb{T}^2 $ é o toro $\mathbb{R}^2 / \mathbb{Z}^2$ e $f_A: \mathbb{T}^2 \to \mathbb{T}^2$ é induzida pela matriz $A$. 

Deixamos como exercício verificar  que $\mathbb{Z}^2$ é o único subgrupo maximal normal abeliano de $G_A.$ 

Pela diagonalização, $A \sim \left( \begin{array}{cc}
\lambda & 0 \\ 
0 & \lambda^{-1}
\end{array}\right)$ com $\lambda > 1$. Isso implica que  existe um mergulho  
$$G_A\hookrightarrow \Sol_3 := \mathbb{R}^2\rtimes_D \mathbb{R},$$ 
onde 
$$
D(t)=
 \left( \begin{array}{cc}
e^t & 0 \\ 
0 & e^{-t}
\end{array}\right): \mathbb{R} \to \mathrm{Aut}(\mathbb{R}^2),\ 
D(t_0)=
 \left( \begin{array}{cc}
\lambda & 0 \\ 
0 & \lambda^{-1}
\end{array}\right).
$$
A imagem de $G_A$ é discreta e cocompacta em $\Sol_3$, além disso, $G_A$ é livre de torção pois ele é isomorfo a um subgrupo do grupo sem torção. O grupo $\Sol_3$ tem uma métrica Riemanniana invariante a esquerda. Portanto, $G_A$ age isometricamente em $\Sol_3$, munido dessa métrica. Assim, cada $G_A$ como acima é quasi-isométrico a $\Sol_3$. Construiremos agora dois grupos $G_A$, $G_B$ dos tipos mencionados acima e que não são virtualmente isomorfos entre si. Escolha duas matrizes $A, B\in \Hyp(2,\mathbb{Z})$ tais que para cada $n, m \in \mathbb{Z}\setminus\{0\}$, $A^n$ não seja conjugada a $B^m$ em $\SL_2(\mathbb{R})$. Por exemplo, 
$$A=
 \left( \begin{array}{cc}
2 & 1 \\ 
1 & 1
\end{array} \right), \quad B=
 \left( \begin{array}{cc}
3 & 2 \\ 
1 & 1
\end{array} \right).$$  
A propriedade sobre as potências de $A$ e $B$ seguem considerando-se os autovalores de $A$ e $B$ e observando que os corpos que eles geram são extensões quadráticas distintas de $\Q$.
Vamos verificar que $G_A$ não é virtualmente isomorfo a $G_B$. Em primeiro lugar, uma vez que ambos $G_A$, $G_B$ são livres de torção, basta mostrar que eles não são comensuráveis, isto é, não contêm subgrupos de índice finito isomorfos. Seja $H_A$ um subgrupo de índice finito em $G_A$. Então $H_A$ intersecta o subgrupo normal, livre abeliano, de posto $2$ de $G_A$ ao longo de um subgrupo  $L_A$  livre abeliano de posto $2$. Ou seja $L_A=H_A \cap \mathbb{Z}^2.$

A imagem de $H_A$ pela projeção $G_A\to G_A/\mathbb{Z}^2 = \Z$  deve ser um subgrupo cíclico infinito, gerado por algum $n \in \N$. Portanto, $H_A$ é isomorfo a $\Z^2 \rtimes_{A^n} \Z$. Pela mesma razão, $H_B \cong \Z^2 \rtimes_{B^m} \Z$. Qualquer isomorfismo $H_A\to H_B$ deve levar $L_A$ isomorficamente em $L_B$. Contudo, isto implicaria que $A^n$ é conjugado com $B^m$,o que é  uma contradição.
\end{example}

\begin{example}
Seja $S$ uma superfície compacta de gênero $g\geq 2$. Considere $G_1 =\pi_1(S)\times \Z$, $M$ o espaço total do fibrado unitário sobre $S$ e $G_2 = \pi_1(M)$. Temos então a sequência exata $$1\to \Z\to G_2\to \pi_1(S)\to 1,$$
$$G_2 = \langle a_1, b_1,\cdots, a_g, b_g, t\mid [a_1, b_1] \cdots [a_g, b_g]t^{2n-2}, [a_i, t], [b_i, t], \forall i \rangle.$$

Deixamos ao leitor verificar que a passagem para qualquer subgrupo de índice finito em $G_2$ não o torna uma extensão central trivial do grupo fundamental da superfície hiperbólica. Portanto, $G_1$ e $G_2$ não são virtualmente isomorfos. Por outro lado, como o grupo $\pi_1(S)$ é hiperbólico, no sentido do Capítulo \ref{cap8}, segue que os grupos $G_1$ e $G_2$ são quasi-isométricos.
De fato, nesta última afirmação usamos o seguinte teorema (veja \cite[Teorema~11.159]{kd} ou \cite{gersten1992bounded} para o caso particular $A=\Z$):

\begin{thm} Dada uma co-extensão central
$$1\to A\to G\to H\to 1,$$ onde $A$ é um grupo abeliano finitamente gerado e $H$ um grupo hiperbólico, tem-se que $G$ é quasi-isométrico a $A\times H$. 
\end{thm}

\end{example}

\begin{example}
Em seu trabalho \cite{thurston1982three}, Thurston deixou uma lista com 24 problemas envolvendo  3-variedades e grupos Kleinianos, dos quais apenas um ainda está em aberto atualmente: ``\textit{Mostrar que os volumes de 3-variedades hiperbólicas não são todos racionalmente relacionados.}'' Suponhamos que  existam $M_1$ e $M_2$ variedades compactas de dimensão $3$ e curvatura $-1$ tais que o quociente de seus  volumes  seja irracional. Os grupos fundamentais dessas variedades serão quasi-isométricos a $\mathbb{H}^3$, pelo Corolário~\ref{fundgroup} e, por transitividade, são quasi-isométricos entre si. Se $\pi_1(M_1)$ é comensurável a $\pi_1(M_2)$, então existem subgrupos $H_i \leq \pi_1(M_i)$, de índice finito $n_i$, os quais são isomorfos. Logo, existem recobrimentos finitos $\overline{M}_i$, com $n_i$ folhas, das variedades $M_i$, tais que $\pi_1(\overline{M}_i) \cong H_i$. Assim, pelo teorema da rigidez de Mostow (vide \cite{mostow1968quasi}), $\overline{M}_1$ é isométrico a $\overline{M}_2$ e então $\vol(\overline{M}_1)= \vol(\overline{M}_2)$. No entanto, $\vol(\overline{M}_i) = n_i \vol(M_i)$, o que contradiz a irracionalidade do quociente dos volumes das variedades $M_i$. De fato, é suficiente considerar duas $3$-variedades hiperbólicas compactas não comensuráveis (cuja existência é conhecida) e, usando a rigidez de Mostow, deduzir que seus grupos não são virtualmente isomorfos, enquanto são quasi-isométricos por Milnor--Schwarz.
\end{example}

\begin{exercise}
Mostre que para todos $m,n\geq 2$, os grupos  $F_m$ e $F_n$ são  quasi-isométricos e virtualmente isomorfos.
\end{exercise}

O resultado apresentado a seguir, devido a Gromov \cite{gromov1993asymptotic}, fornece um critério alternativo para verificar se dois grupos são quasi-isométricos. 

\begin{definition} Sejam $G_1, G_2$ dois grupos. Um \textit{acoplamento topológico} \index{acoplamento topológico} de $G_1$ e $G_2$ é um espaço metrizável $X$, localmente compacto, junto com duas ações $\rho_i : G_i \curvearrowright X$ que comutam, são cocompactas e propriamente descontínuas. Duas ações $\rho_i : G_i \to \mathrm{Homeo}(X)$ comutam se $\rho_1(g_1)\rho_2(g_2)= \rho_2(g_2)\rho_1(g_1)$, para quaisquer $g_i \in G_i$.
\end{definition}

\begin{thm}
\label{thm:topologicalcoupling}
Dois grupos finitamente gerados são quasi-iso\-mé\-tricos se, e somente se, existe um acoplamento topológico entre eles.
\end{thm}

\begin{exercise}
    Forneça uma prova do Teorema \ref{thm:topologicalcoupling} (\textit{sugestão:} veja  \cite[Teorema 0.2.C$_2$]{gromov1993asymptotic} ou \cite[páginas 97 e 98]{de2000topics}).
    \end{exercise}

\begin{definition}
Dizemos que dois grupos $G_1$ e $G_2$ possuem um \textit{modelo geométrico comum}\index{modelo geométrico comum} se existe um espaço métrico quasi-geodésico próprio $(X,d)$ no qual ambos, $G_1$ e $G_2$, agem geometricamente.
\end{definition}

Note que o teorema de Milnor--Schwarz implica que, se $G_1$ e $G_2$ tem um modelo geométrico comum, então eles são quasi-isométricos. No entanto, não vale a recíproca:

\begin{thm}[\cite{mosher2003quasi}]
Dados $p\neq q$ primos ímpares distintos, considere os grupos $$G_1= \Z_p \ast \Z_p, \quad G_2= \Z_q \ast \Z_q,$$ os quais são virtualmente isomorfos a $F_2$. Então $G_1$ e $G_2$ são quasi-isométricos, mas não tem modelo geométrico comum.
\end{thm}


\section{Um levantamento sobre rigidez quasi-iso\-mé\-trica}

Já conseguimos transformar grupos finitamente gerados em espaços métricos, tratando-os como objetos geométricos quasi-isométricos a eles, por exemplo, seus grafos de Cayley. Gostaríamos agora de saber até que ponto somos capazes de reconstruir o grupo $G$, como objeto algébrico, a partir de alguns de seus modelos geométricos, definidos a menos de quasi-isometrias. Em outras palavras, gostaríamos de saber até que ponto o mapa de ``geometrização''
\begin{equation*}
    geo:\{ \mbox{Grupos Finitamente Gerados}\} \to \{ \mbox{Espaços Métricos}\}/\QI
\end{equation*}
é injetivo.

\begin{definition}
Um grupo finitamente gerado $G$ é dito \textit{$\QI$-rígido}\index{rigidez quasi-isométrica} se todo grupo $G_0$ que é quasi-isométrico a $G$  satisfaz $G_0 \vi G$.

Já uma classe de grupos $\mathcal{C}$ é dita \textit{$\QI$-rígida} se, para todo $G'$ com $G' \qi G \in \mathcal{C}$, existe algum $G'' \in \mathcal{C}$ tal que $G''\vi G'$ (veja a Definição~\ref{def:vi} para a relação de  isomorfismo virtual $\vi$ ).
\end{definition}

\begin{definition}
Uma propriedade $\mathcal{P}$ de grupos é dita \textit{geométrica} ou  \textit{$\QI$-invariante} se a classe de todos os grupos que satisfazem $\mathcal{P}$ é $\QI$-rígida.\index{propriedade geométrica}
\end{definition}

Boa parte dos problemas em teoria geométrica de grupos foca em estudar rigidez quasi-isométrica. Listamos a seguir alguns exemplos de grupos e classes de grupos $\QI$-rígidas:
\begin{itemize}
    \item Todos os grupos livres não abelianos (J.~Stallings, \cite{Stal68});
    \item Todos os grupos livres abelianos (M.~Gromov \cite{Gromov81} e P.~Pan\-su \cite{Pan83});
    \item Classe dos grupos nilpotentes (M.~Gromov \cite{Gromov81});
    \item A classe de todos os grupos fundamentais de superfícies fechadas, isto é, superfícies compactas e sem bordo, de gênero $g\geq 2$  (uma combinação dos trabalhos de P.~Tukia \cite{Tuk88, Tuk94}, de  D.~Gabai \cite{Gab91}, e de A.~Casson e D.~Jungreis \cite{CJ94});
    \item Classe dos grupos fundamentais de variedades fechadas de dimensão $3$ (uma combinação dos trabalhos de R.~Schwartz \cite{Sch95}; de M.~Kapovich e B.~Leeb \cite{KL97}; de A.~Eskin, D.~Fisher e K.~Whyte \cite{EFW12, EFW13} e, o mais importante, da solução da conjectura de geometrização por G.~Perelman);
    \item Classe de grupos finitamente apresentados (E. Ghys e P. de la Harpe
 \cite[Seção~10.4]{GH89});
    \item Classe dos grupos hiperbólicos (veja o Capítulo \ref{cap8}; esse item segue de uma aplicação do Lema de Morse para quasi-geodésicas, como no Corolário \ref{cor:ripsQI});
    
    \item Classe dos grupos amenáveis (pode ser verificado em E. Ghys e P. de la Harpe
 \cite[Seção~10.4]{GH89});
 
    \item Classe dos grupos fundamentais de variedades fechadas hiperbólicas de dimensão $n\geq 3$ (P.~Tukia \cite{Tuk86});
    \item Classe dos reticulados cocompactos em um grupo de Lie simples não compacto, por exemplo, em $G=\mathrm{SL}(n,\R)$ (esta é uma combinação dos trabalhos de P.~Tukia \cite{Tuk86},  P.~Pansu \cite{Pan89},  R.~Chow \cite{Cho96},  B.~Kleiner e B.~Leeb \cite{KlL97} e de A.~Eskin e B.~Farb \cite{EF97});
    
    \item Todo subgrupo discreto cofinito e não cocompacto em um grupo de Lie simples não compacto. Por exemplo, se tivermos $\Gamma \qi \mathrm{SL}(n,\Z) \leq \mathrm{SL}(n,\R)$, então $\Gamma \vi \mathrm{SL}(n,\Z) $ (R.~Schwartz \cite{Sch95, Sch96} e A.~Eskin \cite{Esk98}).
\end{itemize}

Intuitivamente, quanto mais perto um grupo ou uma classe de grupos está de ser um grupo de Lie simples, maior a chance de haver $\QI$-rigidez.
Provaremos a rigidez quase-isométrica de alguns desses exemplos mais adiante neste livro, depois de desenvolver as ferramentas apropriadas.

\medskip

Por outro lado, temos os seguintes grupos ou classes de grupos que não são $\QI$-rígidos:
\begin{itemize}
    \item Todo grupo fundamental de uma  variedade hiperbólica fechada de dimensão $n\geq 3$: para um dado $n$, existe uma quantidade enumerável de classes distintas por $\vi$ destes grupos e, pelo teorema de Milnor--Schwarz, eles são todos quasi-isométricos entre si.
    \item A classe dos grupos finitamente gerados solúveis (A.~Erschler \cite{Er00});
    \item A classe dos grupos simples finitamente apresentados: $F_2\times F_2$ é $\QI$ a um grupo simples (M.~Burger e S.~Mozes \cite{BM00});
    \item A classe de grupos residualmente finitos (M.~Burger e S.~Mozes \cite{BM00}).
\end{itemize}

\chapter{Fins de um espaço topológico}
\label{cap5}

\section{O espaço de fins de um espaço topológico}
Nesta seção vamos introduzir a noção de fins de um espaço topológico X. Este é o mais antigo invariante por quasi-isometrias que se conhece. Podemos pensar intuitivamente na seguinte motivação: um dos invariantes topológicos mais simples é o número de componentes conexas de um espaço, ou ainda o número $\pi_0(X)$ de componentes conexas por caminhos de $X$. Suponha, no entanto, que trabalhemos com um espaço conexo. Poderíamos pensar num eventual invariante topológico que envolva o  número de componentes conexas do complementar de um ponto, ou ainda, de um conjunto finito nesse espaço. De fato, se um espaço pode ser desconectado removendo um de seus pontos, e outro não, então esses dois espaços não podem ser homeomorfos. 

No contexto de geometria grosseira, pontos em espaços métricos são indistinguíveis de conjuntos limitados. É então natural olharmos para as componentes conexas do complementar de conjuntos limitados, como as bolas em espaços métricos. Para espaços topológicos em geral, veremos que é possível trabalhar com compactos no lugar de bolas, desde que o espaço admita uma exaustão por compactos.

\subsection{Número de fins de um espaço}
Seja $X$ um espaço topológico não vazio,  segundo contável, conexo, localmente compacto, localmente conexo por caminhos e Hausdorff. Podemos considerar, por exemplo, qualquer espaço métrico próprio e geodésico. 

Para cada subconjunto $B \subset X$, denote por $B^c$ o seu complementar $X\setminus B$. Se $B\subset X$ é fechado, denote por $U_B$ a união de todas as componentes conexas por caminhos de $B^c$ que não são relativamente compactas em $X$. Caso $X$ seja um espaço métrico, $U_B$ pode ser tomado como a união de todas as componentes conexas por caminhos ilimitadas de $B^c$. Assim, para $B\subset X$ fechado, defina $\pi_0^u(B^c):= \pi_0(U_B)$.

\begin{definition}
    O \textit{número de fins} \index{número de fins de um espaço topológico} de $X$ é definido como o supremo, tomado sobre todos os subconjuntos compactos $K \subset X$, das cardinalidades  $\pi_0^u(K^c)$. Esse número será denotado por $\eta(X)$.
\end{definition}

\begin{figure}[h!]
    \centering
    \begin{tikzpicture}[scale=0.8]
      \draw (-3,1.5) .. controls (0,0) and (0,0) .. (-3,-2); 
      \draw (5,1) .. controls (3,0) and (3,0) .. (5,-1); 
      \draw (-3,2.5) .. controls (1,1) and (1,1) .. (5,1.5); 
      \draw (-3,-2.5) .. controls (1,-1) and (1,-1) .. (5,-1.3); 

    \draw[dashed] (0.5,1.3) .. controls (0.1,0) and  (0.1,0).. (0.5,-1.3);
    \draw[dashed] (1.5,1.1) .. controls (2,0) and  (2,0).. (1.5,-1.1);
    \draw (1,0) node{$K_1$};

    \draw (5,-2) node{$X$};
   
    \draw[thick, dotted]  (-1,1.7) to [bend right] (-1.3,0.6);
    \draw[thick, dotted]  (-1,-1.7) to [bend left] (-1.3,-0.8);
   \draw[thick, dotted] (2.5,1.2) .. controls (2.7,0) and  (2.7,0).. (2.5,-1.2);
    \draw (-0.5,0.7) node{$K_2$};

    \end{tikzpicture}
    \caption{Um espaço com número de fins igual a 4.}
    \label{fig:fins}
\end{figure}
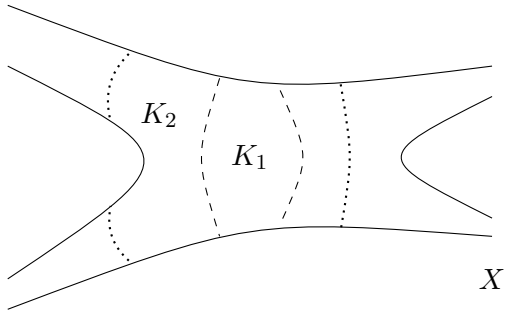

Na próxima seção, definiremos o espaço de fins de um  espaço topológico $X$, denotado por  $e(X)$. Este é também um espaço topológico, cuja cardinalidade satisfaz $\card(e(X)) \geq \eta(X)$, valendo a igualdade se algum deles for finito. Em teoria de grupos, veremos que $e(x)$ terá a cardinalidade do contínuo se for infinito. No entanto, nosso interesse com respeito a grupos estará em investigar a finitude ou infinitude do número de fins, e sua cardinalidade no caso finito. Assim, a distinção entre $\eta(X)$ e $\card(e(X))$ não será tão relevante.
\begin{example}
    \begin{enumerate}[(1)]
        \item $\eta(X) = 0$ se, e somente se, $X$ é compacto;
        \item $\eta(X) = 1$ se, e somente se,  $X$ é não compacto e, para cada compacto $K\subset X$, $K^c$ tem apenas uma componente ilimitada.
    \end{enumerate}
\end{example}

\begin{exercise} Mostre que 
    \begin{enumerate}[(1)]
        \item $\eta(\R) = 2$;
        \item $\eta(\R^n) = 1$, se $n \geq 2$;
        \item Se $T$ é uma árvore regular de valência finita $n\geq 3$, então \mbox{$\eta(T) = \infty$}.
    \end{enumerate}
\end{exercise}

\begin{lemma} \label{lemafins}
    O número de fins $\eta(X)$ de um espaço métrico próprio  e geodésico $X$ é um invariante por quasi-isometrias.
\end{lemma}

\begin{proof}
    Seja $f:X \to Y$ uma $(L,A)$-quasi-isometria entre dois espaços métricos próprios e geodésicos.  
    Suponha que $\eta(X) \geq n$, para algum $n\in \N$, ou seja, existe uma bola $B = \overline{B}(x,R) \subset X$ tal que $B^c$ consiste de pelo menos $n$ componentes ilimitadas. Assim, $f(B)$ será limitada, já que $f$ é quasi-isometria. Além disso, a imagem de qualquer componente ilimitada $C \subset B^c$ é também ilimitada.
    
    No entanto, $f(C)$ pode não ser conexa ou não estar contida no complementar de $f(B)$. Além disso, as imagens de componentes conexas distintas de $B^c$ podem estar contidas numa mesma componente conexa de $f(B)^c$ em $Y$.

    Vamos lidar com cada um desses problemas separadamente. Seja $B'= B(x, R')$ outra bola com $R'>R$.
    \begin{enumerate}
        \item Se $R'-R >t$, onde $(L^{-1}t-A) > 0$, então cada componente conexa $C'$ de $(B')^c$ tem imagem disjunta de $f(B)$. Como o número de fins independe do compacto que escolhemos, podemos trocar $R$ por algum $R'> R+AL$ e assim, $f(C')$ está contida em $f(B')^c$ para qualquer componente conexa $C' \subset (B')^c$. 
        
        \item Se $x_1, x_2 \in X $ são tais que  $d_X(x_1,x_2) \leq 1$, então vale que $d_Y(f(x_1), f(x_2))\leq L+A =: r$. Assim, $\mathcal{N}_r(f(C'))$ é conexa por caminhos em $Y$. Para que esta vizinhança seja disjunta de $f(B)$, é suficiente trocar $R$ por alguma constante $R'$ tal que $L^{-1}R' - A >r$. Assim, se  $R'> R + L(A+r)$, teremos  $f(C')$ conexa por caminhos e disjunta de $f(B)$.
        \item O último problema é um pouco mais sutil. Até o momento, usamos apenas que $f$ é um mergulho quasi-isométrico. Neste ponto, a hipótese da sobrejetividade grosseira de $f$ será importante (veja o  Exercício~\ref{exercfins}).
    
Sejam $C_1$ e $C_2$ componentes conexas por caminhos ilimitadas distintas de $B^c$ e $x_i \in C_i$, para $i=1,2$, pontos mapeados em $y_i=f(x_i)$ que pertencem à mesma componente conexa por caminhos de $(\mathrm{cl}(f(B)))^c$. Considere um caminho $p$ em $(\mathrm{cl}(f(B)))^c$ ligando $y_1 $ a $ y_2$. Seja $\Bar{f}$ uma quasi-inversa de $f$.
Quando tomamos a composição $\Bar{f}\circ p$, este não precisa ser um caminho em $X$, mas podemos torná-lo um caminho grosseiro. 

Considere uma sequência de pontos $z_1,\ldots,z_n$ em $\mathrm{Im}(p)$ com $z_1 = y_1$, $z_n=y_2$ e $\dist(z_i,z_{i+1})\leq 1$ para $i=1, \ldots,n-1$. 

Os pontos $w_i:= \Bar{f}(z_i) \in X$ satisfazem $\dist(w_i,w_{i+1})\leq L+A$, para $i=1, \ldots,n-1$. Temos também que $\dist(w_0, w_1)\leq A$ e $\dist(w_n, w_{n+1})\leq A$, onde  $w_0 := x_1$ e $w_{n+1} := x_2$. 

Conectando os pontos consecutivos $w_i$ e $w_{i+1}$   $(i = 0,\ldots, n)$ por segmentos geodésicos em $X$, obtemos um caminho $q$, que conecta $x_1$ a $x_2$. Substituímos por $q$ o caminho $\Bar{f}\circ p$ (possivelmente descontínuo) em $X$. 

Precisamos garantir que a imagem do caminho contínuo $q$ é disjunta de $B$, o que nos leva a uma contradição, já que assumimos que $C_1$ e $C_2$ eram componentes conexas por caminhos disjuntas de $B^c$. Se a imagem de $p$ está no complementar de $B(y,r')$, onde $y = f(x)$, então $\dist_X(x, \mathrm{Im}(q)) \geq R'' := L^{-1}r' - 3A -L$. 
Escolhemos $r'$ tal que $R'' > R$. Assim, se $x_1$ e $x_2$ estão suficientemente distantes de $x$ (já que $C_1,C_2$ são ilimitados) obtemos que $y_1$ e $y_2$ estão em componentes distintas de $B(y, r')^c$.

Portanto, existe um conjunto limitado   $B' = B(y, r')$ de $Y$ cujo complementar tem pelo menos $n$ componentes ilimitadas, e assim concluímos que $\eta(Y) \geq \eta(X)$. \end{enumerate}

Invertendo os papéis de $X$ e $Y$, o mesmo argumento mostra que na verdade $\eta(Y) = \eta(X)$. 
\end{proof}

\begin{exercise}
    \label{exercfins} Considere o exemplo do mergulho isométrico de $\R$ em $\R^2$ para verificar o que falha no sentido do lema acima caso não tenhamos a sobrejetividade grosseira de $f$.
\end{exercise}

Como consequência do Lema~\ref{lemafins}, podemos fazer a seguinte definição de número de fins de um grupo finitamente gerado:
\begin{definition}
    Sejam $G$ um grupo finitamente gerado e $S$ um conjunto de geradores de $G$. O \textit{número de fins de $G$}\index{número de fins de um grupo} é o número de fins do grafo  $\cay (G,S)$.
\end{definition}

O número de fins de um grupo está bem definido pois diferentes conjuntos finitos de geradores do grupo fornecem grafos de Cayley quasi-isométricos.

\subsection{Limites diretos e inversos de espaços}

Para proceder à definição do espaço dos fins de um espaço topológico, necessitaremos das noções de limites diretos e limites inversos, da teoria de categorias.

Seja $I$ um conjunto \textit{dirigido}, isto é, um conjunto parcialmente ordenado (também chamado \textit{poset}), tal que quaisquer dois elementos $i, j \in I$ tenham uma cota superior, ou seja, existe $k \in I$ com $i,j \leq k$. Os exemplos básicos de conjuntos dirigidos são o conjunto de números reais, ou números reais positivos, ou números naturais.

Um \textit{sistema direto}\index{sistema direto} de conjuntos (resp. espaços topológicos, grupos) indexado por $I$ é uma coleção $\{A_i \mid i \in I\}$  de  conjuntos (resp. espaços topológicos, grupos), junto de uma coleção $\{f_{ij}: A_i \to A_j\}$ de funções (resp. funções contínuas, homomorfismos), definidas para $i \leq j$,  satisfazendo as seguintes condições de compatibilidade:
\begin{enumerate}[(1)]
\item $f_{ik}=f_{jk}\circ f_{ij}$ para todos $i\leq j \leq k;$
\item $f_{ii} = \id$.
\end{enumerate}

Um \textit{sistema inverso}\index{sistema inverso} é definido similarmente, exceto pelos fatos de que $f_{ij}: A_j \to A_i$, para $i \leq j$,
e que a primeira condição de compatibilidade é dada por $f_{ik}=f_{kj}\circ f_{ji}$ para todos $i\leq j \leq k.$

O \textit{limite direto}\index{limite direto} de um sistema direto é o conjunto $$A=\limdir A_i =\displaystyle\bigsqcup_{i\in I}A_i\Bigg/\sim , $$ onde $a_i \sim a_j$ quando $f_{ik}(a_i) = f_{jk}(a_j)$ para algum $k \in I$. Temos aplicações naturais $f_k :A_k\to \limdir A_i$ dadas por $f_k(a_k) = [a_k]$. Note que $\limdir A_i =\displaystyle \bigcup_{k\in I} f_k(A_k).$ 

Se os conjuntos $A_i$ são grupos, equipamos o limite direto $\limdir A_i $ com a operação de grupo $[a_i]\cdot[a_j] =[f_{ik}(a_i)\cdot f_{jk}(a_j)]$ onde $k \in I$ é uma cota superior para $i$ e $j.$ Note que este produto está bem definido pela condição de compatibilidade. 

Se os $A_i$'s forem  espaços topológicos, definimos no limite direto $\limdir A_i$ a topologia fraca induzida pelos $f_k,$ isto é, $U \subset \limdir A_i$ é aberto, se e somente se, $f_k^{-1}(U)$ é aberto em $A_k$ para cada $k \in I$.

O \textit{limite inverso}\index{limite inverso} de um sistema inverso como acima é dado por 
$$\liminv A_i := \left\{(a_i)_{i\in I}\in\displaystyle\prod_{i\in I}A_i \mid a_i = f_{ij}(a_j) \mbox{ para todo }  i\leq j\right\}.$$

Se os conjuntos $A_i$ são grupos, equipamos  $\liminv A_i $ com a operação de grupo induzida pelo produto direto dos grupos $A_i$, ou seja, dados $(a_i)_{i \in I}$ e $(b_i)_{i \in I}\in\liminv A_i$ seu produto é $(a_i\cdot b_i)_{i \in I} \in \liminv A_i.$

Caso sejam espaços topológicos, definimos em  $\liminv A_i$ a topologia inicial, isto é, a topologia gerada por conjuntos da forma  $f_k^{-1}(U_k)$, onde  $U_k \subset A_k$ são abertos e $f_k: \liminv A_i  \to A_k$ é a $k$-ésima projeção.

\begin{exercise}
    Mostre que $\liminv  A_i$ é fechado em $\displaystyle\prod_{i\in I} A_i$.  Conclua portanto que, se cada $A_i$  for compacto,  então $\liminv  A_i$ também o é.
\end{exercise}

\subsection{O espaço de fins} \label{sec:esp de fins}

Seja $X$ um espaço topológico segundo contável, Hausdorff, localmente compacto, conexo e localmente conexo. Então $X$ admite uma exaustão por uma família enumerável $B_n$ de subconjuntos compactos.
Por exemplo, se $X$ for um espaço métrico próprio, então um exemplo de tal exaustão é $X= \displaystyle \bigcup_{n \in \N} \overline{B(x,n)}$, onde $x\in X$.

Considere o poset  $\mathcal{K} = \mathcal{K}_X$ dos subconjuntos compactos de $X$ com ordem parcial dada pela inclusão: $K_1 \leq K_2 $ se $K_1 \subset K_2 $. Note que se $K_1,K_2 \in \mathcal{K}$, então $K_1 \cup K_2 \in \mathcal{K}$, o que significa que o poset $\mathcal{K}$ é dirigido. 

Dado $K \in \mathcal{K}$, o conjunto $\pi_0(K^c)$ é definido como o conjunto de componentes conexas por caminhos de $K^c$.

Para $K_1 \leq K_2$ defina a aplicação natural
$$f_{K_1,K_2}: \pi_0(K_2^c)\to \pi_0(K_1^c),$$
onde a componente conexa por caminhos $c_2 \in \pi_0(K_2^c) $ é levada na componente conexa por caminhos  $c_1 \in \pi_0(K_1^c) $ que a contém.

\begin{exercise}
\label{exer:fKK}
    Mostre que a aplicação $f_{K_1,K_2}$ definida acima é sobrejetiva e que, para todos $K_1 \leq K_2 \leq K_3$ e $K\in \mathcal{K}$, valem: \begin{eqnarray*}
    f_{K_1,K_2}\circ f_{K_2,K_3}& = & f_{K_1,K_3},\\
    f_{K,K}& = & id.
    \end{eqnarray*}
  
\end{exercise}

Pelo Exercício \ref{exer:fKK}, segue que a coleção $\{\pi_0(K^c) \mid K \in \mathcal{K}\}$, junto das aplicações $\{f_{K_1,K_2} \mid K_1 \leq K_2 \}$ é um sistema inverso, e podemos portanto definir o seu limite inverso:
$$\liminv \pi_0(K^c) = \left\{(c_i)\mid c_i \in \pi_0(K_i^c), \,c_i = f_{K_i,K_j}(c_j)\, \mbox{ para } K_i\leq K_j\right\}.$$

\begin{definition}
    O\textit{ espaço de fins de $X$} \index{espaço de fins de um espaço topológico} é o espaço topológico $e(X)= \liminv\pi_0(\mathcal{K}^c)$, munido da topologia inicial definida pelas projeções, onde consideramos cada $\pi_0(K^c)$ com a topologia discreta.
\end{definition}

\begin{exercise}\label{ex-5.9}
     Mostre que $e(X)$ é um espaço de Hausdorff e totalmente desconexo.
\end{exercise}

\begin{exercise}
    Se $K$ é um subconjunto compacto de $X$, mostre que $\pi_0^u(K^c)$ é finito. (Lembre-se que  definimos $\pi_0^u(K^c)$ como o conjunto de componentes conexas por caminhos que não são relativamente compactas de $K^c$.)
\end{exercise}

Note que poderíamos considerar também os espaços $\liminv \pi_0^u(\mathcal{K}^c)$ ou $\liminv \pi_0^u(\mathcal{B}^c)$,  onde  $\mathcal{B} = \{B_n \mid n \in \N \}$ é uma exaustão de $X$. Os espaços acima também 
 são vistos como espaços topológicos com a topologia inicial definida pelas projeções naturais.

As inclusões $\mathcal{B} \xhookrightarrow{} \mathcal{K}$ e  $\pi_0^u(\mathcal{K}^c) \xhookrightarrow{} \pi_0(\mathcal{K}^c)$ induzem aplicações nos limites inversos:
$$\varphi:\liminv \pi_0^u(\mathcal{B}^c)  \to \liminv \pi_0^u(\mathcal{K}^c) \text{ e }$$
$$\psi:\liminv \pi_0^u(\mathcal{K}^c)  \to \liminv \pi_0(\mathcal{K}^c) = e(X).$$
Como cada $\pi_0^u(K^c)$ é finito, obtemos pelo Teorema de Tychonoff que $\liminv \pi_0^u(\mathcal{K}^c)$ é compacto. Deixamos como exercício a verificação de que os mapas $\varphi$ e $\psi$ são homeomorfismos. Isto permite dar outra caracterização do espaço de fins $e(X)$. Sejam $\mathcal{B} = \{B_n \mid n \in \N \}$ uma exaustão de $X$ e tome $e \in e(X)$. Sabemos que $e$ corresponde a uma cadeia $C_1 \supset C_2\supset \ldots$ de componentes de $B_i^c$, para $i \in \N$.
Podemos, portanto, pensar cada fim $e \in e(X)$  como um mapa $e : \mathcal{K}\to  2^X$, que envia cada compacto $K \subset X $ em uma componente $C$  de $K^c$, tal que $K_1 \subset K_2$ implica que  $e(K_2) \subset e(K_1)$.

A topologia em $e(X)$ se estende a uma topologia em $\Bar{X}:= X \cup e(X)$.
Uma base de vizinhanças de $e \in e(X)$ é a coleção de subespaços $B_{K,e} \subset \Bar{X}$, onde $K\in \mathcal{K}$, $B_{K,e} \cap  X = e(K)$ e $B_{K,e} \cap e(X)$ consiste de todos os mapas $e' : \mathcal{K} \to 2^X$, tais que $e'(K) = e(K)$.
Os conjuntos $e(K)$ serão chamados de vizinhanças de $e$ em $X$. Em $X$, a topologia coincide com a original. É imediato verificar que $X$ é aberto e denso em $\Bar{X}$.

Diremos que um compacto $K\subset X$ \textit{separa} dois fins $e$ e $e'$ de $X$ se $e$ e $e'$ estão em componentes distintas de $\Bar{X}\setminus K$ ou, equivalentemente, se existem componentes ilimitadas distintas $C,C'$ de $K^c$ tais que $(C,e)$ e $(C',e')$ são, respectivamente, vizinhanças de $e$ e $e'$ em $e(X)$.

\begin{exercise}\label{ex-5.11}
    Seja $\eta \in e(X)$ representado por $(U_i)_{i\in \N}$, tal que $U_{i+1}\subset U_i$, para todo $i$, com $U_i$ conexo. Defina
    $$ N_i(\eta):=\{\eta'\in e(X)\mid \eta'\mbox{ é representado por } (U'_i)_{i\in \N} \mbox{ com } U'_i\subset U_i\}.$$ 
    Mostre que 
    $$\mathcal{B}=\{N_{i}(\eta)\mid \eta\in e(X)\mbox{ é representado por } (U'_i)_{i\in \N}\}.$$ 
    é base para uma topologia em $e(X)$.
\end{exercise} 

\begin{exercise}
    Mostre que toda ação topológica $G \curvearrowright X$ se estende a uma ação  topológica de $G$ em $\Bar{X}$.
\end{exercise}

Uma terceira abordagem para se definir fins de um espaço topológico é feita a partir de raios e pode ser vista com mais detalhes em \cite[Capitulo~I.8]{bridson2013metric}.

Seja $X$ um espaço topológico. Um \textit{raio}\index{raio} em $X$ é uma aplicação \mbox{$r: [0,\infty)\to X$}. Sejam $r_1, r_2 :[0,\infty)\to X$ dois raios próprios. Dizemos que $r_1$ e $r_2$ são convergentes para o mesmo fim em $X$ se, para cada compacto $K$ de $X$, existe  $N\in \N$ tal que $r_1([N,\infty))$ e $r_2([N,\infty))$ estão contidos na mesma componente conexa por caminhos de $K^c$.

Isto define uma relação de equivalência sobre os raios contínuos e próprios em $X$. A classe de equivalência de $r$ é denotada por $e(r)$ e o conjunto de todas as classes de equivalência é o espaço $e(X)$.

Com esta abordagem, a convergência nos fins é definida da seguinte forma: dizemos que $e(r_n)\to e(r)$ se, e somente se, para cada compacto $K$ de $X$ existe uma sequência $N_n$ de números reais tal que  $r_n([N_n,\infty))$ e $r([N_n,\infty))$ estão contidos na mesma componente conexa por caminhos de $X\setminus K$ para $n$ suficientemente grande.
 A topologia em $e(X)$ é descrita a partir dos seus fechados. Um subconjunto $B\subset e(X)$ é fechado se satisfizer a seguinte condição: caso $e(r_n) \in B$ para todo $n$, então $e(r_n)\to e(r)$ implica $ e(r)\in B.$
 
 \begin{proposition} \label{finsinvarianteqi}
 Se  $X$ e $Y$ são espaços métricos próprios, toda quasi-isometria $f : X\to Y$ induz um homeomorfismo $f_{e}:  e(X)\to e(Y)$ entre seus espaços de fins.
  A aplicação que a cada $f\in \QI(X)$ associa o mapa $f_{e}\in \Homeo(e(X))$ é um homomorfismo de grupos.
 \end{proposition}
 
\begin{exercise}
Dar uma prova da Proposição \ref{finsinvarianteqi}.
\end{exercise}

Vejamos que as noções de fins são equivalentes: Para cada $C_i$, em uma cadeia $(C_i)$ representando um fim $e \in e(X)$, escolha um ponto $x_i$. Defina assim uma sequência $(x_i)$, que diremos representar o fim  $e$. Dada uma sequência $(x_i)$ representando $e$, podemos conectar cada $x_i$ a $x_{i+1}$ por um caminho contínuo contido em $C_i$. A concatenação de todos esses caminhos é um raio em $X$, com  $r(i) = x_i$. 

Por outro lado, dado um raio $r$ em $X$, toda sequência $t_i \in \R_+$ divergindo monotonicamente para infinito define uma sequência $x_i = r(t_i)$ que representa um fim $e$ de $X$.
Esse fim é independente da escolha da sequência $t_i$.

\begin{exercise}
    Mostre que:
    \begin{enumerate}[(1)]
        \item O espaço $\Bar{X} = X \cup e(X)$ é Hausdorff;
        \item Se $X$ é segundo contável, então $\Bar{X}$ também o é;
        \item Uma sequência $(x_i)$ em $X$ representa o fim $e$ se, e somente se, ela converge para $e$ na  topologia de $\Bar{X}$;
        \item Se $X$ é um espaço métrico e $(x_i), (x'_i)$ são sequências com distância limitada, isto é:
        $$\sup_i \{\dist(x_i, x'_i)\} < 1,$$
e $(x_i)$ representa $e \in e(X)$ então $(x'_i)$ representa o mesmo fim.
    \end{enumerate}
\end{exercise}

\begin{proposition}
    O espaço $\Bar{X}$ é compacto.
\end{proposition}

\begin{proof}
    Seja  $\mathcal{U} = \{U_{\lambda}\}_{\lambda \in L}$ uma cobertura aberta de $\Bar{X}$. Como $e(X)$ é compacto, deve existir um subconjunto finito $\mathcal{U}_1 = \{(K_{\lambda_i}, e_i) \mid i = 1, \ldots n\} \subset \mathcal{U}$ que cobre $e(X)$, onde $K_{\lambda_i} \in \mathcal{K}$ e $e_i \in e(X)$.

Considere os abertos $C_i = e_i(K_{\lambda_i})$ (estamos olhando para os fins de $X$ como mapas $\mathcal{K}
\to  2^X$). Afirmamos que o fechado $
A = X\setminus \bigcup_{i=1}^n C_i$ é compacto em $X$. Caso contrário, deve existir uma sequência $x_k \in A$  a qual não esteja contida em nenhum dos compactos $K_{\lambda_i}$. Passando a uma subsequência se necessário, obtemos uma  sequência decrescente de conjuntos complementares
$C_{k_l} \subset X \backslash K_{k_l}$, tais que  $x_{k_l} \in C_{k_l}$. Essa sequência define um fim $e \in e(X)$ que não é coberto por nenhum dos conjuntos $(K_{\lambda_i}, e_i)$, $i = 1, \ldots,n$, o que gera uma contradição.

Logo, $A \subset X $ é compacto. Assim, existe outro subconjunto finito $\mathcal{U}_2 \subset \mathcal{U}$, o qual cobre $A$. Portanto, $\mathcal{U}_1 \cup \mathcal{U}_2 $ é uma subcobertura finita de $\Bar{X}$.
\end{proof}

\section{Fins de grupos}

 \begin{definition}[{\it Fins de um grupo}] Sejam $G$ um grupo finitamente gerado e $\cay(G,S)$ seu grafo de Cayley com respeito a um conjunto finito de geradores $S.$ Definimos o espaço de fins de $G$ como $e(G) = e(\cay(G,S))$.
 \end{definition}
 
Pelo Teorema de Milnor--Schwarz e pela Proposição~\ref{finsinvarianteqi}, o espaço de fins de um grupo finitamente gerado não depende da escolha do conjunto de geradores $S$. 

O seguinte teorema é essencialmente devido a Hopf. Nossa exposição é uma adaptação do Teorema~8.32 no livro \cite{bridson2013metric}.

\begin{thm} \label{thm:fins grupos}
O espaço de fins de um grupo $G$ tem as seguintes propriedades: 
\begin{enumerate}
\item O espaço $e(G)$ é compacto, Hausdorff e totalmente desconexo.
\item A cardinalidade de $e(G)$ vale $0$, $1$, $2$, ou então $e(G)$ tem a cardinalidade do contínuo. Além disso, se $\card(e(G))=2^{\omega}$ então $e(G)$  é um conjunto perfeito (i.e. fechado e não possui pontos isolados). 
\item Vale $\card(e(G)) =2$ se e somente se $G$ é virtualmente cíclico infinito, e $\card(e(G)) = 0$ se e somente se $G$ é finito.
\item  Se $G$ é de uma  das formas $A\ast_F B$ ou $A\ast_F$ com $F$ finito,  então $\card(e(G)) > 1$.
\end{enumerate}
\end{thm}

\begin{proof} Fixamos um conjunto finito $S$ de geradores  de $G$, e escrevemos  $X = \cay(G,S)$. Seguindo a definição, temos $e(G) = e(X)$. Pelo Exercício \ref{ex-5.9}, $e(X)$ é Hausdorff e totalmente desconexo.
Dada uma sequência de fins $e_n$, escolhemos  raios geodésicos $r_n: [0, \infty) \to X$ com $r_n(0) = 1$ e $e(r_n) = e_n$. O teorema de Arzela--Ascoli mostra que uma subsequência destes raios converge em subconjuntos compactos para algum raio geodésico $r$, e então a sequência correspondente de fins converge para $e(r)$. Assim, $e(X)$ é sequencialmente compacto e, uma vez que satisfaz o primeiro axioma de enumerabilidade (veja o Exercício~\ref{ex-5.11}), é compacto. 

Agora vamos provar o item 2. Sabemos que a ação de $G$ por translação à esquerda se estende a uma ação por isometrias em  $X$. Obtemos assim um homomorfismo $G\to \Homeo(e(X))$. Seja $H$ o núcleo desse homomorfismo. 

Se $e(X)$ é finito, então $\Homeo(e(X))$ é finito e portanto $H$ tem índice finito em $G$. Suponhamos que $\card(e(G)) \geq 3$ e tomemos $e_0,e_1,e_2 \in e(G)$ três fins distintos. Vamos fixar dois raios geodésicos $r_1, r_2 $ partindo da identidade, com $e(r_i)=e_i$. Como $H$ tem índice finito (em particular é  quasi-isométrico a G), podemos construir um raio próprio $r_0 :[0, \infty)\to X$ com $e(r_0)=e_0$, $d(r_0(n),1)\leq n$, e $r_0(n)\in H$ para cada $n\in \N.$ Seja $h_n = r_0(n)$. Fixaremos $\rho>0$ tal que $r_0([\rho,\infty)),r_1([\rho,\infty))$ e $r_2([\rho,\infty))$ estão em componentes conexas distintas do complementar da bola de centro $1$ e raio $\rho$.  Se $t,t'> 2 \rho$, então $d(r_1(t),r_2(t'))> 2 \rho$, pois cada caminho ligando $r_1(t)$ a $r_2(t')$ tem que passar por $B(1,\rho)$. Como $H$ age trivialmente em $e(X)$, temos que $e(h_nr_i) = e(r_i)$ para $i = 1$, $2$. Seja $n>3 \rho$. Daí,  $h_nr_i$ começa numa componente conexa e vai para outra, donde concluímos que deve passar por $B(1,\rho)$. Assim, existem $t,t'$ tais que $h_nr_1(t),h_nr_2(t')\in B(1,\rho)$, com $t,t'>2\rho$, donde $$2\rho<d(r_1(t),r_2(t')) = d(h_nr_1(t),h_nr_2(t')) < 2\rho,$$
uma contradição. Isso mostra que $\card(e(G)) < 3$.

Um espaço compacto de Hausdorff no qual cada ponto é um ponto de acumulação é não-enumerável. Assim basta provar que $e(X)$ é perfeito, isto é, que  caso $e(X)$ seja infinito, temos que todo $e \in e(X)$ é um ponto de acumulação. 

Dados $e_1$, $e_2 \in e(X)$, seja $D = D(e_1, e_2)$ o maior  inteiro tal que $r_1([D, \infty))$ e $r_2([D, \infty))$ estejam na mesma componente conexa por caminhos de $X \setminus B(1, D)$ para alguns (portanto todos) raios geodésicos $r_i$ com $e(r_i) = e_i$ e $r_i(0) = 1$, $i = 1,2$. Observe que $e_n \to e$ se, e somente se, $D(e_n, e) \to \infty$ quando $n \to \infty$.

Fixe $e_0$ em $e(X)$ e um raio geodésico $r_0$ com $r_0(0) = 1$ e $e(r_0) = e_0$. Seja $g_n = r_0(n)$. Construiremos uma sequência de fins $e^m$ tal que  $e^m \neq e_0$ e $D(e^m, e_0) \to \infty$ quando $m \to \infty$. Tome $e_1$, $e_2$ fins distintos e diferentes de $e_0$, e seja $r_i$ um raio geodésico com $r_i(0) = 1$ e $e(r_i) = e_i$, para $i = 1,2$. Seja $M = \max\{ D(e_i, e_j) \mid i,j = 0, 1, 2\}$. Se fixarmos $\rho > M$, argumentando como acima, vemos que, para $n$ suficientemente grande, no máximo um dos raios $g_n r_i$ pode passar por $B(1, \rho)$. Para cada um dos outros dois raios $g_n r_j$ aplicamos o teorema de Arzela--Ascoli para construir um raio geodésico $r_j'$ com $r_j'(0) = 1$ e $e_j' = e(r_j') = e(g_n r_j)$. Como $r_0$ e $r_j'$ possuem segmentos terminais na mesma componente conexa por caminhos de $X \setminus B(1, \rho)$, temos $D(e_0, e_j') > \rho$. Os fins $e'_j$ são distintos, então pelo menos um deles não é igual a $e_0$: renomeie este fim por $e^1$.

Agora repetimos o argumento acima com $e_1$ substituído por $e^1$ e chamamos o novo raio construído com o mesmo argumento de $e^2$. Substituímos $e^1$ por $e^2$ e repetimos novamente. A cada iteração deste processo, o inteiro $M$ é incrementado. Assim obtemos a sequência desejada de fins $e^m$, com $e^m \neq e_0$ e $D(e^m, e_0) \to \infty$ quando $m \to \infty$. Isso conclui a prova do item~2.

No item 3, o caso $\card(e(G)) = 0$ é trivial.  Também é fácil ver que $G = \Z$ tem $2$ fins, e então cada grupo virtualmente isomorfo a $\Z$ tem $2$ fins. Reciprocamente, seja $\card(e(G)) = 2$. Escolhamos uma bola $B = B(1, \rho)$ em $X$ que separa os dois fins. Seja $g \in G\setminus B$ um elemento que age trivialmente em $e(X)$ (lembramos que o núcleo da ação de $G$ em $e(X)$ tem índice finito em $G$ portanto tais elementos existem). Seja $r$ um raio geodésico que começa em $g$ e não passa por $B$. Para cada $n > 0$, $g^n r$ é um raio geodésico que não passa por $B$ (um exercício), então $g^n \neq 1$. Assim o elemento $g$ tem ordem infinita em $G$. As sequências $g^n$ e $g^{-n}$ convergem a diferentes fins de $X$ e qualquer ponto de $X$ está a uma distância limitada de $\langle g \rangle$. 

Para provar o item 4, considere a ação de $G$ na árvore de Bass--Serre \cite[Seção~I.4]{Serre} e conclua de modo semelhante ao caso do grupo livre não abeliano. 
\end{proof}

\begin{exercise}\

\begin{enumerate}
\item Dê exemplos de grupos com $0$, $1$, $2$ e um número infinito de fins.
\item Mostre que $\card(e(G)) = 2$ se e somente se $G$ é da forma $A\ast_F B$ com $|A:F| = |B : F| = 2$ e $F$ finito ou da forma  $F\ast_F$ com $F$ finito.
\end{enumerate}
\end{exercise}

A recíproca do item $4$ do Teorema~\ref{thm:fins grupos} também é verdadeira. Este é um resultado importante que foi provado por Stallings em 1968 para os grupos sem torção e posteriormente estendido por Bergman para o caso geral:

\begin{thm}[Stallings \cite{Stal68}]\label{thm:Stallings}
Se $\card(e(G)) > 1$, então $G$ se divide não trivialmente sobre um subgrupo finito.
\end{thm}
 
Atualmente, existem várias provas diferentes deste teorema, mas todas elas são bastante difíceis. Os fins de grupos lidam com o problema-chave da teoria geométrica de grupos: relacionar as propriedades geométricas de um grupo e sua estrutura algébrica. Quase toda vez que se consegue recuperar a estrutura algébrica de um grupo a partir de informações geométricas sobre ele, algum milagre tem que acontecer. Isso pode explicar por que não há provas fáceis dos resultados como o Teorema de Stallings. Nós nos referimos a \cite[Capítulo 21]{kd} para os detalhes de uma das provas mais recentes.

\chapter{Ultralimites e cones assintóticos}
\label{cap6}

Ultrafiltros e ultralimites apareceram na teoria geométrica dos grupos por conta de sua relação com os grupos amenáveis. Esses elementos de análise não padronizada forneceram um conjunto de ferramentas poderosas que permitiram simplificar e fortalecer vários resultados que remontam a Banach, von Neumann e outros. Mais tarde, ultrafiltros foram usados para definir cones assintóticos. Estudaremos grupos amenáveis e as aplicações de cones assintóticos nos capítulos seguintes, fornecendo agora os pré-requisitos necessários.

\section{Filtros e ultrafiltros}
\label{sec:filtrosultrafiltros}

Seja $S$ um conjunto qualquer. Um \textit{filtro} \index{filtro} em $S$ é uma família $\Fc$ de subconjuntos de $S$
 com as seguintes propriedades:
\begin{enumerate}[(i)]
    \item $S\in \Fc$ e $\emptyset \notin \Fc$;
    \item $A,B \in \Fc \Rightarrow A\cap B \in \Fc$;
    \item $A \in \Fc, A \subset B \Rightarrow B \in \Fc$.
 \end{enumerate}

\begin{example}
\label{frechet}
Se $S$ é um conjunto infinito, a coleção formada pelos complementares de conjuntos finitos em $S$ é um filtro em $S$, chamado \textit{filtro de Fréchet}.\index{filtro de Fréchet} Ele é usado em topologia, para definir a topologia cofinita.
\end{example}

\begin{example}
\label{filtrovizinhanças}
    Em um espaço topológico $T$, o conjunto de todas as vizinhanças de um ponto $x\in T$ é um filtro, conhecido como \textit{filtro das vizinhanças}\index{filtro das vizinhanças} de $x$.
\end{example}

De modo geral, a ideia de filtro é que ele carrega consigo a informação de quais são todos os subconjuntos ``grandes'' o suficiente de $S$, onde a noção exata do que significa um conjunto grande é determinada pelo contexto. No Exemplo~\ref{frechet}, ser ``grande'' significa ter complementar finito, enquanto no Exemplo~\ref{filtrovizinhanças} essa condição é caracterizada por ser um conjunto que contém $x$ em seu interior.

\begin{definition} \label{ultrafiltro}
Dado um filtro $\Fc$ em $S$, dizemos que $\Fc$ é um \textit{ultrafiltro} \index{ultrafiltro} se, para todo $A\subset S$, tem-se $A\in \Fc$ ou $S\setminus A \in \Fc$ (mas não ambos). Equivalentemente, podemos dizer que $\Fc$ é um ultrafiltro em $S$ se $\Fc$ é maximal no conjunto ordenado dos filtros em $S$.  
\end{definition}

\begin{exercise}
Verifique que as noções de ultrafiltros na Definição~\ref{ultrafiltro} são equivalentes.  
\end{exercise}

Um exemplo de ultrafiltro é a família de todos os subconjuntos contendo um elemento fixado $s \in S$. Esse tipo de ultrafiltro é dito \textit{ultrafiltro principal}.\index{ultrafiltro principal} 
Caso $S$ seja finito, todo ultrafiltro de $S$ é principal. Já no caso onde  $S$ é infinito, o Lema de Zorn mostra que existem ultrafiltros contendo o filtro de Fréchet, os quais não são principais. Na verdade, todo ultrafiltro não principal contém o filtro de Fréchet, o qual é a interseção de todos os ultrafiltros que o contém, conforme veremos na proposição a seguir. Note que o filtro de Fréchet não é um ultrafiltro, pois podemos exibir em $S$ um subconjunto $A$ infinito com $S\setminus A$ também infinito.

\begin{proposition}\label{prop:Frechet}
   Um  ultrafiltro em  um conjunto infinito $S$ é não principal se, e somente se, ele contém o filtro de Fréchet.
\end{proposition}

\begin{proof}

Seja $\Fc$ um ultrafiltro  não principal em $S$ e considere $x\in S$ qualquer. Como $\Fc$ é um ultrafiltro, exatamente um dos conjuntos $\{x\}$ ou $S\setminus \{x\}$ pertence a $\Fc$. Sendo $\Fc$ não principal, a primeira opção não pode ocorrer. Assim,  $S\setminus \{x\} \in \Fc$ para todo $x \in S$. Mas com isso, dado um subconjunto finito $F\subset S$ qualquer,  obtemos que $S\setminus F = \displaystyle\bigcap_{x\in F} S\setminus \{x\}$ está em $\Fc$, pois $\Fc$ é fechado para interseções finitas. Logo, $\Fc$ contém o filtro de Fréchet.

A recíproca pode ser provada por contrapositiva. Um ultrafiltro principal $\Fc_s$ em $S$ definido por um ponto $s$ contém $\{s\}$. Assim,  ele não contém $S\setminus \{s\}$, o qual é um elemento do filtro de Fréchet.
\end{proof}

O resultado principal sobre a existência dos ultrafiltros é fornecido pelo seguinte lema.

\begin{lemma}[Lema do ultrafiltro]
Todo filtro em um conjunto $S$ é um subconjunto de algum ultrafiltro em $S$.    
\end{lemma}
\begin{proof}
Seja $\mathcal{F}$ um filtro em $S$. Pelo Lema de Zorn, existe um filtro maximal $\mathcal{U}$ em $S$ contendo $\mathcal{F}$. Pela maximalidade, $\mathcal{U}$ é um ultrafiltro. 
\end{proof}

Este resultado é particularmente interessante quando começamos com o filtro de Fréchet $\mathcal{F}$ em um conjunto infinito $S$. Neste caso, pela Proposição~\ref{prop:Frechet}, o lema do ultrafiltro implica a existência de um ultrafiltro não principal em $S$ contendo $\mathcal{F}$.


Vimos que filtros capturam informações sobre conjuntos serem suficientemente grandes, segundo algum contexto. Já um ultrafiltro divide todos os subconjuntos entre ``grandes'' e ``pequenos'', onde objetos de tamanho intermediário não são reconhecidos. Embora se saiba que o número de ultrafiltros em um conjunto infinito $S$ seja da ordem de $2^{|S|}$, não existem construções explícitas gerais para nenhum ultrafiltro não principal nesse caso.

Observamos que utilizamos uma forma do Axioma da Escolha (o Lema de Zorn) para estabelecer a existência dos ultrafiltros não principais. Percebe-se, por sua vez, que a existência dos ultrafiltros não principais implica o teorema da extensão de Hahn--Banach, que lembramos agora:
\begin{thm}[Hahn--Banach]\label{theorema:Hahn-Banach}\index{teorema de Hahn--Banach}
Sejam $V$ um espaço vetorial real, $U$ um subespaço de $V$ e $\varphi : U \to \R$ uma função linear. Seja $p : V \to \R$ um mapa com as seguintes propriedades:
$$ p(\lambda x) = \lambda p(x)\text{ e } p(x + y) \leq p(x) + p(y),\ \forall x, y \in V,\ \lambda\in [0, +\infty),$$
tal que $\varphi(x) \leq p(x)$ para todo $x \in U$. Então existe uma extensão linear de $\varphi$, $\Bar{\varphi}: V \to \R$ tal que $\Bar{\varphi}(x) \leq p(x)$ para todo $x \in V$.    
\end{thm}
Assim, o teorema de Hahn--Banach pode ser visto como o Axioma da Escolha do analista, numa versão mais fraca. Com estas observações em mente, não é surpreendente que alguns resultados obtidos por meio de ultrafiltros também possam ser comprovados usando o teorema de Hahn--Banach. 

Algumas propriedades de ultrafiltros são dadas a seguir, e o leitor pode se referir a \cite{goldbring2022ultrafilters} e a \cite{kd} para mais detalhes, exemplos e propriedades a respeito de filtros e ultrafiltros.

\begin{proposition}
\label{propultrafiltro}
Seja $S$ um conjunto e $\Fc$ um ultrafiltro em $S$. Então:
\begin{enumerate}[(1)]
    \item Se $S = A_1\cup  \ldots \cup A_n$, então $A_i \in \Fc$, para algum $i\in \{1,\ldots,n\}$;
    \item Se $T \in \Fc$, e $T = A_1 \cup \ldots \cup A_n$, então $A_i\in \Fc$ para algum $i$;
    \item Se  $G$ é uma família de subconjuntos não vazios de $S$, tais que a interseção de qualquer número finito de membros de $G$ é não vazia, então $G$ está contida em algum filtro em $S$.
\end{enumerate}
\end{proposition}

\begin{exercise}
    Prove a Proposição \ref{propultrafiltro}.
\end{exercise}

No que segue, dados dois conjuntos quaisquer $A$ e $S$, denote por $A^S$ o conjunto de todas as funções de $S$ para $A$.

\begin{definition}
   Seja $\Fc$ um filtro em $S$, e $A$ algum conjunto. A \textit{potência reduzida de $A$ módulo $\Fc$},\index{potência reduzida} denotada por  $A^S/\Fc$ é o conjunto das classes de equivalência em $A^S$, onde dois elementos $f, g \in  A^S$ são definidos como equivalentes se  $\{s \in S \mid f(s) =
g(s)\} \in \Fc$. 
\end{definition}

Em outras palavras, duas funções de $S$ para $A$ são equivalentes se elas coincidem em um elemento de $\Fc$, isto é, em algum conjunto ``grande''. Neste caso, diremos que tais funções coincidem em quase todo ponto (q.t.p.).
Se  $\Fc$ for um ultrafiltro, uma potência  reduzida módulo $\Fc$ é chamada de  \textit{ultrapotência}.\index{ultrapotência}

\section{Ultralimites}

Vamos verificar como a existência de ultrafiltros não principais implica o teorema de extensão de Hahn--Banach (cf. Teorema~\ref{theorema:Hahn-Banach}) em um caso especial.  Sejam $V$ o espaço vetorial real de sequências limitadas de números reais $\mathbf{x} = (x_n)$, $U \subset V$ o subespaço de sequências convergentes de números reais, $p$  a função a norma do sup:
$$p(\mathbf{x}):= ||\mathbf{x}||_\infty = \sup_{n \in \N} |x_n|,$$
e $\varphi : U \to \R$ a função limite, ou seja,
$$\varphi(\mathbf{x}) = \lim_{n \to\infty} x_n.$$
Em vez da norma do sup, podemos também tomar
$$ p: \mathbf{x} \to \limsup x_n.$$

Em outras palavras, usando um ultrafiltro não principal, pode-se estender a noção de limite de sequências convergentes para sequências limitadas. A principal ferramenta nesta extensão é o conceito de \textit{ultralimite}.


Fixemos um ultrafiltro não principal $ \Fc$ em $\N$. Sejam $T$ um espaço topológico e $(x_n)$ uma sequência em $T$. Para cada $x\in T$, e cada vizinhança $U$ de $x$, defina   $O(x,U) := \{ n \in \N \mid x_n \in U\}$.

\begin{definition}
    Diremos que $x$ é o  \textit{$\Fc$-limite} \index{ultralimite} de $(x_n)$ se, para cada vizinhança $ U$ de $x$, o subconjunto $O(x,U)$ está em $\Fc$. Neste caso, escrevemos $x = \Fc\lim x_n$.
\end{definition}

Com isto, $x$ será o  $\Fc$-limite de uma sequência  se toda vizinhança de $x$
contém quase todos os membros daquela sequência.

\begin{exercise}
    Mostre que um ponto $x\in T$ é o  limite de $(x_n)$, no sentido usual se, e somente se, $x$ é o $\Fc$-limite de $(x_n)$, para todo ultrafiltro não principal $\Fc \subset \mathbb{N}$.
\end{exercise}

\begin{proposition} \label{prop:ultralim} Sejam $T$ um espaço topológico e $\Fc$ um ultrafiltro não principal em $\N$.
    \begin{enumerate}
        \item Se $T$ é Hausdorff, então o  $\Fc$-limite de uma sequência é único.
        \item Se $T$ é compacto, então toda sequência é $\Fc$-convergente. 
    \end{enumerate}
\end{proposition}

\begin{proof}
 \begin{enumerate}
     \item Sejam $x$ e $y$ dois pontos distintos, e considere $U$ e $V$ vizinhanças disjuntas desses pontos.  Assim, os conjuntos $O(x,U)$ e $O(y, V )$ são disjuntos, e no máximo um deles pode pertencer a $\Fc$, pelas propriedades (i) e (ii) de filtros.
\item Suponha que $\{x_n\}$ não é  $\Fc$-convergente. Então, para cada $y\in T$, existe uma vizinhança $U_y$ tal que  $O(y,U_y) \notin \Fc$. Como $\Fc$ é um ultrafiltro, o conjunto $O^*(y,U_y) :=
\{n \in \N \mid x_n \notin U_y\}$ está em $\Fc$. Dados  finitos pontos $y_1, \ldots , y_k$,  temos $O^* =
\bigcap O^{*}(y_i, U_{y_i} ) \in \Fc$ e, portanto, esta interseção é não vazia. Se  $i \in O^{*}$, então $x_i$ está fora de todas as vizinhanças $U_{y_i}$, e portanto $ T$ não pode ser recoberto  por uma quantidade finita dessas vizinhanças, o que contradiz a  compacidade de $T$.
 \end{enumerate} 
\end{proof}

\begin{corollary}
Toda  sequência limitada de números  reais é  $\Fc$-convergente, e seu $\Fc$-limite é único.
\end{corollary}

\begin{exercise}
    Sejam $x_n$ e $y_n$ duas sequências limitadas de números reais, e $c$ um número real. Mostre que:
    \begin{enumerate}[(1)]
        \item $\Fc \lim(x_n + y_n) = \Fc \lim x_n + \Fc \lim y_n$; 
        \item $\Fc \lim(c x_n) = c\Fc \lim x_n $;
        \item Se $x_n \leq y_n$ para todo $n$, então $\Fc \lim x_n \leq \Fc \lim y_n$.
    \end{enumerate}
\end{exercise}

\begin{exercise}
    Um número real $x$ é o $\Fc$-limite de uma sequência $(x_n)$ em $\R$ para algum ultrafiltro
não principal $\Fc$ se, e somente se, $x$ é o limite usual de alguma subsequência de $(x_n)$.
\end{exercise}

Percebemos assim que um ultrafiltro em $\N$ é uma maneira de escolher, para cada sequência em $T$,  algum ponto de acumulação. Além disso, há uma uniformidade nessa escolha. Dadas sequências $(x_n)$ e $(y_n)$, seus $\Fc$-limites podem ser obtidos através da ``mesma'' subsequência de ambas, isto é, existe uma sequência $(n_i)$ de índices tal que $ \Fc \lim x_n = \lim x_{n_i}$ e $\Fc \lim y_n = \lim y_{n_i}$.
Isso segue do  fato de que dois conjuntos em $\Fc$ tem  interseção não-vazia, e essa interseção deverá ser infinita. A uniformidade se 
estende a qualquer número finito de sequências.

Finalmente, podemos provar o teorema de Hahn--Banach para o espaço de sequências convergentes $U$, o espaço de todas as sequências limitadas $V$ e o funcional $\varphi := \lim : U \to \R$. Seja
$$\Bar{\varphi}((x_n)) = \Fc\lim x_n.$$
A Proposição~\ref{prop:ultralim} implica que toda função limitada $f : \N \to \R$ possui um ultralimite. No caso em que existe o limite ordinário $\displaystyle \lim_{n\to\infty} x_n$, ele é igual ao ultralimite $\Fc\lim x_n$. É fácil verificar que a desigualdade
$$\Fc\lim x_n \leq p((x_n))$$
vale para ambos $p((x_n)) = \sup_{n \in \N} |x_n|$ e $p(x_n) = \limsup x_n$. Isso prova o teorema de Hahn--Banach (no caso especial).

\section{Cones assintóticos}

O cone assintótico de um espaço é um espaço métrico que representa o espaço visto do infinito. Como a geometria em pequena escala desaparece nesse limite, ele se torna uma boa ferramenta para estudar propriedades quase-isométricas de grupos. Construiremos os cones assintóticos usando ultrafiltros. Recordemos que um ultrafiltro $\Fc$ nos permite definir o $\Fc$-limite de qualquer sequência limitada: os ultralimites usam o axioma da escolha para estender a noção de limite a sequências limitadas arbitrárias; eles são lineares e o ultralimite de uma sequência é um ponto limite da sequência. Essas são as propriedades-chave para a construção de cones assintóticos com os quais procedemos nesta seção.

O conceito de cone assintótico foi introduzido pela primeira vez na teoria geométrica de grupos por van den Dries e Wilkie em \cite{vdDriesWilkie}, embora sua versão para grupos de crescimento polinomial já tivesse sido usada por Gromov em \cite{Gromov81}, que usou a convergência de Gromov--Hausdorff como ferramenta em vez de ultrafiltros. Cones assintóticos para espaços métricos gerais foram definidos por Gromov em \cite{gromov1993asymptotic}.


Considere um par $(X, x)$, onde $X$ é um espaço métrico com distância $d$ e $x \in X$ é um ponto fixado. Mantemos a escolha de um ultrafiltro $\Fc$ em $\N$.

\begin{definition}
    Uma sequência $(x_n)$ em $X$ é dita \textit{moderada}\index{sequência moderada} se satisfizer $d(x_n,x) \leq A n$, para todo $n$, onde  $A$ é  alguma constante (que depende da sequência). Vamos denotar por $M$ o \textit{conjunto  das sequências moderadas} em $X$.
    
Dadas sequências moderadas $\mathbf{x} = (x_n)$ e $\mathbf{y} = (y_n)$, a \textit{distância} entre elas é definida por $d_M(\mathbf{x},\mathbf{y}) = \Fc \lim\left(\dfrac{d(x_n, y_n)}{n}\right)$.
\end{definition}

Essa distância satisfaz a desigualdade triangular. No entanto, é possível que  duas sequências distintas tenham distância zero. Desse modo, diremos que duas sequências moderadas são  equivalentes se a distância entre elas é igual a  zero. Obtemos assim uma relação de equivalência $\sim$ em $M$,  o que permite definir a distância $d_K$ entre duas classes de equivalência de elementos de $M/{\sim}$ como a distância $d_M$ entre quaisquer de seus representantes. O espaço métrico 
$$K = (M/{\sim}, d_K)$$ 
obtido desse modo é chamado de \textit{cone assintótico} \index{cone assintótico} definido por $(X,x)$ e $\Fc$. Note que duas sequências que originam o mesmo elemento na ultrapotência $X^{\N} / \Fc$, isto é, que coincidem em algum conjunto de $\Fc$,  são equivalentes.  Isso significa que podemos trocar os elementos de qualquer sequência em um conjunto fora de $\Fc$  sem alterar o elemento correspondente em $K$. Mas duas sequências podem ser 
equivalentes mesmo não sendo  idênticas em um conjunto de $\Fc$. Podemos dizer que duas
sequências são iguais no cone assintótico  $K$ se elas forem  ``infinitesimalmente próximas''.

Seja $G$ um  grupo finitamente gerado, e  fixe um conjunto finito $S$ de geradores de $G$. Seja $\ell_S$ a função de comprimento com respeito a $S$, e considere  o espaço métrico $G$
 munido da distância $d_S(x, y) = \ell_S(x^{-1}y)$. Tomemos a identidade $e$ de $G$ como ponto base, e definamos $K$  como o cone assintótico  definido por $(G, e)$
e $\Fc$. Também denotamos por $e$ o elemento de $K$ definido pela sequência  constante $(e)$.

Uma propriedade importante dos cones assintóticos é que os cones dos espaços quase-isométricos são bi-Lipschitz. Isso ocorre porque, quando reduzimos a métrica por uma sequência de fatores divergentes, os termos constantes aditivos na definição de quase-isometria se reduzem a zero.

\begin{examples}
    \begin{enumerate}
        \item Se $G=\Z^m$, o grupo livre abeliano em $m$ geradores, com a métrica das palavras (a qual coincide com a métrica induzida pela métrica da soma em $\R^m$). O cone assintótico definido por $(G,0)$ e por qualquer ultrafiltro não principal $\Fc$ em $\N$ é o grupo topológico $\R^m$, também com  a métrica da soma $d(x,y) = |x_1 - y_1| + \ldots + |x_m - y_m|$.  De fato, um ponto $x\in \R^m$ é representado pela sequência moderada $x_n = \lfloor nx \rfloor$ em $G$, onde a função piso é aplicada a cada componente de $x$, e qualquer sequencia moderada $(x_n)$ representa o ponto $\Fc\lim(x_n/n) \in \R^m$.

        \item  Se $G = F_m$, o grupo livre de posto $m \ge 2$, seu cone  assintótico $K$ é uma $\R$-árvore onde cada ponto possui grau incontável, que novamente independe do ultrafiltro fixado. Por definição, uma $\R$-árvore é um espaço métrico geodésico onde cada triângulo é um trípode (Definição~\ref{def:tripode}). Para ver isso, consideremos as geodésicas em $K$. Seja $x_n = x_0$ para todo $n$. Há uma quantidade incontável de geodésicas passando por $x_0$ no grafo de Cayley de $G$, que é uma árvore regular, e cada uma delas corresponde a uma geodésica em $K$ que passa pelo ponto $x$ representado pela sequência $(x_0)$. Agora, considere outro ponto $y \neq x$ em $K$ representado por uma sequência $(y_n)$. Observe que quando $n$ aumenta, os $y_n$ se afastam cada vez mais de $x_0$. O limite das geodésicas que conectam $x_0$ e $y_n$ é uma geodésica única entre $x$ e $y$ em $K$. O mesmo é válido para qualquer ponto no cone assintótico. Para verificar se todo triângulo é um trípode, precisamos tomar o terceiro ponto $z\in K$, distinto de $x$ e $y$, e considerar as geodésicas que conectam $x$, $y$ e $z$ (cf. \cite[Seção~2]{gromov1993asymptotic} para o caso mais geral de espaços $\delta$-hiperbólicos). 
    \end{enumerate}
\end{examples}


Vamos prosseguir com algumas propriedades gerais dos cones assintóticos de grupos finitamente gerados.
Note que  o conjunto $M$ forma um grupo com a operação $\mathbf{x} \cdot \mathbf{y} := (x_ny_n)$, se $\mathbf{x} = (x_n)$, $\mathbf{y} = (y_n) \in M$. O conjunto $M_0$ de sequências equivalentes a $e$ é um subgrupo de $M$, e existe uma  correspondência  $1-1$ entre $K$ e as classes laterais de $M_0$ em $M$. No entanto, $M_0$ não é um  subgrupo normal, e portanto $K$ não herda uma estrutura  natural de grupo. Ainda assim, $K$  possui boas propriedades como espaço métrico.

\begin{thm}\label{thm:AsCone} 
O cone assintótico $K$ de um grupo finitamente gerado $G$ é um espaço métrico  homogêneo, conexo por caminhos, localmente conexo e completo.
\end{thm}
\begin{proof}
Assim como  em qualquer grupo, a  multiplicação à esquerda em $M$ (por algum elemento $\mathbf{x}=(x_n)\in M$) é uma permutação de $M$. Além disso, com respeito à métrica $d_M$, esse mapa é uma isometria de $M$, que por sua vez induz
uma isometria em  $K$. Em particular, se dois elementos $a, b \in  K$ são representados
pelas sequências $(x_n)$ e $(y_n)$, respectivamente, então o  mapa  que manda cada elemento de $K$, representado por alguma sequência $(z_n)$, no  elemento representado por $(y_nx^{-1}_n z_n)$ é uma
isometria de $K$ que leva  $a$ em  $b$. Com isso, $K$ é  homogêneo.

Provaremos que $K$ é conexo por caminhos mostrando que cada elemento $a\in K$ pode ser conectado a $e$ por um caminho contínuo. Seja $\{x_n\} \in M$ uma sequência que representa $a$. Para cada $n\in \N$ fixe uma escrita de $x_n\in G$ como uma palavra de tamanho $\ell_S(x_n)$ em termos dos geradores $S$ de $G$. Seja $0 \leq \alpha \leq 1$ e, para cada palavra $w$ de comprimento $ k$ em $S$, denote por $w(\alpha)$ a palavra consistindo das primeiras $\lceil k \alpha \rceil$ letras
em $w$. Defina $f:[0,1] \to K$ pondo  $f(\alpha)$ como o elemento representado pela sequência $\{x_n(\alpha)\}$.
É claro que $f(0) = e$ e $f(1) = a$. Além disso, caso $0 \leq \alpha \leq \beta \leq 1$, obtemos $$(\beta-\alpha) \ell_S(x_n)-1 \leq d_S(w(\alpha), w(\beta)) \leq (\beta-\alpha)\ell_S(x_n)+1.$$ Como $\ell_S(x_n) \leq An$, concluímos que $d_K(f(\alpha), f(\beta)) \leq A(\beta -\alpha)$, e portanto 
$f$ é contínua. Concluímos que $K$ é conexo por caminhos, portanto conexo.
Além disso, $d_K(e, f(\beta)) \leq  d_K(e, a)$, e assim o caminho ligando $e$ a $a$ está contido na bola de raio $d_K(e, a)$ centrada em  $e$. Isso mostra que toda bola centrada em $e$ é conexa por caminhos. Pela homogeneidade de $K$, isso vale para qualquer bola, e portanto $K$ é localmente conexo.

Resta verificar que $K$ é completo. Para tal, usaremos a seguinte observação: suponha que $ a$ e $b$ são dois pontos de $K$ e que $d_K(a, b) < A$, para algum número $A$. Logo, podemos encontrar sequências $(x_n)$ e $(y_n)$ representando  $a$ e $b$, tais que a desigualdade $d_S(x_n, y_n) < An$ vale em algum conjunto $S$ do ultrafiltro $\Fc$ em $\N$. Para $n \notin S$, podemos trocar $y_n$ por $x_n$. Essa troca não altera o ponto $b$, mas garante que a desigualdade acima vale para todo $n$. Seja agora $(a_i)$ uma sequência de Cauchy em $K$ e, para cada $i$, tome uma sequência representante ${x_{in}}$ de $a_i$. Podemos assumir que $d_K(a_1, a_i) < 1$
para todo $i$. Trocando como acima as sequências $x_{in}$, para $i > 1$, podemos assumir que para cada $n$ e cada $i$ temos $d_S(x_{1n}, x_{in}) < n$,  e portanto $d_S(x_{in}, x_{jn}) < 2n$, para todo par $i, j \in \N$. 

Seja $k$ o primeiro índice tal que $k > 1$
e $d_K(a_i, a_j) < \frac{1}{2}$ para todos $i, j \geq k$. Em seguida, troque $x_{in}$, para $i > k$,  de forma a obter $d_S(x_{kn}, x_{in}) < \frac{n}{2}$ para todo $i > k$. Isso não altera as desigualdades anteriores pois, no pior dos casos, teremos trocado $x_{in}$ por $x_{kn}$, que ainda satisfaz $d_S(x_{1n}, x_{kn}) < n$. Continuando esse processo, vemos que é possível tomar as sequências $x_{in}$ de modo que para cada $n$ e cada $\varepsilon > 0$, a desigualdade $d_S(x_{in}, x_{jn}) <\varepsilon n$ vale, desde que $i$ e $j$ sejam grandes o suficiente. 

Fixe $n\in \N$, e escolha $\varepsilon$ tal que $\varepsilon n < 1$. Então $x_{in}$ e $x_{jn}$
estão a uma distância menor que 1, e portanto são iguais em $G$. Daí, segue que a sequência ${x_{in}}$ é  eventualmente constante. Denote por $x_n$ o eventual valor de ${x_{in}}$.
Se $m$ é um índice tal que $d_S(x_{in}, x_{jn}) <\varepsilon n$ para $i, j \geq m$, então em  particular $d_S(x_n, x_{jn}) < \varepsilon n$ para todo $j \geq m$ (e para todo $n$). Seja $a$ o ponto de $K$ representado pela sequência $(x_n)$. A última desigualdade implica que $d_K(a, a_j) \leq \varepsilon$. Então a sequência  $(a_i)$ converge para $a$, como queríamos demonstrar.
\end{proof}

\section{Ação de um grupo em seu cone assintótico}

 Seja $G$ um grupo gerado por um  conjunto finito $S$ e tome  $x \in G$. Para cada sequência $(x_n)$ tal que $ \ell_S(x_n) \leq An$, temos $\ell_S(xx_n) \leq An+\ell_S(x)$, e portanto a sequência $(xx_n)$ é também  moderada e o  mapa $(x_n)\mapsto (xx_n)$ é uma isometria de $M$. Todas as isometrias de $M$
preservam equivalência, e portanto  induzem isometrias de $K$. Isso define uma ação de $G$ sobre $K$, isto é,  um homomorfismo  
\mbox{$\Phi: G \to \mathrm{Isom}(K)$}. 
Dados $x \in G$  e $a \in K$, denotamos por $xa$ o elemento $\Phi(x)(a)$, e por $N$ o núcleo de $\Phi$.

\begin{thm}
\label{thm:gpactiononK}
    Sejam $G$ um grupo infinito finitamente gerado ,  $K$ o cone assintótico de $ G$ e  $I := \mathrm{Isom}(K)$ o grupo de isometrias  de $K$. Então o homomorfismo $\Phi : G \to  I$ com núcleo $N=\ker (\Phi)$ definido acima satisfaz uma das seguintes afirmações:

    \begin{itemize}
        \item[(i)] O  quociente $G/N$ é infinito;
        \item[(ii)] O subgrupo $N$ é virtualmente  abeliano;
        \item[(iii)] Para cada vizinhança $ O$ da identidade em $I$, existe um homomorfismo $\varphi_O : N \to I$, tal que $\mathrm{Im} (\varphi_O)\cap O$ contém elementos não triviais.
    \end{itemize}
\end{thm}

Para estudar a ação de $G$ sobre $K$ via $\Phi$ e demonstrar o teorema acima, precisaremos de alguns resultados intermediários. Defina, para $x \in G$, o \textit{deslocamento} de $x$ como a função  em termos de $k\in \N$  dada por $D(x, k) = \max \{d(a, xa)\} =
\max \{\ell_S(a^{-1}xa)\}$, onde o máximo é tomado sobre os elementos $a \in G$ tais que  $\ell_S(a) \leq k$.
Se $x \in H \leq G$, e restringimos a definição acima para $a\in H$, escrevemos $D_H(x, k)$ para o máximo correspondente. As seguintes proposições fornecem propriedades da função deslocamento, que serão utilizadas na demonstração do Teorema~\ref{thm:gpactiononK}.

\begin{proposition}
\label{prop:D1} 
    Mantendo as notações acima, se $x \in N$, então $$\Fc \displaystyle \lim_{k \to \infty} \frac{D(x, k)}{k} = 0.$$
\end{proposition}

\begin{proof}
Para cada $k\in \N$, escolha $a_k \in G$ tal que $\ell_S(a_k) \leq k $ e $\ell_S(a^{-1}_k xa_k) =
D(x, k)$. A sequência $\alpha := (a_k)$ é moderada. Se $x \in N$, então $x\alpha = \alpha$.
Portanto $\Fc \lim \frac{D(x, k)}{k} = d_K(x\alpha, \alpha) = 0$.
\end{proof}

\begin{proposition}
\label{prop:D2}
    A função $D(x, k)$ é limitada se, e somente se, $x$ possui apenas finitos conjugados em $G$, e nesse caso $x \in N$.
\end{proposition}

Esse fato segue da  definição. Elementos que possuem apenas finitos conjugados serão chamados de FC-elementos. Eles formam um subgrupo normal de $G$ (também chamado FC-centro de $G$).

Suponha que $|G : H| <\infty$ e $x \in H$. Então $x$ tem apenas finitos conjugados em $G$ se, e somente se, ele tem apenas  finitos  conjugados em $H$. Assim, é equivalente dizer que $D(x, k)$ é limitada ou que $D_H(x, k)$ o é.

\begin{proposition}
\label{prop:D3}
    Para $x, y \in G$ e  $k, n\in \N$ temos $D(x, k +n) \leq D(x, k) + 2n$ e $D(y^{-1}xy, k) \leq D(x,k) + 2\ell_S(y)$.
\end{proposition}

\begin{proof}
Seja $\ell(a) \leq k + n$. Então podemos escrever $a = bc$ com $\ell_S(b) \leq k$ e
$\ell_S(c) \leq n$. Logo, $d(xa, a) = d(xbc, bc) \leq d(xbc, xb) + d(xb, b) + d(b, bc) = d(xb, b) + 2d(b, bc) \leq d(xb, b) + 2n \leq D(x, k) + 2n$, e daí  $D(x, k+n) \leq
D(x, k)+2n$. Com isso, se $l(a) \leq k$, então $d(y^{-1}xya, a) = d(xya, ya) \leq D(x, k+\ell_S(y))$, e aplicamos a desigualdade anterior.
\end{proof}

\begin{proposition}
\label{prop:infclasses}
Um grupo $G$ não pode consistir da união de finitas classes laterais de subgrupos de índice infinito.
\end{proposition}

\begin{proof}
Suponha que $G$ é dado por uma tal união, e que as classes pertencem a $k$
 subgrupos distintos. Usamos indução em $k$, onde o caso $k = 1$ é direto.
Para $k > 1$, seja $H$ um dos subgrupos envolvidos. Como $|G : H| = \infty$,
alguma classe $Hx$ não ocorre na união, e como é disjunta  das classes de $H$ que aparecem, ela está contida na união das classes dos outros subgrupos. Toda classe $Hy$ pode ser escrita como  $Hx\cdot x^{-1}y$, e isso mostra que $Hy$ está também contido em uma união  finita de classes dos outros subgrupos. Isso implica que todas as classes de $H$ que aparecem na união estão contidas em  uniões finitas de classes de outros  subgrupos, e portanto $G$ é a união de finitas classes dos outros $k - 1$ 
subgrupos, o que contradiz a hipótese de indução.    
\end{proof}

\noindent \textit{Prova do Teorema \ref{thm:gpactiononK}.}
Como queremos mostrar que uma das três afirmações é verdadeira, podemos assumir que $|G : N|$ é finito. Caso contrário, já vale $(i)$ e não temos nada a demonstrar.
Então $N$ também é finitamente gerado, digamos por $\{y_1,\ldots , y_d\}$. Se $D(y_j, k)$ for limitada para todo $j=1,\ldots,d$, então $|N : \displaystyle \bigcap_{j=1}^d C_N(y_j)|$ é finito. Como essa interseção coincide com o  centro de $N$, concluímos que $(ii)$ é verdadeira nesse caso.

Vamos assumir, portanto, que $N$ contém  elementos não-FC, isto é, elementos que não possuem apenas finitos conjugados. Como os FC-elementos formam um subgrupo próprio de $N$, podemos também assumir que nenhum dos geradores é um FC-elemento.

Fixe números $k\in \Z$ e $\varepsilon>0$. Para cada $j=1,\ldots,d$, os conjuntos 
$\{y \in N \mid D(y^{-1}y_jy, k) \leq \varepsilon k\}$ e 
$\{y \in N \mid \ell_S(a^{-1}y^{-1}y_jya) \leq \varepsilon k, \mbox{ para todo } a\in G \mbox{ tal que } \ell_S(a) \leq k\}$ são iguais. Isso implica que o conjugado $y^{ya}_j$ de $y_j$ é um dos finitos elementos, e portanto o elemento $ya$ está em uma das  finitas classes laterais do centralizador $C_G(y_j)$.  Em particular, tomando $a = 1$, vemos que o próprio $y$ está em uma dessas finitas classes. Como cada um dos
$y_1,\ldots, y_d$ possui infinitos conjugados em N, ou seja, seus centralizadores possuem índice infinito, a Proposição~\ref{prop:infclasses} mostra que $N$ não é a união de finitas classes dos centralizadores dos geradores, e portanto existe algum $z_k \in N$ com $D(z^{-1}_k y_jz_k, k) > \varepsilon k$ para todo $j=1,\ldots,d$.

Escrevemos $z_k$ em termos dos geradores $y_j$, e escolhemos a primeira subpalavra inicial $x_k$ de $z_k$  para a qual
$D(x^{-1}_k y_jx_k, k) >\varepsilon k$, para algum $j$. Escolhemos também para cada $k$ um tal índice
$j = j(k)$, e para cada $i \leq d$ escrevemos $S(i) := k$ tal que $j(k) = i$. Os finitos conjuntos $S(i)$ particionam $\N$, e portanto um deles, digamos $S(i_0)$, está em $\Fc$.
Seja $l$ o máximo entre os comprimentos, em termos dos geradores de $G$, dos elementos $y_j$. Podemos tomar $k$ grande o suficiente, de modo que $D(y_j, k) \leq \varepsilon k$, pela Proposição~\ref{prop:D1}. Logo,  $x_k \neq 1$, e podemos escrever $x_k = w_ky$, onde $w_k$ é a subpalavra inicial de $z_k$ que antecede $x_k$, e $y$ é algum gerador. Então a  Proposição~\ref{prop:D3} mostra que $D(x^{-1}_k y_jx_k, k) \leq D(w^{-1}_ky_jw_k, k) + 2l \leq \varepsilon k + 2l$. Isso vale para cada $j$, mas para  $i_0$ também temos $D(x^{-1}_k y_{i_0}x_k, k) > \varepsilon k$. Segue que  $\Fc\lim \dfrac{D(x^{-1}_ky_{i_0}x_k, k)}{k} =\varepsilon$. Sempre vale que  $\ell_S(x) \leq D(x, k)$,
para todo $k$, e portanto a desigualdade anterior mostra que $ \ell_S(x^{-1}_k y_jx_k) \leq \varepsilon k + 2l$, para cada $j$, e assim, se $y \in N$ tem comprimento $m$ nos geradores $y_j$, obtemos $l(x^{-1}_kyx_k) \leq m \varepsilon k + 2ml$.

Portanto, multiplicação à esquerda pela  sequência ${x^{-1}_kyx_k}$ preserva  sequências moderadas, e induz uma isometria em  $K$. Definimos $\varphi(y)$ como essa isometria. Então $\varphi$ é um  homomorfismo $N \to I$. 
Aplicando a Proposição~\ref{prop:D1} a esse homomorfismo, obtemos que $\varphi(y_{i_0} )\neq 1$,
pois $\Fc \lim \dfrac{D(x^{-1}_k y_{i_0}x_k, k)}{k} \neq 0$. Por outro lado, $d_K(\varphi(y_{i_0} )(\alpha), \alpha) \leq \varepsilon$ para todo $\alpha \in K$, o que mostra que $\varphi(y_{i_0} )$ pode ser tomado em qualquer vizinhança dada da identidade em $I$, tomando $\varepsilon$ pequeno o suficiente, como queríamos demonstrar.
\hfill$\square$

\chapter{Espaços hiperbólicos}
\label{cap7}
\section{Geometria hiperbólica}
\label{sec:esphiperbólico}
A geometria hiperbólica surge após uma longa história de tentativas  de estabelecer uma resposta à seguinte questão: 

\begin{quote}\textit{``Na geometria euclidiana, o axioma das retas paralelas é independente, ou decorre dos outros axiomas?''}
\end{quote}

Uma resposta foi dada por N.I.~Lobachevsky, J.~Bolyai e C.F.~Ga\-uss de forma independente no início do século XIX. Eles mostraram que o axioma de paralelas é de fato independente. Foram desenvolvidos modelos de geometrias não-euclidianas, onde o quinto postulado de Euclides é substituído pelo axioma:
\begin{quote}
\textit{``Para cada ponto $P$ que não pertence a uma reta $L$, existem infinitas retas contendo $P$ e que são paralelas a $L$.''}    
\end{quote}  
O primeiro modelo de espaço satisfazendo esse axioma foi o do plano hiperbólico,  construído por E.~Beltrami em 1868. Outros modelos importantes que  vieram depois foram os de Klein (modelo projetivo), Poincaré (modelo do disco e modelo do semiplano superior) e Lorentz (modelo do hiperboloide). Em \cite{milnor1982hyperbolic}, há uma bela descrição da história do desenvolvimento da geometria hiperbólica. Os exemplos de espaços hiperbólicos apresentados a seguir servirão como motivação para definições mais gerais de espaços hiperbólicos de Gromov e Rips. 

Veremos nesta seção algumas propriedades básicas da geometria hiperbólica, começando com o modelo do semiplano superior de Poincaré $\hip$. Em seguida, definiremos o espaço hiperbólico \mbox{$n$-di}\-men\-sional $\mathbb{H}^n$, com o objetivo de entender como a geometria do plano hiperbólico pode ser generalizada para espaços de dimensão maior. Existem muitos textos clássicos que podem ser consultados para encontrar demonstrações dos resultados e outras propriedades aqui brevemente discutidas. O leitor pode se referir, por exemplo,  a \cite{anderson2006hyperbolic}, \cite{beardon2012geometry} ou \cite{katok1992fuchsian}. 

\subsection{O modelo do semiplano superior de Poincaré}

Seja $\hip=\{z\in\mathbb{C}\mid \im(z)>0\}$ o semiplano superior em $\C$, munido da métrica definida por
\begin{equation}
\label{distformulaH}
\cosh{d(z,w)}= 1+\dfrac{|z-w|^2}{2\,\im(z)\,\im(w)}.
\end{equation}
Alternativamente, podemos definir uma métrica riemanniana em $\mathbb{H}^2$ por: $$ds^2 = \dfrac{dx^2+dy^2}{y^2},$$ onde $x=\re(z)$ e $y =\im(z).$ Com essa métrica, obtemos que o comprimento de uma curva $\alpha:[a,b]\to \hip$  dada por $\alpha(t)=x(t)+iy(t)$ é: 
$$L(\alpha) = \int_{a}^{b}\dfrac{\sqrt{x'(t)^2+y'(t)^2}}{y(t)}dt.$$ 

Poderíamos também ter definido a métrica em $\hip$ determinando que a distância entre dois pontos $p,q\in \hip$ é o ínfimo dos comprimento de todas as curvas suaves por partes ligando $p$ a $q$, isto é, $$d_{\hip}(p,q)=\inf_{\alpha} L(\alpha),$$ onde o ínfimo é tomado entre todas as curvas diferenciáveis por partes ligando os dois pontos (essa condição facilita a demonstração da desigualdade triangular e a definição seria equivalente se considerássemos apenas curvas diferenciáveis). 

\begin{exercise}
Verifique que as funções $d$ e $d_{\hip}$ são iguais e definem uma métrica em $\hip$.
\end{exercise}

O espaço $\hip$ munido da métrica $d_{\hip}$ é chamado  \textit{plano hiperbólico}\index{plano hiperbólico}.

\subsubsection*{Isometrias} 
O grupo $\SL(2,\R)$, de matrizes reais $2 \times 2$ com determinante igual a 1, age em $\mathbb{H}^2$ através de transformações fracionais lineares:
\begin{equation}
\label{eq:mobiusaction}
\left( \begin{array}{cc}
a & b \\ 
c & d
\end{array} \right)\cdot z = \dfrac{az+b}{cz+d}.  
\end{equation}

\begin{exercise}
Verifique que a ação acima mapeia $\hip$ em $\hip$ preservando a métrica.
\end{exercise}

As transformações do tipo $T: z \mapsto \dfrac{az+b}{cz+d}$, associadas a matrizes $g= \left(\begin{array}{cc}
a & b \\ 
c & d
\end{array}\right) \in \SL(2,\R)$, formam um grupo de homeomorfismos de $\hip$, e são chamadas de \textit{transformações de Möbius}\index{transformações de Möbius}. Note que a matriz que representa a composição de duas transformações de Möbius é dada pelo produto das matrizes correspondentes, bem como a matriz que representa $T^{-1}$ é a inversa da matriz que representa $T$. Cada transformação de Möbius $T$ está representada na verdade por um par de  matrizes $\pm g$. Por isso, o conjunto de todas as transformações de Möbius de $\hip$ é denotado por $\PSL(2,\R)$ e corresponde ao grupo $\SL(2,\R)/\{\pm \id\}$.

Essa ação é isométrica, transitiva, e bitransitiva sob pares de pontos equidistantes. Além disso, o grupo $\mathrm{Isom}^+(\mathbb{H}^2)$ das isometrias de $\hip$ que preservam orientação é exatamente $\PSL(2,\R)$. É um bom exercício verificar todas essas propriedades, cujas demonstrações também podem ser encontradas nos livros mencionados acima.  

Sejam $w \in \mathbb{H}^2$ e $R \in \R_{>0}$ dados. Então um {\it círculo hiperbólico} de raio $R$ centrado em $w$ é o conjunto 
$$C_{\mathbb{H}^2}(w,R) = \{z \in \mathbb{H}^2 \mid d_{\mathbb{H}^2}(z,w)=R\}.$$

 \begin{exercise}\label{exer-circulo}
 Mostre que o círculo hiperbólico $C_{\mathbb{H}^2}(i,R) $ em $\hip$ coincide, como conjunto, com um círculo euclidiano, com centro $z= i \cosh(R)$ e raio $\delta =\sinh(R)$.
\end{exercise}

Observamos que a ação de $\PSL(2,\R)$ em $\hip$ leva \textit{círculos generalizados} em círculos generalizados, onde estamos nos referindo com o termo círculo generalizado a qualquer elemento do conjunto dos círculos euclidianos e das retas euclidianas em $\C$, que podem ser representados por uma equação do tipo $$\alpha z \Bar{z} + \beta z + \Bar{\beta}\Bar{z} + \rho = 0.$$ onde $\alpha, \rho \in \R$ e $\beta \in \C$. Caso $\alpha=0$, teremos uma reta euclidiana e, caso contrário, uma circunferência. Além disso, o subconjunto dos círculos generalizados que são ortogonais ao eixo real também é mantido invariante por essa ação.

\begin{exercise}
\label{exer:geodésicas}
    Mostre que o conjunto dos círculos generalizados e o conjunto dos círculos generalizados ortogonais ao eixo real são mantidos invariantes pela ação de transformações de Möbius (você pode pensá-las como definidas em $\C$). Além disso, a ação é transitiva em cada um dos conjuntos.
\end{exercise}

\subsubsection*{Geodésicas} 
Um \textit{caminho geodésico} em um espaço métrico $X$ é um caminho de comprimento mínimo que une dois pontos em $X$ e é parametrizado pelo comprimento de arco. Uma \textit{geodésica}\index{geodésica} em $X$ é uma curva em que cada segmento local é um caminho geodésico.

Para descrever as geodésicas em $\hip$, vamos primeiro considerar um caso especial de dois pontos em uma reta ortogonal ao eixo real $\im(z)=0$. 

Sejam $z_0 = y_0i$, $z_1 = y_1i \in \hip$, com $0 < y_0 < y_1$. Vamos verificar que 
$$d(z_0, z_1) = \log\left(\frac{y_1}{y_0}\right)\text{ e }\phi(t) = i y_0 e^{t},\ t \in [0, \log\left(\frac{y_1}{y_0}\right)]$$ 
é o único caminho geodésico parametrizado pelo comprimento de arco que une $z_0$ e $z_1$. 
Temos:
$$d(z_0, z_1) \leq L(\phi) = \int_{0}^{\log(\frac{y_1}{y_0})}\dfrac{y_0 e^{t}}{y_0 e^{t}}dt = \log\left(\frac{y_1}{y_0}\right).$$
Por outro lado, suponha que $\psi : [a, b] \to \hip$ seja outro caminho de $z_0$ para $z_1$. Seja $\psi(t) = \psi_x(t) + i\psi_y(t)$. Então
\begin{align*}
L(\psi) &= \int_{a}^{b}\dfrac{\sqrt{\psi_x'(t)^2+\psi_y'(t)^2}}{\psi_y(t)}dt 
\geq \int_{a}^{b}\dfrac{|\psi_y'(t)|}{\psi_y(t)}dt
\geq \int_{a}^{b}\dfrac{\psi_y'(t)}{\psi_y(t)}dt \\
&= \log(\psi_y(b)) - \log(\psi_y(a)) = \log\left(\frac{y_1}{y_0}\right).
\end{align*}
A igualdade nesta fórmula implica que $\psi_x'(t) = 0$ para todo $t\in [a,b]$ e $\psi_y'(t) > 0$, assim $\psi(t)$ é igual a $\phi(t)$ a menos de uma reparametrização monótona. 

Isso mostra que o raio vertical $\re(z) = 0$ é uma geodésica em $\hip$. Aplicando as isometrias de $\hip$, tendo em vista o Exercício~\ref{exer:geodésicas}, obtemos imediatamente que as geodésicas de $\hip$ são semicírculos centrados no eixo real $\im(z)=0$ e semirretas verticais. 
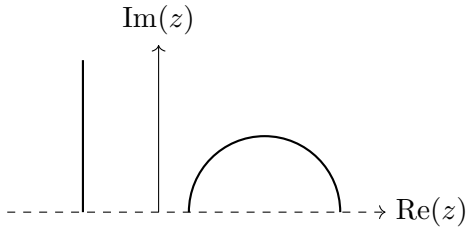
\begin{figure}[!ht]
\label{fig:geodesicsH}
\centering
\begin{tikzpicture}[scale=1]
		\draw[dashed, ->] (-2,0) -- (3,0) node [right] {$\re(z)$};
		\draw[->] (0,0) -- (0,2.2) node [above] {$\im(z)$};

           \coordinate (A) at (-1,0.1);
  \coordinate (B) at (-1,0);
  \coordinate (C) at (-0.9,0);
  \rightAngle{A}{B}{C}{0.2};

           \coordinate (D) at (0.4,0.1);
  \coordinate (E) at (0.4,0);
  \coordinate (F) at (0.5,0);
  \rightAngle{D}{E}{F}{0.2};

             \coordinate (G) at (2.4,0.1);
  \coordinate (H) at (2.4,0);
  \coordinate (I) at (2.5,0);
  \rightAngle{I}{H}{G}{0.2};

		\draw[thick](-1,2) -- (-1,0);
  
		\draw[thick](0.4,0) arc (180:0:1cm) --(2.4,0);	
\end{tikzpicture}
\caption{Geodésicas em  $ \mathbb{H}^2 $. }
\end{figure}

Como corolário, obtemos que, para qualquer par de pontos $z, w \in \mathbb{H}^2$, existe uma única geodésica unindo tais pontos. Além disso, para cada ponto $z$ que não pertence a uma geodésica $L$, existem infinitas geodésicas contendo $z$ e que não se intersectam com $L$.


\subsubsection*{Círculos e discos em $\hip$} 

O \textit{círculo} hiperbólico de raio $R$ com centro $p\in \hip$ é definido por 
$$ C(p,R) = \{x \in \mathbb{H}^2 \mid d(p, x) = R\}.$$
O círculo $C(p,R)$ delimita um \textit{disco} 
$$ D(p,R) = \{x \in \mathbb{H}^2 \mid d(p, x) \le R\}.$$

Como já vimos no Exercício~\ref{exer-circulo}, todo círculo hiperbólico no semiplano superior $\hip$ é um círculo euclidiano (com centro e raio diferentes).

Vamos calcular o comprimento hiperbólico de um círculo e a área de um disco. 

Usando isometrias e o observação anterior, podemos assumir que nosso círculo $C$ é um círculo euclidiano com centro em \((0, y_0)\) e raio euclidiano \(\delta < y_0\). Em coordenadas o círculo é dado por:
\begin{equation*}
\left\{
\begin{aligned}
x(t) &= \delta  \cos(t); \\
y(t) &= y_0 + \delta  \sin(t),\ t\in [0, 2\pi].
\end{aligned}
\right.
\end{equation*}
Então o comprimento hiperbólico de $C$ é
$$L = \int_C \dfrac{\sqrt{dx^{2}+dy^{2}}}{y} = \int _{0}^{2\pi}\dfrac{\delta }{y_{0}+\delta \sin (t)}\,dt = \dfrac{2\pi \delta }{\sqrt{y_{0}^{2}-\delta ^{2}}}.
$$
Resta traduzir esse resultado para os parâmetros hiperbólicos. Observe que os pontos $y_{0}+\delta $ e $y_{0}-\delta $ do círculo são diametralmente opostos. Segue-se que o raio hiperbólico 
$$ R=\frac{1}{2}\ln \left(\frac{y_{0}+\delta }{y_{0}-\delta }\right),
\text{ então } \sinh(R) = \dfrac{\delta }{\sqrt{y_{0}^{2}-\delta ^{2}}}.
$$
Isso implica que $L = 2\pi\sinh(R)$.

Para calcular a área hiperbólica do disco $D(p, R)$, precisamos integrar o elemento de área \(dA=\frac{dx\,dy}{y^{2}}\) sobre a região delimitada por $C$. Deixamos os detalhes do cálculo como exercício. Como resultado, obtemos 
$$A = 4\pi \sinh^{2}\left(\frac{R}{2}\right).$$

Observe que, para um raio $R$ suficientemente grande, o comprimento e a área dependem \textit{exponencialmente} do raio, enquanto que, para $R$ próximo de zero, as fórmulas convergem para o comprimento e a área euclidianos.

\subsubsection*{Classificação das isometrias} 
As isometrias de $\hip$ que preservam orientação são classificadas em termos dos seus pontos fixos. Para isto, definimos o \textit{bordo}, ou \textit{fronteira ideal}, do espaço $\mathbb{H}^2$ como $\partial \mathbb{H}^2 = \{z \in \C \mid \im(z) =0 \}\cup \{\infty \}$. As isometrias são então classificadas entre os seguintes tipos:
\begin{itemize}
    \item $g$ é dita \textit{elíptica}\index{isometria elíptica} se ela possui um único ponto fixo, e ele está em $\hip$;
    \item $g$ é dita \textit{parabólica}\index{isometria parabólica} se ela possui um único ponto fixo, contido em $\partial \mathbb{H}^2$;
    \item $g$ é dita \textit{hiperbólica}\index{isometria hiperbólica} se ela possui dois pontos fixos, ambos em $\partial \mathbb{H}^2$. Uma transformação hiperbólica $g$ deixa invariante uma geodésica $\gamma$ em $\hip$, a qual liga os dois pontos fixos do bordo e é chamada de \textit{eixo}\index{eixo} de $g$. Ao longo dessa geodésica, $g$ age como uma translação por uma distância positiva $\tau_g$, chamada o \textit{número de translação}\index{número de translação} de $g$. 
\end{itemize}

De maneira geral, dada $g\in \mathrm{Isom}(X)$, onde $(X,d)$ é um espaço métrico, podemos definir o \textit{número de translação} de $g$ por 
$$\tau _g := \displaystyle \inf _{x\in X}d(x,gx).$$
Com essa noção, a classificação de isometrias de $\hip$ que preservam orientação dada acima se traduz em:\begin{itemize}
    \item $g$ é  \textit{elíptica} se $\tau_g = 0$ e o ínfimo é atingido em $X$;
    \item $g$ é \textit{parabólica}  se $\tau_g = 0$ e o ínfimo não é atingido em $X$;
    \item $g$ é  \textit{hiperbólica} se $\tau_g > 0$.
\end{itemize}

Levando em consideração que $\mathrm{Isom}^+(\mathbb{H}^2) \cong \PSL(2,\R)$, a classificação pode ser feita a partir dos traços das matrizes $\pm g  \in \SL(2,\R)$ que representam uma isometria. Lembramos que o traço de $g= \left(\begin{array}{cc}
a & b \\ 
c & d
\end{array}\right) $ $\in \SL(2,\R)$ é $\tr(g)=a+d$. Assim, obtemos a seguinte caracterização:

\begin{proposition}
\label{prop:caracterizaçãoisomH}
Seja $g \in \mathrm{Isom}^+(\hip)$. Então 
    \begin{enumerate}[(1)]
        \item $g$ é elíptica se, e somente se, $\tr(g) \in (-2,2)$;
        \item $g$ é parabólica se, e somente se, $\tr(g) \in \{-2,2\}$;
        \item $g$ é hiperbólica se, e somente se, $\tr(g) \notin [-2,2]$.
    \end{enumerate}
\end{proposition}

\begin{exercise}
    Provar a Proposição \ref{prop:caracterizaçãoisomH}.
\end{exercise}

\subsubsection*{Outros modelos}
Existem diversos outros modelos para o plano hiperbólico, sendo o modelo do disco de Poincaré um dos mais usados. O disco de Poincaré é o conjunto $\mathbb{D} = \{z \in \C \mid |z| < 1\}$, munido da métrica 
definida de modo que a função $h:\mathbb{H}^2 \to \mathbb{D}$ dada por $h(z) = \dfrac{z-i}{iz-1}$ seja uma isometria. Isso permite mostrar que as geodésicas no modelo do disco de Poincaré são os arcos de círculos e diâmetros em $\mathbb{D}$ que cortam o bordo de $\mathbb{D}$ ortogonalmente.

\begin{figure}[!ht]
\label{fig:geodesicsD}
\centering
\begin{tikzpicture}[scale=1]
		\draw[dashed, ->] (0,0) circle (1);
        \draw (1,-1) node {$\mathbb{D}$};
  
		\draw[thick] (-0.707,0.707) -- (0.707,-0.707) ;
 \coordinate (O) at (0,0);
  \coordinate (X) at (0.707,-0.707);
  \coordinate (Y) at (0,-1.4);
  \rightAngle{O}{X}{Y}{0.2};

\coordinate (F) at (-0.707,0.707);
  \coordinate (G) at (-1.4,0);
  \rightAngle{G}{F}{O}{0.2};
  
  \coordinate (B) at (1,0);
  \coordinate (C) at (1,-1);
  \rightAngle{O}{B}{C}{0.2};

    \coordinate (D) at (0,1);
  \coordinate (E) at (-1,1);
  \rightAngle{E}{D}{O}{0.2};
  
		\draw[thick](1,0) to [bend left=45] (0,1);	
\end{tikzpicture}
\caption{Geodésicas em  $ \mathbb{D} $. }
\end{figure}

Uma vantagem do modelo do disco é que o disco unitário $\mathbb{D}$ é um conjunto limitado em $\R^2$. Isso facilita a visualização de todo o plano hiperbólico em uma folha de papel e torna mais simples fazer alguns desenhos. Uma vantagem do modelo do semiplano superior com respeito ao modelo do disco é a possibilidade de uso das coordenadas cartesianas em cálculos.

Referimo-nos a \cite{beardon2012geometry} para as descrições dos modelos de Klein e hiperboloide. Uma discussão mais aprofundada sobre os diferentes modelos de espaços hiperbólicos pode ser encontrada em \cite{Geometry2_EMS_part1}.

\subsection{O espaço hiperbólico \texorpdfstring{$n$}{n}-dimensional}

Semelhantemente à maneira como o grupo $\PSL(2, \R)$ age no plano hiperbólico, podemos identificar $\PSL(2, \C)$ com o grupo de isometrias do espaço hiperbólico $3$-dimensional. Este espaço  pode ser realizado no modelo do semiespaço superior
$$\mathbb{H}^3 = \{(z, t) : z \in \mathbb{C},\ t > 0 \}$$
com a métrica riemanniana de curvatura constante $-1$ dada por $ds^2 = \frac{|dz|^2 + dt^2}{t^2}$. Dada $\gamma = \begin{pmatrix} a & b \\ c & d \end{pmatrix} \in \mathrm{SL}(2,\mathbb{C})$, a ação em $(z, t) \in \mathbb{H}^3$ é:

$$\gamma \cdot (z, t) = \left( \frac{(az+b)\overline{(cz+d)} + a\bar{c}t^2}{|cz+d|^2 + |c|^2 t^2},\ \frac{t}{|cz+d|^2 + |c|^2 t^2} \right).$$
Deixamos para o leitor a verificação de que essa ação é isométrica.

Mais geralmente, podemos definir o \textit{espaço hiperbólico $n$-dimen\-sio\-nal}\index{espaço hiperbólico}  a partir do modelo do semiespaço superior $$\mathbb{H}^n = \{(x_1, \ldots,x_n) \in \R^n \mid x_n>0\},$$ com métrica riemanniana $ds^2= \dfrac{dx_1^2+\ldots+dx_n^2}{x_n^2}$. 

A variedade riemanniana  $(\mathbb{H}^n, ds^2)$ é usualmente chamada de espaço hiperbólico \textit{real} $n$-dimensional, para distingui-lo  de outros espaços também ditos hiperbólicos (por exemplo, espaço hiperbólico complexo, espaço hiperbólico quaterniônico, espaços Gromov-hiperbólicos, etc.).
Usaremos a expressão espaço hiperbólico para $\mathbb{H}^n$, com a adição do adjetivo real caso outras  noções de hiperbolicidade estejam envolvidas na discussão. Qualquer que seja $n\geq 2$, o espaço $\mathbb{H}^n$ é uma variedade riemanniana completa.

As isometrias de $\mathbb{H}^n$ que preservam orientação são dadas pelas transformações de Möbius generalizadas, que passamos a descrever. 
Vamos pensar na esfera $\s^m$ como a compactificação a um ponto do espaço euclidiano $ \R^m$, isto é, $\s^m = \widehat{\R^m} = \R^m \cup \{\infty\} $.

Consequentemente, olharemos a  compactificação a um ponto de um hiperplano em $\R^n$ como uma esfera de raio infinito, e a  compactificação a um ponto de uma reta em $\R^n$ como um círculo.
\begin{definition}
    A \textit{inversão}\index{inversão} na esfera $\Sigma_0 = \{x \in \R^{n} \mid |x| = r\}$ é o mapa $J_{\Sigma_0}:\R^{n}\cup\{\infty \} \to \R^{n}\cup\{\infty \} $ dado por 
\begin{equation*}
\left\{\begin{array}{ccl}
     J_{\Sigma_0}(x) &=& r^2 \frac{x}{|x|^2}, \mbox{ se } x \neq 0, \infty,  \\
     J_{\Sigma_0}(0) &=& \infty, \\
     J_{\Sigma_0}(\infty) &=& 0.
\end{array}
    \right.
\end{equation*}
\end{definition}

Definimos a inversão na esfera $\Sigma_a = \{x \in \R^{n} \mid |x - a| = r\} $ pela
fórmula $J_{\Sigma_a} = T_a \circ J_{\Sigma_0 } \circ T_{-a}: \R^{n}\cup\{\infty \}  \to \R^{n}\cup\{\infty \} $,
\begin{equation*}
     J_{\Sigma_a}: x \mapsto r^2 \frac{x-a}{|x-a|^2}+a, \,\,\, J_{\Sigma_a}(a) = \infty, \,\,\, J_{\Sigma_a}(\infty)=a,
\end{equation*}
onde $T_a: x \mapsto x+a$ é a translação pelo vetor $a\in \R^{n}$. Inversões mapeiam esferas  em esferas e  círculos em círculos. Elas também preservam ângulos euclidianos. Veremos a reflexão sobre um hiperplano euclidiano como uma inversão que 
fixa $\infty$.

\begin{definition}
Uma \textit{transformação de Möbius de $\R^n$}\index{transformações de Möbius} (ou, mais precisamente, de $\s^n$) é uma composição de finitas inversões em $\R^n$. O grupo de todas as transformações de Möbius de $\R^n$ é denotado por $\mathrm{Mob}(\R^n)$ ou $\mathrm{Mob}(\s^n)$.
\end{definition}
Em particular, transformações de Möbius preservam ângulos, enviam círculos em círculos e esferas em esferas.

Observamos que a  métrica $ds^2$ em $\mathbb{H}^n$ é claramente invariante pelas translações horizontais
euclidianas  $x \mapsto x + v$, onde $v = (v_1, \ldots,v_{n-1}, 0)$, já que as translações preservam a métrica euclidiana e a coordenada $x_n$. Analogamente, $ds^2$ é invariante por dilatações
$h : x \mapsto \lambda x$,  $\lambda > 0$
já que $h$ reescala tanto o numerador quanto o denominador na expressão da métrica. Por fim, $ds^2$ é invariante
por rotações que fixam o eixo $x_n$ (pois  preservam a
coordenada $x_n$). O grupo gerado por essas isometrias de $\mathbb{H}^n$ age transitivamente em $\mathbb{H}^n$, o que implica que $\mathbb{H}^n$ é uma variedade riemanniana homogênea.

Podemos verificar, por meio de um cálculo, que a inversão $J_{\Sigma_0}$ na esfera unitária $\Sigma_0$ centrada na origem é uma isometria de $\mathbb{H}^n$. Todas as outras transformações de Möbius que preservam $\mathbb{H}^n$ são composições da inversão $J_{\Sigma_0}$ com translações horizontais e dilatações. Portanto, todo elemento de $\mathrm{Mob}(\s^n)$ que deixa $\mathbb{H}^n$ invariante é uma isometria de $\mathbb{H}^n$. Usando o modelo da bola do espaço hiperbólico, podemos mostrar que, de fato, este conjunto coincide com $\mathrm{Isom}(\mathbb{H}^n)$, como pode ser visto na Seção 4.2 de \cite{kd}.

Utilizando o outro modelo do espaço hiperbólico $n$-dimensional, o modelo do hiperboloide, é possível identificar o grupo de isometrias completo com o grupo ortogonal projetivo $\mathrm{PO}(n,1)$. Sugerimos que o leitor se refira a \cite{thurston1997three} e \cite{ratcliffe2006hyperbolic} para todos os detalhes sobre este modelo.

De forma semelhante ao caso planar, podemos observar que as geodésicas de $\mathbb{H}^n$ são arcos de círculos ortogonais ao plano $x_n=0$ ou segmentos de retas verticais. Além disso, para cada geodésica $\alpha$, existe uma isometria de $\mathbb{H}^n$ que leva $\alpha$ a um segmento do eixo $x_n$. Dado qualquer par de pontos em $\mathbb{H}^n$, existe uma única geodésica que os une.

\subsubsection*{Esferas e horoesferas em $\mathbb{H}^n$} 
Escolha um ponto $p \in \mathbb{H}^n$ e um numero real positivo $R$. A \textit{esfera hiperbólica}\index{{esfera hiperbólica}} de raio $R$ centrada em $p$ é o conjunto 
$$ S(p,R) = \{x \in \mathbb{H}^n \mid d(p, x) = R\}.$$
A região delimitada por uma esfera é uma \textit{bola hiperbólica}\index{bola hiperbólica}
$$ B(p,R) = \{x \in \mathbb{H}^n \mid d(p, x) \leq R\}.$$

\begin{exercise}
\label{exerc:círculos}
    \begin{enumerate}
        \item Prove que $S(e_n,R)\subset \mathbb{H}^n$, $e_n = (0,\ldots, 0,1)$,  coincide com a esfera euclidiana de centro $\cosh(R)e_n$ e raio $\sinh(R)$.
        \item Suponha que  $S \subset \mathbb{H}^n$ é a esfera euclidiana com raio $r$ e centro $x= (x_1,\ldots, x_n)$ de modo que  $x_n = a$. Então $S$ coincide com a esfera hiperbólica $S(p, R)$, onde $R = \frac{1}{2}(\log(a + r) -\log(a- r))$.
    \end{enumerate}
\end{exercise}

A área de uma esfera e o volume de uma bola podem ser calculados por integração, generalizando o cálculo em $\hip$. Isso fornece a área
$$A_{n-1}(R)=c_{n-1}\sinh ^{n-1}(R),$$
onde \(c_{n-1}\) é a área da superfície de uma esfera unitária euclidiana, dada pela fórmula
\(c_{n-1}=\dfrac{2\pi ^{\frac{n}{2}}}{\Gamma \left(\frac{n}{2}\right)}.\)

O volume de uma bola $B(p,R)$ é encontrado integrando a área da superfície ao longo do raio de \(0\) a \(R\):
$$V_{n}(R)=c_{n-1}\int _{0}^{R}\sinh ^{n-1}(t)dt.$$

Por exemplo, em $\mathbb{H}^3$ temos:
$$A_2(R) = 4\pi\sinh^2(R),\ V_3(R) = \pi(\sinh(2R)-2R).$$

\begin{exercise}
Verifique estas fórmulas.   
\end{exercise}

\begin{remark}
Constatamos que o volume de bolas e a área de esferas hiperbólicas crescem exponencialmente em relação ao raio. Disso decorre que a medida (ou volume) de uma  bola hiperbólica suficientemente grande se concentra fortemente em um anel estreito próximo à sua fronteira. Isso difere notavelmente do espaço euclidiano, onde a medida se distribui uniformemente por toda a extensão da bola. Esse fenômeno tem importantes consequências geométricas e analíticas, além de aplicações em áreas modernas, como aprendizado de máquina.
\end{remark}

No modelo do semiespaço superior de $\mathbb{H}^n$, podemos considerar também as esferas euclidianas tangentes à fronteira. Na métrica hiperbólica, essas esferas têm raio infinito. Essa construção nos fornece uma classe interessante de subespaços, denominados horoesferas, que discutiremos agora.
 
Podemos definir horoesferas em um espaço métrico qualquer $(X,d)$. Considere um raio geodésico $\gamma: [0, \infty)\to X$. A função $b_{\gamma}(x,t)=d(x,\gamma(t))-t$ é uma  função decrescente de $t$ (devido à desigualdade triangular) e limitada por baixo por $-d(x, \gamma(0))$ (veja Figura \ref{fig:busemann}). Portanto, existe um limite \begin{equation*}
b_{\gamma}(x)=\lim_{t\to \infty}b_{\gamma}(x,t).
\end{equation*}

\begin{figure}[!ht]
\centering
\begin{tikzpicture}[scale=1]
		\draw[dashed] (-2.9,0) -- (3.9,0);

 \draw (2,0) arc(0:60:3);
\draw (1.25,1.98) -- (1.25, 2.905);
\draw (1.25,2.9)arc(104.48:120:3);

  \node[fill=black, circle, inner sep=1pt] at (1.25,1.98) {};
  \node[fill=black, circle, inner sep=1pt] at (1.25,2.9) {};
  \node[fill=black, circle, inner sep=1pt] at (0.5,2.6) {};
  \node at (0.33,2.77) {\tiny{$\gamma(0)$}};
\node at (1.1,1.8) {\tiny{$\gamma(t)$}};
\node at (1.35,3)   {\tiny{$x$}};
 \node at (0.75,2.2)   {\tiny{$t$}}; 

           \coordinate (A) at (2,0.1);
  \coordinate (B) at (2,0);
  \coordinate (C) at (2.1,0);
  \rightAngle{A}{B}{C}{0.2};

\end{tikzpicture}
\caption{Função de Busemann}
\label{fig:busemann}
\end{figure}

\begin{definition}
A função $b_{\gamma}: X \to \R$ é chamada de  \textit{função de Busemann}\index{função de Busemann} do raio $\gamma$.
\end{definition}

\begin{proposition}
Se $\gamma_1$ e $\gamma_2$ são dois raios assintóticos, então a função $b_{\gamma_1} - b_{\gamma_2}$ é uma função constante.
\end{proposition}

\begin{proof}
Dois raios \(\gamma _{1}\) e \(\gamma _{2}\) são assintóticos se existe uma constante \(C > 0\) tal que
$$d(\gamma _{1}(t),\gamma _{2}(t))\le C\quad \text{para\ todo\ }t\ge 0.$$ 
Fixe um ponto \(x \in X\). Queremos avaliar a diferença das funções de Busemann 
$$ b_{\gamma _{1}}(x)-b_{\gamma _{2}}(x)=\lim _{t\rightarrow \infty }[d(x,\gamma _{1}(t))-d(x,\gamma _{2}(t))].$$
Pela desigualdade triangular aplicada para os pontos de $\gamma_1$ e $\gamma_2$, obtemos 
$$d(x,\gamma _{1}(t))\le d(x,\gamma _{2}(t))+d(\gamma _{2}(t),\gamma _{1}(t)).$$
Como \(\gamma _{1}\) e \(\gamma _{2}\) são assintóticos, \(d(\gamma_2(t), \gamma_1(t)) \le C\), logo:
$$d(x,\gamma _{1}(t))-d(x,\gamma _{2}(t))\le C.$$
Invertendo os papéis de \(\gamma _{1}\) e \(\gamma _{2}\) na desigualdade triangular, obtemos o limite oposto. Ao tomar o limite quando \(t \to \infty\), o termo \(C\) se mantém e o resultado converge para um valor exato. Como isso vale para qualquer ponto \(x \in X\), o valor do limite independe do ponto escolhido. Logo, a diferença entre as duas funções é uma função constante.
\end{proof}

Em particular, segue-se que os conjuntos de subnível e os conjuntos de nível de uma função de Busemann no espaço hiperbólico não dependem do raio $\gamma$, mas apenas do ponto no infinito que $\gamma$ representa.

\begin{definition}
Um subnível da função de Busemann $b_{\gamma}^{-1}(-\infty,a]$, $a \in \R$ é chamado \textit{horobola}\index{horobola} (fechada) com centro $\xi = \gamma(\infty)$. Um nível $b_{\gamma}^{-1}(a)$  da função de Busemann é chamada \textit{horoesfera}\index{horoesfera} centrada em $\xi = \gamma(\infty)$.  
\end{definition}

\begin{figure}[!ht]
\centering
\begin{tikzpicture}[scale=1]
		\draw[dashed, ->] (-1.5,0) -- (3.5,0) node [right] {$\re(z)$};
		\draw[->] (0,0) -- (0,2.5) node [above] {$\im(z)$};
 \draw[thick](3,1) arc (0:360:1cm) --(3,1);	
 \draw[thick] (-1.5,2.2) -- (3.5,2.2);
 \draw[thick] (2,0.5) circle (14pt);
 \draw (2,-0.3) node {$\xi$};
 \draw[fill=white] (2,0) circle (1pt);
\end{tikzpicture}
\caption{Horocírculos em  $ \mathbb{H}^2 $. }
\label{fig:horocirculo}
\end{figure}
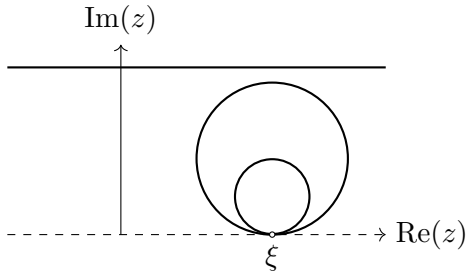

Uma horoesfera em $\mathbb{H}^n$ é uma esfera generalizada que passa por um ponto da fronteira $\partial \mathbb{H}^n := \{(x_1,\ldots,x_n) \in \R^n \mid x_n=0\}\cup \{\infty\}$ de $\mathbb{H}^n$. Observamos que as isometrias de $\mathbb{H}^n$ mapeiam horoesferas em horoesferas. No caso de $\hip$, representado na Figura~\ref{fig:horocirculo}, as horoesferas são também chamadas de horocírculos. Note que, nesse caso, as linhas horizontais também são horocírculos, que passam por $\infty$. A restrição da métrica hiperbólica a uma horoesfera coincide com a métrica euclidiana ali: isso é claro para as horoesferas horizontais e decorre para as demais mediante a aplicação de isometrias.

\subsubsection*{Classificação das isometrias} 

Analogamente ao que ocorre em $\hip$, podemos classificar as isometrias de $\mathbb{H}^n$ em termos de pontos fixos ou do número de translação. Lembre-se que dada $g\in \mathrm{Isom}(\mathbb{H}^n)$, o \textit{número de translação} de $g$ é dado por 
$$\tau_g = \displaystyle \inf _{x\in \mathbb{H}^n}d(x,gx).$$
Para $g \in \mathrm{Isom}^+(\mathbb{H}^n)$ temos as seguintes opções:


\begin{itemize}
    \item $g$ é  \textit{elíptica} se $\tau_g = 0$ e o ínfimo é atingido em $\mathbb{H}^n$; ela possui um único ponto fixo no interior de $\mathbb{H}^n$;
    \index{isometria elíptica}
    \item $g$ é \textit{parabólica} se $\tau_g = 0$ e o ínfimo não é atingido em $\mathbb{H}^n$; ela fixa um único ponto em $\partial \mathbb{H}^n$;
    \index{isometria parabólica}
    \item $g$ é \textit{hiperbólica} se $\tau_g > 0$; ela fixa dois pontos no bordo e deixa invariante a única geodésica que une esses pontos, chamada de \textit{eixo} de $g$. \index{isometria hiperbólica}
\end{itemize}

Observamos que o conjunto de todas as isometrias elípticas que fixam um ponto $p \in \mathbb{H}^n$ forma um subgrupo compacto de isometrias cujas órbitas, sob a ação em $\mathbb{H}^n$, são os círculos hiperbólicos com centro em $p$. De modo semelhante, as órbitas das isometrias parabólicas que fixam $q \in \partial \mathbb{H}^n$ são as horoesferas com centro em $q$.



\subsection{Triângulos no espaço hiperbólico}

Triângulos hiperbólicos desempenham um papel fundamental na extensão da noção de hiperbolicidade a espaços métricos gerais, o que, em particular, conduz à definição de grupos hiperbólicos (veja o Capítulo \ref{cap8}). Nesta seção, consideraremos algumas propriedades básicas dos triângulos hiperbólicos.

Um \textit{triângulo hiperbólico}\index{triângulo hiperbólico} é um triângulo $\Delta = \Delta(a,b,c)$ em $\mathbb{H}^n$ composto por três vértices $a,b,c$, os quais podem ser pontos de $\mathbb{H}^n $ ou pontos da fronteira ideal $\partial \mathbb{H}^n = \{(x_1, \ldots,x_n) \in \R^n \mid x_n=0\}\cup\{\infty\}$, e de três arestas geodésicas ligando esses vértices. No caso em que algum dos vértices de $\Delta$ é um ponto da fronteira ideal, diremos que o triângulo é um \textit{triângulo ideal}\index{triângulo ideal}. Neste último caso, note que o ângulo em $\Delta$ correspondente a um vértice ideal é zero. 

\begin{figure}[h!]
\begin{center}
\begin{tikzpicture}
		\draw[dashed] (-3,0) -- (3,0);
\begin{scope}[scale=0.6]
 \draw (1.37,1.46) arc(46.75:120:2);
\draw  (1.37,1.46)arc(31.58:148.42:2.78);
\draw  (-1,1.73)arc(60:133.24:2);
\end{scope}

\begin{scope}[scale=0.8]
	\draw[dashed] (-2.9,0) -- (3.9,0);

 \draw (3,0) arc(0:12:13.95);
 \draw (3,0) arc(0:60.5:3.01);
\draw (2.7,2.9)arc(84.48:121.3:2);
\end{scope}

\end{tikzpicture}
\end{center}
    \caption{Triângulos no plano hiperbólico.}
    \label{fig:placeholder}
\end{figure}



\begin{exercise} \label{exerc:subesphiperbolico}
    Todo triângulo hiperbólico $\Delta \subset \mathbb{H}^n$ está  contido na  compactificação de um subespaço hiperbólico bidimensional totalmente geodésico $\hip \subset \mathbb{H}^n$. 
    
    \textit{Dica:} Considere um
triângulo $\Delta = \Delta(a,b,c)$, onde $ a, b $ pertencem a uma reta vertical comum.
\end{exercise}

A primeira propriedade básica dos triângulos hiperbólicos que consideraremos é a fórmula para a sua área.

\begin{thm}[Gauss--Bonnet]
Para qualquer triângulo hiperbólico $\Delta$ com ângulos $\alpha$, $\beta$ e $\gamma$, sua área é dada pela fórmula
$$A(\Delta) = \pi - (\alpha + \beta + \gamma).$$
\end{thm}
\begin{proof}
Pelo exercício anterior, podemos supor que $\Delta \subset \hip$.
Suponha, primeiramente, que o ângulo $\gamma$ no vértice c é igual a zero. Aplicando isometrias de $\hip$, podemos assumir que $c = \infty$ e que os vértices $a$ e $b$ estão sobre $|z| = 1$. 

\begin{figure}[!ht]
\centering
\begin{tikzpicture}[scale=1]
  \draw[dashed] (-1.9,0) -- (1.9,0);

  \draw[dashed] (1,0) -- (1,3.5);
  \draw (-1,1.936) -- (-1,3.5);
  \draw [dashed](-1,0) -- (-1,3.5);
  \draw(1,1.323) -- (1,3.5);

   \draw[dashed] (-0.5,0) -- (1,1.323);
 \draw[dashed] (-0.5,0) -- (-1,1.936);

 \coordinate (C) at (-0.5,0);
\coordinate (V) at (1,1.323);
\coordinate (W) at (-1,1.936);
\coordinate (A) at (1,0);
\coordinate (B) at (-1,0);
\coordinate (D) at (-0.8,1.97);
\coordinate (E) at (1,2);
\coordinate (F) at (-1,2);
\coordinate (G) at (0.855,1.45);

\pic [draw,blue, angle radius=0.3cm] {angle = D--W--F};
\pic [draw,blue, angle radius=0.2cm] {angle = A--C--V};
\pic [draw, blue, angle radius=0.3cm] {angle = W--C--B};
\node at (-0.75,2.25) {$\alpha$};
\node at (-0.85,0.28) {$\alpha$};
\pic [draw,blue, angle radius=0.3cm] {angle = E--V--G};
\pic [draw, blue, angle radius=0.2cm] {angle = E--V--G};
\pic [draw,blue, angle radius=0.3cm] {angle = A--C--V};
\node at (0.85,1.8) {$\beta$};
\node at (0,0.2) {$\beta$};

  \draw (1,1.323) arc (41.41:104.48:2);

  \node[fill=black, circle, inner sep=1pt] at (-1,1.936) {};
  \node[fill=black, circle, inner sep=1pt] at (1,1.323) {};

\end{tikzpicture}
\caption{Triângulo hiperbólico com um vértice na fronteira ideal.}
\label{fig:Gauss-Bonnet1}
\end{figure}

Com base na Figura~\ref{fig:Gauss-Bonnet1}, verificamos que
$$A(\Delta) = \int_\Delta \dfrac{dx\;dy}{y^2} = \int_{\cos(\pi-\alpha)}^{\cos(\beta)}\left( \int_{\sqrt{1-x^2}}^{\infty}\dfrac{dy}{y^2}\right) dx = \pi - (\alpha + \beta),$$
o que fornece o resultado desejado quando $\gamma = 0$. De modo geral, qualquer triângulo é a diferença entre dois triângulos desse tipo, conforme ilustrado na Figura \ref{fig:Gauss-Bonnet2}.

\begin{figure}[!ht]
\centering
\begin{tikzpicture}
  \draw[dashed] (-2.9,0) -- (3.9,0);

 \draw (1.25,1.98) arc(82.82:180:2);
 \draw (1.25,1.98) arc(41.40:60:3);
\draw (1.25,1.98) -- (1.25, 2.905);
\draw (1.25,2.9)arc(104.48:180:3);

  \node[fill=black, circle, inner sep=1pt] at (1.25,1.98) {};
  \node[fill=black, circle, inner sep=1pt] at (1.25,2.9) {};
  \node[fill=black, circle, inner sep=1pt] at (0.5,2.6) {};
  \node at (0.4,2.8) {$a$};
\node at (1.45,2) {$b$};
\node at (1.45,2.9) {$c$};

\draw[fill=lightgray, opacity=0.4](1.25,1.98) -- (1.25,2.9) arc(104.48:120:3) arc(60:41.40:3) ; 
  
\end{tikzpicture}
\caption{Um triângulo como diferença de dois triângulos ideais.}
\label{fig:Gauss-Bonnet2}
\end{figure}
\end{proof}

Mostraremos a seguir que triângulos hiperbólicos são ``magros'' no seguinte sentido: dado um triângulo geodésico $\Delta = \Delta(a,b,c)$ em $\hip$, qualquer uma de suas arestas está contida numa $\delta$-vizi\-nhan\-ça da união das outras duas arestas, para $\delta= \log(3)$. Um  triângulo satisfazendo essa condição é dito \textit{$\delta$-magro}.\index{triângulo $\delta$-magro} Este tipo de comportamento será a base para a definição de espaços hiperbólicos num contexto mais geral, na próxima seção. 

Suponha que $\Delta$ é um  triângulo hiperbólico em $\hip$ com lados $\alpha_i$,  para $ i = 1, 2, 3$, de modo que $\Delta$ limita o triângulo sólido $\blacktriangle$. Para cada  ponto $x \in \blacktriangle$, defina
$$\lambda_x(\Delta) := \displaystyle \max_{i=1,2,3} d(x, \alpha_i)$$ e
 $$\lambda(\Delta) := \displaystyle \inf_{x \in \blacktriangle} \lambda_x(\Delta).$$
Para mostrar que os triângulos são $\delta$-magros, usaremos uma estimativa para $\lambda(\Delta)$. É imediato verificar que o ínfimo na definição de $\lambda(\Delta)$ é  realizado por um ponto $x_0\in \blacktriangle $, o qual é equidistante de todos os lados de $\Delta$, isto é, pelo ponto de interseção dos bissetores do triângulo.

\begin{exercise}[Monotonicidade da  distância hiperbólica] \label{exerc:triangles} Seja $\Delta_i$ para  $i = 1$, $2$ o triângulo retângulo hiperbólico com vértices $a_i,b_i,c_i$ (onde $a_i$ ou $b_i$ podem ser vértices ideais) tal que $a = a_1 = a_2$, $\overline{ab_1}\subset \overline{ab_2}$ e $\overline{ac_1}\subset \overline{ac_2}$, como na  Figura \ref{fig:monotonicity}. Então $\alpha_1 \leq \alpha_2$.
\end{exercise}

\begin{figure}[!ht]
\centering
\begin{tikzpicture}[scale=0.8]
		\draw   (-1.5,0) -- (3.5,0) node [right] {$c_2$} ;
		\draw (-1.5,0)  to [bend right=10] (3.5,2) node [right] {$b_2$};
 \draw (3.5,0) -- (3.5,2);	
\draw (-1.7,0) node  {$a$};
\draw (2,-0.2) node  {$c_1$};
\draw (2,1.4) node  {$b_1$};
\draw (2.3,0.6) node  {$\alpha_1$};
\draw (3.8,1) node  {$\alpha_2$};
  \draw (2,0) -- (2,1.15);
\end{tikzpicture}
\caption{Monotonicidade da distância hiperbólica. }
\label{fig:monotonicity}
\end{figure}
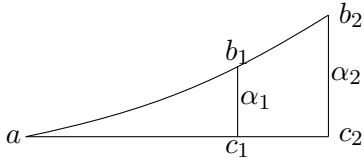

Defina o \textit{raio inscrito} $\mathrm{inrad}(\Delta)$ de $\Delta$ como o  supremo dos raios de discos hiperbólicos
 contidos em $ \blacktriangle$.

\begin{lemma}
    Para qualquer triângulo hiperbólico $\Delta = \Delta(a,b,c)$, vale $\lambda(\Delta) = \mathrm{inrad}(\Delta)$.
\end{lemma}

\begin{proof}
    Suponha que $D = B(x,R)\subset\blacktriangle$ é um  disco hiperbólico. Caso  $\Bar{D}$ não toque dois dos lados de $\Delta$, é fácil ver que existe um disco $D_0 = B(x_0,R_0)\subset \blacktriangle$ que contém propriamente $D$ e, portanto, tem raio maior. Assim, $R\neq\mathrm{inrad}(\Delta)$.
    
    Suponha agora que o fecho de $D \subset \blacktriangle$  toque exatamente duas arestas de $\Delta$, digamos $\overline{ab}$ e $ \overline{bc}$. Então o centro $x$ de $D$ pertence ao bissetor $ \sigma $ de um dos ângulos de $\Delta$, neste caso o ângulo correspondente a $b$. Podemos mover o centro $x$ de $D$ ao longo do bissetor $\sigma$ no sentido que se afasta do vértice $b$ e obter um novo disco $D' = B(x',R')$ de modo que $\overline{D'}$ ainda toque apenas os lados $\overline{ab}, \overline{bc}$ de $\Delta$. Afirmamos que o raio $R'$ de $D'$ é maior que o raio $R$ de $D$. Para provar essa afirmação, considere  triângulos hiperbólicos $\Delta(x,y,b)$ e $\Delta(x',y', b)$, onde $y,y'$ são os pontos de tangência entre $D,D'$ e o lado $\overline{ba}$.
    Esses triângulos retângulos tem ângulo comum em $b$  e satisfazem $$d(b,x) \leq d(b,x').$$
   Assim, a desigualdade $R \leq R'$ segue do Exercício \ref{exerc:triangles}. 
   
   Desse modo, concluímos que se $R = \mathrm{inrad}(\Delta)$, então o maior disco de raio $R$ contido em $\blacktriangle$ toca os três lados de $\Delta$, e seu centro é o ponto de interseção dos bissetores desse triângulo. Mas isto significa que $\lambda(\Delta) = \mathrm{inrad}(\Delta)$, como queríamos.    
\end{proof}

Seja $S$ um triângulo hiperbólico com lados $\sigma_1, \sigma_2, \sigma_3$. Como esses lados são segmentos geodésicos, podemos estender um deles, digamos $\sigma_1$ e encontrar um triângulo  hiperbólico ideal (com três vértices ideais) $\Delta \subset \hip $ cujos lados limitam um triângulo sólido $\blacktriangle$, de modo que $ S \subset \blacktriangle$ e um dos lados de $\Delta$ é a extensão de $\sigma_1$, como na figura abaixo, onde isto é visto no modelo do disco de $\mathbb{H}^2$.

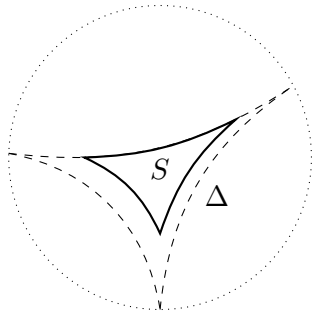
\begin{center}
\begin{tikzpicture}[scale=0.5]
 \path (0,0) coordinate (A) (4,1) coordinate (B) (2,-2) coordinate (C);
 \draw[thick] 
 (A)  to[bend right=12] 
 (B)  to[bend right=15] 
 (C)  to[bend right=20] cycle;
\draw[dashed] (-2,0.11) to [bend right=20] (5.51,1.88) ;
 \draw[dotted] (2,0) circle (4);
\draw (2,-0.3) node {$S$};
\draw (3.5,-1) node {$\Delta$};
\draw[dashed] (-2,0.11) to [bend left=35] (2,-4) ;
\draw[dashed] (2,-4) to [bend left=25] (5.51,1.88) ;

\end{tikzpicture}
\end{center}

Assim, para estimar por cima o raio inscrito de um dado triângulo $S$, basta estimar o raio inscrito de um triângulo com três vértices ideais que o contenha no sentido acima. Como todos os triângulos com três vértices ideais são congruentes, já que os três ângulos de um triângulo hiperbólico o determinam, a menos de isometrias, basta estimarmos o raio inscrito do triângulo ideal $\Delta(\infty, -1, 1)$ (veja a Figura \ref{fig:triangideal}).  O Exercício~\ref{exerc:círculos} fornece $\mathrm{inrad}(\Delta(\infty, -1, 1)) = \frac{\log 3}{2}$.
Unindo essas observações ao Exercício~\ref{exerc:subesphiperbolico}, obtemos a seguinte proposição.

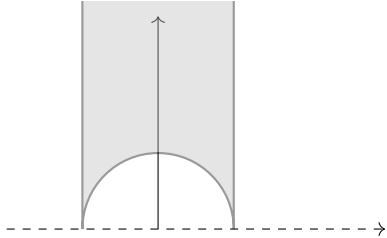
\begin{figure}[!ht]
\centering
\begin{tikzpicture}[scale=1]
		\draw[dashed, ->] (-2,0) -- (3,0);
		\draw[->] (0,0) -- (0,2.8) ;
  
		\draw[thick, fill=lightgray, opacity=0.4](-1,3) -- (-1,0) arc (180:0:1cm) --(1,0) --(1,3);
\end{tikzpicture}
\caption{Triângulo  $ \Delta(\infty, -1, 1) $. }
\label{fig:triangideal}
\end{figure}

\begin{proposition}
 Para cada  triângulo hiperbólico $\Delta$ em $\mathbb{H}^n$,  existe um ponto $p \in  \mathbb{H}^n$ tal que a distância de $p$ a qualquer um dos três lados de $\Delta$ é menor ou igual que $\frac{\log(3)}{2}$. 
\end{proposition}

 Na próxima seção, veremos outra maneira de mostrar que os triângulos hiperbólicos são $\delta '$-magros, desta vez com $\delta ' = \mathrm{arccosh}(\sqrt{2})$.

\section{Espaços hiperbólicos segundo Rips e Gromov}
\label{sec:gromovhip}
Gromov e Rips perceberam que a geometria global de espaços de curvatura negativa pode ser capturada pela propriedade de que todos os triângulos geodésicos são magros (vide \cite{gromov1987hyperbolic}). Nesta seção, descreveremos duas noções de hiperbolicidade, mostrando que elas são equivalentes se estivermos em espaços métricos geodésicos.

Seja $(X, d)$ um espaço métrico geodésico. Recordamos que, para $A\subset X$ e $\delta>0$,  $\mathcal{N}_{\delta}(A)$ denota a união de todas as bolas abertas de raio $\delta$ centradas em pontos de $A$.  Um \textit{triângulo geodésico} \index{triângulo geodésico}  $\Delta(a,b,c)$ em $X$ é a concatenação de três segmentos geodésicos $\overline{ab}$, $\overline{bc} $ e $\overline{ca}$, chamados lados do triângulo, que ligam  os seus vértices $a,b$ e $c$  na ordem cíclica natural.
Também denotaremos o triângulo geodésico por $T_{\alpha_1,\alpha_2,\alpha_3}$, onde $\alpha_i$ são os comprimentos dos lados. 

\begin{definition}[Rips]
Seja $\delta \geq 0$ dado. O espaço $X$ é dito \textit{$\delta$-hiperbólico por Rips} \index{espaços $\delta$-hiperbólicos (Rips)} se, para qualquer triângulo geodésico $\Delta(x,y,z)$, tem-se $\overline{xy} \subset \mathcal{N}_\delta(\overline{xz} \cup \overline{yz})$. Ou seja, se dado qualquer $p\in \overline{xy}$ existe $q \in \overline{xz}\cup\overline{yz}$  com $d (p, q)\leq \delta$. 
\end{definition}
Em outras palavras, dizemos que $X$ é $\delta$-hiperbólico por Rips se todo triângulo $\Delta(x,y,z)$ em $X$ é \textit{$\delta$-magro}\index{triângulo $\delta$-magro}. O ínfimo $\delta_0$ das constantes $\delta$ para as quais $X$ é $\delta$-hiperbólico é dito \textit{constante de hiperbolicidade de $X$}.\index{constante de hiperbolicidade} Caso $\delta_0<\infty$, diremos que $X$ é $\delta_0$-hiperbólico ou \textit{Rips-hiperbólico}.

 
Uma desvantagem da definição de hiperbolicidade de Rips é que ela depende das geodésicas do espaço $X$. Apresentaremos agora uma segunda definição, devida a Gromov, menos intuitiva porém mais geral que a definição de Rips, e que se mostra bastante útil em certas situações. Por exemplo, o espaço métrico $X$ em questão não precisa ser geodésico.

Seja $(X,d)$ um espaço métrico. Tome $p \in X$ um ponto base. Para cada $x \in X$, seja $|x|_p:=d(p,x)$ e defina o {\it produto de Gromov} \index{produto de Gromov} de $x,y\in X$ em $p$ por:

$$(x,y)_p:=\dfrac{1}{2}\left(|x|_p+|y|_p-d(x,y)\right).$$

Note que, pela desigualdade triangular, vale que  $(x,y)_p\geq 0$ para todos $x,y,p \in X$

\begin{example} Suponha que $X$ é uma árvore. Então $(x,y)_p$ é a distância $d(p,\gamma)$ de $p$ para a geodésica $\gamma = xy.$ 
\end{example}

\begin{remark}
    O produto de Gromov é uma generalização para o contexto de espaços métricos do produto interno em espaços vetoriais, onde o ponto base $p$ é escolhido como a origem: se $X = \R^n$, e $p=0$, então $|v|_0 = \left\|v\right\|$ e 
    $$\frac{1}{2}(|x|_0^2+|y|_0^2 -\left\|x-y\right\|^2) = x\cdot y.$$    
\end{remark}

\begin{definition}[Gromov]
Um espaço métrico $X$ é dito \textit{$\delta$-hiperbó\-li\-co no sentido de Gromov} \index{espaços Gromov-hiperbólicos} se, para todos $x,y,z,p \in X$, vale: $$(x,y)_p \geq \min\{(x,z)_p,(y,z)_p\}-\delta.$$
\end{definition}

\begin{lemma} 
\label{lemma:ripsgromov1} Suponha que $X$ é $\delta$-hiperbólico no sentido de Rips. Então o produto de Gromov em $X$ é comparável a $d(p,\overline{xy})$, i.e., para cada $x,y,p \in X$ e toda geodésica $\overline{xy}$, $$(x,y)_p \leq d(p,\overline{xy})\leq (x,y)_p+ 2\delta .$$
\end{lemma}
\begin{proof}
Seja $z \in \overline{xy}$. Então $d(x,y) = d(x,z)+d(z,y)$. Logo, $$2(x,y)_p = |x|_p+|y|_p-d(x,y) = |x|_p+|y|_p-d(x,z)-d(z,y)$$ $$= |y|_p+|z|_p -2(p,x)_z -d(z,y)= 2|z|_p-2(p,x)_z-2(p,y)_z. $$ Portanto, $(x,y)_p\leq d(p,z)$, para todo $z \in \overline{xy}$, donde segue a primeira desigualdade.

Pelo que vimos, temos $(x,p)_z \leq d(z,\overline{px})$ e $(y,p)_z \leq d(z,\overline{py})$. Logo, \begin{equation}
\label{eq:lemaripsgromov1}
    \min \{(x,p)_z ,(y,p)_z \}\leq \min\{d(z,\overline{px}),d(z,\overline{py})\}.
\end{equation} Como o triângulo $\Delta(p,x,y)$ é $\delta$-magro, vale $\overline{xy}\subset \mathcal{N}_{\delta}(\overline{px})\cup \mathcal{N}_{\delta}(\overline{py})$. Daí, $\min\{d(z,\overline{px}),d(z,\overline{py})\}\leq \delta$ para cada $z \in \overline{xy}.$ Logo, segue de  \eqref{eq:lemaripsgromov1} que 
\begin{equation}
\label{eq:lemaripsgromov2}
    \min \{(x,p)_z ,(y,p)_z \}\leq \delta, 
\end{equation} 
para cada $z \in \overline{xy}$.

Usando a conexidade da geodésica $\overline{xy}$ e a continuidade das funções $z \mapsto (x,p)_z, (z,p)_z$, obtemos que os conjuntos $A:=\{z \in \overline{xy} \mid (x,p)_z>\delta\}$ e $B:=\{z \in \overline{xy} \mid (y,p)_z>\delta\}$ são abertos disjuntos por \eqref{eq:lemaripsgromov2}. Portanto  existe $z \in \overline{xy}$ tal que $(x,p)_z ,(y,p)_z\leq \delta$ donde $$|z|_p-(x,y)_p = (x,p)_z +(y,p)_z\leq 2\delta.$$
Como $|z|_p \geq d(p,\overline{xy})$, concluímos que  $d(p,\overline{xy})\leq (x, y)_p + 2\delta$.
\end{proof}

\begin{lemma}\label{lemma:ripsgromov2} Se um espaço métrico geodésico $X$ é $\delta$-hiperbólico no sentido de Rips, então ele é $3\delta$-hiperbólico no sentido de Gromov.
\end{lemma}
\begin{proof}
Vamos mostrar que os vértices de um triângulo $\delta$-magro $\Delta(x,y,z)$ satisfazem $(x,y)_p\geq \min\{(y,z)_p,(x,z)_p\}-3\delta$, qualquer que seja $p\in X$. 
De fato, seja $m \in \overline{xy}$ o ponto que realiza distância entre $p$ e a geodésica $\overline{xy}.$ Como $\Delta(x,y,z)$ é $\delta$-magro existe $n \in \overline{yz}\cup \overline{xz}$ tal que $d(m,n)\leq \delta$. 
Suponhamos, sem perda de generalidade, que $n \in \overline{yz}$. Daí, pelo Lema 
\ref{lemma:ripsgromov1}, usando que $n \in \overline{yz}$ e a desigualdade triangular, obtemos
\begin{eqnarray*}
    (y,z)_p &\leq& d(p,\overline{yz})\\
    &\leq& d(p,n)\\
    &\leq& d(p,m)+d(m,n) \\
    &=& d(p,\overline{xy})+d(n,m)\\
    &\leq& (x,y)_p+3\delta.
\end{eqnarray*}
\end{proof}

\begin{lemma}\label{lemma:gromovrips}
Se um espaço métrico geodésico $X$ é $\delta$-hiperbólico no sentido de Gromov, então ele é  $4\delta$-hiperbólico no sentido de Rips.
\end{lemma}

\begin{proof}
Vamos primeiro mostrar que se $\alpha$ é uma geodésica ligando $x$ a $y$ e $p$ um ponto em $X$ tal que \begin{equation}\label{eq:4delta}
|x|_p+|y|_p\leq d(x,y)+2 \delta,
\end{equation}
 então $d(p,\alpha)\leq 4 \delta $.
 
 Suponhamos inicialmente que $|x|_p \leq |y|_p.$ Se $d(p,x) \geq d(x,y)$ então $d(p,\alpha)\leq d(p,x)\leq 4\delta$ e o resultado segue. Portanto, vamos também supor que $d(p,x) < d(x,y)$. Seja $z \in \alpha$ tal que $d(x,z) = d(x,p).$ Como $X$ é $\delta$-hiperbólico no sentido de Gromov, $$(x,y)_p\geq \min \{(x,z)_p,(y,z)_p\}-\delta.$$ Se $(x,y)_p \geq (x,z)_p-\delta$. Então, $$d(y,p)-d(x,y)\geq d(z,p)-d(x,z)-2\delta,$$  e portanto 
$$\arraycolsep=0pt\def\arraystretch{1.2}
\begin{array}{rcl}
d(z,p)\*&\;\leq\;&\*  d(y,p)+d(x,z)-d(x,y)+2\delta \\
\*&\;=\;&\* d(y,p)+d(x,p)-d(x,y)+2\delta   \\
\*&\;=\;&\*2(x,y)_p +2\delta   \\
\*&\;\leq\;&\* 4\delta,
\end{array}$$
 onde na última desigualdade foi usada a equação  \eqref{eq:4delta}. Agora, se $(x,y)_p \geq (y,z)_p-\delta$, então $$d(x,p)-d(x,y)\geq d(z,p)-d(y,z)-2\delta.$$  
 Logo, $d(z,p)\leq 2\delta +d(x,p)-d(x,y) + d(y,z) = 2\delta +d(x,z)-d(x,y) + d(y,z) = 2\delta .$

Considere agora um triângulo geodésico $\Delta(x, y, p)$ em $X$ e seja $z \in\overline{xy}$. Nosso objetivo é mostrar que $z$ pertence a $\mathcal{N}_{4\delta}(\overline{px}\cup \overline{py})$. Temos que 
$(x, y)-p\geq \min ((x,z)_p, (y, z)_p) -\delta$. Supondo que $(x, y)_p\geq (x, z)_p - \delta$, obtemos
$d(y, p)\geq d(y, z) + d (z, p) -2\delta$. 
Portanto, pela discussão inicial, obtemos $z \in \mathcal{N}_{4\delta} (\overline{xp}\cup \overline{yp})$ e, assim, o triângulo $\Delta(x, y, z)$ é $4\delta$-magro.
\end{proof}

\begin{corollary}
    Em espaços métricos geodésicos, as noções de hiperbolicidade de Gromov e de Rips são equivalentes.
\end{corollary}

\begin{example}[A hiperbolicidade de Gromov não é um invariante por quasi-isometrias]

Por outro lado, vejamos o que ocorre em $X$. Seja $p = (0,0)\in X$ e, para cada $t>0$, sejam $x=(2t,2t)$, $y=(-2t,2t)$ e $z = (t,t)$. Notemos que 
$$\min\{(y,z)_p,(x,y)_p\}-(x,z)_p \geq c t,$$ onde $c$ é uma constante positiva. De fato,
\begin{eqnarray*}
 2(x,z)_p &=& d(x,p)+d(z,p)-d(x,z) = 2\sqrt{2}t\\
 2(x,y)_p &=& d(x,p)+d(y,p)-d(x,y) = 4\sqrt{2}t-4t\\
 2(y,z)_p &=& d(y,p)+d(z,p)-d(y,z) = 3\sqrt{2}t-\sqrt{10} t.
\end{eqnarray*}
Portanto, a quantidade $\min\{(y,z)_p,(x,y)_p\}-(x,z)_p $ é ilimitada à medida que $t \to \infty$  e assim, $X$ não é $\delta$-hiperbólico para nenhum valor de $\delta < \infty$. Em particular, $X$ é QI ao espaço Gromov–hiperbólico $\R$, mas não é Gromov–hiperbólico.
\end{example}

\begin{exercise}
    Mostre que a hiperbolicidade de Gromov é invariante por $(1,L)$-quasi-isometrias.
\end{exercise}

Provaremos mais tarde que, no contexto dos espaços geodésicos, a hiperbolicidade é um $\QI$-invariante. Esse resultado seguirá como consequência do Lema de Morse.

\begin{definition} \label{def:tripode}
    Uma \textit{trípode} \index{trípode} $\Tilde{T}$ é um grafo métrico que, como grafo, é isomorfo ao grafo $T$ que possui 4 vértices $v_0, v_1, v_2, v_3$, onde dois  vértices são conectados por uma  única aresta se, e somente se, um desses vértices é $v_0$ e o outro é  diferente de $v_0$. O vértice $v_0$ é o centro da ``estrela'' e as arestas são chamadas de pernas de $\tilde{T}$. 
\end{definition}

\begin{figure}[!ht]
    \centering
    \begin{tikzpicture} \coordinate[label=left:$v_1$] (x) at (0,0); 
    \coordinate[label=right:$v_2$] (y) at (3.5,-1); 
    \coordinate[label=above:$v_0$] (m) at (2,0); 
    \coordinate[label=right:$v_3$] (z) at (3,1); 
    \draw (x)--(m);
    \draw (y)--(m); 
    \draw (z)--(m);
    \draw[fill] (x) circle [radius=0.04]; \draw[fill] (y) circle [radius=0.04]; \draw[fill] (z) circle [radius=0.04];
    \draw[fill] (m) circle [radius=0.04];
    \end{tikzpicture}
    \caption{Trípode.}
    \label{fig:my_label}
\end{figure}

Por abuso de notação, veremos uma trípode $\tilde{T}$ como um  triângulo geodésico neste grafo, cujos vértices são os pontos extremos $v_1,v_2, v_3$ de  $\tilde{T}$. Portanto, usaremos a notação $ \Tilde{T} = T_{\alpha_1, \alpha_2, \alpha_3} =
T(v_1, v_2, v_3)$. Da mesma forma, os comprimentos dos lados da trípode são os comprimentos dos lados do triângulo correspondente.

\begin{exercise}
    Sejam $\alpha_1,\alpha_2, \alpha_3 \in \R$ satisfazendo as desigualdades triangulares $\alpha_i\leq \alpha_j+\alpha_k$. Mostre que existe uma única trípode $T_{\alpha_1,\alpha_2,\alpha_3}$ com arestas de tamanho $\alpha_i$, como na Figura \ref{fig:tripode}, a menos de isometria.
\end{exercise}


 Seja $\Delta =\Delta(v_1,v_2,v_3) \subset X$ um triângulo geodésico em um espaço métrico $X$, com lados de comprimento $\alpha_i, i=1,2,3$. Então existe uma única aplicação, a menos de pós-composição com isometrias,  $$\kappa : \Delta \to T_{\alpha_1,\alpha_2,\alpha_3},$$
isométrica nas arestas, chamada o \textit{colapso} \index{colapso}  de $\Delta$.

\begin{figure}[!ht]
    \centering
    \begin{tikzpicture} \coordinate[label=left:$v_1$] (x) at (0,0); 
    \coordinate[label=right:$v_2$] (y) at (3.5,-1); 
    \coordinate[label=above:$v_3$] (z) at (3,1); 
    \draw (x)--(y) ;
    \draw (1.6, 0.8) node {$\alpha_1$};
    \draw (y)--(z); 
    \draw (1.6, -0.7) node {$\alpha_2$};
    \draw (z)--(x);
    \draw (3.6,0) node {$\alpha_3$}; 
    \draw[fill] (x) circle [radius=0.04]; \draw[fill] (y) circle [radius=0.04]; \draw[fill] (z) circle [radius=0.04];
    \draw[->] (4.4,0)--(5.3,0);
    \draw (4.9,0.2) node {$\kappa$}; 
    \end{tikzpicture} \quad
    \begin{tikzpicture} \coordinate[label=left:$v_1$] (x) at (0,0); 
    \coordinate[label=right:$v_2$] (y) at (3.5,-1); 
    \coordinate (m) at (2,0); 
    \coordinate[label=right:$v_3$] (z) at (3,1); 
    \draw (x)--(m) ;
    \draw (1.6, 0.5) node {$\alpha_1$};
    \draw (y)--(m); 
    \draw (1.6, -0.5) node {$\alpha_2$};
    \draw (z)--(m);
    \draw (2.8,0) node {$\alpha_3$}; 
    \draw[dotted, <->] (0,0.2) .. controls (1.8,0.2) and (1.8,0.2) .. (2.8,1.1);
    \draw[dotted, <->] (0,-0.2) .. controls (1.8,-0.2) and (1.8,-0.2) .. (3.4,-1.1);
    \draw[dotted, <->] (3.2,0.9) .. controls (2.1,0) and (2.1,0) .. (3.6,-0.9);
    \draw[fill] (x) circle [radius=0.04]; \draw[fill] (y) circle [radius=0.04]; \draw[fill] (z) circle [radius=0.04];
    \draw[fill] (m) circle [radius=0.04];
    \end{tikzpicture}
    \caption{Colapso.}
    \label{fig:tripode}
\end{figure}
 
\begin{exercise}
 Mostre que a aplicação de colapso preserva os produtos de Gromov $(x_i,x_j)_{x_k}.$
\end{exercise}

 
 O \textit{centroide}\index{centroide} de $\tilde{T}$ é o ponto de interseção $o \in \tilde{T}$ das três arestas. Note que o produto  de Gromov satisfaz 
 $$(v_i,v_j)_{v_k}= d(v_k,\overline{v_iv_j}) = d (v_k,o).$$ 
 Tomando as pré-imagens de $o \in \tilde{T}$ pelo mapa $\kappa$ , obtemos três pontos $x_{ij} \in \overline{v_iv_j}$, chamados de pontos centrais do triângulo $\Delta$. Note que:
$$d(v_i, x_{ij}) = (v_j , v_k)_{v_i}.$$

\subsection{Exemplos de espaços hiperbólicos}

Listaremos nesta seção alguns exemplos clássicos de espaços hiperbólicos por Rips. 

\begin{example}
O exemplo básico é o espaço $\mathbb{H}^n$, o qual é sempre $\delta$-hiperbólico, para $\delta = \mathrm{arccosh}(\sqrt{2})$. 

Para verificar esse fato, notemos primeiramente que é suficiente considerar o caso $n=2$, já que qualquer triângulo hiperbólico $\Delta$ em $\mathbb{H}^n$ está contido em algum $2$-subespaço totalmente geodésico de  $\mathbb{H}^n$, o qual é isométrico a $\hip$. 
  

Todo triângulo $\Delta$ em $\hip$ tem pelo menos um lado que não é uma geodésica vertical. Podemos estender esse lado indefinidamente, como na Figura \ref{fig:extensãotriangulo}, de modo que $\Delta$ esteja contido em um triângulo ideal $\Delta_{\infty}$.

\begin{figure}[!ht]
\centering
\begin{tikzpicture}[scale=1.2]
		\draw[dotted, ->] (-2,0) -- (2.3,0);
		\draw[dotted, ->] (0,0) -- (0,2.1) ;
  
		\draw(-1,2) -- (-1,0) arc (180:0:1cm) --(1,0) --(1,2);
  \draw[thick, fill=lightgray, opacity=0.4] (-0.5,0.87) arc (120:60:1cm)--(0.5,0.87) arc (30:68.5:2cm)  -- (-0.5,0.87);
\end{tikzpicture}
\caption{Todo triângulo está contido em um triângulo ideal. }
\label{fig:extensãotriangulo}
\end{figure}
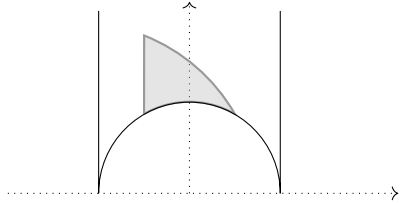

Se mostrarmos que todo triângulo ideal $\Delta_{\infty}$ é $\delta$-magro, então cada triângulo $\Delta$ em $\mathbb{H}^n$ será também $\delta$-magro. Isso se deve ao seguinte fato: sejam $\sigma_i, i=1,2,3$ os lados de $\Delta$, sendo $\sigma_1$ o lado que estendemos acima. Se chamarmos de $ \tau_1, \tau_2, \tau_3$ os lados de $\Delta_{\infty}$, onde $\tau_1$ é a extensão de $\sigma_1$, então para cada  $p \in \sigma_1 \subset \tau_1$, temos $$d(p, \sigma_2 \cup \sigma_3) \leq d(p, \tau_2 \cup \tau_3),$$ já que todo segmento geodésico de $p$ a $\tau_j$ intersecta $\sigma_j$, para $j=2,3$.

Uma propriedade importante em $\mathbb{H}^2$ é que todos os triângulos ideais são isométricos. Assim, podemos supor que  $\Delta_{\infty}$ é o triângulo ideal com vértices $a_1= \infty, a_2= -1 \mbox{ e } a_3=1$. Parametrizando as  geodésicas $\tau_j = a_{j-1}a_{j+1}$, com índices $j=1,2,3$ tomados módulo 3, podemos verificar, usando o Exercício \ref{exerc:triangles}, que para cada $j$, a função $f_j(t) = d(\tau_j(t), \tau_{j-1})$ é monotonicamente crescente. Portanto, o supremo 
$\displaystyle \sup_t d(\tau_1(t), \tau_2 \cup \tau_3 )$
é atingido no ponto médio $m=\tau_{1}(t_0) = i$, e, usando a fórmula \eqref{distformulaH}, obtemos
$$\displaystyle \sup_t d(\tau_1(t), \tau_2 \cup \tau_3 ) = d(i, -1+i\sqrt{2}) = \mathrm{arccosh}(\sqrt{2}).$$
\end{example}




Para o próximo exemplo, recordamos a noção de espaço modelo e de triângulo de comparação.

\begin{definition}
Dado $\kappa \in \R$, chamamos de \textit{espaço modelo de curvatura $\kappa$}\index{espaço modelo} uma variedade Riemanniana completa e simplesmente conexa $M_{\kappa}$, com métrica $d_{\kappa}$ de curvatura constante igual a $\kappa$. Dado um triângulo geodésico $\Delta = \Delta(p,q,r)$ em um espaço métrico $(M,d)$, um \textit{triângulo de comparação}\index{triângulo de comparação} de $\Delta$ em um espaço modelo $(M_{\kappa}, d_{\kappa})$ é um triângulo $\Delta(\bar{p},\bar{q},\bar{r})$ em $M_{\kappa}$ tal que $\{\bar{p},\bar{q},\bar{r}\}$ é uma cópia isométrica de $\{p,q,r\}$.
\end{definition}

Um triângulo de comparação, quando existe, é único a menos de isometrias. A existência é sempre garantida para $\kappa \leq 0$. Para $\kappa > 0$, isso pode ser garantido pela condição adicional 
$$d(p,q) + d(q,r) + d(r,p) \le \frac{2\pi}{\sqrt{\kappa}}.$$

\begin{example}
Qualquer variedade Riemanniana completa, simplesmente conexa de curvatura estritamente negativa é um exemplo de espaço $\delta$-hiperbólico. Isso decorre dos teoremas de comparação em geometria riemanniana (veja \cite[Seção~I.3]{ballmann-lectures}):
\begin{thm}
\label{thm:Toponogov}
Seja $M$ uma variedade Riemanniana completa, simplesmente conexa, com curvatura seccional $\leq\kappa$. Denotaremos por $M^2_{\kappa}$ o espaço modelo de curvatura $\kappa$ e dimensão $2$. Dado um triângulo geodésico $\Delta = \Delta(p,q,r)$ em $M$ (com perímetro $\leq 2\pi/\sqrt{\kappa}$, caso tenhamos $\kappa>0$), seu triângulo de comparação $\Bar{\Delta}$ em $M^2_{\kappa}$ cumpre a propriedade de comparação de triângulos: 
para todos os pontos $x$, $y$ nas arestas de $\Delta$ e os pontos correspondentes $\bar{x}$, $\bar{y}$ nas arestas de $\Bar{\Delta}$ temos $d(x,y) \leq d_{\kappa}(\bar{x}, \bar{y})$.
\end{thm}
\end{example}

Este teorema dá origem à importante definição a seguir. O nome ``$\CAT(\kappa)$'' foi sugerido por Gromov e deriva de Cartan, Aleksandrov e Toponogov, cujos trabalhos desempenharam um papel fundamental nesta área:

\begin{definition}
Um espaço métrico $(X,d)$ é dito $\CAT(\kappa)$ \index{espaços $\CAT(\kappa)$} se 
\begin{itemize}
    \item quaisquer dois pontos $x,y \in X$ podem ser ligados por uma geodésica de tamanho $d(x,y)$;
    \item dado qualquer triângulo geodésico $\Delta(p,q,r)$ (com perímetro $\leq 2\pi/\sqrt{\kappa}$ no caso em que $\kappa>0$) em $X$, se $\Delta(\bar{p},\bar{q},\bar{r}) $ é um triângulo de comparação de $\Delta(p,q,r)$ em $M^2_{\kappa}$, então para todos os pontos $x$, $y$ nas arestas de $\Delta$ e os pontos correspondentes $\bar{x}$, $\bar{y}$ nas arestas de $\Bar{\Delta}$ temos $d(x,y) \leq d_{\kappa}(\bar{x}, \bar{y})$.
\end{itemize}
\end{definition}

Essa definição nos permite estender o exemplo anterior a espaços métricos geodésicos gerais, por exemplo, aos grafos ou complexos simpliciais.

\begin{example}
Qualquer espaço  $\CAT(\kappa)$, com $\kappa<0$, é hiperbólico no sentido de Rips, já que todo triângulo geodésico em $X$ é no máximo tão espesso quanto o triângulo de comparação correspondente no espaço modelo simplesmente conexo de dimensão 2 e curvatura constante $\kappa$,  que é o plano hiperbólico $\mathbb{H}^2_\kappa$.
\end{example}
\begin{exercise}
   Se $\kappa< 0$ então $M^2_\kappa = \mathbb{H}^2 _{\kappa}$ é o plano hiperbólico real $\hip$ com métrica reescalada por um fator $ \frac{1}{\sqrt{|\kappa|}}$.  Prove que, nesse caso, $$\delta_{M^2_{\kappa}} = \frac{1}{\sqrt{|\kappa|}}\mathrm{arccosh}(\sqrt{2}).$$
\end{exercise}

\begin{exercise}
Mostre que, se $X$ é um espaço $\CAT(\kappa)$, então ele é um espaço $\CAT(\kappa')$ para todo $\kappa' > \kappa$.
\end{exercise}

Na teoria geométrica de grupos, os espaços $\CAT(\kappa)$ nos permitem definir \textit{grupos $\CAT(\kappa)$}:

\begin{definition}
Um grupo $G$ é dito $\CAT(\kappa)$ se ele age geometricamente em um espaço métrico $\CAT(\kappa)$.\index{grupo $\CAT(\kappa)$}
\end{definition}

Tendo em vista esta definição, duas classes de grupos são de interesse especial: os grupos $\CAT(0)$ e os grupos $\CAT(-1)$. Um problema que ainda está em aberto é o seguinte: construir um grupo  Gromov-hiperbólico (definido no próximo capítulo) que não é $\CAT(-1)$.

Nas últimas décadas, a teoria dos espaços $\CAT(\kappa)$ tornou-se uma grande área, com muitos resultados e aplicações importantes. Remetemos aos livros \cite{burago2022course} e \cite{bridson2013metric} para uma introdução a essa teoria.

\subsection{Geometria dos espaços hiperbólicos}

Nesta seção, listamos algumas propriedades simples, porém bastante úteis, de espaços geodésicos $\delta$-hiperbólicos. Ao final, falaremos do Lema de Morse, que afirma que quasi-geodésicas em um espaço $\delta$-hiperbólico estão uniformemente próximas de geodésicas.

Ao longo da seção, assumiremos que $X$ é um espaço métrico, geodésico e $\delta$-hiperbólico.

\begin{lemma}[Propriedade de bígonos]
\label{lemma:bigons}
Todo par de geodésicas $\overline{xy}$,  $\overline{xz}$ em $X$ com $d(y,z)\leq D$ estão a uma distância de Hausdorff no máximo $D+\delta.$ Em particular, se $\alpha$ e $\beta$ são duas geodésicas ligando dois pontos $x,y \in X$ então $d_{Haus}(\alpha, \beta)\leq \delta.$
\end{lemma}

\begin{proof}
Para todo $p\in \overline{xy}$, vale $d(p,\overline{xz})\leq \delta$ ou $d(p,\overline{yz})\leq \delta $, já que $\Delta(x,y,z)$ é $\delta$-magro. No primeiro caso, é direto concluir que $d(p,\overline{xz}) \leq D+\delta$. No segundo caso, se $m\in \overline{yz}$ é tal que $d(p,\overline{yz})= d(p,m)$, então pela desigualdade triangular vale que  $d(p,\overline{xz})\leq  d(p,m)+ d(m,\overline{xz}) \leq  \delta + d(y,z) \leq \delta+ D$.  Portanto, $d_{Hauss}(\overline{xy},\overline{xz})\leq D+\delta.$
\end{proof}

\begin{lemma}[Propriedade de companheiros de viagem]
\label{lemma:companheirosviagem}
Sejam $\alpha(t)$ e $\beta(t)$ geodésicas parametrizadas por comprimento de arco em $X$, tais que $\alpha(0)=\beta(0)=p$ e $d(\alpha(t_0), \beta(t_0))\leq D$ para algum $t_0\geq 0$. Então, para todo $t \in [0,t_0],$ $$d(\alpha(t),\beta(t))\leq 2(D+\delta).$$
\end{lemma}

\begin{proof}
Pelo Lema \ref{lemma:bigons}, para cada $t \in [0,t_0]$ existe $s\in [0,t_0]$ tal que $$d(\alpha(t),\beta(s))\leq D+ \delta .$$ Aplicando a desigualdade triangular, obtemos $$s=d(\beta(s),p)\leq d(\alpha(t),\beta(s))+d(\alpha(t),p),$$
$$t=d(\alpha(t),p)\leq d(\alpha(t),\beta(s))+d(\beta(s),p).$$ Logo, $$|t-s|\leq d(\alpha(t),\beta(s))\leq D+ \delta.$$ Portanto,
\begin{eqnarray*}
    d(\alpha(t),\beta(t))&\leq& d(\alpha(t),\beta(s))+d(\beta(s),\beta(t))\\
    &=& d(\alpha(t),\beta(s)) + |t-s|\\
    &\leq& 2(D+\delta).
\end{eqnarray*}
\end{proof}

A noção de triângulos magros se generaliza naturalmente para o conceito de \textit{polígonos magros}. Um  $n$-ágono geodésico $P$ em um  espaço métrico $X$ é a concatenação de segmentos geodésicos $\sigma_i$, $i = 1, \ldots, n$, que conectam pontos $p_i$, $ i = 1, \ldots, n$, na ordem cíclica natural. Os pontos $p_i$ são os \textit{vértices} de $P$ e as geodésicas $\sigma_i$
são os \textit{lados} ou \textit{arestas} de $P$. Um polígono $P$ é dito $\delta$-\textit{magro}\index{polígono magro} se todo lado de $P$ estiver  contido em uma  $\delta$-vizinhança da união dos outros lados.

\begin{exercise}
Mostre que, se $X$ é um espaço métrico,
geodésico e $\delta-$hiperbólico, então todo $n$-ágono em $X$ é $\delta(n-2)$-magro. 
\textit{Dica:} Faça uma triangulação de um  $n$-ágono $P$ por $n-3$ diagonais partindo de um mesmo vértice. Em seguida, use indutivamente que triângulos são $\delta$-magros.
\end{exercise}

\begin{lemma}[Polígonos magros]
\label{lemma:poligonosmagros}
Para $n\geq 2$, cada $(n+1)$-ágono geodésico em $X$ é $\delta_n$-magro com $\delta_n = \delta \log_2(n)$.
\end{lemma}

\begin{proof}

Mostraremos a estimativa por indução. Defina $m=n+1$, o número de lados do polígono. Para $m = 3 = 2^1+1$, não há o que demonstrar já que o espaço é $\delta$-hiperbólico. 

Suponha que $k\geq 1$ e que a estimativa vale para todo polígono com $m \leq (2^{k} + 1)$ lados. Considere  um polígono  geodésico $P$ com $m \leq 2^{k+1}+1$ lados e fixe uma de suas arestas, digamos $\overline{xy}$. Se $m$ for par, então ainda teremos $m+1 \leq 2^{k+1}+1$. Desse modo, ao adicionar um vértice no interior de alguma aresta diferente de $\overline{xy}$, poderemos reduzir o problema a resolver o caso em que $m$ é ímpar.

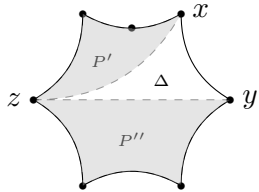
\begin{figure}[!ht]
\centering
\begin{tikzpicture}[scale=1.3]
    
    \draw (1,0) to [bend left] (0.5,0.87) to [bend left] (-0.5,0.87) to [bend left] (-1,0) to [bend left] (-0.5,-0.87) to [bend left] (0.5,-0.87) to [bend left] (1,0);

 \draw (1,0) node {\tiny{$\bullet$}};
 \draw (-1,0) node {\tiny{$\bullet$}};
 \draw (-0.5,0.87) node {\tiny{$\bullet$}};
 \draw (-0.5,-0.87) node {\tiny{$\bullet$}};
 \draw (0.5,-0.87) node {\tiny{$\bullet$}};
 \draw (0.5,0.87)node {\tiny{$\bullet$}};
 \draw (0,0.72)node {\tiny{$\bullet$}};
 \draw (-1.2,0)node {$z$};
 \draw (0.7,0.9)node {$x$}; 
 \draw (1.2,0)node {$y$}; 
\draw (0,-0.4)node {\tiny{$P''$}}; 
\draw (-0.3,0.4)node {\tiny{$P'$}}; 
\draw (0.3,0.2)node {\tiny{$\Delta$}};

\draw[fill=lightgray,dashed, opacity=0.4] (1,0) -- (-1,0) to [bend left]  (-0.5,-0.87) to [bend left] (0.5,-0.87) to [bend left] (1,0);

\draw [fill=lightgray,dashed, opacity=0.4] (-1,0) to [bend right] (0.5,0.87) to [bend left] (-0.5,0.87) to [bend left] (-1,0);
 
\end{tikzpicture}
\caption{Polígonos em espaços hiperbólicos são magros. }
\label{fig:poligonosmagros}
\end{figure}

 Subdividimos $P$ em três polígonos menores, introduzindo as diagonais $\overline{xz}$ e $\overline{yz}$, onde  $z$ é um vértice de $P$ tal que o número de arestas entre $x$ e $z$ seja o mesmo que  entre $y$ e $z$, a saber, $\frac{m-1}{2}$. Os polígonos resultantes dessa divisão são dois $\left(\frac{m-1}{2} + 1\right)$-ágonos  $P'$ e  $P''$ e um triângulo  $\Delta =\Delta(x,y,z)$, como na Figura \ref{fig:poligonosmagros}.

Por hipótese de indução, os polígonos $P'$ e $ P''$ são $\delta\log_2(\frac{m-1}{2})$-magros, enquanto o triângulo $\Delta$ é $\delta$-magro. Assim, $\overline{xy}$ está contida numa $\delta \log_2(\frac{m-1}{2})+\delta = \delta( \log_2(\frac{m-1}{2}) + \log_2(2)) = \delta \log_2(m-1)$-vizinhança da união dos outros lados de $P$, como queríamos demonstrar.    
\end{proof}

\begin{lemma}[Triângulos parecem trípodes]
\label{lemma:triangulosparecemtripodes}
Seja $\Delta = \Delta(a_1,a_2,a_3)$ um triângulo geodésico em $X$ com pontos centrais $a_{ij}\in \overline{a_ia_j}$. Então, para $i, j, k \in \{1,2,3\}$, valem: \begin{eqnarray*}
    d_{Haus}(\overline{a_ia_{ij}},\overline{a_ia_{ik}})&\leq & 7 \delta, \\
    d(a_{ij},a_{jk})&\leq& 6 \delta.
\end{eqnarray*} 
Além disso, dadas parametrizações 
$\alpha_{ji}, \alpha_{jk} : [0, t_j] \to  X$  por comprimento de arco, onde $ t_j = d(a_j , a_{ij}) = d(a_j , a_{jk})$, dos  segmentos geodésicos $\overline{a_ja_{ji}}$ e $\overline{a_ja_{jk}}$, respectivamente, vale:
$$d(\alpha_{ji}(t), \alpha_{jk}(t)) \leq 14\delta
, \mbox{ para todo } t \in [0, t_j ].$$
\end{lemma}
\begin{proof} A geodésica $\overline{a_ia_j}$ pode ser coberta pela união dos fechados $\overline{\mathcal{N}}_\delta(\overline{a_ia_k})$ e  $\overline{\mathcal{N}}_\delta(\overline{a_ja_k})$. Portanto, pela conexidade de $\overline{a_ia_j}$, deve existir um ponto $p$ nesse segmento com distância até $\overline{a_ia_k}$ e $\overline{a_ja_k}$ limitada por $\delta$. Sejam $p'\in \overline{a_ia_k}$ e $p''\in \overline{a_ja_k}$ pontos com distância menor ou igual que $\delta$ de $p$. Assim, $d(p',p'') \leq 2\delta$. Temos
\begin{eqnarray*}
    (a_j , a_k)_{a_i} &=& \frac{1}{2}
[d(a_j,a_i)+ d(a_k,a_i)- d(a_k,a_j)] \\ &=&\frac{1}{2}
[d(a_i, p) + d(p, a_j) + d(a_i, p') + \\& &\quad d(p',a_k)- d(a_j, p'') - d(p'', a_k)] \\
&\leq &\frac{1}{2}
[d(a_i, p) + d(p,p'') + d(a_i, p') + d(p',p'')]\\
&\leq &\frac{1}{2}
[d(a_i, p) + d(p,p'') + d(a_i,p) +d(p, p') + d(p',p'')]\\
&\leq &\frac{1}{2}
[2d(a_i, p) +\delta  + \delta +2\delta]\\ 
&=& d(a_i, p) +2\delta.
\end{eqnarray*}

Por outro lado, 
\begin{eqnarray*}
   d(a_i, p)- (a_j , a_k)_{a_i} &=& \frac{1}{2}
[2d(a_i,p) + d(a_k,a_j) - d(a_j,a_i) - d(a_k,a_i)] \\ &=&\frac{1}{2}
[2d(a_i, p) + d(a_k, p'')+d(p'',a_j) - d(a_j,p) \\& &\quad  -d(p,a_i)- d(a_k, p')- d(p',a_i) ] \\ 
&\leq &  2\delta.
\end{eqnarray*}

Logo,
$$d(p,a_{ij})= |d(a_i,p)-d(a_i, a_{ij})| = | d(a_i, p) - (a_j , a_k)_{a_i}| \leq 2\delta.$$

Analogamente, $d(a_{ik}, p')\leq 3\delta$, donde segue que $d(a_{ij} , a_{ik}) \leq 6\delta$.
Usando o Lema \ref{lemma:bigons}, concluímos que $d_{Haus}(\overline{a_ia_{ij}},\overline{a_ia_{ik}}) \leq 7 \delta$.
Para a última parte, aplicamos o Lema \ref{lemma:companheirosviagem}. 
\end{proof}

\begin{corollary}
A aplicação de colapso $\kappa : \Delta \to T$, onde $\Delta$ é um triângulo e $T$ é uma trípode, como na Figura \ref{fig:tripode}, é uma quasi-isometria uniforme.
\end{corollary}

\begin{proof}
    O mapa $\kappa$ é sobrejetivo e $1$-Lipschitz. Por outro lado, pela parte final do Lema \ref{lemma:triangulosparecemtripodes}, obtemos $$d(x, y) - 14\delta \leq d(\kappa(x), \kappa(y)),$$ para todos $x,y\in T$. Portanto, $\kappa$ é uma $(1,14\delta)$-quasi-isometria.
\end{proof}

\begin{exercise}
Mostre que, com a notação do Lema \ref{lemma:triangulosparecemtripodes}, vale $|d(a_k,a_{ij})-d(a_k,\overline{a_{i}a_j})|\leq 8 \delta$.
\end{exercise}

Seja $\gamma$ uma geodésica em $X$. Para um ponto $x \in X$,  podemos tentar definir a projeção de $x$ sobre $\gamma$,  $p =\pi_{\gamma}(x)$ como algum ponto de $\gamma$ mais próximo de $x$. Em geral, essas projeções   não estão bem-definidas. No entanto, em espaços geodésicos hiperbólicos, o seguinte lema mostra  que essa noção é  ``quasi''-bem definida. 

\begin{lemma}[A projeção é quasi bem-definida] Seja $\gamma$ uma geodésica em $X$. Se, dado um ponto $x \in X$, escolhemos $p =\pi_{\gamma}(x)$ como algum ponto de $\gamma$ mais próximo de $x$ então dado  $p' \in \gamma$ tal que $d(x,p')< d(x,p)+R$ temos 
$$d(p,p')\leq 2(R+2\delta).$$ 
Em particular, se $d(p,x) = d(p',x)$ então $$d(p,p')\leq 4 \delta.$$
\end{lemma}

\begin{proof}
Considere $\alpha, \alpha'$ geodésicas conectando $x$ a $p$ e $p'$, respectivamente.  Se $d(x,p') \leq R+\delta$, então, pela desigualdade triangular e usando que $p$ é o ponto de $\gamma$ mais próximo de $x$, teremos $d(p,p')\leq d(p,x)+d(x,p')\leq 2d(x,p')\leq 2 (\delta + R)\leq 2(R+2\delta)$. Por outro lado, se $d(x,p') > R+\delta$, então existe $q'\in \overline{xp'}$ com $d(p',q')=\delta + R.$ Agora, usando que o triângulo geodésico $\Delta(p,p',x)$ é $\delta$-magro, existe algum ponto $q\in \overline{pp'}\cup \overline{px}$ tal que $d(q,q')\leq \delta$. Se $q \in \gamma$ e $q\neq p$, obtemos uma contradição com o fato que $p$ é ponto mais próximo de $x$. Com efeito, 
\begin{eqnarray*}
    d(x,q)&\leq& d(q,q')+d(q',x)\\ 
    &= & d(q,q') +d(x,p')-d(p',q')\\ 
    &\leq & \delta +d(x,p')-d(p',q')\\
    &=& d(x,p')-R\\
    &<&d(x,p).
\end{eqnarray*}
Portanto, $q\in \overline{xp}$, e assim obtemos 
\begin{eqnarray*}
    d(x,p')-(R+ \delta)&=&d(x,q')\\
    &\leq & d(x,q)+d(q,q') \\
    &\leq & d(x,p)-d(p,q)+\delta.
\end{eqnarray*}
Desse modo, 
\begin{equation*}
    d(p,q)\leq d(x,p)-d(x,p')+R+2\delta \leq R+2\delta.
\end{equation*} 
Além disso, $d(p',q)\leq R+2\delta$, e portanto $d(p,p')\leq 2(R+2\delta)$.
\end{proof}

\begin{lemma}[Divergência exponencial de geodésicas]
Se $\overline{xy}$ é uma geodésica em $X$ de comprimento $2r$ e $m$ é o seu ponto médio, então cada caminho ligando
$x$ a $y$ fora da bola $B (m, r)$ tem comprimento pelo menos $2^{\frac{r-1}{\delta}}.$
\end{lemma}

\begin{proof}
Seja $\gamma$ um tal caminho de comprimento $l.$ Divida $\gamma$ em $2^k$ arcos de comprimentos $\leq \dfrac{l}{2^k}$, onde $k =  \lfloor \log_2 l \rfloor$ é escolhido de modo que $\dfrac{l}{2^k}\leq 2$. 

Sejam $x_0 = x, x_1,\ldots , x_{2^k} = y$ os pontos consecutivos em $\gamma$ que são extremos dos novos segmentos obtidos após a subdivisão. Pelo Lema \ref{lemma:poligonosmagros}, aplicado ao polígono geodésico com esses vértices (que possui portanto $2^k+1$ lados) tem-se que $$m \in \mathcal{N}_{k \delta}\displaystyle\left( \bigcup_{i=0}^{n-1}\overline{x_ix_{i+1}}\right).$$

\begin{figure}[H]
     \centering
 \begin{tikzpicture}
\draw [gray,thick,opacity=0.7] plot [smooth] coordinates {(0,0) (1,2) (2,2) (3,2) (4,0)};

\draw[thick] (0,0) to [bend left = 5] (4,0);
\draw (0,-0.2) node{$x$};
\draw (4,-0.2) node{$y$};
\draw[gray, thick] (3.4,1.6) node{$\gamma$};

\draw (0.45,1) node{\tiny{$\bullet$}};
\draw (1,2) node{\tiny{$\bullet$}};
\draw (2,2) node{\tiny{$\bullet$}};
\draw (3.6,0.95) node{\tiny{$\bullet$}};
\draw (2.8,2.1) node{\tiny{$\bullet$}};
\draw (1.7,0.1) node{\tiny{$\bullet$}};

\draw[gray] (1.7,1.9) node{\tiny{$\bullet$}};
\draw[gray] (1.7,1.7) node{\tiny{$p$}};

\draw (1.7,-0.1) node{\tiny{$m$}};
\draw (0.4,1.25) node{\tiny{$x_1$}};
\draw (1,2.2) node{\tiny{$x_2$}};
\draw (2,2.2) node{\tiny{$x_3$}};
\draw (4.1,1) node{\tiny{$x_{2^k-1}$}};
\draw (2.8,2.3) node{\tiny{$x_4$}};

\draw (3.15,1.7) node{\tiny{$\ddots$}};

\draw[thick] (0,0) to [bend right =30] (0.45,1);
\draw[thick] (0.45,1) to [bend right =30] (1,2);
\draw[thick] (1,2) to [bend right =30] (2,2);
\draw[thick] (2,2) to [bend right =30] (2.8,2.1); 
\draw[thick] (3.6,0.95) to [bend right=30] (4,0);
\end{tikzpicture}
         \caption{Divergência exponencial}
     \label{fig:divexp}   
\end{figure}

Seja $p$ o ponto da curva geodésica por partes $\sigma = \displaystyle \bigcup_{i=0}^{n-1}\overline{x_ix_{i+1}} $  mais próximo de $m$. Pelo que vimos acima, $d(p,m) \leq k\delta$. Como $d (x_j, m)\geq r$ para cada $j=0,\ldots,n$ e existe algum índice $i_0$ para o qual $d (x_{i_0}, p)\leq 1$, obtemos:  $$r \leq d(x_{i_0},m) \leq d (x_{i_0}, p) + d (p, m) \leq 1+ k \delta. $$
 Portanto, 
 $r-1\leq\delta  \log_2 (l) $ e concluímos finalmente que  $l\geq 2^{ \frac{r-1}{\delta}}$.
 \end{proof}

Apresentamos a seguir uma sequência de lemas que serão úteis na demonstração do Lema de Morse, ou Teorema de estabilidade de geodésicas, que é o principal objetivo desta seção. A demonstração desses lemas é deixada como exercício ao leitor, e sugerimos que sejam utilizados os lemas anteriores.

\begin{lemma}
\label{lemma:12delta}
Seja $S(a,r)$ uma esfera em $X$. Tome $x,y\in S(a,r)$, tais que $d(x,y) = 2r$. Então cada caminho ligando $x$ a $y$ fora da bola $\overline{B (a, r)}$ tem comprimento pelo menos $2^{\frac{r-1}{\delta}-6}-12\delta.$
\end{lemma}

\begin{lemma}
\label{lemma:14delta}
Seja $S(a, r_1+r_2)$ uma esfera em $X$. Tome dois pontos $x, y \in S (a, r_1 + r_2)$ de modo que existam duas geodésicas $\overline{xa}$ e $\overline{ya}$ cruzando a esfera $S (a, r_1)$ em dois pontos $x_0, y_0$ a uma distância de pelo menos $7\delta$.  Então cada caminho ligando $x$ a $y$ fora da bola $\overline{B (a, r_1+r_2)}$ tem comprimento pelo menos $2^{\frac{r_2-1}{\delta}-7}-14\delta.$
\end{lemma}

\begin{exercise}
    Demonstre os Lemas \ref{lemma:12delta} e \ref{lemma:14delta}.
\end{exercise}

Finalmente, chegamos ao enunciado do Lema de Morse.

\begin{thm}[Lema de Morse]
\label{thm:Morse}
Considere constantes $\lambda \geq 1$ e $\delta, \varepsilon\geq 0 $. Sejam $X$ um espaço $\delta$-hiperbólico, $c$ uma $(\lambda,\varepsilon)$-quasi-geodésica e $\gamma$ uma geodésica ligando os pontos finais de $c$. Então a distância de Hausdorff entre $c$ e $\gamma$ é limitada por uma constante $R=R(\lambda,\varepsilon,\delta)$.
\end{thm}

\begin{remark}
    O resultado não é verdadeiro de maneira geral para espaços $\mathrm{CAT}(0)$, não valendo por exemplo em $\R^2$.  
\end{remark}

\begin{corollary}
\label{cor:ripsQI}
A hiperbolicidade de Rips é um invariante por quasi-isometrias.
\end{corollary}

\begin{proof} De fato, sejam $X \qi X'$, onde $X$ é $\delta$-hiperbólico. Dados pontos $a,b,c \in X'$ e uma $(L,A)$-quasi-isometria $f:X'\to X$, como $X$ é $\delta$-hiperbólico, o triângulo $\Delta(f(a), f(b), f(c))$ é $\delta$-magro e, pelo Lema de Morse, obtemos que $f(\Delta(a,b,c))$ é um triângulo quasi-geodésico $(2R+\delta)$-magro. Portanto, $\Delta(a,b,c)$ é $(L(2R+\delta)+A)$-magro e concluímos com isso que $X'$ é $(L(2R+\delta)+A)$-hiperbólico por Rips. 
\end{proof}

Para provar o Lema de Morse, usaremos o seguinte resultado:

\begin{lemma}
\label{lemma:auxiliarMorse}
    Sejam $X$ um espaço métrico geodésico e $c:[a,b]\to X$ uma $(\lambda,\varepsilon)$-quasi-geodésica. Então existe  $c':[a,b]\to X$, uma $(\lambda',\varepsilon')$-quasi-geodésica, com as seguintes propriedades:
\begin{enumerate}[(1)]
    \item $c'(a) = c(a)$, $c'(b) = c(b)$;
    \item $\varepsilon ' = 2(\lambda + \varepsilon)$;
    \item $\ell(c'\vert_{[t,t']}) \leq K_1(d(c'(t), c'(t')))+K_2$, para quaisquer $t,t' \in [a,b]$, onde $K_1 = \lambda(\lambda + \varepsilon)$ e $K_2 = (\lambda \varepsilon' +3)(\lambda + \varepsilon)$;
    \item $d_{Haus}(e,e') \leq \lambda + \varepsilon$, onde $e, e'$ são as imagens de $c$ e $c'$, respectivamente.
\end{enumerate}
    
\end{lemma}

\begin{exercise} 
\label{exer:auxiliarMorse}
    O objetivo desse exercício é provar o Lema \ref{lemma:auxiliarMorse}. Definiremos $c'$ da seguinte forma: considere o conjunto $\Sigma:=\{a,b\}\cup (\Z \cap (a,b))$. Para cada $t \in \Sigma$, ponha $c'(t) = c(t)$. Para cada subintervalo $(t_i,t_{i+1}) \subset [a,b]$ com extremos em $\Sigma$, escreva uma reparametrização $\gamma_i(s)$ da geodésica que liga $c(t)$ a $c(t')$ nesse intervalo, e defina $c'(s)$ como $\gamma_i(s)$.
Prove que a quasi-geodésica $c'$ assim definida satisfaz as propriedades $(1)$ a $(4)$.

\begin{figure}[H]
     \centering
 \begin{tikzpicture}
\draw [gray,thick,opacity=0.7] plot [smooth] coordinates {(0,0) (1,1) (2,-1) (3,0.5)(4,-0.8) (5,1) (6,-0.6)(7,0)};

\node[scale=0.7,gray] at  (0,-0.2) {$c(a)$};
\node[scale=0.7,gray] at (7.2,-0.2) {$c(b)$};
\node[scale=0.7, gray] at  (3,0.7) {$c(t)$};

\node[scale=0.7] at  (1,0.1) {$c'$};

\node[scale=0.5,gray] at (1.65,0.35) {$c(t_1)$};
\node[scale=0.5,gray] at (2.8,-0.1) {$c(t_2)$};
\node[scale=0.5,gray] at (5.9,0.15) {$c(t_{n-1})$};

 \draw (0,0) node{\tiny{$\bullet$}};
 \draw (7,0) node{\tiny{$\bullet$}};
 \draw (1.4,0.2) node{\tiny{$\bullet$}};
 \draw (2.63,0) node{\tiny{$\bullet$}};
\draw (4.45,0) node{\tiny{$\bullet$}};
\draw (5.62,0) node{\tiny{$\bullet$}};
\draw (3.4,0) node{\tiny{$\bullet$}};

\draw (3.95,0) node{\tiny{$\ldots$}};

\draw (0,0) to [bend left =20] (1.4,0.2);
\draw (1.4,0.2) to [bend left =10] (2.63,0);
\draw (2.63,0) to [bend left =30] (3.4,0);
\draw (4.45,0) to [bend left =15](5.62,0); 
\draw  (5.62,0) to [bend right=20] (7,0);
\end{tikzpicture}
         \caption{Exercício \ref{exer:auxiliarMorse}}
     \label{fig:auxMorse}   
\end{figure}

\end{exercise}

\textit{Demonstração do Teorema \ref{thm:Morse}.}
Substitua $c$ por uma quasi-geodésica $c'$ com as propriedades do Lema \ref{lemma:auxiliarMorse}. Sejam $p=c(a)$ e $q=c(b)$.

Defina $D= \sup\{d(x,e') \mid x \in \overline{pq}\}$ e tome um ponto $x_0 \in \overline{pq}$ com $d(x,e')= D$. Escolha $y \in \overline{px_0} \subset \overline{pq}$ tal que $d(y,x_0) = 2D$ ou, caso $d(p,x_0) \leq 2D$, tome $y=p$. De maneira análoga, escolha um ponto $z \in \overline{x_0q}$.

Tome pontos $y',z' \in e':=\mbox{imagem}(c')$ tais que $d(y,y') \leq D$ e $d(z,z') \leq D$. O caminho $\sigma$ formado pela concatenação da geodésica $\overline{yy'}$ com a parte de $c'$ entre $y'$ e $z'$ e finalmente com a geodésica $\overline{z'z}$ está fora da bola $B(x_0,D)$. Usando o Lema \ref{lemma:auxiliarMorse} e o Lema \ref{lemma:poligonosmagros}, obtemos:
\begin{eqnarray*}
    \ell (\sigma) &\leq& 6DK_1 + K_2 +2D;\\
    d(x_0, \mbox{imagem}(\sigma)) = D &\leq & \delta |\log_2(\ell(\sigma))|+1.
\end{eqnarray*}
Assim, 
$$D-1 \leq  \delta\log_2(6DK_1 + K_2 +2D),$$
e portanto $D \leq D_0 = D_0(\lambda, \varepsilon, \delta)$, pois funções lineares crescem mais rápido que $\log_2$.

Por fim, $\mbox{imagem}(c')\subset \mathcal{N}_{R'}(\overline{pq})$, com $R' = D_0(1+K_1) + \frac{K_2}{2}$ e podemos tomar $R = R' + \lambda + \varepsilon$, usando o Lema~\ref{lemma:auxiliarMorse}.
\qed

\begin{remark}
    A primeira versão desse teorema foi provada por Morse em \cite{morse1924fundamental}, para o seguinte contexto: considere uma superfície compacta $S$ munida de duas métricas Riemannianas $g_1, g_2$ de curvatura negativa. Levantando essas métricas ao recobrimento universal de $S$, obtém-se que cada geodésica com respeito ao levantamento $\tilde{g}_1$ de
    $g_1$ é uma  quasi-geodésica (uniforme) com respeito ao levantamento $\tilde{g}_2$ de $g_2$. Morse provou que todas as geodésicas com respeito a $\tilde{g}_1$ são uniformemente próximas das geodésicas com respeito a $\tilde{g}_2$, desde que seus pontos extremos sejam os mesmos. Mais tarde, Busemann mostrou em \cite{busemann1966extremals} uma versão desse lema no caso de $\hipn$, onde a métrica $g_2$ não precisava ser Riemanniana. Uma versão em termos de quasi-geodesicas é devida a Mostow (\cite{mostow1973strong}), para o caso de  espaços simétricos de curvatura negativa. A primeira prova para o caso de espaços métricos geodésicos $\delta$-hiperbólicos em geral foi feita por Gromov  em \cite{gromov1987hyperbolic}. Em \cite{kd}, são apresentadas duas demonstrações do Lema de Morse, uma delas usando o conceito de ultrafiltros e outra parecida com a versão apresentada nesta seção, na qual são tornadas explícitas estimativas para os valores de $R$ e $R'$, em termos de $\lambda, \varepsilon$ e $\delta$. 
\end{remark}

\subsection{Quasi-convexidade}
Em um espaço métrico $(X,d)$, diremos que um ponto $b$ está entre $a$ e $c$ se $d(a, b)+d(b, c) = d(a, c)$. Denotamos por $I(a, c)$ o subconjunto de pontos de $X$ que estão entre $a$ e $c$, e chamamos esse conjunto de \textit{intervalo} entre $a$ e $c$. Um subconjunto $Y\subset X$ é dito \textit{convexo} \index{conjunto convexo} se, para todos  $a, b\in Y $, o intervalo $I(a, b)$ está contido em  $Y$.
Observamos que a noção de convexidade métrica não é equivalente à noção usual de convexidade em espaços euclidianos ou variedades riemannianas.

O \textit{fecho convexo}\index{fecho convexo} de um subconjunto $Y\subset X$ é a interseção de todos os subconjuntos convexos de $X$ que contém $Y$. Existem exemplos de espaços $\delta$-hiperbólicos que são o fecho convexo de um número finito de seus pontos. Logo, a noção de fecho convexo particularmente não é muito útil no caso de espaços métricos geodésicos hiperbólicos. Nesse contexto, essa noção é substituída pela noção de \textit{quasi-convexidade}\index{quasi-convexidade}.

\begin{definition}
Seja $X$ um espaço métrico geodésico e $Y\subset X$. Então o \textit{fecho quasi-convexo} \index{fecho quasi-convexo} $H (Y)$ de $Y$ em $X$ é a união de todas as geodésicas $\overline{y_1y_2}\subset X$, com pontos extremos $y_1, y_2$ contidos em $Y$.
Consequentemente, um subconjunto $Y$ de $X$ é chamado $R$-\textit{quasi-convexo} se $H (Y) \subset \overline{\mathcal{N}_R (Y)}$. Um subconjunto $Y\subset X$ é chamado \textit{quasi-convexo} \index{conjunto quasi-convexo} se for $R$-quasi-convexo para algum $R <\infty$.
\end{definition}

\begin{exercise}
    Seja $X$ um espaço métrico $\delta$-hiperbólico e geodésico. Mostre que:
    \begin{enumerate}[(a)]
        \item qualquer bola $B(x,R) \subset X$ é $\delta$-quasi-convexa;
        \item se $Y_i \subset X$ são subconjuntos $R_i$-quasi-convexos, para $i=1,2$, e $Y_1\cap Y_2 \neq \emptyset$, então $Y_1 \cup Y_2$ é $(R_1+R_2+\delta)$-quasi-convexo (observe a diferença com respeito à convexidade --- a união de dois subconjuntos convexos não precisa ser convexa);
        \item qualquer que seja o subconjunto $Y\subset X$,  seu fecho quasi-convexo $H(Y)$ é $2\delta$-quasi-convexo.
    \end{enumerate}
\end{exercise}

\begin{example}
    Um exemplo importante de conjuntos que não são quasi-convexos em $\hipn$ são as horoesferas. O fecho quasi-convexo de uma horoesfera $S$ é a horobola limitada por $S$.
\end{example}

\begin{thm}[Bowditch]
Sejam $X, Y$ espaços métricos geodésicos, onde $X$ é $\delta$-hiperbólico.  Então, para cada mergulho quasi-isométrico $f: Y\to X$, a imagem $f (Y)$ é quasi-convexa em $X$.
\end{thm}

\begin{proof}
Sejam $y_1,y_2 \in Y$ e considere uma geodésica $\alpha = \overline{y_1y_2}$. Dada uma $(L,A)$-quasi-isometria $f:Y \to X$, a curva $\beta = f(\alpha)$ é uma $(L,A)$-quasi-geodésica em $X$. Pelo Lema de Morse, 
$$d_{Haus}(\beta, \beta^*) \leq R=R(L,A, \delta),$$
onde $\beta^* = \overline{f(y_1)f(y_2)}$ é geodésica em $X$. Assim, $\beta^* \subset \mathcal{N}_R(f(Y))$ e portanto $f(Y)$ é quasi-convexo.
\end{proof}

A seguir, mostraremos uma ``recíproca'' para o teorema anterior, também devida a Bowditch. 

\begin{definition}
Um espaço métrico  $(X,d)$ é dito \textit{grosseiramente conectado}\index{espaço grosseiramente conectado} se existe um número real $c$ tal que, para cada par $x, y \in X$, existe uma sequência finita de pontos $x_i\in X$, onde $x_0=x, x_n=y$ e $d (x_i, x_{i + 1}) \leq c$ para cada $i=0,\ldots,n$.
\end{definition}

Dado um espaço métrico $(Y, d_Y)$, denote por $d_{Y, C}$ uma nova métrica em $Y$ dada pelo ínfimo dos comprimentos de caminhos $c$-grosseiros que conectam pontos de $Y$.

\begin{thm}
Suponha que $Y \subset X$ seja  $c$-grosseiramente conectado e que $Y$ seja quasi-convexo em $X$. Então o mapa identidade $f: (Y, d_{Y, C})\to (X, d_X)$ é um mergulho quasi-isométrico para todo $C\geq 2c + 1$. 
\end{thm}

\begin{proof}
    Seja $C>0$ um número real tal que o fecho quasi-convexo de $Y$, $H(Y)$, esteja contido em $\mathcal{N}_C(Y)$. Se $d_Y(y,y') \leq C$, então também vale $d_X(y,y') \leq C$.

    Tome $y,y'$ pontos de $Y$ e $\gamma$ uma geodésica em $X$ de comprimento $L$ que conecta $y$ a $y'$. Subdividindo $\gamma$ em $n=\lfloor L\rfloor$ segmentos de comprimento $1$ e mais um segmento de comprimento $L-n$, obtemos pontos $z_0 = y, z_1, z_2, \ldots z_{n+1}=y' $.

    Como cada $z_i \in\mathcal{N}_c(Y)$, devem existir pontos $y_i \in Y$, com $d_X(y_i, z_i)$ tais que $d_X(y_i, z_i) \leq c$, onde $y_0=z_0$ e $y_{n+1} = z_{n+1}$.

    Então $d_X(z_1, z_{i+1}) \leq 2c+1 $. Logo, concluímos que 
    $$d_{Y,c}(y,y') \leq C(n+1) \leq Cd_x(y, y') +C.$$
\end{proof}


A noção de quasi-convexidade tem várias aplicações no estudo de grupos finitamente gerados. Como um primeiro exemplo, temos o seguinte resultado:

\begin{thm} \label{thm:quasiconv}
Suponha que $G$ é um grupo finitamente gerado e que  $H$ é um subgrupo  $R$-quasi-convexo de $G$. Então $H$ é finitamente gerado e a inclusão $H \hookrightarrow G $ é um mapa bi-Lipschitz.
\end{thm}

\begin{proof}
Seja $S$ um conjunto finito de geradores de $G$ e denote por  $S'$ o conjunto dos elementos de $H$ que distam no máximo $2R+1$ da identidade $e$ em $\cay(G, S)$.

Suponha que $\rho$ é um caminho geodésico  em $\cay(G, S)$ ligando $1$ a um  elemento $h$ de $H$. Vamos inserir desvios no caminho $\rho$ da seguinte forma: ao chegar a cada vértice, viaje por uma geodésica para um elemento  de $H$ mais próximo e depois volte. Pela quasi-convexidade de $H$, cada desvio tem comprimento no máximo $R$. Então esse novo caminho é a concatenação de, no máximo, $ {d_S(1,h)}$ caminhos, onde cada um dos quais viaja entre elementos de $H$ com comprimento $\leq 2R+1$. Isto é, escrevemos $h$ em termos de elementos de $S'$, já que dois elementos consecutivos $h_i, h_{i+1}$ de $H$ no novo caminho distam de, no máximo, 2R+1 e portanto $ h_i^{-1} h_{i+1}\in S'$. 

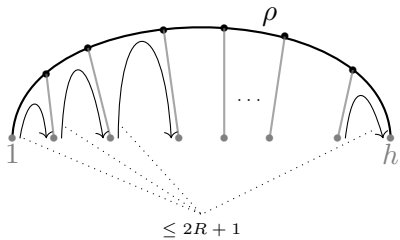
\begin{figure}[H]
     \centering
\begin{tikzpicture}
\draw[thick] (0,0) to [bend left = 88] (5,0);
\draw[gray] (0,-0.2) node{\tiny{$1=h_0$}};
\draw[gray] (5,-0.2) node{\tiny{$h = h_n$}};
\draw[gray] (2.2,-0.2) node{\tiny{$h_i$}};
\draw[gray] (2.8,-0.2) node{\tiny{$h_{i+1}$}};
\draw[gray] (0,0) node{\tiny{$\bullet$}};
\draw[gray] (5,0) node{\tiny{$\bullet$}};

\draw[thick] (3.4,1.6) node{$\rho$};

\draw (0.45,0.85) node{\tiny{$\bullet$}};
\draw (1,1.19) node{\tiny{$\bullet$}};
\draw (2,1.43) node{\tiny{$\bullet$}};
\draw (2.8,1.45) node{\tiny{$\bullet$}};
\draw (3.6,1.34) node{\tiny{$\bullet$}};
\draw (4.5,0.9) node{\tiny{$\bullet$}};

\draw[thick, gray, opacity=0.7] (0.45,0.85) to (0.55,0);
\draw[thick, gray, opacity=0.7] (1,1.19) to (1.3,0);
\draw[thick, gray, opacity=0.7] (2,1.43) to (2.2,0);
\draw[thick, gray, opacity=0.7] (2.8,1.45) to (2.8,0); 
\draw[thick, gray, opacity=0.7] (3.6,1.34) to (3.4,0);
\draw[thick, gray, opacity=0.7](4.5,0.9) to (4.3,0);

\draw[gray] (0.55,0) node{\tiny{$\bullet$}};
\draw[gray] (1.3,0) node{\tiny{$\bullet$}};
\draw[gray] (2.2,0) node{\tiny{$\bullet$}};
\draw[gray] (2.8,0) node{\tiny{$\bullet$}};
\draw[gray] (3.4,0) node{\tiny{$\bullet$}};
\draw[gray] (4.3,0) node{\tiny{$\bullet$}};

\draw[->] (0.1,0) .. controls (0.2,0.6) and (0.4,0.6) .. (0.45,0);
\draw[->] (0.65,0) .. controls (0.7,1.2) and (1,1.2) .. (1.2,0);
\draw[->] (1.4,0) .. controls (1.4,1.7) and (2,1.7) .. (2.1,0);
\draw (3.15,0.5) node{\tiny{$\cdots$}};
\draw[->] (4.4,0) .. controls (4.5,0.8) and (4.7,0.7) .. (4.9,0);

\draw[dotted] (0.15,0) to (2.5,-1); 
\draw[dotted] (0.7,0.15) to (2.5,-1);
\draw[dotted](1.45,0.15) to (2.5,-1);
\draw[dotted](4.75,0.1) to (2.5,-1);
\draw (2.5,-1.2) node{\tiny{$\leq 2R +1$}};
\end{tikzpicture}
         \caption{Teorema \ref{thm:quasiconv} -- pontos em cinza são pontos de $H$.}
     \label{fig:quasiconv}   
\end{figure}

Isso implica que $S'$ gera $H$, já que podemos escrever $$h=(h_0^{-1}h_1)(h_1^{-1}h_2)\cdots(h_{n-1}^{-1}h_n),$$ e, além disso, a inclusão $H \hookrightarrow G$ é bi-Lipschitz, pois:
$$\displaystyle \frac{1}{2R+1} d_S(1,h)  \leq  d_{S'}(1,h)  \leq d_S(1,h).$$
A desigualdade da direita decorre do fato de que cada aresta em $\rho$ produz um gerador de $S'$, enquanto a desigualdade da esquerda segue ao aplicarmos sucessivas vezes a desigualdade triangular à distância $d_S(1,h)$, quando escrevemos $h$ em termos dos geradores de $S'$, além do fato de que cada gerador de $S'$ tem tamanho no máximo $2R+1$ na métrica $d_S$. 
\end{proof}

\section{Fronteiras ideais}
A todo espaço métrico Gromov-hiperbólico $X$ pode-se associar sua \textit{fronteira ideal}  $\partial_{\infty}X$. Este é um invariante por quasi-isometrias em espaços métricos próprios, geodésicos e hiperbólicos, o qual tem se mostrado uma ferramenta bastante útil, especialmente para o estudo de grupos hiperbólicos, assunto que será discutido no próximo capítulo. A fronteira ideal tem, em geral, uma estrutura rica (topológica, dinâmica, métrica, quasi-conforme, e algébrica) com vasta aplicabilidade. Nesta seção, apresentaremos três definições de fronteira ideal de um espaço métrico $X$, as quais são equivalentes se $X$ for geodésico, próprio e $\delta$-hiperbólico.

\subsection{Fronteiras definidas a partir de raios geodésicos} 
Se $X$ é um espaço métrico geodésico, dois raios geodésicos $\rho_1, \rho_2: [0, \infty) \to X$ em $X$ são ditos {\it assintóticos}\index{raios assintóticos} se suas imagens estão a uma distância de Hausdorff finita ou, equivalentemente, se
$$\sup_{t\in [0,\infty)} d(\rho_1(t),\rho_2(t))< \infty.$$

Podemos assim definir uma relação de equivalência no conjunto de todos os raios geodésicos em $X$, dizendo que dois raios são equivalentes se, e somente se, eles forem assintóticos. Denotamos por $\rho(\infty)$ a classe de equivalência de um raio $\rho$, isto é, o conjunto de todos os raios geodésicos em $X$ que são assintóticos a $\rho$.
\begin{definition}
A \textit{fronteira ideal} \index{fronteira ideal} de um espaço métrico geodésico $X$ é a coleção $\partial_{\infty}(X)$ de todas as classes de equivalência de raios geodésicos via a relação de equivalência definida acima. 
\end{definition} 

\begin{example}
    A fronteira ideal da reta $\R$ com a métrica usual consiste em dois pontos, assim como a fronteira ideal do grafo escada infinita da Figura \ref{fig:escada}, cujas arestas têm comprimento igual a um.
\end{example}

\begin{figure}[!ht]
    \centering
    \begin{tikzpicture}[scale=0.8]

\draw[thick] (0,1) -- (10.4,1);
\draw[thick] (0,0) -- (10.4,0);

\foreach \x in {1,2.2,3.4,4.6,5.8,7,8.2,9.4}
{
    \draw[thick] (\x,0) -- (\x,1);
}

\end{tikzpicture}
    \caption{Grafo escada infinita.}
    \label{fig:escada}
\end{figure}

Um segundo modelo para a fronteira ideal é o conjunto das classes de equivalência pela mesma relação de equivalência acima, definida no subconjunto dos raios geodésicos em $X$ que começam em um mesmo ponto base:

\begin{definition}
Dado um ponto de base $p\in X$, a \textit{fronteira ideal com base $p$} denotada por $\partial_{\infty}^p(X)$ é a coleção de todas as classes de equivalência de raios geodésicos $\rho:[0,\infty) \to X$ com $\rho(0)=p$ via a relação de equivalência acima. 
\end{definition}

Os lemas a seguir mostram que, assumindo que $X$ é um espaço próprio e hiperbólico, então as duas definições são equivalentes. Nesse caso, denotaremos também  $\partial_{\infty}^p(X)$ por $\partial_{\infty}(X)$, já que esse conjunto não dependerá do ponto base.

\begin{lemma}
\label{lemma:raiospontobase}
Seja $X$ um espaço métrico próprio, geodésico e $\delta$-hiperbólico. Para cada $x \in X$ e cada $\xi \in \partial_{\infty}X$, existe um raio geodésico $\rho:[0,\infty) \to X$ tal que $\rho(0)=x$ e $\rho(\infty) = \xi$.
\end{lemma}

\begin{proof}
Seja $\rho'$ um raio geodésico com $\rho'(\infty) = \xi$. Considere a sequência de segmentos geodésicos $\rho_n : [0,d(x,\rho'(n))]\to X$, conectando $x$ a $\rho'(n)$. A $\delta$-hiperbolicidade de $X$ implica que a imagem de $\rho_n$ está a uma distância de Hausdorff de no máximo $\delta+d(x,\rho'(0))$
de $\rho'|_{[0,n]}$. Pelo teorema de Arzelà--Ascoli, passando a uma subsequência se necessário, $\rho_n$ converge a um raio $\rho,$ tal que $\rho(0) = x$. Como cada $\rho$ está a uma distância de Hausdorff de no máximo $\delta+d(x,\rho'(0))$ de $\rho'|_{[0,n]}$, segue que $\rho$ está a uma distância de Hausdorff de no máximo $\delta+d(x,\rho'(0))$ de $\rho'$. Em particular, $\rho$ e $\rho'$ são assintóticos.
\end{proof}

O lema a seguir é uma generalização da propriedade de ``companheiros de viagem'' das geodésicas (Lema \ref{lemma:companheirosviagem}).

\begin{lemma}[Raios assintóticos estão uniformemente próximos]
\label{lemma:raiosassintóticospróximos}
Sejam $\rho_1$ e $\rho_2$  raios geodésicos assintóticos em um espaço métrico próprio, geodésico e $\delta$-hiperbólico $X$, tais que  $\rho_1(0)=\rho_2(0)=x$. Então, para cada $t>0$, $$d(\rho_1(t),\rho_2(t))\leq 2 \delta.$$
\end{lemma}

\begin{proof}
Dado $t \in [0,\infty)$ escolha $t_0 \gg t$. 
Uma vez que os raios são assintóticos, deve existir $C>0$ tal que $d_{Haus}(\rho_1,\rho_2) \leq C$ e, portanto, existe $t_1>0$ com $d(\rho_1(t_0),\rho_2(t_1))\leq C$. 

Sendo $X$ um espaço $\delta$-hiperbólico, o triângulo $\Delta(x,\rho_1(t_0),\rho_2(t_1))$ é $\delta$-magro, donde concluímos que deve existir um ponto $m$ em um dos segmentos geodésicos $\overline{x\rho_2(t_1)}$ ou $\overline{\rho_1(t_0)\rho_2(t_1)}$ com $d(\rho_1(t),m) \leq \delta$. Como o comprimento de $\overline{\rho_1(t_0)\rho_2(t_1)}$ é menor ou igual a $C$ e $t_0$ é muito maior que $t$, o ponto $m$ deve pertencer ao segmento $\overline{x\rho_2(t_1)}$. Portanto, existe $t'>0$ tal que $d(\rho_1(t),\rho_2(t'))\leq \delta$. 

Pela desigualdade triangular, vale $|t-t'|\leq \delta$, e assim obtemos $d(\rho_1(t),\rho_2(t))\leq 2 \delta$.
\end{proof}

\subsection{Fronteira sequencial de \texorpdfstring{$X$}{X}} 

Um terceiro modelo para a fronteira ideal de $X$ é obtido através da noção de  convergência no infinito de sequências de $X$, definida usando o produto de Gromov. Fixado um ponto base $p \in X$ em um espaço métrico hiperbólico $(X,d)$, dizemos que uma sequência $(x_n)_{n\geq 1}$ de pontos de $X$ \textit{converge para o infinito} \index{sequência que converge para infinito} se 
$$\lim_{i,j \to \infty} (x_i,x_j )_p =\infty.$$

Note que, para quaisquer
$x_i, x_j, p, p' \in X$,
$$
\left| (x_i,x_j)_p - (x_i,x_j)_{p'} \right|
\leq d(p,p').
$$
Assim, a definição acima é independente da escolha de $p \in X$.

Fixado o ponto base $p$ acima, diremos que duas sequências $(x_n)$ e $(y_n)$ que convergem para o infinito são equivalentes, escrevendo $(x_n) \sim (y_n)$, se 
$$\lim_{i,j \to \infty} (x_i,y_j )_p =\infty.$$

A relação $\sim$ é uma relação de equivalência no conjunto de sequências que convergem para o infinito. Vamos denotar por $[(x_n)]$ a classe de equivalência de uma sequência $(x_n)$ nesse conjunto.

\begin{exercise}
    Verifique que a noção de sequências equivalentes via $\sim$ também não depende da escolha do ponto base $p$.
\end{exercise}

\begin{exercise}
    Mostre que  $\sim$ não é uma relação de equivalência no conjunto de sequências que convergem para infinito em $X=\R^2$.
\end{exercise}

\begin{example} Suponha que $X$ seja o grafo métrico representado pela Figura \ref{fig:exsequentboud}, onde os comprimentos das arestas são todos iguais a um. Então,
\[
(x_i,x_j)_p=(x_i',x_j')_p=\min\{i,j\}
\quad \text{e} \quad
(x_i,x_i')_p=i.
\]
Logo as sequências $(x_i)$ e $(x_i')$ convergem para o infinito e são equivalentes.
\end{example}

\begin{figure}[!ht]
\centering
\begin{tikzpicture}[scale=0.9]

\draw[->] (0,0) -- (6,0);

\filldraw (0,0) circle (1pt);
\node[below] at (0,0) {\small{$p$}};

\foreach \x/\lab in {
1/\small{$x_1$},2/\small{$x_2$},3/\small{$x_3$},4/\small{$x_4$},5/\small{$x_5$}
}{
    \filldraw (\x,0) circle (1pt);
    \node[below] at (\x,0) {\lab};

    \draw (\x,0) -- (\x,1);

    \filldraw (\x,1) circle (1pt);
}

\node[above] at (1,1) {\small{$x_1'$}};
\node[above] at (2,1) {\small{$x_2'$}};
\node[above] at (3,1) {\small{$x_3'$}};
\node[above] at (4,1) {\small{$x_4'$}};
\node[above] at (5,1) {\small{$x_5'$}};

\node at (6,0.5) {$\cdots$};

\end{tikzpicture}
\caption{Sequências equivalentes.}
\label{fig:exsequentboud}
\end{figure}

\begin{definition}
    Seja $(X,d)$ um espaço métrico $\delta$-hiperbólico. A \textit{fronteira sequencial}\index{fronteira sequencial} de $X$ é o conjunto $\partial_s(X)$ das classes de equivalência de sequências que convergem ao infinito pela relação $\sim$ acima. 
\end{definition}

\begin{proposition}
    Seja $(X,d)$ um espaço métrico geodésico próprio $\delta$-hiperbólico e $p\in X$ um ponto-base. Então existem bijeções naturais entre os conjuntos $\partial_{\infty}(X)$, $\partial_{\infty}^p(X)$ e $\partial_{s}(X)$.
\end{proposition}

\begin{proof}
Se $\rho$ é um raio geodésico em $X$ começando em $p$, então $\rho(\infty)$ pertence a ambos $\partial_{\infty}(X)$ e $\partial_{\infty}^p(X)$. A inclusão é portanto um mapa injetivo natural de $\partial_{\infty}^p(X)$ para $\partial_{\infty}(X)$. Pelos Lemas \ref{lemma:raiospontobase} e \ref{lemma:raiosassintóticospróximos}, esse mapa é também sobrejetivo.

 Agora, se $\rho$ é um raio geodésico em $X$, o que assumimos ser parametrizado por comprimento de arco, então a sequência $\rho(n)$ converge para infinito. Escrevendo $\iota(\rho(\infty)):=[(\rho(n))_{n\geq1}]$, obtemos um mapa natural $\iota: \partial_{\infty}X \to \partial_sX$. Para verificar que esse mapa está bem-definido, dados dois raios geodésicos parametrizados por comprimento de arco $\rho, \gamma$, com $\rho(\infty) = \gamma(\infty)$, mostremos que $[(\rho(i))_{i\geq1}] =[(\gamma(j))_{j\geq1}]$. Tomamos como ponto base, por exemplo, $p=\rho(0)$ e aplicamos a desigualdade triangular ao produto de Gromov das sequências $x_i = \rho(i)$ e $y_j=\gamma(j)$:

 \begin{eqnarray*}
     (x_i,y_j)_p &=& \dfrac{1}{2}\left( d(p, \rho(i))+  d(p, \gamma(j)) - d(\rho(i),\gamma(j))\right)\\
     &\geq & \dfrac{1}{2}\left( i-  d(p, \gamma(0))+j - d(\rho(i), \rho(j)) - d(\rho(j), \gamma(j))\right)\\
     &\geq & \dfrac{1}{2}\left( i+j - |i-j| - 2C\right),
 \end{eqnarray*}
onde $C = \sup_{t\geq 0} d(\rho(t),\gamma(t))$. Desse modo,  $[(x_i)] =[(y_j)]$, como queríamos demonstrar. Deixamos como exercício verificar que $\iota$ é um mapa injetivo.

Por outro lado, seja $(x_i)_{i\geq 1}$ uma sequência de $X$ que converge ao infinito, para um ponto $\xi = [(x_i)] \in \partial_sX$. Para cada $i\geq 1$, escolhemos um segmento
geodésico $[x_0, x_i]$ de um ponto base  $x_0$ até $x_i$. Novamente usando o teorema de Arzelà--Ascoli, mostramos como no Lema \ref{lemma:raiospontobase} que existe uma subsequência $(y_i)_{i\geq 1}$ de $(x_i)_{i\geq 1}$
tal que os segmentos $[x_0, y_i]$ convergem a um raio $\rho$. Temos portanto que o mapa $\iota$ é sobrejetivo, pois $\iota(\rho(\infty)) = [(x_i)]$. 
\end{proof}

\begin{definition}
Identificamos a partir de agora os três modelos descritos acima como um único conjunto, que chamaremos de \textit{bordo} ou \textit{fronteira} \index{fronteira ideal} de um espaço geodésico próprio $\delta$-hiperbólico $X$, e denotaremos apenas por $\partial X$.
\end{definition}

Cada modelo tem suas vantagens. Por exemplo, para mostrar que esse conjunto é compacto -- após inserir uma topologia em $\partial X$, de modo que os mapas acima sejam homeomorfismos -- usaremos o modelo $\partial_{\infty}(X)$. O modelo sequencial tem a vantagem de fazer sentido mesmo que $X$ não seja geodésico ou próprio. Por fim, existem outras maneiras equivalentes de definir a fronteira ideal de um espaço métrico. Alguns exemplos podem ser encontrados em \cite{ghys2013groupes} e nas referências lá mencionadas.

\begin{exercise}
A fronteira de um espaço hiperbólico $X$ é claramente vazia se $X$ possui diâmetro finito. Mostre que a recíproca vale se $X$ for geodésico e próprio. 
\end{exercise}
Ainda a respeito do exercício acima, no caso de espaços não-próprios, observe o que acontece com o espaço $0$-hiperbólico não localmente compacto
obtido ao munir da métrica canônica de árvores a união sobre $n \in \N$ dos segmentos da forma $[0, ne^{i\frac{\pi}{n}} ] $ no plano complexo.

\subsection{Topologias na fronteira ideal} 
Nesta seção, a menos de menção contrária, $X = (X,d)$ denotará um espaço métrico próprio, geodésico e $\delta$-hiperbólico. Nosso objetivo é munir $\partial X$ de uma topologia natural que torne este conjunto um espaço compacto.  Considere o espaço $R([0,\infty),X)$  dos raios geodésicos parametrizados por comprimento de arco em $X$, visto como subespaço de $C([0,\infty),X)$, munido da topologia compacto-aberta ou, equivalentemente, a topologia da convergência uniforme sobre compactos. Podemos portanto trabalhar com a topologia quociente, que será denotada por $\tau$, em $\partial X$. 

\begin{lemma}
\begin{enumerate}[(1)]
\item A fronteira ideal $\partial X$, munida da topologia $\tau$, é compacta.
\item Para todos $\xi\neq \eta \in \partial X$, existe uma geodésica completa em $X$ que é assintótica a ambos $\xi$ e $\eta.$
\end{enumerate}
\end{lemma}

\begin{proof}
\begin{enumerate}[(1)]
\item Pelo teorema de Arzelà--Ascoli, o espaço dos raios geodésicos partindo de $p\in X$ é compacto. Com isso, $\partial X$ será compacto quando munido da topologia quociente.
\item Dados $\xi \neq \eta \in \partial X$, considere raios geodésicos $\rho, \rho'$ tais que  $\rho(0) = \rho'(0) = p \in X$, $\rho(\infty) = \xi$ e $\rho'(\infty) = \eta$. 

Como $\xi \neq \eta$, obtemos, para cada $R>0$, que o conjunto 
$$K(R):=\{x \in X \mid d(x,\rho)\leq R,\ d(x,\rho')\leq R\}$$ 
é compacto. Considere as sequências $x_n = \rho(n)$ e $x'_n= \rho'(n)$ em $X$. Para cada $n\in\N$, o triângulo $\Delta(x_n,x'_n,p)$ é $\delta$-magro, e, portanto, $\gamma_n \cap K(\delta) \neq \emptyset$, onde $\gamma_n = \overline{x_nx'_n}$.

Novamente, usamos o teorema de Arzelà--Ascoli para concluir que a sequência de segmentos geodésicos $\gamma_n$ converge para uma geodésica completa $\gamma$ em $X$ que, por construção, está contida em $\mathcal{N}_{\delta}(\rho \cup \rho')$. A geodésica $\gamma$ é assintótica a $\xi$ e a $\eta$, como queríamos.
\end{enumerate}
\end{proof}

\begin{exercise}
    A fronteira de uma árvore $n$-valente, para $n \geq 3$, é homeomorfa ao conjunto de Cantor.
\end{exercise}

\subsection*{Topologia da sombra em $X \cup \partial X$} 

Podemos estender a topologia $\tau$ definida em $ \partial X $ a uma  topologia na união $\Bar{X} = X \cup \partial X$. A ideia é que $\Bar{X}$ seja uma compactificação natural de $X$ pela união com sua fronteira ideal (compare com o espaço de fins considerado na Seção~\ref{sec:esp de fins}).

Dados um ponto base $x\in X$ e um número $k\geq 3\delta$, podemos definir uma nova topologia $\tau_{x,k}$ em $\Bar{X}$, onde a base de vizinhanças de um ponto $a\in X$ consiste das bolas abertas $B(a,r)$, enquanto a base de vizinhanças de um ponto $ \xi = \rho(\infty)$, onde $\rho(0)=x$, é dada pelos conjuntos $\mathcal{U}_y(\xi) = \mathcal{U}_{k,x,y}(\xi)$, onde $y=\rho(t)$, com $t\in [0,\infty)$, definidos por:
\begin{eqnarray*}
\mathcal{U}_y(\xi)=\{z \in  \Bar{X} &\mid  \mbox{ para cada segmento ou raio geodésico }
\overline{xz}, \\
 &\overline{xz} \cap B(y, k) \neq 0\}.
\end{eqnarray*}

\begin{exercise}
Fixados um ponto base $x\in X$ e um número $k\geq 3\delta$, mostre que $\mathcal{B}_{x,k}:=\{\mathcal{U}_y(\xi) \mid \xi = \rho(\infty) \in \partial X,\ y = \rho(t),\ t\geq 0\} \cup \{B(a,r) \mid a\in X,\ r>0\}$ 
forma uma base para uma topologia em $\Bar{X}$, a qual denotaremos por $\tau_{x,k}$.
\end{exercise} 

\begin{definition}
A topologia $\tau_{x,k}$ em $\Bar{X}$ é chamada \textit{topologia da sombra}. \index{topologia da sombra}   
\end{definition}

\begin{figure}[H]
\label{fig:shadowtopology}
\centering
\begin{tikzpicture}[scale=1]
    \draw[->] (-2,0) -- (3.2,0) node [right] {$\xi$};
    \draw(2.18,1.75) arc (60:-60:2cm) node [below] {$\partial X$};	
    \draw (-2,0) -- (1.6,0.6) node [right] {$z$};
     \draw[opacity=0.7,line width=2pt, lightgray] (2.758,1.265) arc (38:-40:2cm);

    \draw[lightgray,] (-2,0) -- (2.73,1.245) ;
    \draw[lightgray] (-2,0) -- (2.725,-1.245) ;
    

    \node[circle,fill=black,inner sep=0pt,minimum size=3pt]  at (1.6,0.6) {};
    \node[circle,fill=black,inner sep=0pt,minimum size=3pt,label=left:{$x$}]  at (-2,0) {};
    \node[circle,fill=black,inner sep=0pt,minimum size=3pt,label=below:{$y$}]  at (0.4,0) {};
    \draw[fill=gray, opacity=0.2] (1,0) arc (0:360:0.6cm);
    \end{tikzpicture}
\caption{Topologia da sombra em  $\Bar{X}$. }
\end{figure}
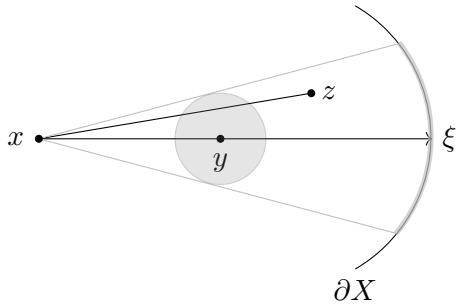

\begin{exercise}
Se $X=\mathbb{H}^n$, então $\partial X$ é homeomorfo a $\mathbb{S}^{n-1}.$
\end{exercise}

\begin{proposition}
Com respeito à topologia $\tau_{x,k}$ definida acima temos:
\begin{enumerate}[(1)]
    \item  $X$ é um aberto denso em $\Bar{X}$;
    \item $\Bar{X}$ é um espaço primeiro contável;
    \item $\Bar{X}$ é Hausdorff.
\end{enumerate}

\end{proposition}
\begin{proof}
\begin{enumerate}[(1)]
    \item  O fato de que $ X$ é aberto segue diretamente da definição da topologia. A densidade de $X$ segue do fato que, para cada  $\xi = \rho(\infty) \in \partial X$, a sequência $(\rho(n))_{n\in\N}$ converge para $\xi$.
    \item Pontos de $X$ claramente possuem base enumerável de vizinhanças. Para os pontos $\xi = \rho(\infty) \in \partial X$, os conjuntos $U_{\rho(n)}$ com $n \in N$ formam uma base  de vizinhanças de $\xi$, já que para todo $t > 0$, $U_{\rho(t_0)} \subset U_{\rho(t)}$ sempre que $t_0 \leq t + k +\delta$. Esse fato segue da $\delta$-hiperbolicidade de $X$.
    \item  Vamos verificar que qualquer dois pontos distintos em $\Bar{X}$ possuem vizinhanças disjuntas. O caso de pontos em $X$ é direto, e o caso de um ponto em $X$ e outro em $\partial X$ é deixado como  exercício ao leitor. Seja $\xi_1 = \rho_1(\infty)$ e $\xi_2 = \rho_2(\infty) \in \partial X$ dois pontos distintos. Como os raios geodésicos $\rho_1$ e $\rho_2$ não são assintóticos, devem existir pontos $y_i = \rho_i(t_i)$ tais que $\dist(y_i,\rho_j) > k + \delta$ para $ i\neq j \in {1, 2}$.
    Afirmamos que as vizinhanças básicas $U_{y_1} $ e $U_{y_2}$ são disjuntas. Caso contrário, supondo que exista $u \in U_{y_1}\cap U_{y_2}$, deveria existir $z \in \overline{xu} \cap B(y_2, k)$ e uma geodésica $\overline{xz}$ a qual intersecta $B(y_1, k)$. Como o triângulo
    $\Delta(x, y_2, z)$ é $\delta$-magro, segue que o ponto $z$ está a uma distância menor do que $\delta$ do segmento $\overline{xy_2}$. No entanto, isso contradiz a hipótese de que a  distância mínima de $y_1$ a $\rho_2$ é maior do que $k + \delta$. Trocando os papéis dos pontos $y_1$ e $ y_2$,  concluímos que  $U_{y_1} \cap U_{y_2} = \emptyset$.
\end{enumerate}
\end{proof}

Chamaremos de  $G_x([0,\infty),X)$ o conjunto das geodésicas finitas ou infinitas em $X$ que partem do ponto $x$ (aqui estamos considerando todas as geodésicas $\gamma$ definidas no intervalo $[0,\infty)$, pondo $\gamma(t)=y$, para todo $t\geq t_0$ se $\gamma$ for uma geodésica finita em $X$ ligando $x$ a $y$). Munido da topologia compacto-aberta, $G_x([0,\infty),X)$ é compacto e Hausdorff. Podemos definir um mapa natural 
$$\varepsilon : G_x([0,\infty),X) \to \Bar{X},$$ 
dado por $\varepsilon(\overline{xy}) =y$ com $y \in \Bar{X}$.

\begin{exercise}
    Mostre que o mapa $\varepsilon$ definido acima é contínuo, se $G_x([0,\infty),X)$ está munido da topologia compacto-aberta e $\Bar{X}$ está munido da topologia $\tau_{x,k}$ para algum $k\geq 3\delta$.
\end{exercise}

Da continuidade da função $\varepsilon$, segue que $(\Bar{X}, \tau_{x,k})$ é também compacto e, como $X$ é um aberto denso em $\Bar{X}$, concluímos que com essa topologia, $\Bar{X}$ é uma compactificação de $X$.

É possível mostrar que a topologia da sombra $\tau_{x,k}$ em $\Bar{X}$ não depende do ponto base $x \in X$, nem do valor de $k\geq 3\delta$ e também é independente da escolha do raio geodésico $\rho$ usado na definição da base de vizinhanças de um ponto $\xi \in \partial X$. Além disso, quando olhamos a topologia induzida por $\tau_{x,k}$ na fronteira $\partial X$, ela coincide com a topologia quociente $\tau$ discutida no início da seção:

\begin{proposition}
Dados $x\in X$ e $k \geq 3\delta$, as topologias $\tau$ e $\tau_{x,k}$ coincidem em $\partial X$.
\end{proposition}

\begin{proof}
Tome $\xi \in \partial X$ e fixe um raio geodésico $\rho$ tal que $\rho(0)=x$ e  $\xi = \rho(\infty)$. 
Seja $\{\rho_j\}$ uma sequência de raios geodésicos infinitos com $\rho_j(0)=x$ e tal que $ \rho_j (\infty) \notin \mathcal{U}_{y}(\xi)$, para algum $y = \rho(n)$.
Se $\gamma = \lim_{j\to \infty} (\rho_j)$, então $\gamma \notin \mathcal{U}_{y}(\xi)$ e assim, pelo Lema \ref{lemma:raiosassintóticospróximos}, segue que $\gamma(\infty) \neq \rho(\infty)$.

Reciprocamente, se para qualquer $y = \rho(n)$ e todo $j$ suficientemente grande tem-se  $\rho_j \in \mathcal{U}_{y}(\xi)$, então a sequência $(\rho_j)$ converge a um raio $\rho' \in \mathcal{U}_{k,n,y} $, para todos $k,n, y = \rho(n)$, e portanto $\gamma(\infty) = \rho(\infty)$.
\end{proof}

A topologia definida acima funciona bem em espaços métricos geodésicos. Gromov estendeu esta ideia para espaços métricos próprios e $\delta$-hiperbólicos arbitrários, a partir da noção de fronteira sequencial.
Em \cite[Capítulo 7]{ghys2013groupes}, Ghys e de la Harpe mostram que essa topologia na fronteira de $X$ é metrizável, e portanto $(\partial X, \tau)$ pode também ser pensado do ponto de vista de espaço métrico.

O teorema a seguir retrata a importância de fronteiras ideais como um invariante por quasi-isometrias no contexto de espaços hiperbólicos. 

\begin{thm}
Sejam $X, X'$ espaços geodésicos, próprios e hiperbólicos. Se $f: X \to X'$ é um mergulho quasi-isométrico, então o mapa $\rho(\infty)\to (f\circ \rho )(\infty)$ define um mergulho topológico $f_{\partial}: \partial_{\infty}X\to  \partial_{\infty}X'$. Se $f$ for uma quasi-isometria, então $f_{\partial}$ será um homeomorfismo.
\end{thm}

\begin{proof}
    Suponhamos que $f: X \to X'$ é um $(\lambda, \varepsilon)$-mergulho quasi-isométrico. Dado $p\in X$, denote por $p'$ o ponto $f(p) \in X'$.

    Se $\rho_1, \rho_2:[0,\infty) \to X$ são raios geodésicos com $\rho_i(0)=p$, então cada $f\circ \rho_i$ é uma $(\lambda, \varepsilon)$-quasi-geodésica com $f\circ \rho_i(0) = p'$. Esses raios quasi-geodésicos são assintóticos se, e somente se, $\rho_1$ e $\rho_2$ são assintóticos. Assim, $f_{\partial}$ está bem-definida e é injetiva.
Usando a definição da topologia em $\partial_{\infty}X$, podemos também verificar que $f_{\partial}$ é contínua.

Por fim, se  $f: X \to X'$ é uma quasi-isometria, então ela possui um quasi-inverso $f':X'\to X$, e como $d(f'f(x), x)$ é limitada para $x\in X$, segue que $(f'f)_{\partial} = f'_{\partial}f_{\partial}$ é a função identidade em $\partial_{\infty}X$. Portanto, $f_{\partial}$ é um homeomorfismo.    
\end{proof}

\begin{exercise}
    Verifique que o mapa $f_{\partial}$ definido acima é contínuo.
\end{exercise}

\begin{corollary}Os espaços hiperbólicos reais $\mathbb{H}^n$ e $\mathbb{H}^m$ são quasi-isométricos se, e somente se, $m=n$.
\end{corollary}

\subsection*{Topologia do cone}
Ainda em $\Bar{X} = X \cup \partial X$, definiremos uma topologia cuja restrição a $X$ coincida com a topologia métrica original. Essa construção não é equivalente à topologia $\tau_{x,k}$ vista na seção anterior: com poucas exceções, $\partial X$ em geral não é compacto na topologia do cone.

\begin{definition}\label{conetop}
Dizemos que uma sequência $x_n \in X$ converge a $\xi = \rho(\infty) \in \partial_{\infty}X$ na \textit{topologia do cone}\index{topologia do cone} se existir uma constante $C>0$ tal que $x_n \in \mathcal{N}_C(\rho)$ e a sequência de segmentos geodésicos $\overline{x_1x_n}$ converge para um raio assintótico a $\rho.$
\end{definition}

\begin{exercise}
    Se uma sequência $(x_i)$ converge para $\xi \in \partial X$ na topologia do cone, então ela também converge para $\xi$ na topologia $\tau_{x,k}$ em $\Bar{X}$.
\end{exercise}

\begin{example}
Considere $X = \mathbb{H}^n$ no modelo do semiespaço superior. Tome $\xi =0 \in \R^{n-1}$ e $\ell$ dada pela geodésica vertical a partir da origem. Então uma sequência $x_n \in X$ converge a $\xi$ na topologia do cone se e somente se todos os pontos $x_n$ pertencem ao cone euclidiano com eixo $\ell$ e a distância euclidiana de $x_n$ a $0$ tende a zero.
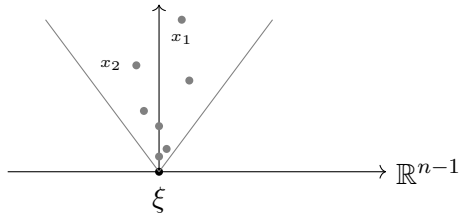
\begin{figure}[!ht]
\label{fig:conetopology}
\centering
\begin{tikzpicture}[scale=1]
		\draw[->] (-2,0) -- (3,0) node [right] {$\R^{n-1}$};
		\draw[->] (0,0) -- (0,2.2);
         \node[circle,fill=black,inner sep=0pt,minimum size=3pt,label=below:{$\xi$}] at (0,0) {};
         \node[circle,fill=gray,inner sep=0pt,minimum size=3pt,label=below:{\tiny{$x_1$}}]  at (0.3,2) {};
         \node[circle,fill=gray,inner sep=0pt,minimum size=3pt,label=left:{\tiny{$x_2$}}]  at (-0.3,1.4) {};
         \node[circle,fill=gray,inner sep=0pt,minimum size=3pt]  at (-0.2,0.8) {};
         \node[circle,fill=gray,inner sep=0pt,minimum size=3pt]  at (0.4,1.2) {};
         \node[circle,fill=gray,inner sep=0pt,minimum size=3pt]  at (0,0.6) {};
         \node[circle,fill=gray,inner sep=0pt,minimum size=3pt]  at (0.1,0.3) {};
         \node[circle,fill=gray,inner sep=0pt,minimum size=3pt]  at (0,0.2) {};
		\draw[gray](0,0) -- (-1.5,2);
		\draw[gray](0,0) -- (1.5,2);
  
\end{tikzpicture}
\caption{Topologia do cone em  $\Bar{\mathbb{H}}^n$. }
\end{figure}
\end{example}

\begin{exercise}
    Suponha que uma sequência $(x_i)$ converge a um ponto $\xi \in \mathbb{H}^n$ ao longo de uma horoesfera centrada em $\xi$. Mostre que a sequência $(x_i)$ não contém subsequências  convergentes na topologia do  cone em  $\Bar{\mathbb{H}}^n$.
\end{exercise}

\begin{exercise}
    Mostre que uma sequência $(x_n)$ em $X$ converge a $\xi \in \partial X$ na topologia de cone se, e somente se, dado $p\in X$ qualquer, $(x_i,x_j)_p\to \infty$, quando $i,j \to \infty$.
\end{exercise}

\subsection{Estendendo o produto de Gromov a \texorpdfstring{$\Bar{X}$}{fronteira}} 

\begin{definition}
Sejam $X$ um espaço $\delta$-hiperbólico e $p \in X$ um  ponto base. Estendemos o produto de Gromov para $\Bar{X} = X\cup \partial X$ por: 
$$(x,y)_p=\sup \liminf_{i,j\to \infty} (x_i,y_j)_p$$
onde o supremo é tomado sobre todas as sequências $(x_n), (y_n)$ em $X$ tais que $\displaystyle\lim_{n\to \infty} x_n=x$ e $\displaystyle\lim_{n\to \infty} y_n =y$. 
\end{definition}

O exemplo a seguir ilustra a necessidade de considerar o limite inferior e o  supremo, em vez de apenas tomar o limite na definição acima.

\begin{example}
Seja $X=\cay(\Z\times \Z_2, S)$ o grafo de Cayley de $\Z\times \Z_2=\langle a,b \mid [a,b], b^2\rangle.$ Note que $X$ é hiperbólico e que $\partial X$ consiste de dois pontos, $\partial X =\{a_{-},a_{+}\}.$ 
Considere as sequências $x_n=ba^{-n}$, $y_n=a^n$ e $z_n=ba^n$ em $X$. Defina uma nova sequência $w_n$, que é  igual a $y_n$ se $n$ é par e igual a $z_n$  se $n$ é ímpar.

Para quaisquer índices $i,j$, temos: $$(x_i,y_j)_1=(1+i)+j-(1+i+j)=0,$$  $$2(x_i,z_j)_1=(1+i)+(1+j)-(i+j)=2.$$ 
Assim, $0=\displaystyle \liminf_{i,j\to \infty}(x_i,y_j)_1$ não é igual a $\displaystyle\liminf_{i,j\to \infty}(x_i,z_j)_1 =1$. Em particular, $\displaystyle \lim_{i,j\to \infty} (x_i,w_j)_1$ não existe.
\end{example}

\begin{remark}
\label{rem:gromovprodbordo}
Seja $X$ um espaço $\delta$-hiperbólico e fixe $p \in X$. 
\begin{enumerate}[(1)]
\item O mapa $(\,,\,)_p$ é contínuo em $X\times X$, pela continuidade da distância. No entanto, pelo exemplo anterior, vemos que ele não é contínuo em $\Bar{X} \times \Bar{X}$. De fato, para $x_n \to a_-, y_n \to a_+$, temos $(x_i,y_j)_1 = 0$ e $(a_-,a_+)_1 \geq 1$.
\item $(x,y)_p=\infty$ se, e somente se, $x=y \in \partial X$.
\item Para quaisquer $x,y \in \Bar{X}$, existem sequências $(x_n)$ e $(y_n)$ em $X$ tais que 
$$\lim_{n\to \infty} x_n =x,\ \lim_{n\to \infty} y_n = y\text{ e }(x,y)_p =\lim_{n\to \infty}(x_n,y_n)_p.$$ 
Se algum dos pontos está em $X$, então podemos tomar a sequência constante. Já no caso em que os pontos estão em  $\partial X$, podemos construí-la da seguinte forma: para cada $N\in\N$, existem sequências $(x_n), (y_n)$ em $X$ tais que $$(x,y)_p \geq \displaystyle \liminf_{i,j \to \infty} (x_i,y_j)_p \geq (x,y)_p -\frac{1}{N}.$$
Logo, existem $r,s \in \N$ tais que 
$$(x,y)_p +\frac{1}{N} \geq (x_r,y_s)_p \geq (x,y)_p -\frac{1}{N}.$$
Escrevendo $\tilde{x}_N =x_r, \tilde{y}_N =y_s$, obtemos $\tilde{x}_N \to x$, $\tilde{y}_N \to y$ e $(\tilde{x}_N,\tilde{y}_N)_p \to (x,y)_p$, quando $N\to\infty$.

\item Para $x,y,z \in \Bar{X}$, temos $(x,y)_p\geq \min\{(x,z)_p,(z,y)_p\}-2\delta$.

\item Para todos os pontos  $x,y \in \partial X$ e sequências $(x_n)$ e $(y_n)$ em $X$ com $\displaystyle\lim_{n\to \infty} x_n =x$, $\displaystyle\lim_{n\to \infty} y_n = y,$ temos $$(x,y)_p-2\delta\leq \displaystyle\liminf_{i,j\to \infty}(x_i,y_j)_p \leq (x,y)_p.$$

\item Se $X$ é próprio e geodésico, e se $(\xi_n)$ é uma sequência em $\partial X$, então $\xi_n\to\xi$ na topologia $\tau_{x,k}$ se, e somente se, $(\xi_n,\xi)_p\to \infty.$  
\end{enumerate}
\end{remark}

\begin{exercise}
    Prove os itens (2), (4), (5) e (6) da Observação~\ref{rem:gromovprodbordo}.
\end{exercise}

\begin{exercise} \label{exerc:propelem9}
    Suponha que $(x_n)$ e $(y_n)$ são sequências em $X$ como acima, onde  $(x_n)$ converge para $\infty$  e $(x_n , y_n)_p \leq K $ para todo $n\in \N$ e alguma constante $K \in \R$. Então existe $N \in \N$ tal que, para  $m,n \geq N $, temos $(x_n, y_m)_p \leq K + \delta$.

\noindent
\textit{Dica}: use que
$K \geq (x_n , y_n)_p \geq \min\{(x_n , y_m)_p,(y_n , y_m)_p\} - \delta$ e que $(y_n , y_m)_p\to \infty$.
\end{exercise}

\begin{exercise}
\label{exerc:gromovbordo}
    Seja $X$ um espaço métrico próprio, geodésico e  \mbox{$\delta$-hiperbó\-li\-co}, com $\delta>0$. Mostre que existe uma constante $k=k(\delta)$ tal que $$|d(p,\im(\gamma))-(\xi,\xi')_p|\leq k,$$ para quaisquer $p \in X$,  $\xi, \xi' \in \partial X$ e para cada  geodésica $\gamma$ em $X$ com $\gamma(\infty)= \xi'$ e $\gamma(-\infty)=\xi$.
\end{exercise}

\subsection{Métrica visual em espaços hiperbólicos}
Nesta subseção, exibiremos uma métrica explícita sobre o bordo do espaço hiperbólico $\mathbb{H}^n$, $n\geq 2$. Se fixamos $p \in \mathbb{H}^n$, então a função 
$$(\xi,\xi') \mapsto \sphericalangle_p(\xi,\xi'),$$
dada pelo ângulo entre as duas geodésicas em  $\mathbb{H}^n$ que ligam $p$ a $\xi$ e $\xi'$, respectivamente, define uma métrica em $\partial \mathbb{H}^n$, tornando-o isométrico a $\mathbb{S}^{n-1}$. Como $\sphericalangle_p$ descreve a geometria de $\partial \mathbb{H}^n$ vista de $p$,  ela é dita uma métrica visual em $\partial\mathbb{H}^n$.

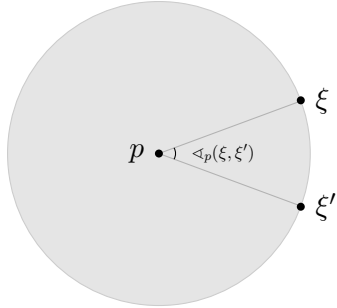
\begin{figure}[!ht]
\label{fig:visualmetric}
\centering
\begin{tikzpicture}[scale=0.5]

    \draw[lightgray,] (-2,0) -- (1.8,1.4) ;
    \draw[lightgray] (-2,0) -- (1.8,-1.4) ;
    
     \draw[fill=gray, opacity=0.2] (2,0) arc (0:360:4cm);

     \draw (-1.6,0.15) arc (30:-30:0.3);

     \node[scale=0.6] at (-0.3,0){$\sphericalangle_p(\xi,\xi')$};

\node[circle,fill=black,inner sep=0pt,minimum size=3pt,label=left:{$p$}]  at (-2,0) {};

\node[circle,fill=black,inner sep=0pt,minimum size=3pt,label=right:{$\xi$}]  at (1.75,1.4) {};

\node[circle,fill=black,inner sep=0pt,minimum size=3pt,label=right:{$\xi'$}]  at (1.75,-1.4) {};
\end{tikzpicture}
\caption{Métrica visual em $\partial \mathbb{H}^n$. }
\end{figure}

É possível mostrar que  $\sphericalangle_p$ satisfaz a seguinte relação:

\begin{equation}
    \label{anglep}
    k_1 e^{-(\xi,\xi')_p} \leq \sphericalangle_p(\xi,\xi')\leq  k_2 e^{-(\xi,\xi')_p},
\end{equation}
para algumas constantes positivas $k_1,k_2$ que não dependem de $p$.

\begin{definition}
    Dado um espaço $\delta$-hiperbólico $X$,  qualquer métrica em $\partial X$ que  satisfaça \eqref{anglep} será chamada de \textit{métrica visual} em $\partial X$.\index{métrica visual} 
\end{definition}

Em espaços $\delta$-hiperbólicos, é sempre possível definir 
uma métrica visual na fronteira $\partial X$, de modo que a topologia induzida por essa métrica em $\partial X$ coincida com a topologia $\tau$. Para mais detalhes sobre métricas visuais, recomendamos a leitura de \cite[Seção~III.H.3]{bridson2013metric}.

\chapter{Grupos hiperbólicos}
\label{cap8}
O estudo de grupos hiperbólicos foi introduzido e desenvolvido por Gromov~\cite{gromov1987hyperbolic}. A inspiração veio de várias áreas da matemática, em particular, da geometria hiperbólica, da topologia de baixa dimensão e da teoria combinatória de grupos. Grupos hiperbólicos fornecem exemplos da aplicabilidade de métodos geométricos em teoria de grupos.

\section{Definição e propriedades básicas}

\begin{definition}
\label{def:grupohiperbolico}
Um grupo finitamente gerado $G$ é dito \textit{hiperbólico}  \index{grupo hiperbólico} se seu grafo de Cayley, com respeito a algum conjunto de geradores $S$, é Rips-hiperbólico.
\end{definition}

Todo grupo finito é trivialmente hiperbólico, já que seu grafo de Cayley com respeito a um conjunto de geradores qualquer terá diâmetro finito. Alguns exemplos menos triviais são os seguintes:

\begin{example}
Em uma árvore, qualquer triângulo é  $0$-magro. Portanto, todo grupo livre é hiperbólico.
\end{example}

\begin{example}
Grupos discretos e cocompactos de isometrias de $\mathbb{H}^2$ são hiperbólicos, pelo teorema de Milnor--Schwarz.
\end{example}




Como a mudança de conjuntos geradores não afeta o tipo de quasi-isometria do grafo de Cayley e a hiperbolicidade de Rips é invariante sob quasi-isometrias (veja Corolário~\ref{cor:ripsQI}), concluímos que um grupo $G$ é hiperbólico se e somente se todos os seus grafos de Cayley forem hiperbólicos. Assim, a hiperbolicidade de um grupo não depende da escolha do conjunto finito de geradores. Além disso, se os grupos $G_1$ e $G_2$ forem quasi-isométricos, então $G_1$ é hiperbólico se e somente se $G_2$ for hiperbólico. Em particular, se $G_1$, $G_2$ forem virtualmente isomorfos, então $G_1$ é hiperbólico se e somente se $G_2$ for hiperbólico. Por exemplo, todos os grupos virtualmente livres são hiperbólicos.

Em vista do Teorema de Milnor--Schwarz, temos que se um grupo $G$ age geometricamente em um espaço métrico Rips-hiperbólico, então $G$ também é hiperbólico.

Nem todo subgrupo de um grupo hiperbólico é hiperbólico. Por exemplo, a construção de Rips na Seção~11.18 de \cite{kd}  mostra que existem  subgrupos finitamente gerados de  grupos hiperbólicos que não são hiperbólicos. Mais ainda, em \cite{brady1999branched}, N.~Brady fornece um exemplo de um subgrupo finitamente apresentado de um grupo hiperbólico que não é hiperbólico. Tendo isso em mente, podemos usar o Teorema~\ref{thm:quasiconv} para provar o seguinte resultado, que  mostra que quasi-convexidade garante a passagem da propriedade de hiperbolicidade a subgrupos.

\begin{thm}
    Subgrupos quasi-convexos de grupos hiperbólicos são hiperbólicos.
\end{thm} 
\begin{proof} Seja $G$ um grupo hiperbólico com conjunto finito de geradores $S$.
Suponha que $H$ seja um subgrupo $R$-quasi-convexo de $G$. 
Pelo Teorema~\ref{thm:quasiconv}, $H$ é gerado por um conjunto finito $S'$ e a inclusão $i : H \to G$ é um mapa bi-Lipschitz. Da prova do Teorema~\ref{thm:quasiconv}, para quaisquer $h_1, h_2 \in H$ obtemos:
$$ d_{S'}(h_1, h_2) \leq d_S(h_1, h_2) \leq (2R+1)d_{S'}(h_1, h_2).$$
Sendo assim as métricas $d_{S'}$ e $d_S|_H$ são bi-Lipschitz equivalentes, e
como a hiperbolicidade é invariante por quasi-isometrias,
concluímos que $H$ é hiperbólico.
\end{proof}

\section{Desigualdade isoperimétrica linear e o algoritmo de Dehn}\label{sec:8-isoperim}




Se $G =\langle S\mid R\rangle$ é um grupo finitamente apresentado e $w$ é uma palavra no grupo livre $F(S)$, então toda vez que 
$w = 1$ em $G$, podemos escrever 
$$w = \prod^{n}_{i=1}u_ir^{\pm1}_iu^{-1}_i,$$
para alguns $u_i \in  F(S)$ e $r_i \in R$. 
Dada uma palavra $w$ em $S$ tal que $w=1$ em
$G$. Definimos a \textit{área combinatória} \index{area combinatória@área combinatória} de $w$ por
$$A(w) = \min\{n \in \N \mid w =\prod^{n}_{i=1}u_ir^{\pm1}_iu^{-1}_i\}.$$

Para entendermos melhor essa definição, vejamos um exemplo. Queremos pensar na área combinatória como o número de vezes que precisamos usar informações sobre os relatores para chegar na identidade quando testamos se uma palavra é trivial.

\begin{example}
Seja $\Gamma = \mathrm{BS}(3, 2) := \langle x, y \mid x^{-1}y^3xy^{-2}\rangle$ um grupo de  Baumslag--Solitar.
Tome $$w := xy^{-2}xy^2x^{-1}y^{-1}x^{-1}
.$$ Podemos reescrever $w$ como
$$w = (xy)(y^{-3}xy^2x^{-1})(xy)^{-1}
.$$
Removemos as letras do meio aplicando o relator $r^{-1} =y^2x^{-1}y^{-3}x$ uma vez, mesmo que essa palavra não seja exatamente nosso relator. Reescremos $r^{-1} = (y^2x^{-1})(y^{-3}xy^2x^{-1})(xy^{-2})$,
onde definimos $x_1 := y^2x^{-1}$
para corresponder com nossa definição de área. Assim, obtemos 
$$w = (xy)(y^{-3}xy^2x^{-1})(xy)^{-1} =_{F(S)} (xy)x^{-1}_1r
^{-1}x_1(xy)^{-1} = 1,$$
e portanto a área de $w$ é igual a 1. De outro modo, poderíamos repetidamente tentar reduzir subpalavras e usar o fato de que $A(ww') \leq A(w) + A(w')$.
\end{example}

\begin{definition}
\label{dehnfunction}
Dizemos que uma função $f : \N \to \N$ é uma \textit{função de Dehn} \index{função de Dehn} (ou função isoperimétrica) para $G=\langle S\mid R\rangle$ se, para todo $w \in F(S)$ com $w=1$ em $G$, tem-se $A(w)\leq f(|w|_S)$. Dizemos que $G$ \textit{satisfaz uma desigualdade 
isoperimétrica linear, quadrática, polinomial ou exponencial} se possui uma função de Dehn com a mesma propriedade.\index{desigualdade isoperimétrica}
\end{definition}

Sejam $G$ e $H$  grupos finitamente apresentados quasi-isométricos. Se alguma apresentação  finita de $G$ possui função isoperimétrica  $f(n)$, então qualquer apresentação finita de $H$ possui uma função isoperimétrica equivalente a $f(n)$, no seguinte sentido: dadas duas funções $f,g : \N \to \N$, escrevemos  $f \preceq g$ se existirem constantes $A,B,C \in \R$ tais que $f(n) \leq Ag(Bn) + Cn$. Definimos uma relação de equivalência no conjunto das funções de $\N$ para $\N$ pondo $f \sim g$ se $f \preceq g$ e $g \preceq f$.

Em particular, o fato acima vale quando $G = H$, isto é, quando se trata do mesmo grupo, dado possivelmente por duas apresentações finitas distintas.
Consequentemente, para um grupo  finitamente apresentado, o tipo de crescimento de sua função de Dehn, não depende da escolha de sua apresentação finita. Mais geralmente, se dois grupos finitamente apresentados são quasi-isométricos, então suas funções de Dehn são equivalentes.

\begin{exercise}
Mostre que a função de Dehn $f(n)$ da apresentação $\langle a \mid a^m \rangle$ do grupo cíclico de ordem $m$ é o maior inteiro menor ou igual a $n/m$.    
\end{exercise}

\begin{exercise}
Mostre que qualquer apresentação finita de um grupo finito $G$ tem função de Dehn $f(n) \le Cn$ com $C>0$.
\end{exercise}

\begin{example}
Agora vamos verificar que  $G = \Z^2$ com apresentação $\langle a,b \mid aba^{-1}b^{-1}\rangle$  satisfaz uma desigualdade isoperimétrica quadrática, mas não uma linear. 

Precisamos estabelecer um limite superior quadrático e um limite inferior quadrático para a área de palavras triviais.
\medskip

\noindent\textit{Limite Superior: $f(n) \le C n^2$.}

Seja $w$ qualquer palavra de comprimento $|w| = n$ tal que $w = 1$.

Defina $x$ como o número total de letras $a$ e $a^{-1}$ em $w$, e $y$ como o número total de letras $b$ e $b^{-1}$. Como o comprimento total é $n$, temos $x + y = n$.

Podemos transformar sistematicamente a palavra $w$ movendo todas as letras $a^{\pm 1}$ para o lado esquerdo da palavra e todas as letras $b^{\pm 1}$ para o lado direito usando as relações:
\[ ab = ba, \quad ab^{-1} = b^{-1}a, \quad a^{-1}b = ba^{-1}, \quad a^{-1}b^{-1} = b^{-1}a^{-1} \]

Cada vez que uma letra $a^{\pm 1}$ é trocada com uma letra $b^{\pm 1}$ adjacente, exatamente uma relação é aplicada. No pior cenário possível, cada letra $a^{\pm 1}$ deve cruzar todas as letras $b^{\pm 1}$. O número máximo de trocas necessárias para separar completamente as letras é $x \cdot y$.

Pela desigualdade MA-MG, o produto $x(n-x)$ é maximizado quando $x = y = n/2$:
\[ x \cdot y \le \left(\frac{x+y}{2}\right)^2 = \left(\frac{n}{2}\right)^2 = \frac{1}{4}n^2 \]
\medskip

\noindent\textit{Limite Inferior: $f(n) \ge C' n^2$.}

Para estabelecer o limite inferior, construímos uma família de palavras que requer um número quadrático de aplicações de relatores. Para qualquer inteiro $k \ge 1$, considere a palavra:
\[ w_k = a^k b^k a^{-k} b^{-k}.\]

O comprimento de $w_k$ é $n= |w_k| = 4k$. Logo, $k = n/4$.

Para reduzir $w_k$ à identidade usando a relação $ab = ba$ (e suas variantes para inversos), cada letra $b$ deve comutar por todas as letras $a^{-1}$ para se cancelar com as letras $b^{-1}$, ou equivalentemente, os blocos de $a$ e $b$ devem se cruzar um pelo outro.

Como existem $k$ letras $a$ e $k$ letras $b$, mover todas as $k$ letras $b$ pelas $k$ letras $a$ requer exatamente $k^2$ trocas adjacentes. Cada troca corresponde precisamente a uma aplicação do relator $[a,b]$.

Portanto, a área de $w_k$ é exatamente $k^2$. Substituindo $k = n/4$:
\[ A(w_k) = \left(\frac{n}{4}\right)^2 = \frac{1}{16}n^2.\]

Como existe uma palavra de comprimento $n$ com área de $\frac{1}{16}n^2$, a função de Dehn deve crescer pelo menos quadraticamente:
\[ f(n) \ge \frac{1}{16}n^2.\]

\end{example}

\begin{exercise}
Mostre que, para qualquer apresentação finita de qualquer grupo abeliano finitamente gerado, a função de Dehn é no máximo $Cn^2$.
\end{exercise}

A seguir, mostraremos que desigualdades isoperimétricas lineares caracterizam a hiperbolicidade de grupos. Este resultado se deve a Gromov.
Como corolários, mostraremos que tais grupos
são finitamente apresentados e, na seção seguinte, que o problema da palavra é solúvel para grupos hiperbólicos.

Fixemos agora um grupo hiperbólico $G$,  o seu grafo de Cayley $\Gamma = \cay(G,S)$, com respeito a um conjunto finito de geradores $S$, de modo que todos os triângulos geodésicos sejam $\delta$-magros. Podemos assumir que $\delta \geq 2$ é um número inteiro. 

\begin{thm}[A hiperbolicidade implica numa desigualdade isoperimétrica linear] \label{prop:hiperbolicimpliesdil}
Nas condições estabelecidas acima, para cada palavra $w = 1$ em $G$, vale
$$A(w)\leq K|w|_S.$$
\end{thm}

\begin{proof}

    Seja $K = \max\{A(w) \mid |w|_S \leq 10\delta\}$. Mostraremos por indução em $|w|_S$ que $A(w) \leq K|w|_S$ para toda palavra $w \in F(S)$ com $w = 1$ em $G$. Se $|w|_S \leq 10\delta$ isso é óbvio. Suponha que a afirmação é verdadeira para $|w|_S \leq n$, com $n > 10\delta$. Seja $w$ uma palavra com $w=1$ em $G$, tal que  $|w|_S = n + 1$. Denotaremos também por $w = w(t)$ o laço em $\Gamma$, parametrizado por comprimento de arco, que representa a palavra $w$. Podemos considerar três casos:

\textbf{Caso 1:} Suponha que, para todos os vértices $w(i)$ de $w$, nós temos $d(w(i), 1) < 5\delta$. Podemos escolher um caminho mais curto $p$ ligando $1$ a $w(5\delta)$. 
Então $w = w_1p^{-1}pw_2$ e, pela hipótese de indução, obtemos
$$A(w) \leq A(w_1p^{-1}) + A(pw_2) \leq K + Kn = K(n + 1) = K|w|_S.$$
\begin{figure}[H]
     \centering
 \begin{tikzpicture}
\draw [gray,thick,opacity=0.7] plot [smooth cycle] coordinates {(0,0) (1,1) (3,1) (3,0) (2,-0.5)};
\draw [thick] (3.3,0)node{$w_2$};
\draw [thick] (0.25,0.8)node{$w_1$};
 \draw (0,0)node{$\bullet$};
 \draw (-0.2,0)node{$1$};
 \draw (3,1)node{$\bullet$};
\draw (3.6,1)node{\small{$w(5\delta)$}};
\draw[dotted] (0,0)--(3,1);
\draw (1.8,0.3)node{$p$};
\end{tikzpicture}
         \caption{Caso 1}
     \label{fig:caso1}   
\end{figure}

\textbf{Caso 2:} Caso contrário, existe algum vértice $w(i)$ em $w$ com $d(w(i), 1) \geq 5\delta$. Seja $w(t)$ o vértice de $w$ mais distante de $1$. 

Se $d(w(t), w(t - 5\delta)) < 5\delta$ ou $d(w(t),(w(t + 5\delta)) < 5\delta$ podemos concluir o resultado de maneira análoga.
De fato, se $\gamma$ é uma geodésica de $w(t - 5\delta)$ até  $w(t)$ (ou de $w(t)$ até $w(t + 5\delta)$), considere os laços $w' = w|_{[t-5\delta,t]}\gamma$ (ou $w|_{[t,t+5\delta]}\gamma$) e $w''= w|_{[0,t-5\delta]}\gamma w|_{[t,n]}$ (ou
$w|_{[1,t]}\gamma w|_{[t+5\delta,n]}$). Temos assim 
$A(w) = A(w') + A(w'')$, onde $A(w'') \leq Kn$ por hipótese de indução e  $A(w') \leq K$ pela definição de $K$.

\begin{figure}[H]
     \centering
 \begin{tikzpicture}
\draw [gray,thick,opacity=0.7] plot [smooth cycle] coordinates {(0,0) (1,1) (2,1) (3,1) (4,0) (3,-1) (1.5,-1)};
\draw  (2,1)node{$\bullet$};
\draw (2.1,1.3)node{\tiny{$w(t-5\delta)$}};

 \draw (0,0)node{$\bullet$};
 \draw (-0.2,0)node{$1$};
 \draw (4,0)node{$\bullet$};
 \draw (4.5,0)node{\tiny{$w(t)$}};
 \draw[dashed, thick] (2,1) to [bend right=30] (4,0);
\draw (2.55,0.2)node{\tiny{$\gamma$}};
\draw [dotted] plot [smooth cycle] coordinates {(0.1,0) (1,0.9) (1.9,0.9) (2.5,0) (3.8,-0.2) (3,-0.9) (1.5,-0.9)};
\draw (1.5,-0.7)node{\tiny{$w''$}};
\draw [dotted] plot [smooth cycle] coordinates {(2.1,0.9) (2.9,0.9) (3.8,0.1) (2.8,0.4) };
\draw (2.7,0.7)node{\tiny{$w'$}};

\end{tikzpicture}
         \caption{Caso 2}
     \label{fig:caso2}   
\end{figure}

 \textbf{Caso 3:} Se nenhum dos casos acima ocorre, considere os triângulos geodésicos $
 \Delta(1, w(t -5\delta), w(t))$ e  $\Delta(1, w(t), w(t + 5\delta))$, onde novamente $w(t)$ o vértice de $w$ mais distante de $1$. Sabendo que eles são $\delta$-magros, notamos que os pontos $u_{\pm} := w(t \pm (\delta+1))$ estão a uma distância $\leq \delta$ de  $\overline{1  w(t)}$  ou de algum dos segmentos $\overline{1  w(t -5\delta)}$, $\overline{1  w(t +5\delta)}$.

 Se, por exemplo, $d(u_+, w)\leq \delta$, com $w \in \overline{1  w(t +5\delta)}$, obtemos 
\begin{equation}
\begin{split}
   d(1,w(t)) &\leq d(1,w)+d(w,u_+)+d(u_+,w(t)) \\
   &\leq r + \delta + (\delta +1) = r + 2\delta +1,
 \end{split} 
\end{equation}
\begin{equation}
   d(1,w(t+5\delta)) = r+s \geq r + 3\delta-1 > r + 2\delta +1,
\end{equation}
 onde $r = d(1,w)$ e $s=d(w,w(t+5\delta)) \geq  d(w(t+5\delta),u_{-}) - d(u_{-},w) \geq 4\delta-1 -\delta = 3\delta-1$.

Com isso, $d(1,w(t+5\delta))  > d(1,w(t))$, uma contradição com a escolha de $t$.

Assim, devemos ter algum ponto $w \in \overline{1  w(t)}$ com $d(u_{\pm}, w) \leq \delta$ (note que podemos pegar o mesmo ponto $w$ de modo que a desigualdade acima valha para $u_{\pm}$). Logo, $d(u_+, u_{-}) \leq 2\delta$, e portanto
\begin{equation*}
    \ell(w|_{[t-(\delta+1),t+(\delta+1)]} \cup \overline{u_+u_{-}}) \leq 2\delta + 2 + 2\delta \leq 5 \delta.
\end{equation*}
e, escrevendo $w' = w|_{[t-(\delta+1),t+(\delta+1)]} \cup \overline{u_{-}u_+}$, $w'' =  \overline{1u_{-}} \cup  \overline{u_+u_{-}} \cup  \overline{u_+1}$, obtemos $A(w) = A(w') +A(w'')$,  onde novamente $A(w'') \leq Kn$, por hipótese de indução, e  $A(w') \leq K$, pela definição de $K$.
 
\begin{figure}[H]
     \centering
 \begin{tikzpicture}
\draw [gray,thick,opacity=0.7] plot [smooth cycle] coordinates {(-1,0) (0,1) (2,1) (3,1) (4,0) (3,-1) (2,-1) (0,-1)};
\draw  (1.4,1.05)node{$\bullet$};
\draw (1.1,1.3)node{\tiny{$w(t-5\delta)$}};
\draw  (1.4,-1.05)node{$\bullet$};
\draw (1.1,-1.3)node{\tiny{$w(t+5\delta)$}};

\draw  (3,1)node{$\bullet$};
\draw (3,1.3)node{\tiny{$u_-$}};
\draw  (3,-1)node{$\bullet$};
\draw (3,-1.3)node{\tiny{$u_+$}};
\draw (-1,0)node{$\bullet$};
\draw (-1.2,0)node{$1$};
\draw (4,0)node{$\bullet$};
\draw (4.5,0)node{\tiny{$w(t)$}};
\draw[thick,dashed] (3,1) to [bend right=10] (3,-1);

\draw [dotted] plot [smooth cycle] coordinates {(-0.9,0) (0,0.9) (2,0.9) (2.8,0.9) (2.8,0) (2.8,-0.9) (2,-0.9)  (0,-0.9)};
\draw (0,0.7)node{$w''$};
\draw [dotted] plot [smooth cycle] coordinates {(3.1,0.8) (3.8,0) (3.1,-0.8) (3,0.1) };
\draw (3.2,0.5)node{$w'$};

\end{tikzpicture}
         \caption{Caso 3}
     \label{fig:caso3}   
\end{figure}

\end{proof}


\begin{corollary}
\label{cor:hiperbolicisfinpres}
Todo grupo hiperbólico é finitamente apresentado.
\end{corollary}
\begin{proof}
Se $G$ é um  grupo hiperbólico  gerado por um conjunto finito $S$, então $\langle S\mid R\rangle$ é 
 uma apresentação finita para $G$, onde $R
 = \{w \in F(S) \mid |w|_S \leq 10\delta, w = 1\}$.
\end{proof}

Daremos agora uma interpretação geométrica para a área combinatória de uma palavra. 

\begin{definition} Seja $D$ qualquer 2-complexo simplicial  finito, planar, orientado, conexo e simplesmente conexo. Diremos que $D$ é um \textit{diagrama} sobre um alfabeto $S$ se toda aresta $e$ de $D$ puder ser associada a um rótulo $\phi(e) \in S$ tal que $\phi(e^{-1}) = (\phi(e))^{-1}$, ou seja, a aresta com a orientação oposta está associada ao inverso do rótulo.
\end{definition}

Desse modo, um caminho  (simplicial) através das arestas de um diagrama é rotulado por uma palavra em $S$ e um caminho que não retrocede sobre si mesmo é rotulado por uma palavra reduzida em $S$ via $\phi$.

\begin{definition}
Um \textit{diagrama de van Kampen} \index{diagrama de van Kampen} sobre um grupo $G = \langle S\mid R\rangle$ é um  diagrama $D$ sobre $S$ tal que, para toda face $f$ de $D$, o rótulo da fronteira de $f$ é dado por alguma  relação $r^{\pm 1}$ com $r \in R$. A \textit{área} de tal diagrama é o número de suas faces.
\end{definition}

Por conveniência, chamamos a palavra que rotula a fronteira da componente ilimitada de $\mathbb{R}^2 \backslash D$ de fronteira de $D$.

\begin{example}
    Considere o grupo $\Z^2$ com apresentação $\Z^2 = \langle a,b\mid [a,b]\rangle$.
    Na figura da esquerda a seguir, temos um diagrama sobre $\{a,b\}$, o qual não é um diagrama de van Kampen, pois a única $2$-célula tem como fronteira a palavra $b^2a^{-1}b^{-2}a$, que não é da forma $r^{\pm 1}$, com $r \in R$.

    \begin{figure}[H]
\begin{center}
  \begin{tikzpicture}[scale=0.9]
 \draw (0,0)node{$\bullet$};
  \draw (1,1)node{$\bullet$};
  \draw (2,2)node{$\bullet$};
  \draw (1,0)node{$\bullet$};
  \draw (2,1)node{$\bullet$};
  \draw (3,2)node{$\bullet$};  
  
 \draw (0.5,-0.2)node{$a$};
 \draw (0.4,0.7)node{$b$};
 \draw (1.2,1.6)node{$b$};
 \draw (2.5,2.2)node{$a$};
 \draw (1.75,0.5)node{$b$};
 \draw (2.6,1.3)node{$b$}; 
  
\draw[->]  (0,0) --  (0.5,0.5);
\draw[->]  (0.5,0.5) --  (1.5,1.5);
\draw  (1.5,1.5) --  (2,2);
\draw[->]  (0,0) --  (0.5,0); 
\draw  (0.5,0) -- (1,0);
\draw[->]  (1,0) --  (1.5,0.5);
\draw[->]  (1.5,0.5) --  (2.5,1.5);
\draw  (2.5,1.5) --  (3,2);
\draw[->]  (2,2) --  (2.5,2); 
\draw  (2.5,2) -- (3,2);
 \end{tikzpicture}
 \hspace{1cm}
   \begin{tikzpicture}[scale=0.9]
 \draw (0,0)node{$\bullet$};
  \draw (1,1)node{$\bullet$};
  \draw (2,2)node{$\bullet$};
  \draw (1,0)node{$\bullet$};
  \draw (2,1)node{$\bullet$};
  \draw (3,2)node{$\bullet$};  
  
 \draw (0.5,-0.2)node{$a$};
 \draw (0.4,0.7)node{$b$};
 \draw (1.2,1.6)node{$b$};
 \draw (2.5,2.2)node{$a$};
 \draw (1.75,0.5)node{$b$};
 \draw (2.6,1.3)node{$b$}; 

\draw (1.3,0.8)node{$a$};
  
\draw[->]  (0,0) --  (0.5,0.5);
\draw[->]  (0.5,0.5) --  (1.5,1.5);
\draw  (1.5,1.5) --  (2,2);
\draw[->]  (0,0) --  (0.5,0); 
\draw  (0.5,0) -- (1,0);
\draw[->]  (1,0) --  (1.5,0.5);
\draw[->]  (1.5,0.5) --  (2.5,1.5);
\draw  (2.5,1.5) --  (3,2);
\draw[->]  (2,2) --  (2.5,2); 
\draw  (2.5,2) -- (3,2);

\draw[->]  (1,1) --  (1.5,1); 
\draw  (1.5,1) -- (2,1);
 \end{tikzpicture}
 \caption{Diagramas sobre $\{a,b\}$}
 \label{fig:diagramasZ2} 
\end{center}
\end{figure}

  Já a figura da direita representa um diagrama de van Kampen, pois todas as fronteiras de $2$-células representam relatores. Neste caso, a palavra da fronteira do diagrama é $b^2a^{-1}b^{-2}a$ e sua área  é 2.
\end{example}


Um importante resultado na teoria combinatória de grupos é a seguinte proposição,  também chamada de Lema de van Kampen\index{Lema de van Kampen}:

\begin{proposition}
Sejam $G = \langle S\mid R\rangle$  um grupo finitamente apresentado e $w$ uma palavra em $ F(S)$. Então $w = 1$ em  $G$ se, e somente se, existe um diagrama de van Kampen $D$ sobre $G$ com fronteira rotulada por $w$.
\end{proposition}

\begin{proof}
    Se $w=1$ em $G$, então $w$ pode ser escrita na forma $w= u_1r_1u_1^{-1} u_2r_2u_2^{-1} \ldots u_nr_nu_n^{-1}$, onde $r_i \in R$.
Vamos construir um diagrama de van Kampen com fronteira $w$ da seguinte forma:
 o diagrama consistirá de $n$ $2$-células, associadas aos relatores $r_i$ da escrita de $w$, cujas fronteiras estão subdivididas em $\ell_S(r_i)$ $1$-células, unidas com $n$ segmentos, anexados a cada uma das $2$-células em um de seus vértices e subdivididos em $\ell_S(u_i)$ arestas. Esses segmentos partem todos de um mesmo ponto base. 
\begin{figure}
\begin{center}
 \begin{tikzpicture}
    \draw[ 
        decoration={markings, mark=at position 0.125 with {\arrow{>}}},
        postaction={decorate}
        ]
        (3.5,0.5) circle (0.5);

    \draw[ 
        decoration={markings, mark=at position 0.5 with {\arrow{>}}},
        postaction={decorate}
        ]
        (0.5,0) -- (3,0.5);

    \fill (0.5,0) circle (0.05); 
    \draw (1.7,0)node{$u_2$};
    \draw (4.2,0.5)node{$r_2$};

    \draw[ 
        decoration={markings, mark=at position 0.125 with {\arrow{>}}},
        postaction={decorate}
        ]
        (3,2) circle (0.5);

    \draw[ 
        decoration={markings, mark=at position 0.5 with {\arrow{>}}},
        postaction={decorate}
        ]
        (0.5,0) -- (2.5,2);

    \draw (1.7,0.8)node{$u_1$};
    \draw (3.7,2)node{$r_1$};

    \draw[ 
        decoration={markings, mark=at position 0.125 with {\arrow{>}}},
        postaction={decorate}
        ]
        (3,-2) circle (0.5);

    \draw[ 
        decoration={markings, mark=at position 0.5 with {\arrow{>}}},
        postaction={decorate}
        ]
        (0.5,0) -- (2.5,-2);

    \draw (1.5,-1.25)node{$u_n$};
    \draw (3.8,-2)node{$r_n$};
    \draw (3.3,-0.6)node{$\vdots$};
\end{tikzpicture}
 \end{center}
 \caption{Diagrama de Van Kampen associado a $w$}
\end{figure}
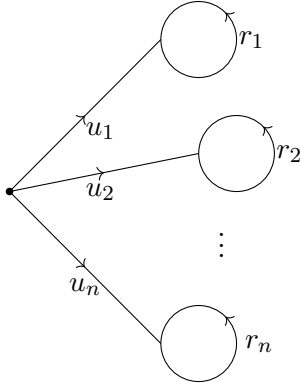


Reciprocamente, suponha que $w$ descreve a fronteira de um  diagrama de  van Kampen $D$. Afirmamos que  $w = 1$ em $G$. A prova é feita por indução sobre a área $k$ de $D$. Para $k = 1$ é óbvio.
Assuma que a afirmação é verdadeira se a área do diagrama é igual a  $k$, e tome um tal diagrama $D$ com $k + 1$ faces. Então  existe uma face $F$ contendo uma aresta $f_1$ em $\partial D$.
Logo, para algumas palavras $u$ e $v$ em $\partial D$, tem-se $w = uf_1v = uf_1f_2f^{-1}_2v$, onde $\partial F$ é rotulada por $f_1f_2$. Mas esta é igual a $(uf_1f_2u^{-1})(uf^{-1}_2v)$ e, por indução, como $D\backslash F$ possui somente $k$ faces, obtemos que $uf^{-1}_2 v = 1$. Mas também $uf_1f_2u^{-1} = 1$ já que $f_1f_2 \in R$. Portanto, $w = 1$ em $G$.   
\end{proof}

\begin{definition}
    Um \textit{diagrama de van Kampen minimal} para uma palavra $w$ é um  diagrama de van Kampen para $w$ com o menor número possível de faces.
\end{definition}

Usaremos diagramas de van Kampen para provar a recíproca do Teorema~\ref{prop:hiperbolicimpliesdil}.

\begin{thm} \label{thm:dilimplieshiperbolic}
    Se um grupo finitamente apresentado $G$ satisfaz uma desigualdade isoperimétrica  linear, então $G$ é hiperbólico.
\end{thm}

\begin{proof}
Seja $\langle S\mid R\rangle$ uma apresentação finita para $G$.
Suponha, por contradição, que para todo  $w \in G$ com $w = 1$ vale $A(w) \leq K\ell(w)$ e que $G$ não é  hiperbólico.
Então os triângulos em $\cay(G,S)$ não são $\delta$-magros, qualquer que seja $\delta \in \R$. Assim, para todo $c > 0$ existe um triângulo geodésico no grafo de Cayley de $G$ o qual contém um ponto cuja distância à união dos outros dois lados é pelo menos $c$. Escolha um tal triângulo $\Delta = \Delta(x, y, z)$ e considere um hexágono $H \subset \Delta $ cujos vértices são os  seis pontos de $\Delta$ que estão a uma distância igual a $\frac{c}{10}$ de $x,y$ e
$z$.

\begin{figure}
\begin{center}
  \begin{tikzpicture}[scale=0.9]
 \draw[font = {\tiny}] (0,0)node{$\bullet$};
  \draw [font = {\tiny}](2,3)node{$\bullet$};
  \draw [font = {\tiny}](4,0)node{$\bullet$};  
  \draw [font = {\tiny}](1.6,0)node{$\bullet$};   
  
 \draw (-0.3,0)node{$x$};
 \draw (2,3.3)node{$z$};
 \draw (4.3,0)node{$y$};
 \draw[font = {\tiny}] (1.6,-0.15)node{$p$};
 \draw[font = {\tiny}] (2.5,2.8)node{$\frac{c}{10}$}; 
 \draw[font = {\tiny}] (0.35,-0.3)node{$\frac{c}{10}$}; 
 \draw[font = {\tiny}] (3.65,-0.3)node{$\frac{c}{10}$};  
 \draw (0.6,1.5)node{$\ell_2$};
 \draw (3.3,1.5)node{$\ell_3$};
 \draw (2.3,-0.3)node{$\ell_1$};
 
\draw [dashed] (0,0) --  (2,3);
\draw [dashed] (0,0) --  (4,0);
\draw [dashed] (2,3) --  (4,0);

\draw [thick, dashed] (0.7,0.35) --  (3.3,0.35);
 \draw[font = {\tiny}] (2,0.5)node{$star(L)$};
\draw[thick] (1.6,2.4) to [bend right=30] (2.4,2.4);
\draw[thick] (0.7,0) to [bend right=30] (0.4,0.6);
\draw[thick] (3.3,0) to [bend left=30] (3.6, 0.6);

\draw[thick] (0.4,0.6) -- (1.6,2.4);
\draw[thick] (2.4,2.4) -- (3.6, 0.6);
\draw[thick] (0.7,0)  -- (3.3,0);
\fill[pattern=north east lines] (0.7,0) rectangle (3.3,0.34);
 \end{tikzpicture}
 \end{center}
 \caption{Prova do Teorema \ref{thm:dilimplieshiperbolic}}
    \label{fig:thmdilimplieshyp}
\end{figure}

Suponha que  $p$ é o ponto no segmento $\overline{xy}$ para o qual
$$d(p, \overline{xz} \cup \overline{yz}) = \max_{q\in\overline{xy}}
\{d(q,\overline{xz} \cup \overline{yz})\}.$$
Seja $D$ um diagrama de van Kampen minimal para $H$. Suponha que  $\ell_1,\ell_2$ e $\ell_3$ são os comprimentos dos lados de $H$
que foram truncados de $\Delta$ e seja $\ell' = \max\{\ell_1, \ell_2, \ell_3\}$, o qual é realizado por um lado  $L$ de $H$ (vamos supor sem perda de generalidade que $\ell(L) = \ell_1$).

Assim,  $\ell(\partial D) < 6\ell '$ e, pela desigualdade isoperimétrica linear, vale $A(D) < 6K\ell '$.  Seja $\rho = \displaystyle \max_{r \in R}\{\ell (r)\}$, obtemos 
$$\mbox{área}(\mathrm{star}(L)) \geq \dfrac{\ell '}{\rho},$$
onde $\mathrm{star}(L)$ é a união de todas faces de $D$ que cruzam $L$. 

Além disso, se $L_1 = \partial (\mathrm{star}(L)) \backslash \partial H$, temos $\ell(L_1) \geq \ell' -2\rho$, pois caso contrário $L$ não seria uma geodésica.

Repetindo esse processo $12K$ vezes, obtemos:
\begin{equation*}
\begin{split}
    \mbox{área}(\mathrm{star}^{12K}(L)) &>\dfrac{\ell'}{\rho}+\dfrac{\ell'-2\rho}{\rho} + \ldots + \dfrac{\ell'-12k\rho}{\rho}\\
&\geq  12k \rho \dfrac{\ell'-12k\rho}{\rho}\\
&= 12k \ell'-12k\rho\\
& > 6k\ell',
\end{split}
\end{equation*} 
e isso contradiz a desigualdade isoperimétrica. Observe que podemos escolher $c$ suficientemente grande, de modo que $\ell'-12k\rho\geq \frac{\ell'}{2}$, já que $\ell'>\frac{c}{10}$.
\end{proof}

Tanto o Teorema~\ref{prop:hiperbolicimpliesdil} quanto sua recíproca podem ser provados para espaços métricos geodésicos quaisquer,  utilizando a noção de $\varepsilon$-preenchimentos de laços em espaços métricos, cuja descrição completa pode ser vista em \cite{bridson2013metric}. A ideia consiste em definir uma noção de área combinatória de laços para estes espaços, que também pode ser aplicada ao grafo de Cayley de grupos.

Seja $\mathbb{D}^2$ o disco de centro na origem e raio $1$ em $\R^2$ e $\partial \mathbb{D}^2= \mathbb{S}^1$. Uma triangulação de $\mathbb{D}^2$ é um homeomorfismo $P:\mathbb{D}^2 \to T$, onde $T$ é um complexo celular onde cada 2-célula é um triângulo. 
 Podemos induzir uma estrutura celular em $\mathbb{D}^2$ via $P$. 

\vspace{0.5cm}
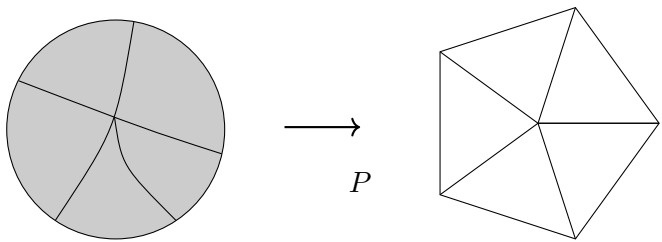
\begin{figure}[!ht]
\begin{center}
\begin{tikzpicture}[scale=0.8]
    \filldraw[fill=gray!40] (0,0) circle (1.8cm);
  \draw (-1.6,0.8) .. controls (0.5,0) and (0.5,0) .. (1.75,-0.4);
  \draw (-1,-1.5) .. controls (0,0) and (0,0) .. (0.3,1.77);
  \draw (1,-1.5) .. controls (0.1,-0.6) and (0.1,-0.6) .. (-0.02,0.2);
\end{tikzpicture}\qquad\tikz[baseline=-\baselineskip]\draw[thick,->] node[above]{$P$}(-1,1) -- ++ (1,0) ;\qquad
\begin{tikzpicture}[scale=0.8]
    \foreach \aa in {0,1,2,...,5}{
    \draw (0,0)coordinate(O) -- ({\aa*360/5}:2)coordinate(P\aa);
}

\draw 
(P0) 
\foreach \aa in  {1,2,...,5}{--
(P\aa)
}
-- cycle;
\end{tikzpicture}        
\end{center}
\caption{Triangulação de $\mathbb{D}^2$}
    \label{fig:triangulaçao}
\end{figure}

Começamos definindo a noção de \textit{curvas retificáveis} \index{curvas retificáveis}. Em um espaço métrico $(X,d)$ qualquer, podemos definir uma noção de comprimento de uma curva contínua $c:[a,b]\to X$ por 
$$\ell(c)
=
\sup_{a=t_0\leq t_1\leq \cdots \leq t_n=b}
\sum_{i=0}^{n-1}
d\bigl(c(t_i),c(t_{i+1})\bigr),$$
onde o supremo é tomado sobre todas as partições $a=t_0\leq t_1\leq \cdots \leq t_n=b$ do intervalo $[a,b]$, sem qualquer restrição sobre o número de pontos da partição. O comprimento de $c$ é um número real não negativo ou infinito. Dizemos que a curva $c$ é \textit{retificável} se o seu comprimento, conforme a definição acima, é finito.

Tome $X$ um espaço métrico e $\gamma: \mathbb{S}^1\to X$ um laço retificável em $X$. Um \textit{$\varepsilon$-preenchimento} \index{$\varepsilon$-preenchimento} de $\gamma$ é um par $(P,\varphi)$, onde $P$ é uma triangulação de $\mathbb{D}^2$ e  $\varphi: \mathbb{D}^2 \to X$ é uma aplicação tal que $\varphi|_{\mathbb{S}^1}=\gamma$ e a imagem por $\varphi$ de cada 2-célula tem diâmetro menor que ou igual a $\varepsilon$. Denotamos por $|\varphi|$ a área do preenchimento $\varphi$, dada por $ |\varphi| =  \# \{2\mbox{-células em }$P$\}$. Em geral, a função $\varphi$ pode não ser contínua.

\begin{definition}
 Dado um laço retificável $\gamma: \mathbb{S}^1\to X$, definimos a \textit{área$_{\varepsilon}$ de $\gamma$} por $$\mbox{área}_{\varepsilon}(\gamma):=\min\{|\varphi| \mid (P,\varphi) \mbox{ é um } \varepsilon\mbox{-preenchimento de } 
  \gamma\}.$$ Se não existe um $\varepsilon$-preenchimento de $\gamma$, escrevemos área$_{\varepsilon}(\gamma) = \infty$.
\end{definition}

Se existir $\varepsilon>0$ tal que cada laço retificável em $X$ tem um \mbox{$\varepsilon$-preenchimento}, então também chamaremos o numero  área$_{\varepsilon}(\gamma)$ de \textit{área combinatorial}\index{area combinatorial@área combinatorial} de $\gamma$, e a denotaremos a partir de agora também por  $A(\gamma)$.

Caso exista $\varepsilon>0$ tal que, para cada laço retificável  $\gamma: \mathbb{S}^1\to X$, $A(\gamma) \leq f(\ell(\gamma))$ onde $f$ é uma função real linear (quadrática, exponencial, etc.) dizemos que $X$ satisfaz uma \textit{desigualdade isoperimétrica grosseira linear}
\index{desigualdade isoperimétrica grosseira} (quadrática, exponencial, etc.).

\begin{remark}
\begin{enumerate}
    \item Se $\varepsilon' > \varepsilon$, então área$_{\varepsilon'}(c) \leq$ área$_{\varepsilon}(c)$ para qualquer laço retificável $c$ em um espaço métrico $X$.
\item Caso $X$ e $Y$  sejam espaços métricos quasi-isométricos de modo que existam $\varepsilon > 0$ e uma função $f: \mathbb{R^+} \to \mathbb{R^+}$ tais que todo laço em $X$ possui um $\varepsilon$-preenchi\-men\-to e, além disso,  área$_{\varepsilon}(c) \preceq f (\ell(c))$ para todo laço retificável $c$ em $X$, então  $Y$ admite um limite isoperimétrico $f'\preceq f $.
\end{enumerate}
\end{remark}

\begin{exercise}
    Dado um espaço métrico $\CAT(k)$ $X$, prove que, para todo $\varepsilon' < \varepsilon <  \frac{2\pi}{\sqrt{k}}$, existe uma constante $C = C(\varepsilon, \varepsilon', k)$ tal que área$_{\varepsilon'}(c) \leq C\,$área$_{\varepsilon}(c)$ para qualquer curva retificável $c$ em $X$.
\end{exercise}

O teorema a seguir generaliza o Teorema~\ref{thm:dilimplieshiperbolic} para espaços métricos que satisfaçam alguma desigualdade isoperimétrica linear.

\begin{thm}[Teorema 2.9 em \cite{bridson2013metric}]
Se $X$ é um espaço métrico geodésico, suponha que existam constantes $K,\varepsilon >0$ tais que área$_{\varepsilon}(\gamma)\leq Kl(\gamma)+K$ para cada laço geodésico por partes. Então, $X$ é $\delta$-hiperbólico, onde $\delta$ é uma constante que depende de $K$ e $\varepsilon$.
\end{thm}

\begin{remark}
    Gromov provou que é suficiente supor desigualdade isoperimétrica subquadrática, para garantir uma desigualdade isoperimétrica linear, e portanto hiperbolicidade.
\end{remark} 

\subsection{O algoritmo de Dehn}\index{algoritmo de Dehn}
Seja $G$ um grupo finitamente apresentado. Dizemos que uma apresentação finita $\langle S \mid R\rangle$ é uma \textit{apresentação de Dehn} \index{apresentação de Dehn} se
\begin{enumerate}
    \item[(i)] ela é uma apresentação simétrica, isto é, $R$ contém todas as palavras obtidas a partir de um relator ao invertê-lo ou  fazer uma permutação cíclica nas suas letras;
    \item[(ii)] para cada palavra reduzida $w$ não trivial representando a identidade em $G$, vale que  $\omega$ contém mais que metade de um relator $r \in R$.
\end{enumerate}
 Vejamos o que isto significa. Se $G =\langle S \mid R\rangle $ é uma apresentação de Dehn de $G$ e $w=1$ é uma palavra reduzida em $S$, isto implica  que existe um relator $r \in R$  que se escreve como um produto $r = r_1r_2$, tal que $|r_1|_S > |r_2|_S$ e $w=w_1r_1w_2$, para $w_i \in G$.
Nestas condições, uma palavra $w \in G$ é dita \textit{Dehn reduzida} se ela não contém mais do que a metade de qualquer relator $r \in R$.

\begin{exercise}
    Mostre que $\Z^2$ não tem apresentação de Dehn.
\end{exercise}

\begin{lemma} \label{lemma:dehnpres}
Se um grupo $G$ possui uma apresentação de Denh, então o problema da palavra é solúvel para $G$.
\end{lemma}
\begin{proof}
Seja $\langle S \mid R\rangle$ uma apresentação de Dehn para o grupo $G$. Tome $w \in F(S)$ uma palavra qualquer. Então ou $w$ não contém mais da metade de algum relator (caso no qual sabemos que $w$ deve ser diferente da identidade) ou ela possui, isto é, $w = ar_1b$ onde existe uma palavra $r_2$ tal que $r_1r_2 \in R$ e $|r_1|_S > |r_2|_S$. Assim, $w = ar^{-1}_2b$ em $G$, a qual é uma palavra de comprimento estritamente menor que a inicial. Repetimos o processo até que, eventualmente, atingimos a palavra vazia, caso no qual $w$ é trivial em $G$, ou reduzimos $w$  a uma palavra que não contenha mais da metade de um relator, caso no qual $w$ não é trivial em $G$. Este algoritmo, conhecido como \textit{algoritmo de Dehn}, \index{algoritmo de Dehn} termina após um número finito de passos, tornando possível solucionar o problema da palavra em $G$.
\end{proof}

A partir do Lema \ref{lemma:dehnpres}, podemos verificar que grupos hiperbólicos possuem o problema da palavra solúvel, visto que a apresentação dada no Corolário~\ref{cor:hiperbolicisfinpres} é uma apresentação de Dehn. Mais ainda, temos o seguinte:

\begin{thm}
\label{thm:hipsseDehn}
Um grupo $G$ é hiperbólico se, e somente se, tem uma apresentação de Dehn.
\end{thm}

\begin{proof}
 Seja $G$ um  grupo hiperbólico  gerado por um conjunto finito $S$ e considere $\langle S\mid R\rangle$, onde $R
 = \{w \in F(S) \mid |w|_S \leq 10\delta, w = 1\}$. Para provar que esta é uma apresentação de Dehn, note que um argumento semelhante ao da prova do Teorema~\ref{prop:hiperbolicimpliesdil} mostra que toda palavra trivial em $G$ possui ao menos metade de um relator. 

 Reciprocamente, se o grupo possui apresentação de Dehn, podemos usar o algoritmo de Dehn para provar que $G$ satisfaz uma desigualdade isoperimétrica linear, e em seguida aplicamos o Teorema \ref{thm:dilimplieshiperbolic}.
 
 De fato, se $X= \cay(G,S)$ e $\gamma$ é um laço em $X$ de comprimento $n$, correspondendo a uma palavra $w=1$ em $G$, e tomamos $\rho$ como sendo o comprimento máximo de uma palavra $v \in R$, então podemos usar indução em $n$ e substituições como no algoritmo de Dehn para mostrar que 
$$\mbox{área}_\rho(\gamma) \leq \rho \ell(\gamma).$$
\end{proof}

Este algoritmo foi inicialmente usado por Dehn no início do século XX para resolver o problema da palavra em grupos fundamentais de superfícies de Riemann (veja \cite{dehn2012papers} para uma coleção dos trabalhos de Dehn). Mais tarde, Cannon estendeu em  \cite{cannon1984combinatorial}  os resultados de Dehn para grupos de variedades compactas de curvatura negativa. Os Teoremas \ref{thm:dilimplieshiperbolic} e \ref{thm:hipsseDehn} são de Gromov (1987). Também é possível mostrar que os problemas da conjugação e do isomorfismo são solucionáveis em grupos hiperbólicos (veja, por exemplo, \cite{BuckleyHolt2013}, \cite{Sela1995} e \cite{Dahmani2011}).

\section{Fronteiras ideias de grupos hiperbólicos}

 A \textit{fronteira ideal}\index{fronteira ideal} de um grupo $G$, denotada por $\partial_{\infty}G$, é definida como a fronteira ideal do espaço métrico $\cay(G,S)$, com respeito a algum conjunto $S$ de geradores de $G$. Para diferentes conjuntos de geradores, as fronteiras correspondentes são homeomorfas,  dado que  os grafos de Cayley são quasi-isométricos.

 \begin{definition}
Seja $G < \mathrm{Homeo}(Z)$ um grupo de homeomorfismos de um espaço topológico compacto Hausdorff $Z$. O grupo $G$ é dito \textit{grupo de convergência}\index{grupo de convergência} se ele age propriamente descontinuamente  em $\mathrm{Trip}(Z)$, onde $\mathrm{Trip}(Z)$ é o conjunto de  triplas de elementos distintos de $Z$, munido da topologia natural de subespaço de $Z^3$. O grupo de convergência  $G$ é dito \textit{uniforme} se $\mathrm{Trip}(Z)/G$ é
 compacto.
 \end{definition}

Para a próxima proposição, usaremos algumas definições e resultados auxiliares. Dizemos que um ponto $ p$ em um espaço métrico próprio $\delta$--hiperbólico $X$ é um  \textit{$R$-centroide} de um triângulo $T$ em $\overline{X}$ se $p$ estiver a uma distância menor ou igual a $R$
dos três lados $\tau_1, \tau_2, \tau_3$ de $T$. Além disso, diremos que $p$ é um \textit{
centroide}\index{centroide} de $T$ se
$d(p, \tau_i) \leq 5 \delta$, para $i = 1, 2, 3$.

Definimos então a correspondência
$$C : \mathrm{Trip}(\partial_{\infty}X) \to  X,$$
que mapeia cada tripla de pontos distintos em $\partial_{\infty}X$ no conjunto de triângulos
$T$ que eles geram, e em seguida no conjunto de centroides desses triângulos ideais.

\begin{exercise}
    Mostre que a distância entre quaisquer dois $R$-centroides de um triângulo $T$  em um espaço métrico próprio $X$ é de, no máximo, $4R + 32\delta$. Com isso, conclua que, para toda tripla de pontos $\xi \in \mathrm{Trip}(\partial_{\infty}X)$, vale que 
$$\mathrm{diam}(C(\xi)) \leq 60\delta.$$
\end{exercise}

\begin{proposition}[\cite{Tuk94}] Seja $G$ um grupo $\delta$-hiperbólico gerado por um conjunto $S$, tal que $\partial_{\infty}G$ consiste de pelo menos três pontos. Então $G$ age em $\partial_{\infty}G$ como grupo  de convergência uniforme.
\end{proposition}
\begin{proof}
 Dados três pontos distintos $\xi, \eta, \rho \in \partial_{\infty}G$, considere o triângulo geodésico $\Delta = \Delta(\xi,\eta, \rho) \subset \cay(G,S)$. O mapa $C$ definido acima é $G$-equivariante e, além disso, a imagem de qualquer compacto $K$ em $\mathrm{Trip}(\partial_{\infty}G)$ por $C$ é limitada.

Dado um subconjunto compacto $K \subset \mathrm{Trip}(Z)$, suponha que o conjunto 
$$G_K := \{ g\in G \mid gK \cap K \neq \emptyset \}$$ seja infinito. Logo, deve existir uma sequência $\xi_n\in K$ e uma sequência infinita $g_n\in G$, com $g_0 = 1 $ e $ g_n(\xi_n)\in K$ para todo $n \in \N$.
 O diâmetro do conjunto $$
E = \displaystyle\bigcup_n
C(g_n(\xi_n))  \subset \cay(G,S) $$
é limitado, e cada $g_n$ envia algum ponto $p_n \in E$ em um elemento de $E$. Isso contradiz o fato de que $G$ age propriamente descontinuamente em $\cay(G,S)$.

Com isso, concluímos que a ação de  $G$ em $\mathrm{Trip}(\partial_\infty G)$ também é propriamente discontínua.
\end{proof}

\begin{proposition}\label{prop:fronteiragrupohip}
Seja $G$ um grupo hiperbólico. 
\begin{itemize}
    \item[(i)] Se $G$ é finito, então $\partial_{\infty}G$ é vazia;
    \item[(ii)]  Se $\partial_{\infty}G$ tem precisamente dois elementos, então $G$ contém um subgrupo cíclico infinito de índice finito;
    \item[(iii)] Em qualquer outro caso, $\partial_{\infty}G$  é infinito não-enumerável.
\end{itemize}

Nos casos (i) e (ii), diremos que o grupo é \textit{elementar}.\index{grupo elementar}
\end{proposition}

\begin{proof}
Seja $X$ um grafo de Cayley de $G$, com respeito a algum conjunto de geradores. Se $G$ for  finito, então $X$ é limitado, e portanto $\partial_{\infty}G = \emptyset$, provando (i). 

Assuma agora que $G$ é infinito. Pelo Exercício \ref{exerc:geodesicabiinfinita}, $X$ contém
uma geodésica completa bi-infinita $\gamma$, e assim $\partial_{\infty}G$ possui pelo menos dois pontos distintos, os pontos-limite de $\gamma$. 


Vamos provar (ii) em dois casos. Suponhamos que $\partial_{\infty}G$ possui exatamente dois elementos. Se $d_{Haus}(\gamma
,X) < \infty$, então $X$ é quasi-isométrico a $\R$ e portanto $e(G)$ possui dois elementos. Assim, $G$ é virtualmente cíclico, pela parte 3 do Teorema \ref{thm:fins grupos}.

Caso $d_{Haus}(\gamma
,X) = \infty$, deve existir uma sequência de vértices $x_n \in X$ satisfazendo $\displaystyle \lim_{n\to\infty} d(x_n,\gamma) =\infty$. Para cada $n \in \N$, seja $y_n\in \gamma$ o vértice de $\gamma$ mais próximo de $x_n$ e $g_n \in G$ um elemento tal que $g_n(y_n) = e \in G$. Aplicando $g_n$ à união de geodésicas $\overline{x_ny_n} \cup \gamma$ e tomando o limite com $n\to \infty$, obtemos uma  geodésica completa $\beta \subset X $ (o limite
de uma subsequência $g_n(\gamma)$) e um raio  geodésico $\rho $ intersectando $\gamma$ em $e$, tais que para todo
$x \in \rho$, $e$ é um ponto mais próximo de $x$ em $\gamma$. Portanto, $\rho(\infty)$ é um ponto diferente de $\gamma(\pm \infty)$, e $\partial_{\infty}G$ contém pelo menos três pontos distintos, uma contradição.

A prova de (iii) segue do Teorema \ref{thm:finsfronteira} a seguir. De fato, se $\partial_{\infty}G$ for  finito, então por esse resultado, obtemos $\partial_{\infty}G = e(G)$. Mas sabemos que o espaço dos fins de um grupo é ou um conjunto finito com 0,1 ou 2 pontos, ou possui um contínuo de pontos, o que conclui a demonstração.
\end{proof}

\begin{thm}\label{thm:finsfronteira}
Seja $G$ um grupo hiperbólico. Então existe uma aplicação $f :\partial_{\infty}G \to e(G)$ contínua, $G$-equivariante, sobrejetiva e tal que  as fibras $f^{-1}(\xi)$, com $\xi \in e(G)$, são conexas.
\end{thm}

\begin{proof}
Seja $\gamma : \R_{+} \to G$ um raio geodésico. Como esse raio foge de todo compacto de $G$, ele define um fim $\xi_{\gamma}$ de $G$. Deixamos como exercício verificar que esse fim depende apenas do ponto $\gamma(\infty)$ do bordo de $G$ definido por $\gamma$ e que a aplicação $f :\partial_{\infty}G \to e(G)$ definida por $f(\gamma(\infty)) = \xi_{\gamma}$ é contínua. Seja $(x_n)_{n\geq 1}$ uma sequência de pontos de $G$ definindo um fim $\xi$ de $G$. Se $a$ é um limite de uma subsequência de $(x_n)$ em $\partial_{\infty}G$, verificamos que $f(a) = \xi$. Assim a aplicação $f$ é sobrejetiva. Como $e(G)$ é totalmente desconexo, as componentes conexas de $\partial_{\infty}G$
estão contidas nas  fibras de $f$.

Resta verificar que as fibras de $f$ são conexas. Suponhamos
que existe uma partição de uma fibra $f^{-1} (\xi)$ em dois fechados não vazios
$F_1$ e $F_2$. Escolhemos dois abertos disjuntos $U_1$ e $U_2$ de $\partial_{\infty}G$ contendo respectivamente $F_1$ e $F_2$. Fixamos um ponto base $w \in G$,
e denotamos por $\hat{U}_j$ a união dos raios partindo de $w$ e terminando em algum ponto de $U_j$ , $j = 1,2$. Se $B_R$ denota a bola de centro $w$ e  raio $R$, onde $R$ é suficientemente grande, verificamos que $\hat{U}_1 \backslash B_R$ e $\hat{U}_2 \backslash B_R$ são
disjuntos e tem distância estritamente positiva um do outro. Seja $\gamma_1$
(respectivamente, $\gamma_2$) um raio saindo de $w$ e terminando em um ponto de $F_1$ (resp., $F_2$). Por hipótese, as sequências $(\gamma_1(n))_{n\geq 1}$ e $(\gamma_2(n))_{n\geq 1}$
definem o mesmo fim. Para todo $N\geq 1$, existem então um inteiro $n_N \geq 1$ e um caminho $\gamma_N : [0,1]\to G$ conectando $\gamma_1(n_N) \in \hat{U}_1$ e
$\gamma_2(n_N) \in \hat{U}_2$  que evita a bola $B_N$. 
 Como $\hat{U}_1 \backslash B_N$ e $\hat{U}_2 \backslash B_N$ tem distância estritamente positiva para $N$ suficiente grande, existe $t_N \in [0,1]$ tal que o ponto $x_N = \gamma_N(t_N)$ não pertence nem a $\hat{U}_1$ nem a $\hat{U}_2$. 
 
 Escolhemos então um segmento geodésico $[w, x_N]$ de $w$ a $x_N$. Como os comprimentos desses segmentos tendem a infinito, deve existir uma subsequência de  $([w,x_N])_{N\geq 1 }$ que
converge para um raio $\gamma_0$ partindo de $w$. É claro que esse raio $\gamma_0$ define o mesmo fim $\xi$ que $\gamma_1$ e $\gamma_2$, mas ele não pode estar contido nem em $F_1$ nem em $F_2$. Assim, obtemos um absurdo, donde resulta que $f^{-1}(\xi)$ é conexa.

\end{proof}

\section{Alternativa de Tits para grupos hiperbólicos}\label{sec8-TitsAlt}

Um resultado fundamental na teoria dos grupos provado por J.~Tits em 1972 é o seguinte teorema.


\begin{thm}[\cite{Tits72}]
 Seja $G$ um grupo linear finitamente gerado. Então ou $G$ é virtualmente solúvel ou $G$ contém um subgrupo isomorfo a um grupo livre não abeliano.
\end{thm}

Relembramos que $G$ é {\it solúvel} se existe cadeia finita de subgrupos $\{G_i\}_{i=0}^n$ de $G$ tal que $1=G_0 \unlhd G_1 \unlhd G_2 \unlhd \ldots\unlhd G_n=G$ e $G_{i+1}/G_{i}$ é abeliano. O grupo $G$ é chamado {\it virtualmente solúvel} se ele possui um subgrupo solúvel de índice finito.

Dizemos que uma classe de grupos satisfaz alguma \textit{Alternativa de Tits}\index{Alternativa de Tits} se cada um de seus elementos satisfizer uma dicotomia do tipo ``ou ele contém um grupo livre não abeliano ou ele possui subgrupo solúvel de índice finito''. Conjectura-se que a Alternativa de Tits é comum entre grupos ``de curvatura não positiva''. No entanto, esse fato está demonstrado para algumas poucas classes 
particulares de grupos. Algumas dessas classes são: grupos Gromov-hiperbólicos \cite{gromov1987hyperbolic},
grupos fundamentais de algumas variedades compactas de curvatura não positiva \cite{ballmann-lectures} e grupos cocompactos agindo isometricamente e propriamente descontinuamente em complexos bidimensionais $\CAT(0)$ \cite{ballmann1999groups}. 
Esta seção é dedicada à demonstração da Alternativa de Tits para o caso de grupos hiperbólicos.

Observamos que, se $G$ age em $X$ por isometrias, essa ação induz uma ação no conjunto das sequências em $X$,
$$g((x_n)) := (g(x_n)).$$
Sendo a ação de $G$ em $X$ isométrica, sempre que $x_n \to \xi \in \partial X$, temos que $(g(x_n))$ converge em $\partial X$ e, se $(x_n), (y_n)$ são sequências com $(x_n,y_n)_p \to \infty$, temos $(g(x_n),g(y_n))_p \to \infty$. 

Assim, podemos definir uma ação isométrica no bordo por $g(\xi) = \displaystyle\lim_{n\to \infty}g(x_n) \in \partial X$, para $x_n \to \xi \in \partial X$.
Nesta seção consideraremos a ação de um grupo hiperbólico $G$ em $X=\cay(G,S)$ com fronteira $\partial X = \partial G$.

\begin{proposition}\label{prop:UV}
Sejam $U,V$ vizinhanças de $a^{+}, a^{-}\in \partial G$, respectivamente, os quais são fixos por um elemento $g \in G$ de ordem infinita.
Então existe $N \in \N$ tal que, para todo $m \geq N$, temos $g^{m}(\partial G \setminus V) \subset U$
e $g^{-m}(\partial G \setminus U) \subset V$.
\end{proposition}

\begin{proof} Seja $X=\cay(G,S)$ um grafo de Cayley de $G$.
Como $g$ tem ordem infinita, podemos associar a esse elemento duas sequências $(g^n)_{n\in \N}$ e $(g^{-n})_{n\in\N}$. Suponha que $g^n \to a^+$ e $g^{-n} \to a^-$ em $\partial G$. Então $g(a^+)= a^+ $ e $g(a^-)= a^- $ e, como a imagem de $\{g^n \mid n\in \Z\}$ é uma quasi-geodésica em $X$, temos $a^+\neq a^-$.

Tome  um ponto  $x \in \partial G \setminus V$. Então existe uma sequência $\{x_n\}$ convergindo a $x$ tal que $\displaystyle\liminf_{n_1,n_2\to \infty}(x_{n_1} , g^{-n_2})_e\leq K$, para algum $K > 0$ (caso contrário,  $\displaystyle\liminf_{n_1,n_2\to \infty}(x_{n_1} , g^{-n_2})_e= \infty$, o que implicaria em $(x,a^-)_e = \infty$ e portanto $x=a^- \in V$, uma contradição). Queremos encontrar um inteiro $m$ tal que $g^m(x) \in U$. É suficiente encontrar $m$ tal que $\displaystyle\liminf_{n\to \infty}(g^m(x_n),g^n)_e \geq M$ onde $M$ tem a seguinte propriedade:
$$(\xi,a^+)_e\geq M \Longrightarrow \xi \in U.$$ Pelo Exercício \ref{exerc:propelem9} existe $N > 0$ tal que, para quaisquer $m,n \geq N$, vale $(x_n, g^{-m})_e \leq  K + \delta$.
Assim, 
$$\dfrac{1}{2}\left(d(x_n, e) + d(g^{-m}, e) - d(x_n, g^{-m})\right)\leq K + \delta,$$
o que, usando que $d(g^m(x_n), e) = d(x_n, g^{-m})$,  implica 
$$\dfrac{1}{2}d(g^m(x_n), e)\geq \dfrac{1}{2}\left(d(x_n, e) + d(g^{-m}, e)\right) - (K + \delta).$$
Portanto, obtemos
\begin{align*}
    (g^m(x_n) , g^n)_e & \geq  \dfrac{1}{2}\Big(d(x_n, e) + d(g^{-m}, e) + d(g^n, e)\\
    &-d(x_n, g^{n-m})\Big) - (K + \delta) \\
    & = [(x_n , g^{n-m})_e - (K + \delta)] + (g^n, g^m)_e,
\end{align*}
onde o termo entre colchetes é limitado e $(g^n, g^m)_e \to \infty$. Tomando $m$ suficientemente grande, obtemos $(g^m(x_n), g^n)_e \geq M$, como queríamos.

\begin{figure}[H]
     \centering
 \begin{tikzpicture}[scale=0.9]

 \draw (0,0) circle (2cm);
\filldraw [fill=gray!10] plot [thin, smooth cycle] coordinates {(0.7,0) (1,0.6) (1.75,0.9) (2,0) (1.75,-0.9) (1,-0.6)};
\filldraw [fill=gray!10] plot [thin, smooth cycle] coordinates {(-0.7,0) (-1,0.6) (-1.75,0.9) (-2,0) (-1.75,-0.9) (-1,-0.6)};

\draw[line width=0.5mm, lightgray, opacity=0.7] (1.77,0.9) arc (27:333:2);

\draw[->, line width=0.4mm, lightgray, opacity=0.8] 
(0,1.9) to [bend left = 30] (-1.2,0.4);

 \draw (0,1)node{$g^m$};
 \draw (-1.4,-0.4)node{\tiny{$U$}};
 \draw (1.5,-0.4)node{\tiny{$V$}};
 \draw (-2.3,0)node{\tiny{$a^+$}};
 \draw (2.3,0)node{\tiny{$a^-$}};
 \node [draw, shape = circle, fill = black, minimum size = 0.05cm, inner sep=0pt] at (0,0){};
\node [draw, shape = circle, fill = black, minimum size = 0.05cm, inner sep=0pt] at (2,0){}; 
\node [draw, shape = circle, fill = black, minimum size = 0.05cm, inner sep=0pt] at (-2,0){};

\draw [->] plot [smooth] coordinates {(2,0) (1.3,0.15) (0.7,-0.2) (0,0) };
\draw  plot [smooth] coordinates {(0,0) (-0.7,0.2) (-1.3,-0.15) (-2,0)};

\draw (-0.1,-0.15)node{\tiny{$e$}};
 
\end{tikzpicture}
         \caption{Proposição \ref{prop:UV}.}
     \label{fig:propUV}   
\end{figure}

\end{proof}

\begin{proposition}\label{prop:stab(a)}
Considere um grupo hiperbólico $G$ e um elemento de ordem infinita $g \in G$. Seja $a \in \partial G $ um ponto fixado por $g$. Então, sob a ação de $G$ em $\partial G$, o estabilizador $\mathrm{Stab}_G(a)$ é uma extensão finita de $\langle g \rangle$, ou seja, é virtualmente cíclico.
\end{proposition}

\begin{proof}
Dada uma classe lateral em $\mathrm{Stab}_G(a) / \langle g \rangle$, tome um representante de classe $c$ de menor comprimento, com respeito a um conjunto de geradores $S$. Mostraremos que $|c|_S$ é limitado.
Seja $a = \lim(g^n) \in \partial G$. Então $\lim(c g_n) = a$ e  $\lim(cg_n,g_n
)_e = \infty$. Se $c_{1,n}
, c_{2,n}$ e $c_{3,n}$ são os pontos centrais dos lados do triângulo $\Delta(e, cg^n
, g^n)$, temos que $\displaystyle \lim_{n\to\infty}(d(e, c_{1,n}
)) = \infty$ e $d(c_{1,n}
, c_{2,n})\leq \delta$ para todo $n$.

\begin{figure}[H]
     \centering
 \begin{tikzpicture}

 \draw (-0.2,-0.2)node{\small{$e$}};
 \draw (3.2,1.2)node{\small{$cg^n$}};
 \draw (3.25,-1)node{\small{$g^n$}};
  \draw (1.9,0.7)node{\tiny{$c_{2,n}$}};
 \draw (1.9,-0.7)node{\tiny{$c_{1,n}$}};
 \draw (2.2,0)node{\tiny{$\leq\delta$}};
 \draw[color=gray] (0.8,-0.4)node{\tiny{$\to\infty$}};
 
 \node [draw, shape = circle, fill = black, minimum size = 0.05cm, inner sep=0pt] at (0,0){};
\node [draw, shape = circle, fill = black, minimum size = 0.05cm, inner sep=0pt] at (3,1){}; 
\node [draw, shape = circle, fill = black, minimum size = 0.05cm, inner sep=0pt] at (3,-1){};
\node [draw, shape = circle, fill = black, minimum size = 0.05cm, inner sep=0pt] at (2,0.52){}; 
\node [draw, shape = circle, fill = black, minimum size = 0.05cm, inner sep=0pt] at (2,-0.52){};

\draw(2,0.52) to [bend right = 10] (2,-0.52);
\draw (0,0) to [bend right = 10] (3,1);
\draw (0,0) to [bend left = 10] (3,-1);
\draw[line width=0.5mm, color=gray, opacity=0.3] (0,0) to [bend left = 8] (2,-0.5);
\draw (3,1) to [bend right = 10] (3,-1);

\end{tikzpicture}
         \caption{Proposição \ref{prop:stab(a)}.}
     \label{fig:propstab(a)}   
\end{figure}

Como a imagem de $\{g^m\mid m \in \Z\}$ é uma quasi-geodésica no grafo de Cayley de $G$, existe $M > 0$ tal que, para todo $n$ sufi\-ci\-en\-temente grande, existem inteiros $k, m$ tais que $d(c_{1,n}, g^k) \leq M$ e $d(c_{2,n}, cg^m) \leq M$. Assim, $d(e, cg^{m-k})\ = d(g^k, cg^m)\leq 2M + \delta$. 
Mas, para $n\neq 0$, temos $|cg^n|_S \geq |c|_S$. Caso contrário, $c_1 :=cg^n$ seria um representante da mesma classe $c\langle g \rangle$ com comprimento menor que o de $c$, o que seria uma contradição. Assim, $|c|_S \leq 2M + \delta$.
\end{proof}

\begin{corollary}
Um grupo hiperbólico $G$ é elementar se, e somente se,  ele é finito ou virtualmente cíclico.
\end{corollary}
\begin{proof}
    Se $G$ é elementar, então $G$ é finito ou a ação de $G$ tem um ponto fixo $a \in \partial G$. Assim, $\mathrm{Stab}_G(a) = G$ e, portanto,  $G$ é virtualmente cíclico. 

    Por outro lado, se $G$ é infinito e virtualmente cíclico, então existe $g \in G$ tal que $\langle g \rangle$ tem índice finito em $G$, onde $g$ tem ordem infinita. Assim, pela Proposição \ref{prop:stab(a)}, $\mathrm{Stab}_G(\displaystyle\lim_ng^n) $ é uma extensão finita de $\langle g \rangle$. 
\end{proof}

\begin{thm} \label{thm:titshiperbolico}
Seja $G$ um grupo hiperbólico e tome $g_1$, $g_2$ elementos de ordem infinita de $G$. Então
 $\langle g_1, g_2\rangle$ é virtualmente cíclico ou existe um inteiro $k > 0$ tal que $\langle g^k_1, g^k_2\rangle$ é um grupo livre de posto 2.
\end{thm}

\begin{proof}
Sejam $X$ o grafo de Cayley de $G$ com respeito a $S$, $a^+, a^-$ os pontos fixos de $g_1$ e $b^+, b^-$ os pontos fixos de $g_2$. Se $\{a^+, a^-\} \cap \{b^+, b^-\} \neq \emptyset$, então $\langle g^k_1, g^k_2\rangle$ é virtualmente cíclico pela Proposição \ref{prop:stab(a)}. Caso contrário, podemos encontrar  vizinhanças abertas $a^+ \in U_1, a^- \in V_1,b^+ \in U_2$ e
$b^{-} \in V_2$, onde $U_1, U_2, V_1 $ e $V_2$ são mutuamente disjuntas. Então existe um inteiro $k$ com $$g^k_1(U_1 \cup U_2 \cup V_2) \subset U_1, \,\, g^k_2(U_2 \cup U_1 \cup V_1) \subset U_2,$$ $$g
^{-k}_1(V_1 \cup U_2 \cup V_2) \subset V_1,\,\, g^{-k}_2(V_2 \cup U_1 \cup V_1) \subset V_2.$$
Logo, para todos $m,n\neq 0$, temos $$(g_1^k)^m( U_2 \cup V_2) \subset U_1 \cup V_1, \,\, (g_2^k)^n( U_1 \cup V_1) \subset U_2 \cup V_2.$$
Pelo Lema do Pingue-Pongue, concluímos que $\langle g^k_1, g^k_2\rangle$ é livre de posto 2.
\end{proof}

\begin{exercise}
\label{exer:elemordeminfinita}
    Se $\Gamma_1$ é um subgrupo de um grupo hiperbólico $\Gamma$,  onde todo elemento de $\Gamma_1$ tem ordem finita, então $\Gamma_1$ é um grupo finito.
\end{exercise}

\begin{corollary}[Alternativa de Tits para grupos hiperbólicos] 
\label{cor:altTitsHip}
Todo subgrupo $H$ de um grupo hiperbólico $G$ contém $F_2$ ou é virtualmente cíclico.
\end{corollary}

\noindent 
\textit{Ideia da Demonstração:}   A prova segue diretamente do Teorema~\ref{thm:titshiperbolico} no caso em que $H$ é gerado por dois elementos de ordem infinita e é trivial no caso em que $H$ é finito.

    Caso $H$ seja um grupo infinito qualquer, daremos uma ideia da prova, seguindo \cite[Capítulo~8]{ghys2013groupes}.  Segue do Exercício~\ref{exer:elemordeminfinita} que $H$ possui um  elemento $h$ de ordem infinita, e portanto existem pelo menos dois pontos distintos $a^+$ e $a^-$ em $\partial H$, os quais são fixados por $h$. 
    
    Caso $\mathrm{Stab}_H(a^+) \neq \mathrm{Stab}_H(a^-)$, deve existir   $\gamma \in \mathrm{Stab}_H(a^-)$ tal que $\gamma(a^+)\neq a^+$. Mas assim podemos mostrar que $h$ e $\gamma h \gamma^{-1}$ geram um subgrupo livre de posto 2. 

    Por outro lado, se  $\mathrm{Stab}_H(a^+) = \mathrm{Stab}_H(a^-)$, dado que esse grupo contém $\langle h \rangle$ como subgrupo de índice finito, se esses estabilizadores forem iguais a $H$, então $H$ é virtualmente cíclico. O mesmo ocorre se $H$  preserva o par $\{a^+, a^-\}$, pois podemos considerar o subgrupo $ H \cap \,\mathrm{Stab}_H(a^+)$, de índice no máximo 2 em $H$. Caso contrário, existe um elemento $\gamma$ de $H$ tal que $\gamma(a^+),\gamma(a^-) \notin \{a^+, a^-\}$. O elemento $h' = \gamma h \gamma^{-1}$ tem, portanto, como pontos fixos os dois pontos $\gamma(a^+)$ e $\gamma(a^-)$,  distintos de $a^+, a^-$. Novamente, uma aplicação do Lema do Pingue-Pongue mostra que,  para $n$ suficientemente grande, $h^n$ e $(h')^n$ geram um grupo livre de posto $2$ em $H$.
    \qed

Um caso particular do resultado anterior é o dos centralizadores de elementos de ordem infinita.

\begin{corollary}
Sejam $G$ um grupo hiperbólico e  $g\in G$ um elemento de ordem infinita. Então o centralizador \index{centralizador de um grupo}
$C_G(g)=\{h\in G: hg=gh\}$ de $g$ em $G$ é virtualmente cíclico.
\end{corollary}

\begin{proof}
Como $C_G(g)$ é um subgrupo de $G$, pelo Corolário~\ref{cor:altTitsHip}, basta mostrar que $C_G(g)$ não contém um subgrupo livre de posto $2$.

Suponha, por absurdo, que $F_2\leq C_G(g)$. Então cada elemento de $F_2$
comuta com $g$, de modo que
$$\langle g,F_2\rangle\cong \langle g\rangle\times F_2.$$
Em particular, $\langle g,F_2\rangle$ contém um subgrupo isomorfo a $\mathbb Z^2$, o que é impossível pelo Teorema~\ref{thm:titshiperbolico}, pois G é hiperbólico. Portanto, $C_G(g)$ não contém um subgrupo livre de posto $2$, e o resultado segue.
\end{proof}


\section{Grupos relativamente hiperbólicos}

Vamos considerar um exemplo ilustrativo. Seja $G = \PSL(2, \mathcal{O}_3)$, onde $\mathcal{O}_3$ é o anel de inteiros do corpo quadrático imaginário  $\Q(\sqrt{-3})$. O grupo $G$ é um subgrupo discreto de $\PSL(2, \C)$ que age de forma propriamente descontínua no espaço hiperbólico tridimensional. De fato, o grupo $G$ desempenha um papel importante na teoria das $3$-variedades hiperbólicas, a partir do trabalho de Riley, que mostrou que ele possui um subgrupo livre de torção de índice $12$, isomorfo ao grupo fundamental do complemento do nó oito em $\mathbb{S}^3$ \cite{Riley75}. Notamos que $G$ não é um subgrupo cocompacto de $\PSL(2, \C)$. Portanto, não podemos aplicar o Teorema de Milnor--Schwarz à sua ação em $\mathbb{H}^3$. Além disso, $G$ possui um subgrupo isomorfo a $\Z^2$ gerado por elementos parabólicos
$$\begin{pmatrix}
1 & 1 \\
0 & 1
\end{pmatrix}
\quad \text{e} \quad
\begin{pmatrix}
1 & \sqrt{-3} \\
0 & 1
\end{pmatrix},$$
correspondente à cúspide do espaço quociente $\mathbb{H}^3/G$.
Portanto, $G$ não é um grupo hiperbólico (pelo Teorema~\ref{thm:titshiperbolico}). Um argumento semelhante se aplica aos outros grupos de Bianchi $\PSL(2, \mathcal{O}_d)$, onde $\mathcal{O}_d$ é o anel de inteiros do corpo quadrático imaginário  $\Q(\sqrt{-d})$, e aos reticulados não cocompactos em espaços hiperbólicos de dimensões maiores. Por outro lado, esses grupos compartilham várias propriedades com os grupos hiperbólicos que estudamos anteriormente neste capítulo. Por esse motivo, é desejável ter uma noção mais geral de hiperbolicidade.

Grupos relativamente hiperbólicos foram introduzidos por Gromov no mesmo artigo \cite{gromov1987hyperbolic} que grupos hiperbólicos. Enquanto grupos hiperbólicos são modelados em reticulados cocompactos em espaços simétricos de curvatura negativa, grupos relativamente hiperbólicos são modelados em reticulados não cocompactos em espaços de curvatura negativa e, mais geralmente, grupos fundamentais de variedades Riemannianas completas de volume finito e curvatura estritamente negativa. Hoje existem várias definições de hiperbolicidade relativa que se aplicam igualmente bem aos exemplos indicados acima, mas que não são equivalentes em geral.

A ideia de Gromov foi elaborada por B.~Bowditch no artigo  \cite{Bow12}, que estava disponível 
antes da publicação como um preprint da Universidade de Southampton e tem sido influente para a pesquisa nesta área.

\begin{definition}[\cite{Bow12}]
Um grupo finitamente gerado $G$ é \textit{hiperbólico em relação a uma coleção de subgrupos}\index{grupo relativamente hiperbólico} finitamente gerados $H_1$,\ldots, $H_m$ se ele admite uma ação isométrica propriamente descontínua em um espaço hiperbólico próprio $X$, de modo que a ação induzida de $G$ na fronteira $\partial X$ satisfaz as seguintes condições:
\begin{enumerate}[(1)]
\item O grupo $G$ age em $\partial X$ como um grupo de convergência geometricamente finito;
\item Os subgrupos parabólicos maximais de $G$ são precisamente os subgrupos de $G$ conjugados a $H_1$, \ldots, $H_m$.
\end{enumerate}
\end{definition}

Lembramos que um grupo G de homeomorfismos de um compacto metrizável $M$ age em $M$ como um \textit{grupo de convergência}\index{grupo de convergência} se a ação induzida no espaço de triplas distintas de elementos de $M$ for propriamente descontínua. Além disso, um subgrupo $H \leq G$ é \textit{parabólico} \index{subgrupo parabólico} se for infinito, fixar algum ponto de $M$ e não contiver elementos $g \in G$ de ordem infinita tais que $\textrm{card}(\textrm{Fix}(g)) = 2$. Neste último caso, o ponto fixo de $H$ é único e é chamado de \textit{ponto parabólico}\index{ponto parabólico}. Um ponto parabólico $x$ é dito \textit{limitado} se $(M \setminus \{x\})/\textrm{Stab}_G(x)$ for compacto. Um ponto $x \in M$ é chamado de \textit{ponto limite cônico}\index{ponto limite cônico} se existirem uma sequência ${g_i}$ em $G$ e dois pontos distintos $a, b \in M$ tais que $g_i x$ converge para $a$ e $g_i y$ converge para $b$, para qualquer $y \in M \setminus \{x\}$. Finalmente, um grupo de convergência $G$ é chamado \textit{geometricamente finito}\index{grupo geometricamente finito} se cada ponto de $M$ for um ponto limite cônico ou um ponto parabólico limitado.

A abordagem de Bowditch à hiperbolicidade relativa foi mais desenvolvida por A.~Yaman em \cite{Yam04}.

Uma outra definição de hiperbolicidade relativa foi sugerida por B.~Fard em \cite{Farb98}. Sejam $G$ um grupo, gerado por um conjunto finito $S$, e $\{H_1, H_2, \ldots, H_m\}$ uma coleção de subgrupos de $G$. Começamos com o grafo de Cayley $\cay(G, S)$ e formamos um novo grafo da seguinte forma: para cada classe lateral à esquerda $gH_i$ de $H_i$ em $G$, adicione um vértice $v(gH_i)$ a $\cay(G, S)$ e  uma aresta $e(gh)$, de comprimento $1/2$, ligando cada elemento $gh$ de $gH_i$ ao vértice $v(gH_i)$. Este novo grafo é chamado de \textit{grafo de Cayley conado}\index{grafo de Cayley conado} de $G$ relativo a $\{H_1, H_2, \ldots, H_m\}$ e é denotado por $\widehat\cay(G, S)$ e munido da métrica de caminhos. Note que $\widehat\cay(G, S)$ não é um espaço métrico próprio, pois bolas fechadas não são necessariamente compactas.

\begin{definition}[\cite{Farb98}]
O grupo $G$ é \textit{hiperbólico relativamente}\index{grupo relativamente hiperbólico} a $\{H_1, H_2, \ldots, H_m\}$ se o grafo de Cayley conado $\widehat\cay(G, S)$ de $G$ relativo a $\{H_1, H_2, \ldots, H_m\}$ for um espaço métrico hiperbólico.
\end{definition}

Em geral, a definição de Farb não é equivalente à de Gromov e Bowditch, mas elas se tornam equivalentes se assumirmos uma propriedade extra chamada penetração de coset limitada. Nos referimos a \cite{Osin06} para os detalhes.

Em \cite{Osin06}, D.~Osin obteve uma caracterização da hiperbolicidade relativa em termos de desigualdades isoperimétricas e adotou técnicas baseadas em diagramas de van Kampen para o estudo de propriedades algébricas e algorítmicas de grupos relativamente hiperbólicos.

Concluímos esta seção com alguns exemplos de grupos relativamente hiperbólicos. 

\begin{examples}
\begin{enumerate}[(1)]
    \item Os grupos hiperbólicos são claramente relativamente hiperbólicos em relação ao subgrupo trivial.
    \item Grupos fundamentais de variedades completas, de volume finito e curvatura estritamente negativa; os subgrupos $H_i$ são os grupos fundamentais de suas cúspides (\cite{Bow12, Farb98}).
    \item Um grupo hiperbólico $G$ é  hiperbólico relativamente a um conjunto de subgrupos $\{H_1, \ldots, H_n\}$ se cada $H_i$ é quasi-convexo e quase malnormal em $G$ (veja \cite{Farb98}). Lembramos que um subgrupo $H \leq G$ é chamado \textit{malnormal} se, para cada $g \in G\setminus H$,  a interseção $|gHg^{-1} \cap H|$ é trivial, enquanto $H$ é dito
    \textit{quase malnormal}\index{subgrupo quase malnormal} se, para cada $g \in G\setminus H$,  a interseção $|gHg^{-1} \cap H|$ é finita. 
    \item Grupos totalmente residualmente livres, conhecidos como grupos limites de Sela; eles têm como subgrupos periféricos $H_i$ uma lista finita de subgrupos abelianos não cíclicos máximos \cite{Dah03}.
\end{enumerate}
\end{examples}

Grupos relativamente hiperbólicos constituem uma área de pesquisa atualmente ativa. Vários resultados provados inicialmente para grupos hiperbólicos podem ser generalizados para o caso relativamente hiperbólico. Nos referimos aos artigos citados acima e trabalhos subsequentes para os resultados e aplicações.

\chapter{Crescimento de grupos e os teoremas de Wolf, Milnor e Gromov}
\label{cap:crescimento}
 \section{A função de crescimento}

\begin{definition} \label{unifdisc}
Um espaço métrico $X$ é dito \textit{uniformemente discreto} \index{espaço uniformemente discreto} se  
 existe uma função  $\varphi : \R_+\to \N$ tal que cada bola fechada $\overline{B(x,r)}\subset X$ contém no máximo $\varphi(r)$ pontos.
\end{definition}

\begin{definition}
Sejam $X$ um espaço métrico uniformemente discreto e $p \in X$ um ponto base. Definimos a \textit{função de crescimento} de $X$  \index{função de crescimento} (com base em $p$) por $$\rho_{X,p}(r)=\card (\overline{B(p,r)}).$$
\end{definition}

 Vamos relembrar quando duas funções reais são ditas assintoticamente iguais.\index{funções assintoticamente iguais} Sejam $f,g: X\subset \R \to \R$. Escrevemos $f \preceq g$ se existir uma constante $K > 0$ tal que $$f(x)\leq Kg(Kx+K)+K, \mbox{ para todo } x\in [0,\infty).$$ Se $f\preceq g$ e $g\preceq f$ então escrevemos  $f\asymp g$ e dizemos que $f$ e $g$ são \textit{assintoticamente iguais}. A relação $\asymp$ é uma  relação de equivalência no conjunto de funções em $X$ (observe que essa noção de equivalência é um pouco mais fina que equivalência usada na Seção~\ref{sec:8-isoperim}.) 

\begin{lemma}[A classe de crescimento é invariante por quasi-isometrias]\label{lemQI} Se $f: X \to Y$ é uma quasi-isometria entre espaços métricos $(X,d_X)$ e $(Y,d_Y)$ uniformemente discretos, com $f(x)=y$ para certos pontos $x \in X$ e $y \in Y$, então $\rho_{X,x}\asymp\rho_{Y,y}.$
\end{lemma}
\begin{proof}
Considere  $f : X \to Y$ e $\overline{f}: Y\to X$,  mergulhos $(L,C)$-quasi-isométricos e quasi-inversos um do outro. Temos  $$L^{-1}d_X(x,x')-C \leq d_Y(f(x),f(x'))\leq L d_X(x,x')+C$$ e $$L^{-1}d_Y(y,y')-C \leq d_X(\overline{f}(y),\overline{f}(y'))\leq L d_Y(y,y')+C.$$ Seja $D =\max \{d_Y(f(x),y),d_X(\overline{f}(y),x)\}$. Então, para cada $r > 0$, obtemos $$f(\overline{B(x,r)})\subset \overline{B(y,Lr+D+C)},\quad \overline{f}(\overline{B(y,r)})\subset \overline{B(x,Lr+D+C)}.$$ 

Com efeito, dado $x' \in \overline{B(x,r)}$, segue da desigualdade triangular que  
\begin{equation*}
\begin{split}
d(f(x'),y)&\leq  d(f(x'),f(x))+d(f(x),y)\\
&\leq Ld(x,x')+C+D\\ 
&\leq LrC+D.
\end{split}
\end{equation*}

Note que, se $f(x')=f(x)$, então $d(x,x')\leq LC$. De fato, $$L^{-1}d(x,x')-C\leq 0=d(f(x),f(x')).$$ O mesmo vale para $\overline{f}.$ Como os espaços $X$ e $Y$ são uniformemente discretos, considerando os mapas $\varphi_X$ e $\varphi_Y$ na Definição~\ref{unifdisc}, obtemos que ambos os  mapas $f$ e $\overline{f}$ tem multiplicidade  menor ou igual a  $m= \max \{\varphi_X(LC),\varphi_Y(LC)\}.$ Segue que 
$$\card (\overline{B(x,r)})\leq m\cdot \card (f(\overline{B(x,r)}))\leq m\cdot \card (\overline{B(y, Lr+C+D)}),$$ 
que implica $\rho_{X,x}(r)\leq \rho_{Y,y}(Lr+(D+C))$, ou equivalentemente, $\rho_{X,x}\preceq\rho_{Y,y}$. Analogamente, podemos mostrar que  $\rho_{Y,y}\preceq\rho_{X,x}$ e portanto, $\rho_{X,x}\asymp\rho_{Y,y}.$ 
\end{proof}

\begin{corollary}
Tem-se $\rho_{X,x}\asymp\rho_{X,x'},$ para todos $x ,x ' \in X$. Com isso, podemos escrever $\rho_{X} := \rho_{X,x}$, qualquer que seja $x\in X$.
\end{corollary}

Se $G$ é  um grupo  finitamente gerado por um conjunto $S$, equipado com a métrica das palavras, usaremos as notações $\rho_{G,S}$ ou $\rho_S$ para a função de crescimento de $G$.

\begin{example}
\label{example:crescZ2}
    Vamos calcular a função de crescimento do grupo $\Z^2$, equipado com os geradores $e_1=(1,0)$ e $e_2=(0,1)$. Tomando como ponto base $x=(0,0)$, temos:

$$\overline{B((0,0),r)} =\{(x,y)\in \mathbb{Z}^2
 \mid |x|+|y|\leq r\}.$$

 Para cada inteiro $k\geq 1$, existem exatamente $4k$ pontos satisfazendo
$
|x|+|y|=k.
$ De fato, para cada ponto no primeiro quadrante, obtemos 3 outros pontos nos demais quadrantes, fazendo variar os sinais. Com isto, basta contar os pontos com $x,y>0$, além dos 4 pontos que estão nos eixos coordenados. Agora, se $x,y>0$ são inteiros tais que $x+y = k$, temos então $k-1$ escolhas possíveis para o par $(x,y)$. Portanto, obtemos $4+4(k-1) = 4k$ pontos no total.

Deste modo,
$$
\begin{aligned}
\rho_{\mathbb Z^2}(r)
&=\card\bigl(\overline{B(p,r)}\bigr)\\
&=1+\sum_{k=1}^{\lfloor r\rfloor}4k\\
&=1+2\lfloor r\rfloor\bigl(\lfloor r\rfloor+1\bigr).
\end{aligned}
$$
Com isto, podemos concluir que $
\rho_{\mathbb Z^2}(r)\sim 2r^2,
$
mostrando que $\mathbb Z^2$ possui crescimento polinomial de grau $2$.
\end{example}

\begin{exercise}\label{exer9.5}
\begin{enumerate}
\item Mais geralmente, com relação ao Exemplo \ref{example:crescZ2}, se $G = \Z^{k}$, mostre que $\rho_S \asymp x^k$ para qualquer conjunto finito de geradores  $S$. Conclua que $\Z^m$ é quasi-isométrico a $\Z^n$ se, e somente se, $m=n$.
\item Prove que, para $n\geq 2$ o grupo $\mathrm{SL}(n,\Z)$ tem crescimento exponencial. 

\item Se $G=F_n$ é o grupo livre de posto $n \geq 2$ e $S$ é um conjunto com $n$ geradores para $G$, mostre que $\rho_S(x) = 1+(q^x-1)\dfrac{q+1}{q-1}$, $q=2n-1$.
\end{enumerate}
\end{exercise}

\begin{proposition}
\label{prop:crescimentogrupos} Seja $G$ um grupo finitamente gerado.
\begin{enumerate}[(1)]
\item Se $S$ e $S'$ são dois conjuntos finitos de geradores de $G$, então $\rho_S \asymp \rho_{S'}$.
\item Se $G$ é infinito, então a restrição $\rho_S\big|_\N$ é estritamente crescente.
\item A função de crescimento é sub-multiplicativa: $$\rho_S(r+t)\leq \rho_S(r)\rho_S(t).$$
\item A função de crescimento satisfaz $\rho_S(r) \preceq 2^r$.
\end{enumerate}
\end{proposition}

\begin{proof}
\begin{enumerate}[(1)]
\item Segue imediatamente do Lema~\ref{lemQI} e do fato  que $(G,d_S)$ e $(G,d_{S'})$ são quasi-isométricos.
\item Sejam $n < m \in \N$. Sendo $G$ infinito, deve existir  $g \in G$ tal que $d=d_S(1,g)\geq m$. Considere uma geodésica   $\gamma:[0,d]\to \cay(G,S)$ ligando $1$ e $g$. O vértice $\gamma(n+1)$ pertence a $\overline{B(1,m)}\setminus \overline{B(1,n)}$. Portanto, $\rho_S(m)\geq \rho_S(n+1)> \rho_S(n)$.
\item Segue imediatamente do fato de que $$\overline{B(1,r+t)}\subset\bigcup_{y\in \overline{B(1,t)}}\overline{B(y,r)}.$$
\item Decorre da existência do epimorfismo $\pi_S: F(S) \to G$.
\end{enumerate}
\end{proof}

A seguir, relembramos o Lema de Fekete, para sequências subaditivas de números reais.

\begin{lemma}[Lema de Fekete]\index{Lema de Fekete}
    Se $(a_n)$ é uma sequência de números reais não negativos tais que $a_{i+j} \leq a_ i + a_j$ para  $i, j \geq 1$, então  existe $\lim_{n \to \infty} \frac{a_n}{n}$ e ele é igual a $\inf  \{\frac{a_n}{n}\}$.
\end{lemma}

A propriedade (3) na proposição acima garante que a função $\log \rho_S$ é subaditiva. Portanto, pelo Lema de Fekete, existe o limite $$\lim_{n\to \infty}\dfrac{\log \rho_S(n)}{n}.$$

Definimos assim a \textit{constante de crescimento} de um grupo $G$ finitamente gerado pelo conjunto $S$: $$\gamma_S = \lim_{n\to \infty}(\rho_S(n))^{\frac{1}{n}}.$$ 

A propriedade (2) implica que $\rho_S(n)\geq n$. Portanto, $\gamma_S \geq 1$ se $G$ for infinito. 

\begin{definition} Se $\gamma_S > 1$,  dizemos que $G$ tem \textit{crescimento exponencial}\index{crescimento exponencial}. Nesse caso, o número $\gamma_S$ é chamado de \textit{ordem de crescimento exponencial} ou \textit{entropia} de $(G, S)$.
Caso $\gamma_S = 1$, dizemos que $G$ tem \textit{crescimento sub\-exponencial}.
\end{definition}

Observamos que a propriedade (1) na Proposição \ref{prop:crescimentogrupos} implica que a definição acima não depende da escolha de geradores. Por exemplo, pelo Exercício~\ref{exer9.5}.4 temos que o grupo livre não abeliano $F_n$ tem crescimento exponencial de ordem $2n-1$.

\begin{proposition}
\begin{enumerate}
\item[(1)] Se $H \leq G$ são finitamente gerados, então $\rho_H \preceq \rho_G$;
\item[(2)] Se $H$ é um subgrupo de índice finito de $G$ então $\rho_H \asymp \rho_G$;
\item[(3)] Se $N$ é um subgrupo normal de $G$ então $\rho_{G/N}\preceq\rho_G$;
\item[(4)] Se $N$ é um subgrupo normal finito de $G$ então $\rho_{G/N}\asymp\rho_G$.
\end{enumerate}\label{prop:cresc}
\end{proposition}
\begin{proof}
(1) Seja $X$ um conjunto finito que gera $H$, e suponha que $ S\supset X$ gera $G$. Então $\cay(H,X)$ é um subgrafo de $\cay(G,S)$ e 
$$\dist_X(1,h) \geq \dist_S(1,h)\quad \forall h\in H.$$
Em particular, temos $\overline{B(1,r)}_{\cay(H,X)} \subseteq \overline{B(1,r)}_{\cay(G,S)}$ e o resultado segue.

(3) Seja $S$ um conjunto finito de geradores de $G$ tal que $S^{-1} = S$ e $1\not\in S$. Então o conjunto $\overline{S} := \{ sN \mid  s\in S,\ s\not\in N\}$ gera $G/N$. 
A projeção $\pi: G \to G/N$ mapeia a bola $\overline{B(1,r)}$ de $\cay(G,S)$ na bola $\overline{B(1,r)}$ de $\cay(G/N, \overline{S})$, o que fornece o resultado desejado.

Os itens (2) e (4) decorrem imediatamente do Lema~\ref{lemQI} e do Teorema de Milnor--Schwarz.

\end{proof}

Agora vamos considerar algumas propriedades específicas de grupos com  \textit{crescimento polinomial},\index{crescimento polinomial} i.e. quando existem constantes $a,b \in \R$ tais que $\rho_{G,S}(x) \leq ax^b$ para todo $x$. Se $G$ tem crescimento polinomial, seu \textit{grau} de crescimento é deﬁnido por 
\begin{eqnarray*}
    d(G) &=& \inf \{ b \mid \text{existe $a$ tal que } \rho_{G,S}(x) \leq a x^b \}\\
    & =& \limsup_{x\to\infty} \frac{\log \rho_{G,S}(x)}{\log x}.
\end{eqnarray*}

\begin{proposition}\label{prop:cresc_pol}
Seja $G$ um grupo  com  crescimento polinomial. 
\begin{enumerate}[(1)]
    \item Se um grupo $H$ satisfaz $\rho_H \asymp \rho_G$, então $d(H)= d(G)$. Em particular, o grau de crescimento não depende da escolha dos geradores.
    \item Se $H \leq G$ tem índice infinito e é finitamente gerado, então $d(H) \leq d(G) - 1$. 
    \item Se $N\triangleleft G$ for infinito e finitamente gerado, então $d(G/N ) \leq d(G) - 1$.
\end{enumerate}
\end{proposition}

\begin{proof}
A prova de (1) segue imediatamente da definição da equivalência assintótica e dos resultados anteriores nesta seção.

Seja $H \leq G$ tal que $|G : H| = \infty$. Considere $X = \{x_1,\ldots, x_m\}$ um conjunto de geradores de $G$, que assumimos conter também um conjunto de geradores de $H$. Denote por $Hu_1,\ldots, Hu_n$ algumas classes laterais distintas de $H$. O conjunto $K := Hu_1 \cup \ldots \cup Hu_n$ não é fechado sob multiplicação à direita pelos geradores de $G$ e seus inversos, pois caso contrário, teríamos $K = G$ e $H$ teria índice finito. Assim, um dos elementos $u_i x_j^{\pm 1}$ representa uma nova classe lateral. Começando com $u_1 = x_1$, este argumento mostra que, para cada $n\in \N$, podemos encontrar $n$ classes distintas de $H$ que são representadas por elementos de comprimento no máximo $n$, então podemos supor que $l(u_i) \leq i$. Se $z_1, \ldots, z_k$ são elementos de comprimento no máximo $n$ de $H$, então os elementos $z_i u_j$ são todos diferentes entre si. Portanto, $\rho_G(2n) \geq n \rho_H(n)$ e, se o crescimento for polinomial, então $d(G) \geq d(H) + 1$.

Agora seja $N$ um subgrupo normal, infinito  e finitamente gerado de $G$. Tome um conjunto finito $X$  de geradores de $G$, contendo um conjunto $Y$ de geradores de $N$. Temos
$$\rho_{G,X}(2n) \geq \rho_{G/N,X/N}(n) \rho_{N,Y}(n) \geq n \rho_{G/N}(n),$$ 
que implica que $d(G) \geq d(G/N) + 1$.
\end{proof}
Em seu livro, \cite[Capítulos~14 e 16]{Mann} discute exemplos e resultados referentes ao crescimento de produtos amalgamados e extensões HNN de grupos. No Capítulo~14, são apresentadas fórmulas explícitas para as funções de crescimento de produtos amalgamados, produtos livres e extensões HNN, em termos das funções de crescimento dos grupos envolvidos. No Capítulo~16, são discutidos resultados a respeito do crescimento de grupos que podem ser expressos a partir de tais construções.

\section{Grupos solúveis, nilpotentes e policíclicos}
\label{sec:sol-nilp-pol}

Nesta seção, trabalharemos com importantes classes de grupos, co\-mo grupos solúveis, nilpotentes e policíclicos. Lembramos agora suas definições e algumas propriedades básicas. Mais detalhes podem ser encontrados em livros sobre teoria de grupos, como  \cite{lang}.

\begin{definition}
Um grupo é dito \textit{solúvel} \index{grupo solúvel} se existem finitos subgrupos 
$G = G_0 \unrhd \cdots \unrhd G_{n-1} \unrhd G_n = \{e\}$, tais que $G_{i}/G_{i+1}$ é abeliano, para cada $i=0,\ldots, n-1$.
\end{definition}

Alternativamente, podemos dizer que $G$ é solúvel se sua \textit{série derivada}\index{serie derivada@série derivada} $G\unrhd G^{(1)}\unrhd G^{(2)} \unrhd \cdots$, onde para cada $i\geq 1$, o grupo $G^{(i+1)}$ é o comutador de $G^{(i)}$, eventualmente chega ao grupo trivial $\{e\}$. Lembramos que o comutador de um grupo $G$ é o subgrupo $[G,G]$, gerado pelo conjunto $\{[g,h] = ghg^{-1}h^{-1} \mid g,h \in G\}$. A condição $G^{(1)} = \{e\}$ significa que o grupo $g$ é abeliano. No caso em que  $G^{(2)} = \{e\}$ o grupo é chamado \textit{metabeliano}\index{grupo metabeliano}.

Subgrupos e  quocientes de grupos solúveis são solúveis, bem como extensões de um grupo solúvel por um outro grupo solúvel.

\begin{definition}
Dado um grupo $G$, definimos indutivamente, para todo  $n\in \N$, subgrupos $C^n(G)$ por $C^0(G) := G$ e, para cada $n$, $C^{n+1}(G) :=[C^n(G), G]$.
A sequência 
$$C^1(G)\unrhd C^2(G) \unrhd \cdots \unrhd C^n(G) \unrhd \cdots $$ é chamada \textit{série central inferior}\index{serie central inferior@série central inferior} de $G$. Dizemos que $G$ é  \textit{nilpotente} de classe $n$ se existe $n\in \N$ tal que $C^n(G)$ é trivial. \index{grupo nilpotente}
\end{definition}

Seja $S$ um conjunto de geradores para $G$. Um
comutador iterado à esquerda com comprimento $n\geq 2$ em $S$ é um comutador do tipo $[[\ldots [x_1, x_2]\ldots, x_{n-1}], x_n]$, onde $x_1, \ldots , x_n$ estão em $S$. 

\begin{exercise}\label{exerc:fgcommutators}
    Se $G$ é um grupo nilpotente finitamente gerado por um conjunto $S$, então para todo $i\geq 2$, o subgrupo $C^iG $ é gerado pelos comutadores iterados à esquerda em elementos de $S$ de comprimento $k\geq i$, juntamente com seus inversos. Note que como $G$ é nilpotente, basta tomar comutadores de comprimento $k \leq n$, onde $n$ é a classe de nilpotência de $G$. Portanto, cada $C^iG$ é finitamente gerado.
\end{exercise}

Podemos também definir grupos nilpotentes em termos de sua \textit{série central superior}:\index{serie central superior@série central superior} 
$$Z_0(G) = \{e\} \unlhd  Z_1(G) \unlhd  \cdots \unlhd Z_i(G) \unlhd  Z_{i+1}(G) \unlhd \cdots, $$
onde os subgrupos normais $Z_i(G)\triangleleft G$ são também definidos indutivamente, da seguinte forma: se $Z_i(G)$ está definido, e $\pi_i: G \to G/Z_i(G)$ é o mapa quociente, então 
$$Z_{i+1}(G) := \pi_i^{-1} (Z(G/Z_i(G))).$$
Estes grupos são normais em $G$ pelo fato de serem a imagem inversa pelo mapa quociente de um subgrupo normal do quociente em questão.

\begin{proposition} \label{prop:nilpotente}
Um grupo $G$ é \textit{nilpotente} se, e somente se, sua série central superior é finita.
\end{proposition}

\begin{exercise}
    Demonstre a Proposição \ref{prop:nilpotente}.
\end{exercise}

Subgrupos, grupos quocientes e produtos diretos de grupos nilpotentes são nilpotentes, mas uma extensão de um grupo nilpotente por outro não precisa ser nilpotente. Este último fato pode ser visto considerando o menor grupo não abeliano, o grupo simétrico sobre três letras $S_3$. O mesmo exemplo mostra que os grupos solúveis não precisam ser nilpotentes, embora fique claro pela definição que os grupos nilpotentes são solúveis.

\begin{example}
    \begin{enumerate}
        \item Todo grupo abeliano é nilpotente, já que o seu comutador é trivial.
    \item O grupo de Heisenberg \index{grupo de Heisenberg discreto}
    $$H = \langle x, y, z \mid [x, z] = 1,\, [y, z] = 1,\, [x, y] = z\rangle$$ 
    é nilpotente. De fato, por conta das relações em $H$, temos $C^1(H) = [H,H]\cong \langle z\rangle_H$ e  $C^2(H) \cong [H, \langle z\rangle_H] = \{e\}$.
    \end{enumerate}
\end{example}
Observamos que grupos virtualmente nilpotentes não precisam ser nilpotentes ou solúveis.
Por exemplo, todo grupo finito é virtualmente nilpotente, mas nem todo grupo finito é nilpotente. Um exemplo é o caso dos  grupos alternados $A_n$, os quais são simples e não abelianos quando $n \geq 5$ e, portanto, não são solúveis \cite[Teorema I.5.5]{lang}.

\begin{exercise}
    Mostre que o produto semidireto
$\Z^2 \rtimes_{\alpha} \Z$, onde $\alpha: \Z \to  \mathrm{Aut}(\Z^2)$ é dado pela ação da matriz $ \left(\begin{array}{cc}
   1 & 1\\
   1 & 2 
   \end{array}\right)$
em $\Z^2 $, é um grupo solúvel que não é  virtualmente nilpotente.
\end{exercise}

\begin{example}
Se $G$ é um grupo finitamente gerado e nilpotente, seu cone assintótico é  homeomorfo a $\R^m$, onde $m$ é a soma  dos postos da  série central inferior de $ G$ (veja \cite[Seção~5]{point1995groups}).
\end{example}

A seguir definiremos uma classe de grupos solúveis que generaliza a classe dos grupos nilpotentes finitamente gerados, os grupos policíclicos. Esses grupos admitem uma série subnormal cujos quocientes são cíclicos, propriedade que lhes confere uma estrutura algébrica suficientemente rígida para permitir diversas demonstrações por indução.

\begin{definition}
Um grupo $G$ é dito \textit{policíclico} \index{grupo policíclico} se existir uma série subnormal descendente 
\begin{equation}
   G = N_0 \rhd N_1\rhd \cdots \rhd N_n = \{e\},
   \label{eq:poli}
\end{equation}
tal que $ N_i/N_{i+1}$ é cíclico para todo $i = 0, \ldots, n-1$.
\end{definition}
Uma série como acima é chamada \textit{série cíclica},\index{serie cíclica@série cíclica} e seu comprimento é o número de subgrupos não triviais na série. Esse número é menor ou igual a $n$. O \textit{comprimento} $\ell(G)$ de um grupo policíclico é o menor comprimento de uma série cíclica de $G$.
Se, além disso, $N_i/N_{i+1}$ é cíclico infinito para todo $ i > 0$, então o grupo $G$ é chamado \textit{poli-$C_{\infty}$} \index{grupo poli-$C_{\infty}$}
e a série é chamada \textit{série}-$C_{\infty}$.
Declaramos que o grupo trivial é  poli-$C_{\infty}$.

\begin{proposition}
\label{prop:polic}
  \begin{enumerate}[(1)]
      \item  Um grupo policíclico tem a propriedade de \textit{geração limitada}:  se $G$ é um grupo com série cíclica  \eqref{eq:poli} de comprimento $n$  e $t_i$  é tal que $t_iN_{i+1}$ é um gerador de $N_i/N_{i+1}$, então todo  $g \in  G$ pode ser escrito como  $g = t_1^{k_1} \ldots t^{k_n}_n $, para  $k_1, \ldots, k_n \in \Z$.\index{grupo limitadamente gerado}
\item Um grupo de torção policíclico é finito.
\item  Todo subgrupo de um grupo policíclico é policíclico e, portanto, finitamente gerado.
\item Se $N$ é um   subgrupo normal em um grupo policíclico $G$, então $G/N$ é policíclico.
\item  Se $N \triangleleft G$ e ambos $N$ e $G/N$ são policíclicos, então $G$ é policíclico.
\item As propriedades (3) e (5) valem com ``poli-$C_{\infty}$'' no lugar de ``policíclico'', mas a propriedade (4) não.
  \end{enumerate} 
\end{proposition}

\begin{exercise}
    Prove a Proposição \ref{prop:polic}.
\end{exercise}

\begin{thm}
    Todo grupo nilpotente finitamente gerado é policíclico. 
\end{thm}

\begin{proof}
   A prova pode ser feita por indução na classe de nilpotência $n$ de $G$. Se $n=1$, então $G$ é abeliano. Logo, $G$ é policíclico, devido ao Teorema Fundamental dos grupos abelianos finitamente gerados.

   Assumindo o resultado verdadeiro para grupos finitamente gerados e  nilpotentes de classe $n$, e tomando $G$ finitamente gerado e nilpotente de classe $n+1$, observe que o grupo $C^1G$ é finitamente gerado, pelo Exercício \ref{exerc:fgcommutators}, e nilpotente de classe $n$. Pela hipótese de indução, $C^1G$ é policíclico. 

   Sendo a abelianização $G/C^1G$ um grupo abeliano finitamente gerado, obtemos que $G/C^1G$ é policíclico. Pelo item $(5)$ da Proposição \ref{prop:polic}, concluímos que $G$ é policíclico.
\end{proof}

Uma observação a respeito do resultado anterior é que a condição de que o grupo seja finitamente gerado é essencial. Por exemplo, o grupo $\displaystyle \bigoplus_{n=1}^{\infty} \mathbb{Z}$ é abeliano (portanto, nilpotente), mas não é policíclico. 

Todo grupo policíclico é solúvel. De fato, isto segue imediatamente por indução sobre o comprimento cíclico de um grupo policíclico $G$ e da propriedade que a extensão de um grupo solúvel por um outro grupo solúvel é solúvel. Existe também uma relação na direção oposta, válida para a subclasse de grupos solúveis noetherianos, da qual tratamos a seguir.

\begin{definition}\label{def:noetherian}
Um grupo é dito \textit{noetheriano}\index{grupo noetheriano}, ou satisfaz a \textit{condição máxima}, se para cada sequência crescente de subgrupos
\begin{equation}\label{eq:noeth}
    H_1 \leq H_2 \leq \cdots \leq H_n \leq  \cdots
\end{equation}
existe $N$ tal que $H_n = H_N$ para todo $n \geq N$.    
\end{definition}

\begin{proposition}\label{prop:noet1}
Um grupo $G$ é noetheriano se, e somente se, todo subgrupo de $G$ é finitamente gerado.    
\end{proposition}

\begin{proof}
Suponha que $G$ seja um grupo noetheriano e seja $H \leq G$ um subgrupo que não é finitamente gerado. Escolha $h_1 \in H \setminus \{1\}$ e seja $H_1 = \langle h_1 \rangle$. Indutivamente, suponha que
$$  H_1 \leq H_2 \leq \cdots \leq H_n $$
é uma sequência estritamente crescente de subgrupos finitamente gerados de $H$. Escolha $h_{n+1} \in H \setminus H_n$ e defina $H_{n+1} := \langle H_n, h_{n+1} \rangle$. Temos, portanto, uma sequência infinita estritamente crescente de subgrupos de $G$, contradizendo a suposição de que $G$ é noetheriano.

Por outro lado, suponha que todos os subgrupos de $G$ sejam finitamente gerados e considere uma sequência crescente de subgrupos como em \eqref{eq:noeth}. Então $H = \displaystyle\bigcup_{n\geq 1} H_n$ é um subgrupo, portanto gerado por um conjunto finito $S$. Existe $N$ tal que $S \subseteq H_N$, portanto $H_N = H = H_n$ para todo $n \geq N$.    
\end{proof}

\begin{proposition}\label{prop:noet2}
Um grupo solúvel é policíclico se, e somente se, for noetheriano.    
\end{proposition}

\begin{proof}
A ida segue imediatamente das Proposições ~\ref{prop:polic}(3) e ~\ref{prop:noet1}. 

Seja $G$ um grupo noetheriano e solúvel. Provamos por indução no comprimento derivado $k$ que $G$ é policíclico. Para $k = 1$ o grupo é abeliano e, como, por hipótese, $G$ é finitamente gerado, o grupo é policíclico. Suponha que a afirmação seja verdadeira para $k$ e considere um grupo solúvel $G$ de comprimento derivado $k + 1$. O subgrupo comutador $G' \leq G$ também é noetheriano, e ele também é solúvel de comprimento derivado $k$. Portanto, pela hipótese de indução, $G'$ é policíclico. A abelianização $G_{ab} = G/G'$ é finitamente gerada (porque $G$ é, por hipótese), portanto é policíclica. Segue-se que $G$ é policíclico pela Proposição~\ref{prop:polic}(5).
\end{proof}

\begin{remark}
Existem grupos noetherianos que não são virtualmente policíclicos, por exemplo os \textit{monstros de Tarski}, que são grupos $G$ finitamente gerados, tais que todo subgrupo próprio de $G$ é cíclico (veja \cite{Ol91}).
\end{remark}

Um exemplo instrutivo de grupo solúvel é o \textit{grupo acendedor de lâmpadas}\index{grupo acendedor de lâmpadas}. Este grupo é o produto entrelaçado $G = \Z_2 \wr \Z$ no sentido da Definição~\ref{def:wr prod}.


Os produtos entrelaçados de grupos finitamente gerados também são finitamente gerados, portanto, o grupo acendedor de lâmpadas é finitamente gerado. Por outro lado, ele possui as seguintes propriedades:
\begin{enumerate}[(1)]
\item Nem todos os subgrupos de $G$ são finitamente gerados. Por exemplo, o subgrupo $\displaystyle\bigoplus_{n\in\Z} \Z_2$ de $G$ não é finitamente gerado.

\item O grupo $G$ não é virtualmente livre de torção: Para qualquer subgrupo de índice finito $H \le G$, a interseção $H \bigcap \displaystyle\bigoplus_{n\in\Z} \Z_2$ tem índice finito em $\displaystyle\bigoplus_{n\in\Z} \Z_2$; em particular, esta intersecção é infinita e contém elementos de ordem~$2$.

Observe que tanto (1) como (2) implicam que o grupo acendedor de lâmpadas não é policíclico.

\item O subgrupo comutador $G'$ do grupo acendedor de lâmpadas coincide com o seguinte subgrupo de $\displaystyle\bigoplus_{n\in\Z} \Z_2$:
$$ C = \{ f: \Z \to \Z_2 \mid \mathrm{Supp}(f) \text{ tem cardinalidade par}\}, $$
onde $\mathrm{Supp}(f) =  \{n\in\Z \mid f(n) = 1 \}$.
\end{enumerate}

Em particular, $G'$ não é finitamente gerado e o grupo $G$ é metabeliano (já que $G'$ é abeliano).

\begin{exercise}
Verifique a descrição do subgrupo comutador de $G$ dada em (3).  
\end{exercise}

\begin{exercise}
Prove que o grupo  acendedor de lâmpadas tem crescimento exponencial.
\end{exercise}

\section{A conjectura de Milnor e o teorema de Wolf}

A direção principal para o desenvolvimento da teoria do crescimento de grupos foi definida pela seguinte conjectura (ou problema) de J.~Milnor e J.~Wolf \cite{Mil68b, Wolf68}:

\begin{conjecture}[Milnor, 1968]\label{conj-Milnor}
O crescimento de um grupo finitamente gerado é polinomial (i.e. $\rho_S(x)\preceq x^d$ para qualquer conjunto de geradores $S$) ou exponencial (i.e. $\gamma_S > 1$ para todo $S$).
\end{conjecture}

O primeiro resultado positivo em relação à conjectura foi o Teorema de Wolf \cite{Wolf68}:

\begin{thm}[Wolf, 1968]\label{thm-Wolf}
Um grupo policíclico é virtualmente nilpotente ou tem crescimento exponencial. 
\end{thm}

Para grupos virtualmente nilpotentes, o crescimento foi calculado explicitamente de forma independente por H.~Bass \cite{Bass72} e Y.~Guivarc'h \cite{Gui70, Gui73}:

\begin{thm}[Bass--Guivarc'h]
\label{thm:BG}
Seja $G$ um grupo finitamente gerado nilpotente de classe $k$. Para qualquer $i\geq 1$, seja $m_i \in \N$ tal que 
$$C^iG/C^{i+1}G \cong \Z^{m_i} \times F_i,$$ 
onde $C^iG = [G, C^{i-1}G]$, $C^1G = G$ e $F_i$  é finito abeliano. 

Nesse caso, $\rho_G(n) \asymp n^d$ com $d = \displaystyle\sum_{i=1}^k im_i$.
\end{thm}

Nosso próximo objetivo é fornecer as provas dos Teoremas~\ref{thm-Wolf} e \ref{thm:BG}, o que exigirá alguns resultados auxiliares. Começamos estudando os automorfismos de grupos abelianos.

\begin{lemma}\label{lema9-1}
Seja $v = (v_1, \ldots, v_n)\in \Z^n$ tal que $\mathrm{mdc}(v_1, \ldots, v_n) = 1$. Então $\Z^n/\langle v \rangle$ é um grupo livre abeliano de posto $n-1$. Além disso, existe uma base $\{ v, y_1, \ldots, y_{n-1}\}$ de $\Z^n$ tal que $\{ y_1 + \langle v \rangle, \ldots, y_{n-1} + \langle v \rangle\}$ é uma base de $\Z^n/\langle v \rangle$.
\end{lemma}

\begin{exercise}
    Prove o Lema \ref{lema9-1}.
\end{exercise}

\begin{lemma}\label{lema9-2}
Suponha que $M \in \GL(n, \Z)$ tem todos os autovalores iguais a $1$. Então existe uma cadeia finita de subgrupos 
$$ \{ 1 \} = \Lambda_0 \subset \Lambda_1 \subset \ldots \ \subset \Lambda_{n-1} \subset \Lambda_n = \Z^n$$
tal que $\Lambda_i \cong \Z^i$, $\Lambda_{i+1}/\Lambda_i \cong \Z$, $M(\Lambda_i) = \Lambda_i$ e $M$ age trivialmente em $\Lambda_{i+1}/\Lambda_i$ para todo $i \ge 0$.
\end{lemma}

\begin{proof}
Como $M$ tem autovalor $1$, existe um vetor $v$ in $\Z^n$, $v =  (v_1, \ldots, v_n)$ tal que  $\mathrm{mdc}(v_1, \ldots, v_n) = 1$ e $Mv = v$. Então $M$ induz um automorfismo de $\Z^n/\langle v \rangle \cong \Z^{n-1}$ e a sua matriz tem todos os autovalores iguais a $1$ (basta escrever a matriz na base $\{ v, y_1, \ldots, y_{n-1}\}$ do Lema~\ref{lema9-1}). Portanto, existe $w + \langle v \rangle \in \Z^n/\langle v \rangle $ com $$M(w + \langle v \rangle) = w + \langle v \rangle.$$ 
Conclui-se a prova por indução.
\end{proof}

\begin{lemma}\label{lema9-3}
Seja $M \in \GL(n, \Z)$ tal que o valor absoluto de todos os seus autovalores complexos é igual a $1$. Então todos os autovalores de $M$ são raízes da unidade. 
\end{lemma}

\begin{proof}
Sejam $\lambda_1, \ldots, \lambda_n$ os autovalores de $M$, contados com multiplicidade. Note que  $$\displaystyle \tr(M^k) = \sum_{i = 1}^n \lambda_i^k.$$

Como $M \in \GL(n, \Z)$, $\tr(M^k) \in \Z$ para todo $k \in \Z$. Os vetores $v_k = (\lambda_1^k, \ldots, \lambda_n^k)$ pertencem ao grupo $(\mathbb{S}^1)^n =: G$, que é um grupo de Lie compacto. Por isso existe uma subsequência $(v_{k_i})$ convergente, e então 
$$ \lim_{i \to \infty}(v_{k_i}v_{k_{i-1}}^{-1}) = 1 \text{ em } G.$$
Sejam $m_i = k_{i+1} - k_i$. Temos $\displaystyle\lim_{i \to \infty} \lambda_j^{m_i} = 1$ para $j = 1,\ldots,n$, o que implica que $s_i = \displaystyle\sum_{j = 1}^n \lambda_j^{m_i} \to n$ quando $i \to \infty$.

Sabendo que  $\tr(M^{m_i})\in \Z$, concluímos que $s_i$ estabiliza para $i$ suficiente grande. Isto é, para todo $i\ge i_0$, temos $s_i$ constante, e então
$$\sum_{j = 1}^n \re(\lambda_j^{m_i}) = n,\ i \geq i_0.$$
Isso implica que  $\re(\lambda_j^{m_i}) = 1$ para todo $j$ e $i\geq i_0$. Portanto,
$\lambda_j^{m_i} = 1$ e os $\lambda_j$ são raízes da unidade.
\end{proof}

\begin{lemma}\label{lema9-4}
Suponha que $M \in \GL(n, \Z)$ tenha um autovalor $\lambda$ tal que $|\lambda| \geq 2$. Então existe um vetor $v \in \Z^n$ tal que a função $F: \displaystyle \oplus_{i \in \N} \Z_2 \to \Z^n$, definida em sequências $(s_i)_{i \in \N}$, com $s_i = 0$ ou $1$ e $s_i = 0$ para $i$ suficiente grande, por
$$ (s_i) \to s_0v + s_1Mv + s_2 M^2 v + \ldots$$
é injetiva. 
\end{lemma}

\begin{proof}    
Considere a aplicação $\phi: \Z^n \to \Z^n$ dada por $\phi(v) = Mv$. O dual $\phi^*$ tem matriz $M^{T}$ em base canônica dual, e então também tem $\lambda$ como autovalor. Portanto, existe uma forma linear $f: \C^n \to \C$ tal que $\phi^*(f) = f\circ \phi = \lambda f$.

Seja $v \in \Z^n \setminus \ker(f)$. 
Suponha que a aplicação $F$, definida a partir de $v$, não é injetiva, i.e. que existe $(t_i)_{i \in \N}$, $t_i \in \{0,1\}$, tal que
$$ t_0v + t_1Mv + t_2 M^2 v + \ldots = 0$$
e $t_i = 0$ para $i > N$.

Temos $M^Nv = r_0v + r_1Mv + \ldots + r_{N-1}M^{N-1}v$, com $r_i \in \{-1,0\}$. Aplicando $ f$ a ambos os lados desta fórmula, obtemos
$$\lambda^N f(v) = (r_0 +r_1\lambda + \ldots + r_{N-1}\lambda^{N-1})f(v);$$
$$ |\lambda|^N \le \sum_{i=0}^{N-1} |\lambda|^i = \frac{|\lambda|^N - 1}{|\lambda| - 1} \le |\lambda|^N -1,$$
uma contradição, o que mostra que $F$ é injetiva.
\end{proof}

Voltamos agora ao estudo dos grupos policíclicos e nilpotentes.

\begin{proposition}\label{prop9-1}
Sejam $G$ um grupo finitamente gerado nilpotente e $\phi: G\to G$ um automorfismo. Então o produto semidireto $S = G \rtimes_{\varphi}\Z$ (que é policíclico) é
\begin{enumerate}[(1)]
    \item virtualmente nilpotente; ou
    \item tem crescimento exponencial. 
\end{enumerate}
\end{proposition}

\begin{proof} O automorfismo $\phi$ preserva os subgrupos 
$$C^iG = [G, C^{i-1}G],\ C^1G = G,$$ 
que formam a série central inferior, então $\phi$ induz automorfismos $\phi_i$ de grupos abelianos $B_i = C^iG / C^{i+1}G \cong F_i\times\Z^{m_i}$, onde
\begin{quote}
    $\phi_i$ induz automorfismos $\psi_i$ de $\mathrm{Tor}(B_i) = F_i$ e $\overline{\phi}_i$ de $B_i/\mathrm{Tor}(B_i) \cong \Z^{m_i}$, onde $\overline{\phi}_i$ é representado por alguma matriz $M_i\in \GL(m_i, \Z)$. 
\end{quote}

Temos dois casos a considerar:

\medskip

\noindent
\textbf{1.} Todas as matrizes $M_i$ tem todos autovalores de módulo $1$. Neste caso, e como $M_i\in \GL(m_i,\Z)$, pelo Lema~\ref{lema9-3} todos os autovalores são raízes da unidade. 

Então existe $N \in \N$ tal que a matriz $M_i$ correspondente a $(\bar{\phi}_i)^N$ tem todos os autovalores iguais a $1$ e, além disso, as  aplicações correspondentes $\psi_i: F_i \to F_i$ são todas iguais à identidade.

Podemos substituir $S$ por 
$G \rtimes_{\varphi}(N\Z) \cong G \rtimes_{\varphi^N}\Z,$
de índice finito em $S$ e que satisfaz as condições do parágrafo anterior.

Pelo Lema~\ref{lema9-2}, aplicado a cada $\overline{\phi}_i$, temos uma extensão da série central inferior:
$$\{1\} = H_n \subset H_{n-1}\subset \ldots \subset H_1\subset H_0 = G,$$
tal que $H_i/H_{i+1}$ é cíclico, $\phi$ preserva cada $H_i$ e induz em $H_i/H_{i+i}$ uma ação trivial.  

Seja $\Z = \langle t \rangle$. Temos, para todo $g \in G$, $tgt^{-1} = \phi(g)$. Como $\phi$ age trivialmente em $H_i/H_{i+1}$, 
$$t^k (hH_{i+1})t^{-k} = hH_{i+1},\ \forall h \in H_i\setminus H_{i+1}.$$
Ou, equivalentemente, $[h,t^k] \in H_{i+1}$ para cada $h \in H_i\setminus H_{i+1}$ e cada $k \in \Z$. 

Considere a sequência exata 
$$ 1 \to G \to S \to \Z \to 1,$$
onde $C^2S = [S,S] \to 1$ em $\Z$. Assim, 
$$C^2S \leq G.$$

Vamos mostrar que, para cada $i\geq 0$, $[H_i, S] \leq H_{i+1}$. Para isto, sejam $h \in H_i$ e $s\in S$. Note que $s = gt^k$ com $g\in G$ e $k \in \Z$. Podemos escrever $$[h, gt^k] = [h,g]\cdot [g, [h,t^k]] \cdot [h, t^k].$$

    Note que $[h, t^k] \in H_{i+1}$. Como a série $\{ H_i \}$ é uma extensão de $\{ C^iG \}$, deve existir $r \geq 1$ tal que $C^rG \geq H_i \geq H_{i+1} \geq C^{r+1}G$.
    Assim, se $h \in H_i$, então $[h,g] \in C^{r+1}G \leq H_{i+1}$.

    Similarmente, $[h, t^k] \in H_{i+1} \leq C^rG$ implica que
    $$[h, [h, t^k]] \in C^{r+1}G \leq H_{i+1.}$$

Por indução, obtemos que $C^{i+2}S \leq H_i$ para todo $i\geq 1$. Em particular, $C^{n+2}S = \{1\}$ e então o grupo $S$ é nilpotente. 

\medskip

\noindent
\textbf{2.} Agora assuma que existe $i$ tal que $M_i$ tem um autovalor com  valor absoluto maior que $1$. 

Podemos supor que este autovalor tem valor absoluto $\geq 2$ (caso contrário, podemos substituir $\phi$ por $\phi^N$, com $N$ suficiente grande, e $S$ por $G \rtimes_{\varphi^N}\Z$, que tem índice finito em $S$). 

Aplicando o Lema~\ref{lema9-4} a $\Z^{m_i}$ e $M_i$, existe $g \in C^iG$ tal que $g \to v \in \Z^{m_i}$ e as sequências $(s_j) \in \displaystyle\oplus_{j \geq 0} \Z_2$ definem elementos $s_0v + s_1Mv + \ldots + s_nM^nv$ \textit{distintos} em $\Z^{m_i}$. Então suas pré-imagens por $p: C^iG \to B_i/\mathrm{Tor}(B_i)$, dadas por $g^{s_0}(tgt^{-1})^{s_1}\ldots (t^n g t^{-n})^{s_n}$ também são distintas.

Considere os elementos que correspondem a sequências $(s_j)$ com $s_j = 0$ para $j \geq k+1$. Temos $2^{k+1}$ tais elementos, e  seus comprimentos no grafo de Cayley de $S$ são $\leq 3k+1$ (se $t$ e $g$ são geradores). Então o grupo $S$ tem crescimento exponencial. 
\end{proof}

\begin{proposition}\label{prop9-2}
Todo grupo policíclico contém um subgrupo normal de índice finito que é poli-$C_{\infty}$.
\end{proposition}

\begin{proof} 
Seja $G$ um grupo policíclico. Usamos indução no comprimento $n$ da série descendente subnormal mais curta de $G$. Para $n = 1$, o grupo $G$ é cíclico e a afirmação é obviamente verdadeira. Suponha que seja verdade para $n$ e considere um grupo policíclico $G$ tendo uma série subnormal

$$ G = N_0 \triangleright N_1 \triangleright \ldots \triangleright N_n \triangleright N_{n+1} = \{e\},$$
tal que $ N_i/N_{i+1}$ é cíclico para todo $i\geq 0$.

A hipótese indutiva diz que $N_1$ contém um subgrupo normal $H$ de índice finito que é poli-$C_{\infty}$. A Proposição~\ref{prop:fin_many_subgrps} implica que $H$ tem um subgrupo de índice finito $H_1$ que é normal em $G$. Um subgrupo de um grupo poli-$C_{\infty}$ é também poli-$C_{\infty}$, então $H_1$ é poli-$C_{\infty}$.

Se $G/N_1$ é finito, então $H_1$ é o subgrupo procurado.

Suponha que $G/N_1$ seja cíclico infinito. Então o grupo $K = G/H_1$ contém o subgrupo normal finito $F = N_1/H_1$, e $K/F$ é isomorfo a $\Z$. Uma sequência exata curta $1 \to F \to K \to \Z \to 1$ sempre cinde, então $K$ é um produto semidireto de $F$ com um subgrupo cíclico infinito $\langle x \rangle$. A conjugação por $x$ define um automorfismo de $F$, e como $\mathrm{Aut}(F)$ é finito, existe $r$ tal que a conjugação por $x^r$ é a identidade em $F$. Portanto, $F\langle x^r \rangle$ é um subgrupo de índice finito em $K$ e é um produto direto de $F$ e $\langle x^r \rangle$. Concluímos que $\langle x^r \rangle$ é um subgrupo normal de índice finito de $K$. Temos que $\langle x^r \rangle = G_1/H_1$, onde $G_1$ é um subgrupo normal de índice finito em $G$, e $G_1$ é poli-$C_{\infty}$ já que $H_1$ é poli-$C_{\infty}$.
\end{proof}

\begin{corollary}
    \begin{enumerate}[(a)]
        \item Um grupo poli-$C_{\infty}$ não tem torção. 
        \item Um grupo policíclico é virtualmente sem torção. 
    \end{enumerate}
\end{corollary}

\begin{proof}
    É suficiente provar a parte (a). Vamos usar indução sobre o comprimento minimal de uma série $(N_i)$ na definição de grupos poli-$C_{\infty}$. Se $n =1$, $G = \Z$ não tem torção. Se supomos que foi provado para todo grupo poli-$C_{\infty}$ com comprimento $\leq n$, seja $g \in G$ de ordem finita. Então $g \to \overline{g} \in G/N_1 \cong \Z$ é a identidade,  portanto $g\in N_1$ e a hipótese de indução se aplica.
\end{proof}

Finalmente, podemos provar o teorema de Wolf: Um grupo policíclico é virtualmente
nilpotente ou tem crescimentos exponencial. A prova abaixo se deve a Bass \cite{Bass72}, o argumento original de Wolf em \cite{Wolf68} utilizou a teoria dos grupos de Lie.

\begin{proof}
Pela Proposição~\ref{prop9-2}, é suficiente considerar os grupos poli-$C_{\infty}$. Temos 
$$ G = N_0 \triangleright N_1 \triangleright \ldots \triangleright N_n \triangleright N_{n+1} = \{1\}, \text{ com } N_i/N_{i+1} \cong \Z.$$

Vamos usar indução em $n$. 
Se $n = 0$, temos $G \cong \Z$, portanto $G$ é nilpotente.

Assuma que o teorema está provado para $n$ e a série de $G$ tem comprimento $n+1$. Pela hipótese de indução, $N_1$ é virtualmente nilpotente ou tem crescimento exponencial.  Se $N_1$ tem crescimento exponencial, o grupo $G \geq N_1$ também terá crescimento exponencial. No caso em que $N_1$ é virtualmente nilpotente, temos 
$$ G = N_1 \rtimes_{\psi}\Z,\ \psi: \Z\to \mathrm{Aut}(N_1)$$
e $N_1 \geq H$ -- um subgrupo nilpotente maximal de índice finito.

Podemos ver que $H$ é invariante por todos os automorfismos de $N_1$, então $\psi(1)(H) = H$. 
Portanto, $H \rtimes_{\psi}\Z \leq G$  tem índice finito. 

Pela Proposição~\ref{prop9-1}, o grupo $H \rtimes_{\psi}\Z$ é virtualmente nilpotente ou tem crescimento exponencial. Isso completa a prova.
\end{proof}

\begin{example} \label{exe:z2z}
O grupo $G = \Z^2 \rtimes_{\varphi}\Z$, com $\varphi = \left(\begin{array}{cc}
2 & 1 \\ 
1 & 1 \end{array}\right)$ é policíclico com crescimento exponencial.
\end{example}

\begin{exercise}
Verifique as afirmações do Exemplo \ref{exe:z2z}.
\end{exercise}

\medskip

\noindent
\textit{Prova do teorema de Bass--Guivarc'h:}
Seja $G$ um grupo finitamente gerado nilpotente de classe $k$. A prova é feita por indução sobre $k$. 

Se $k = 1$, o grupo $G$ é abeliano e a fórmula para o crescimento segue diretamente.
Suponha então que a afirmação é verdadeira para grupos nilpotentes de classe $k-1$, e que $G$ seja nilpotente de classe $k \geq 2$.
Assim, $C^kG$ é um grupo abeliano não trivial e $C^{k+1}G = 1$. 

Seja $d_1 = d - km_k$.
Se $C^kG$ for finito, então $m_k = 0$ e a hipótese de indução se aplica para o grupo $\Tilde{G} = G/C^kG$. Mas $G$ e $\Tilde{G}$ tem funções de crescimento equivalentes, e o resultado segue.

Agora, assumimos que $H:=C^kG$ é infinito, i.e. $m_k \geq 1$.

\medskip

\noindent\textit{Limite inferior:}
Escolhemos geradores para $G$ e $G/H$ de modo que $B_G(1,r)$ seja mapeada sobrejetivamente 
em $B_{G/H}(1,r)$ pela projeção natural. Novamente, a hipótese de indução se aplica ao grupo $\Tilde{G} =G/H$, então existem  $N = \mathrm{card}(B_{G/H}(1,n))\geq \lambda_1 n^{d_1}$ elementos em $G$ de comprimento $\leq n/2$ que não são congruentes módulo $H$. Denote por $g_1,\ldots,g_N \in B_G(1,n)$ os elementos da pré-imagem de $B_{G/H}(1,n)$.

Também podemos mostrar que em $H$ existem pelo menos $ \lambda_2n^{km_k}$ elementos com distância até $1$ em $G$ de no máximo $\frac{n}{2}$. 
Deixamos como exercício a verificação desse fato, sugerindo que o leitor mostre que se $G$ é um grupo finitamente gerado nilpotente de classe $k$ e $H$ é o último termo não trivial na sua série central inferior, então a restrição da função distância de $G$ a $H$ satisfaz
$\dist_G(1, g)\preceq  \dist_H(1, g)^{\frac{1}{k}}$,  para $g \in H$. Logo, existe $\mu \in \R$ tal que $B_G(1, \frac{n}{2}) \cap H \supset B_H(1, \mu \big(\frac{n}{2}\big)^k)$ e, assim, $\card (B_G(1, \frac{n}{2}) \cap H) \geq \lambda_2n^{km_k}$.

 Afirmamos que  $B_G(1, n)$ contém o conjunto $\displaystyle \bigcup_{i=1}^N
g_i\big(B_G(1,\frac{n}{2}\big) \cap H)$, cuja cardinalidade é de no mínimo $N \lambda_2\big(\frac{n}{2}\big)^{km_k}\geq \lambda_1\lambda_2 \big(\frac{n}{2}\big)^{d_1+km_k}$.
De fato, se $x \in B_G(1,\frac{n}{2})\cap H$, então temos $xH=H$ e, com isso, $\dist_{G/H}(xH,g_iH) = \dist_{G/H}(H, g_iH) \leq n$. Assim, 
$$\dist_G(1,g_ix) = \dist_G(g_i^{-1},x)  \leq \dist_G(g_i^{-1},1) + \dist_G(1,x) \leq n,$$ 
ou seja, $g_ix \in B_G(1, n)$. Logo, existe uma constante $\lambda_3$ tal que, para cada $n$, há no mínimo $\lambda_3 n^d$ elementos em $G$ de comprimento $\leq n$. 

\medskip

\noindent\textit{Limite superior:}
Lembramos que se  $G = \langle S \rangle$ e $g\in G$ é tal que $d_S(1,g) \leq n$, então $d(gH, H)\leq n$ em $G/H$. Por hipótese de indução, existem no máximo  $ \lambda_1'n^{d_1}$ classes laterais distintas: 
$$g_1H, g_2H, \ldots , g_NH \quad (N \leq \lambda _1'n^{d_1}),$$
onde podemos supor que todo $g_i$ cumpre $\dist_G(1,g_i) \leq n$.

Assuma que $g \in g_jH$. Então $g = g_jh$ e 
$$\dist_G(1,h) \leq \dist_G(1,g) + \dist_G(1,g_j) \leq 2n.$$
Também é possível mostrar, usando a mesma propriedade mencionada acima a respeito do último subgrupo não trivial da série central inferior de $G$, que há no máximo $\lambda_2'n^{km_k}$  tais $h$'s  distintos.
Então temos no máximo  $ \lambda_1'\lambda_2'n^d$ possibilidades distintas para $g$.

Com isso, a prova do teorema está concluída.
\qed

\begin{exercise}
Considere o grupo de Heisenberg discreto \index{grupo de Heisenberg discreto}
$$ H_3 = \left\{ U_{kln} = 
\left(\begin{matrix} 
1 & k & n \\
0 & 1 & l \\
0 & 0 & 1
\end{matrix}\right),\, 
k, l, n \in \mathbb{Z} \right\}, $$
com o conjunto de geradores $S = \{u^{\pm 1}, v^{\pm 1}, z^{\pm 1}\}$, onde $u = \mathrm{I} + \mathrm{E}_{12}$, \newline $v = \mathrm{I} + \mathrm{E}_{23}$, $z = \mathrm{I} + \mathrm{E}_{13}$ e $\mathrm{E}_{ij}$ é a matriz cuja única entrada não nula vale $1$, e está na posição $(i,j)$.

\begin{enumerate}
\item[(1)] Mostre que $U_{kln} = u^k v^l z^{n-kl}$ (isso implica em particular que qualquer elemento de $H_3$ pode ser escrito \textit{de maneira única} como $u^k v^l z^m$, onde $k, l, m \in \mathbb{Z}$);
\item[(2)] Prove que  $[u^k, v^l] = z^{kl}$ e deduza que $|z^m| \le 6\sqrt{|m|}$ e que  $|u^kv^lz^m| \le |k| + |l| + 6\sqrt{|m|}$ (onde $|g| = \mathrm{dist}_S(1, g)$);
\item[(3)] Prove que $|u^kv^lz^m| \le r$ implica $|k|+|l| \le r$ e $|m|\le r^2$;
\item[(4)] Deduza que $|u^kv^lz^m| \ge \frac12(|k| + |l| + \sqrt{|m|})$;
\item[(5)] Finalmente, mostre que existem constantes $c_2 > c_1 >0$ tais que, para cada $n \ge 1$, 
$$ c_1n^4 \le \rho_S(n) \le c_2n^4.$$
\end{enumerate}
\end{exercise}

\section{Teorema de Milnor}

O próximo resultado fundamental em direção à Conjectura~\ref{conj-Milnor} é devido ao próprio Milnor \cite{Mil68a}. 

\begin{thm}[Milnor, 1968]\label{thm-Milnor}
Um grupo finitamente gerado solúvel é policíclico ou tem crescimento exponencial.
\end{thm}

Também precisaremos demonstrar alguns lemas para chegar à prova desse resultado.

\begin{lemma}\label{lem-mil1}
Se um grupo $G$ é finitamente gerado e tem crescimento subexponencial, então para todos $\alpha,g \in G$, o conjunto $$\{g^k\alpha g^{-k} =: \alpha_k \mid k\in\Z \}$$ gera um subgrupo finitamente gerado. 
\end{lemma}

\begin{remark}
\begin{enumerate}[(1)]
    \item Não é verdadeiro para $G$ com crescimento exponencial. Considere, por exemplo, 
    $G = F_2$.
    \item O mesmo se aplica para um conjunto finito $\{\alpha_1,\alpha_2,\ldots, \alpha_n\}$ no lugar de $\alpha$.
\end{enumerate}\label{remark-mil1}
\end{remark}

\begin{proof}
Dado $m \in \N$, considere o mapa
$$f_m: \prod_{i=0}^m \Z_2 \to G, \quad (s_i) \to g\alpha^{s_0}g\alpha^{s_1}\ldots g\alpha^{s_m}.$$
Se $f_m$ for injetivo para todo $m$, teremos $2^{m+1}$ produtos correspondentes a elementos na bola $B_{m+1}(1)$ em $G$ (assumindo que $g$ e $g\alpha$ são geradores). Isso implicaria numa contradição com a hipótese de crescimento subexponencial.

Então existem $m$ e $(s_i) \neq (t_i)$ em $\displaystyle\prod_{i=1}^m \Z_2$ tais que 
\begin{equation}\label{eq9-lem-mil1} 
g\alpha^{s_0}g\alpha^{s_1}\ldots g\alpha^{s_m} = g\alpha^{t_0}g\alpha^{t_1}\ldots g\alpha^{t_m}.
\end{equation}
Vamos assumir que o número $m$ é o menor possível com esta propriedade (em particular, $s_0 \neq t_0$ e $s_m \neq t_m$).

Observe que podemos escrever
$$ g\alpha^{s_0}g\alpha^{s_1}\ldots g\alpha^{s_m} = \alpha_1^{s_0}\alpha_2^{s_1}\ldots \alpha_{m+1}^{s_m} g^{m+1}, \text{ com } \alpha_i = g^i\alpha g^{-i}.$$
Então \eqref{eq9-lem-mil1} torna-se
$$
\alpha_1^{s_0}\alpha_2^{s_1}\ldots \alpha_{m+1}^{s_m} g^{m+1}  = \alpha_1^{t_0}\alpha_2^{t_1}\ldots \alpha_{m+1}^{t_m} g^{m+1}.
$$
Sendo $s_m\neq t_m$, temos $s_m-t_m = \pm 1$. Isto  nos dá
$$
\alpha_{m+1}^{\pm 1} = \alpha_m^{-s_{m-1}}\ldots \alpha_2^{-s_1}\alpha_1^{t_0-s_0}\alpha_2^{t_1}\ldots \alpha_m^{t_{m-1}}.
$$
Conjugando por $g$, obtemos
$$
\alpha_{m+2}^{\pm 1} = \alpha_{m+1}^{-s_{m-1}}\ldots \alpha_3^{-s_1}\alpha_2^{t_0-s_0}\alpha_3^{t_1}\ldots \alpha_{m+1}^{t_{m-1}}.
$$
Por indução, concluímos que, para $n\geq m+1$, $\alpha_n$ é um produto de $\alpha_1, \ldots, \alpha_m$ e o mesmo se aplica para os inversos, com potências negativas. 
\end{proof}

\begin{corollary}\label{cor:mil1}
Seja $G$ um grupo finitamente gerado com crescimento subexponencial e seja $N \triangleleft G$ tal que $G/N$ é cíclico. Então $N$ também é finitamente gerado.     
\end{corollary}

\begin{proof}
Como todos os subgrupos de índice finito de $G$ são finitamente gerados, podemos assumir que $G/N$ é cíclico infinito. Seja $xN$ um gerador de  $G/N$. Dados quaisquer geradores $\{x_1, \ldots, x_d\}$ de $G$, podemos escrevê-los na forma $x_i = x^{e_i}y_i$, onde $e_i\in \Z$ e $y_i \in N$, e então $x, y_1, \ldots, y_d$ gera $G$. O subgrupo $N$ contém o fecho normal, digamos $K$, dos $y_i$'s. Mas $G/K$ é gerado pelas imagens dos geradores de $G$, portanto por $xK$. Assim, $G/K$ é cíclico infinito, e $G/N$ é um grupo de fatores cíclicos infinito dele, o que só é possível se $K = N$. Agora denote por $K_i$ subgrupo gerado por todos os conjugados $x^{-n} y_i x^n$. Então $N \geq \langle K_1, \ldots, K_d \rangle$, e o último subgrupo contém $y_1, \ldots,y_d$ e é invariante sob conjugação por todos os geradores de $G$, portanto é igual a $N$. O corolário agora segue do Lema \ref{lem-mil1} e da Observação \ref{remark-mil1}(2).    
\end{proof}

\begin{lemma}\label{lem-mil2}
Considere uma sequência exata curta
$$1 \to A \to G \xrightarrow{\varphi} H \to 1,$$
onde  $A$ é um grupo abeliano e $G$ é um grupo finitamente gerado. Se o grupo $H$ é policíclico, então $G$ tem crescimento exponencial ou é policíclico.
\end{lemma}

\begin{proof}
Vamos assumir que $G$ tem crescimento subexponencial e mostrar que ele deve ser policíclico. Para isso, é suficiente mostrar que $A$ é finitamente gerado. De fato, nesse caso, o grupo abeliano $A$ é policíclico e seguindo as hipóteses do lema, temos:
\begin{align*}
G/A &\cong H = H_0 \geq H_1 \geq \ldots \geq H_n = \{ H \};\\
A &= A_0 \geq A_1 \geq \ldots \geq A_k = \{1\} \text{ com } A_i/A_{i+1} \text{ cíclico.} 
\end{align*}
Sejam $Q_i := \varphi^{-1}(H_i)$, $i = 1, \ldots, n$. Temos
\begin{align*}
G & \geq Q_1 \geq \ldots \geq Q_n = A = A_0 \geq A_1 \geq \ldots A_k = \{1\}, 
\end{align*}
que é uma série policíclica de $G$. Vamos provar que $A$ é finitamente gerado. 

Como $H$ é policíclico, ele é \textit{limitadamente gerado}, isto é:
existem $h_1,  h_2,  \ldots, h_q \in H$ tais que, para todo $h\in H$, temos $h = h_1^{m_1}h_2^{m_2}\ldots h_q^{m_q}$ com $m_i\in \Z$.

Sejam $g_i$ as pré-imagens de $h_i$ por $\varphi$. Então, para todo $g\in G$, temos
\begin{equation}\label{eq9-lem-mil2}
    g = g_1^{m_1}\ldots g_q^{m_q}a, \text{ com } m_i\in \Z,\ a\in A.
\end{equation}
\begin{claim}\label{claim9.1}
    Seja $1 \to N \to K \xrightarrow{\pi} Q \to 1$ uma sequência exata de grupos, onde $K$ é finitamente gerado e $Q$ é finitamente apresentado. Então $N$ é normalmente gerado por alguns $n_1$, $n_2, \ldots, n_k$ (i.e. cada elemento de $N$ é um produto de conjugados de $n_i$).
\end{claim}

\noindent
\textit{Prova da Afirmação 1.}\ \ 
Seja $S$ um conjunto finito que gera $K$. Temos que $\overline{S} = \pi(S)$ gera $Q$. O grupo $Q$ é finitamente apresentado, então tome
$$ Q = \langle \overline{S} \mid r_1(\overline{S}), \ldots, r_k(\overline{S}) \rangle$$ uma apresentação finita, 
e sejam $n_i := r_i(S) \in K$. É fácil verificar que os elementos $n_i$ geram normalmente o subgrupo $N$. 
\qed

\medskip

Em nosso caso, a Afirmação \ref{claim9.1} implica que existem $a_1, \ldots, a_k \in A$ tais que cada $a\in A$ é um produto de conjugados de $a_j$ por elementos de $G$.

Por \eqref{eq9-lem-mil2}, os conjugados de $a_j$ têm a forma
\begin{equation}\label{eq9-lem-mil3}
    g_1^{m_1}\ldots g_q^{m_q} a_j (g_1^{m_1}\ldots g_q^{m_q})^{-1}.
\end{equation}

Pelo Lema~\ref{lem-mil1}, o grupo $A_q = \langle g_q^m a_j g_q^{-m} \mid m\in \Z,\ j = 1, \ldots,k \rangle$ é finitamente gerado por, digamos, um conjunto $S_q$ (ver também Observação \ref{remark-mil1} (2)). Assim, temos
$$ g_{q-1}^n g_q^m a_j g_q^{-m} q_{q-1}^{-n} = g_{q-1}^n s q_{q-1}^{-n} \text{ com } n\in Z,\ s\in S.$$
Outra vez usando Lema~\ref{lem-mil1} obtemos que o grupo gerado por 
$$\{ g_{q-1}^n s q_{q-1}^{-n} \mid n\in\Z,\ s\in S\}$$
é finitamente gerado. Então, por indução, o grupo gerado por todos os elementos de forma \eqref{eq9-lem-mil3} é finitamente gerado, e portanto $A$ é finitamente gerado.

Isso conclui a prova da Afirmação~1 e do Lema~\ref{lem-mil2}.
\end{proof}

\medskip

\noindent
\textit{Prova do teorema de Milnor.}
Vamos usar indução no comprimento derivado do grupo $G$. Se o comprimento é igual a $1$, o grupo $G$ é finitamente gerado e abeliano, então ele é policíclico.

Suponha que o teorema vale para os grupos de comprimento $\leq d$ e seja $d+1$ o comprimento derivado do grupo $G$. Note que $H = G/G^{(d)}$ é finitamente gerado e solúvel, de comprimento derivado $d$. Logo, pela hipótese de indução  $H$ é policíclico ou tem crescimento exponencial.

Se $H$ tem crescimento exponencial, então $G$ também tem crescimento exponencial (cf. Proposição~\ref{prop:cresc}(3)).
   
Caso contrário, $H$ é policíclico, e temos 
$$ 1 \to G^{(d)} \to G \to H \to 1,$$
onde $G^{(d)}$ é abeliano pois $G$ tem comprimento derivado $d+1$. Então podemos aplicar o Lema~\ref{lem-mil2} para concluir a prova.

\qed

\begin{corollary}\label{cor:Milnor-Wolf}
Seja $G$ um grupo solúvel finitamente gerado. Então o crescimento de $G$ é exponencial ou polinomial, e o último ocorre se, e somente se, $G$ for virtualmente nilpotente.    
\end{corollary}

\section{Grupos lineares}

Nesta seção, consideramos o crescimento de grupos lineares e alguns outros resultados relacionados, que serão necessários para provar o teorema de Gromov, na seção seguinte.

Relembramos que um \textit{grupo linear} é um grupo isomorfo a um subgrupo do grupo linear geral $\GL(n, F)$, para algum número natural $n$ e algum corpo $F$. O resultado básico resolvendo o problema de crescimento para grupos lineares é a \textit{alternativa de Tits} \cite{Tits72} (cf. Seção~\ref{sec8-TitsAlt}):

\begin{thm}[Tits, 1972]
Seja $G$ um grupo linear finitamente gerado. Então ou $G$ é virtualmente solúvel ou $G$ contém um subgrupo livre não abeliano.
\end{thm}

\begin{corollary}\label{cor-TitsAlt}
O crescimento de um grupo linear finitamente gerado é exponencial ou polinomial, e é polinomial se, e somente se, o grupo for virtualmente nilpotente.    
\end{corollary}

De fato, se o grupo contém um subgrupo livre não abeliano, seu crescimento é exponencial. No outro caso, basta aplicar os teoremas de Milnor e Wolf.

Shalom deu uma prova do Corolário~\ref{cor-TitsAlt} independente da alternativa de Tits para o caso de grupos lineares de característica zero \cite{Shalom98}. Este caso é suficiente para a demonstração do teorema de Gromov descrito a seguir. Outros resultados eficazes sobre a alternativa de Tits foram obtidos mais recentemente por Breuillard, Gelander, Eskin, Mozes, Oh, e outros (veja \cite{BreuGel08} e as referências nele contidas). As provas desses resultados usam o Lema do pingue-pongue (Lema ~\ref{pingponglemma}) e suas variações.

Em seguida, recordamos um importante teorema clássico.

\begin{thm}[Jordan, 1878]\label{thm:Jordan} 
Se $G$ é um subgrupo finito do grupo $\GL(n, \C)$, então $G$ tem um subgrupo abeliano de índice limitado apenas em termos de $n$.
\end{thm}

É notável o fato de que Jordan tenha provado esse teorema em 1878, quando a teoria dos grupos ainda estava em sua infância. Referimo-nos a um recente artigo de E.~Breuillard \cite{Breu23} para a prova do teorema.

\medskip 

A ponte entre esses resultados e o crescimento de grupos gerais, em particular o crescimento polinomial, foi construída por Gromov e se baseia no seguinte resultado fundamental:

\begin{thm}[Gleason--Montgomery--Zippin: Solução do quinto problema de Hilbert \cite{MZ55}]\label{sec9 thm-Montgomery-Zippin}
Seja $X$ um espaço métrico de dimensão finita (de Hausdorff), localmente compacto, localmente conexo, conexo e homogêneo. Então o grupo de isometrias $\Isom(X)$ tem uma estrutura de grupo de Lie com finitas componentes conexas.
\end{thm}

Os leitores podem não estar familiarizados com todos os termos deste teorema, então revisaremos algumas das definições, embora não daremos a definição formal de grupos de Lie. Basta lembrar primeiro que um grupo topológico é um grupo no qual uma topologia é definida de tal forma que as operações do grupo, ou seja, multiplicação e inversão, são contínuas. Grosso modo, um grupo de Lie é um grupo topológico no qual o espaço topológico subjacente é uma variedade, ou seja, cada ponto tem uma vizinhança aberta homeomorfa a um espaço euclidiano e as operações do grupo são não somente contínuas, mas também diferenciáveis. 

Uma propriedade importante dos grupos de Lie com um número finito de componentes conexas é a seguente (veja~\cite[Capítulo~6]{Mann}).

\begin{proposition}\label{sec9 prop-LieGrps}
Seja $L$ um grupo de Lie com um número finito de componentes conexas. Então:
\begin{enumerate}[(1)]
    \item O grupo $L$ tem um subgrupo abeliano normal $Z$ (o centro da componente conexa da identidade), tal que $L/Z$ é isomorfo a um subgrupo de $\GL(k,\C)$, para algum $k$.
    \item Para cada número natural $n$, existe uma vizinhança aberta da identidade em $L$ que não contém nenhum elemento não trivial de ordem finita menor que $n$.
\end{enumerate}
\end{proposition}

Um espaço métrico $X$ é dito \textit{homogêneo} se, para quaisquer dois pontos $x,y \in X$, existe uma isometria de $X$ mapeando $x$ em $y$.\index{espaço métrico homogêneo}

Agora definimos dimensão finita e dimensão de Hausdorff. Um espaço topológico tem dimensão $0$ se cada ponto possui uma vizinhança aberta com fronteira vazia. Ele tem dimensão no máximo $n$, se cada ponto tem uma vizinhança aberta com fronteira de dimensão no máximo $n - 1$. A dimensão é igual a $n$, se é no máximo $n$, mas não é no máximo $n - 1$.
É fácil ver que o espaço euclidiano $\R^n$ tem dimensão no máximo $n$, mas não é tão fácil provar que a dimensão é exatamente $n$. Se considerarmos nesse espaço o conjunto de todos os pontos com coordenadas irracionais, este é um espaço $0$-dimensional. Nos referimos a \cite{HW41} para mais detalhes sobre a dimensão topológica. \index{dimensão topológica}

A dimensão de Hausdorff é um conceito métrico, e não topológico. Lembre-se de que a medida de um subconjunto do espaço euclidiano é determinada observando as coberturas desse subconjunto por pequenas bolas e somando os volumes dessas bolas. O volume de uma bola, por sua vez, é proporcional a uma potência de seu diâmetro, e o expoente dessa potência é a dimensão do espaço. Se usarmos o expoente errado, digamos que tomamos um expoente maior que a dimensão, todos os subconjuntos terão medida 0, enquanto se usarmos um expoente menor que a dimensão, a maioria dos subconjuntos terá medida infinita. A ideia da dimensão de Hausdorff é inverter essas observações: primeiro definimos medidas para todos os expoentes e use o desaparecimento dessas medidas para detectar o expoente que deve ser considerado como a dimensão do espaço. A definição precisa é a seguinte.

\begin{definition}
Seja $X$ um espaço métrico. A \textit{dimensão de Hausdorff} \index{{dimensão de Hausdorff}} de $X$ é o ínfimo do conjunto dos números $s$, tais que para cada $\epsilon > 0$, $X$ pode ser coberto por uma coleção finita ou contável $\{ A_i \}$ de subconjuntos tal que $A_i$ tem diâmetro $d_i$ e $\sum d_i^s < \epsilon$. Em outras palavras, a dimensão de Hausdorff é o ínfimo dos números $s$ tais que a ``medida $s$-dimensional'' de $X$ é $0$.
\end{definition}

Ao contrário da dimensão topológica, a dimensão de Hausdorff não precisa ser um número inteiro e, de fato, qualquer número real não negativo pode ocorrer como a dimensão de Hausdorff de algum espaço. A conexão entre as duas dimensões é dada por:

\begin{thm}
A dimensão de um espaço métrico $X$ é o ínﬁmo das dimensões de Hausdorff dos espaços métricos homeomorfos a $X$.
\end{thm}

Assim, a dimensão não excede a dimensão de Hausdorff. Para a prova, veja \cite[Capítulo~VII]{HW41}. Também precisaremos do seguinte resultado:

\begin{proposition}
Se um espaço métrico $X$ pode ser coberto, para cada $\epsilon > 0$, por $k$ bolas de diâmetro $\epsilon$, onde $k$ pode variar com $\epsilon$, mas o produto $k \epsilon^d$ permanece limitado, então a dimensão Hausdorff de $X$ é no máximo $d$.
\end{proposition}

\begin{proof}
Suponha que $k\epsilon^d < C$, e seja $s > d$. Então $k\epsilon^s < C\epsilon^{s-d}$, e o lado direito pode ser tornado arbitrariamente pequeno, tomando $\epsilon$ suficiente pequeno. Isso mostra que a dimensão de Hausdorff não é maior que $s$.
\end{proof}

\medskip

Para concluir esta seção, recordamos a definição da topologia que torna o grupo de isometrias de um espaço métrico $X$  um grupo de Lie. Fixe algum ponto base $p \in X$. Para quaisquer dois números  $A,\epsilon$ positivos, seja $O(A, \epsilon)$ o conjunto de todas as isometrias $\sigma$ tais que $d(\sigma x, x) < \epsilon$, para todo $x$ que satisfaça $d(x, p) \leq A$. Os conjuntos $O(A, \epsilon)$ definem uma base de vizinhanças da identidade em $\Isom(X)$.

\section{Teorema de Gromov}
\begin{thm}[Gromov, 1981]\label{thm:Gromov} 
Se $G$ é um grupo finitamente gerado com crescimento polinomial, então $G$ é virtualmente nilpotente.
\end{thm}

Considere um grupo $G$ finitamente gerado de crescimento polinomial. Nós gostaríamos de definir um mapa $f: G \to \Isom(X)$, onde $X$ é um espaço métrico de dimensão finita (de Hausdorff), localmente compacto, localmente conexo, conexo e homogêneo. Isso nos permitiria usar os teoremas de Gleason--Montgomery--Zippin, Tits e Jordan juntamente com os resultados anteriores sobre crescimento de grupos para investigar a estrutura de $G$.  A questão é: qual é o espaço $X$? Podemos tentar considerar $X = \cay(G,S)$, mas $X$ não é homogêneo. Se tomarmos X como os vértices de $\cay(G,S)$ ele é homogêneo mas não é conexo. A ideia de Gromov em \cite{Gromov81} foi definir um espaço limite de $G$ para este fim. Intuitivamente, esse espaço vem de olhar para o grafo de Cayley de um grupo de uma distância infinita. Mais tarde, isso levou à definição do cone assintótico $K$ de $G$.

Nós apresentaremos a prova do teorema de Gromov, seguindo principalmente a exposição de A.~Mann \cite{Mann}.

\medskip

Seja $G$ um grupo finitamente gerado, com crescimento polinomial de grau $d$. Embora a sequência $\rho(n) = \rho_G(n)$ seja polinomialmente limitada, ela ainda pode oscilar descontroladamente. Começamos por escolher uma subsequência com um comportamento razoável, no sentido de procurarmos aqueles valores $\rho(n)$ cuja distância dos termos anteriores não seja muito grande.

\begin{proposition}\label{prop:GromovThm}
    Seja $G$ um grupo com crescimento polinomial de grau $d$. Então existem infinitos $n\in \N$ tais que, para todo $i < n$,
    $$\log_2 \rho(2^n) \leq \log_2 \rho(2^{n-i}) + i(d+1).$$
\end{proposition}

\begin{proof}
Pela definição de $d$, temos 
$$\log_2 \rho(n) / \log_2(n) < d+1/2,$$ 
para $n$ suficientemente grande. Escrevendo $l(n) = \log_2(\rho(2^n))$, obtemos em particular que $l(n)/n < d + 1/2$, e portanto $l(n) - n(d + 1) < -n/2$. Assim, 
$$\lim_{n\to\infty} (l(n) - n(d + 1)) = -\infty.$$ 
Para cada inteiro negativo $k$, seja $n=n(k)$ o primeiro inteiro tal que $l(n) - n(d + 1) < k$. Então para este $n$ e para $i < n$, temos 
$$l(n) - n(d + 1) < k \leq l(n - i) - (n - i)(d + 1),$$ 
que é equivalente à desigualdade necessária. Como $n(k)$ assume infinitos valores, a proposição está provada.
\end{proof}

Denote por $S$ o conjunto dos inteiros $n$ que satisfazem as desigualdades da proposição anterior, e seja $T = \{2^n \mid n + 1 \in S\}$. Escolheremos um ultrafiltro $\mathcal{F}$ em $\mathbb{N}$, de modo que ele contenha $T$. A razão para esta escolha particular ficará clara durante a prova seguinte.

\begin{thm}[Teorema do Crescimento Regular]\label{thm:RegGrowth}\index{teorema do crescimento regular}
Sejam $G$ um grupo de crescimento polinomial, $\mathcal{F}$ um ultrafiltro escolhido conforme descrito acima e $\epsilon > 0$ suficientemente pequeno. Então o cone assintótico $K$ definido por $(G,e)$ e $\mathcal{F}$ satisfaz as duas propriedades equivalentes a seguir:
\begin{enumerate}[(1)]
    \item Se uma bola fechada de raio $1$ em $K$ contém $k$ pontos distintos, de forma que as bolas fechadas de raio $\epsilon$ ao redor deles são disjuntas, então $k \leq (1/\epsilon)^{2(d+1)}$.
    \item Se uma bola fechada de raio $1$ em $K$ contém $k$ pontos distintos, de modo que as distâncias entre quaisquer dois deles sejam maiores que $2\epsilon$, então $k \leq (1/\epsilon)^{2(d+1)}$.
\end{enumerate}
\end{thm}

\begin{proof}
Primeiro notamos que as duas propriedades são de fato equivalentes. Suponha que (1) seja válido. Sejam dados $k$ pontos como em (2), então as bolas de raios $\epsilon$ ao redor deles são disjuntas, portanto $k \leq (1/\epsilon)^{2(d+1)}$. Reciprocamente, se vale (2), então para quaisquer $k$ pontos como em (1), suponha que dois deles estejam a uma distância $\delta \leq 2\epsilon$. A prova do Teorema~\ref{thm:AsCone} mostra que existe um caminho contínuo $f(\alpha)$ entre esses dois pontos, tal que $d(f(\alpha), f(\beta)) \leq (\beta - \alpha)\delta$. Então o ponto $f(1/2)$ está nas esferas de raio $\epsilon$ em torno de ambos os pontos, o que é uma contradição. Portanto (2) implica (1).

Por homogeneidade, basta provar a proposição para a bola $B$ de raio $1$ e centro $e$ em $K$. Suponhamos que $k$ pontos $x_1, \ldots , x_k$, como em (2), são dados em $B$. Para cada $r$, suponha que $x_r$ seja representado pela sequência $\{ x(r, n) \}$. Podemos escolher as sequências representativas de modo que cada ponto $x(r, n)$ esteja na bola de raio $(1 + \epsilon)n$ em torno de $e$. Escolha dois índices distintos $r$ e $s$, e seja $N(r, s, \epsilon)$ o conjunto dos $n\in\N$ para os quais $d(x(r, n), x(s, n)) > 2\epsilon n$. A desigualdade $d(x_r , x_s) > 2\epsilon$ significa que $N(r, s, \epsilon) \in \mathcal{F}$, e então também $N(\epsilon) \in \mathcal{F}$, onde $N(\epsilon)$ é a interseção dos conjuntos $N(r, s, \epsilon)$. Seja $i\in \N$ tal que $1/2^i < \epsilon \leq 1/2^{i-1}$, e escolha $n > i$ tal que $2^n \in T \cap N(\epsilon)$. Como $2^n/2^i < 2^n \epsilon$, as bolas de raio $2^{n-i}$ em torno dos pontos $x(r, 2^n )$, para $r = 1, \ldots, k$, são disjuntas, e cada uma delas contém $\rho(2^{n-i})$ pontos. Portanto, obtemos $k\rho(2^{n-i})$ pontos dentro dessas bolas. Um ponto em uma dessas bolas está a uma distância de, no máximo, $(1 + \epsilon)2^n + 2^{n-i}$ de $e$. Então essas bolas estão dentro da bola de raio $2^{n+1}$ ao redor de $e$, e assim obtemos $k\rho(2^{n-i}) \leq \rho(2^{n+1})$. Mas a Proposição~\ref{prop:GromovThm} implica que $\rho(2^{n+1}) \leq 2^{(i+1)(d+1)} \rho(2^{(n+1)-(i+1)})$, e assim $k \leq 2^{(i+1)(d+1)} \leq (1/\epsilon)^{2(d+1)}$.
\end{proof}

\begin{thm}\label{thm:AsConeFinDim}
Se $G$ tem crescimento polinomial, o cone assintótico $K$  definido por $(G,e)$ e $\mathcal{F}$ tem  dimensão finita e é localmente compacto.    
\end{thm}

\begin{proof}
Escolhamos dentro da bola $B$ o maior número possível, digamos $k$, de pontos tais que a distância entre quaisquer dois deles seja maior do que $2\epsilon$. Então cada ponto de $B$ está a uma distância de no máximo $2\epsilon$ de um desses pontos, e o Teorema~\ref{thm:RegGrowth} mostra que $B$ é coberta por no máximo $(1/\epsilon)^{2(d+1)}$ bolas de raio $2\epsilon$. Isso significa que a dimensão Hausdorff de B é no máximo $2(d + 1)$ e, em particular, implica que a dimensão de $B$ é finita. Como a dimensão é determinada pelo comportamento local, a dimensão de $K$ é finita. Agora seja $\{x_n\}$ qualquer sequência em $B$. Mostraremos que ela possui uma subsequência convergente. Isso mostrará que $B$ é compacto e, portanto, $K$ é localmente compacto. Para tanto, dado $i\in \N$, vamos cobrir $B$ por $k_i$ bolas de raio $2^{-i}$. Então uma subsequência infinita de $\{x_n\}$ está em uma das bolas de raio $1$, uma subsequência infinita dela está em uma das bolas de raio $1/2$, etc. Tomando uma sequência diagonal, encontramos uma subsequência de Cauchy da sequência original. Como $K$ é completo (cf. Teorema \ref{thm:AsCone}), essa subsequência converge.
\end{proof}

\begin{corollary}\label{cor:Gr1}
Seja G um grupo infinito de crescimento polinomial. Então existe um grupo de Lie $L$ com um número finito de componentes e um número natural $k$, tal que $G$ contém um subgrupo normal $C$ de índice finito, para o qual vale uma das seguintes condições:    
\begin{enumerate}[(i)]
    \item O grupo $C$ tem um quociente abeliano infinito;
    \item O grupo $C$ possui um quociente infinito em $\GL(k, \C)$;
    \item Existem homomorfismos $\phi_n : C \to L$, para todos os números naturais $n$, tais que $|C/\ker(\phi_n)| \geq n$.
\end{enumerate}
\end{corollary}

\begin{proof}
Seja $N$ como no Teorema~\ref{thm:gpactiononK}. Pelos Teoremas~\ref{thm:AsCone} e \ref{thm:AsConeFinDim}, o cone assintótico $K$ satisfaz as hipóteses do Teorema~\ref{sec9 thm-Montgomery-Zippin}. Pela Proposição~\ref{sec9 prop-LieGrps}, $G/N$ tem um subgrupo abeliano normal $Z/N$ tal que $G/Z$ é isomorfo a um subgrupo de $\GL(k, \C)$, para algum $k \in \N$. Se $G/Z$ for infinito, a prova está completa, bem como no caso em que $G/Z$ é finito e $G/N$ é infinito. Assim, assumimos que $G/N$ seja finito. No caso (ii) do Teorema~\ref{thm:gpactiononK}, a prova também termina. Assumimos então que  o caso (iii) desse mesmo teorema se aplica. Para concluir esse caso, aplicamos a Proposição~\ref{sec9 prop-LieGrps}.    
\end{proof}

Agora podemos deduzir o seguinte:

\begin{corollary}\label{cor:Gr2}
Um grupo infinito de crescimento polinomial contém um subgrupo de índice finito que possui uma imagem homomórfica cíclica infinita.
\end{corollary}

\begin{proof}
Seja $G$ infinito de crescimento polinomial, e seja $C$ como no corolário anterior. No caso (ii) do Corolário~\ref{cor:Gr1}, sabemos, pelos teoremas de Tits e Milnor--Wolf, que o grupo quociente linear infinito $G/N$ é virtualmente nilpotente, então $G$ contém um subgrupo de índice finito $H \geq N$ tal que $H/N$ é nilpotente e infinito, e o último termo da série derivada de $H$ que ainda tem índice finito em $G$ é o subgrupo que queremos. O caso (i) é óbvio, então assumimos que (iii) vale. Para cada $n\in \N$ como em (iii), seja $K_n = \ker(\phi_n)$. A Proposição~\ref{sec9 prop-LieGrps} implica que $C$ contém subgrupos normais $B_n$, tais que $B_n/K_n$ é abeliano e $C/B_n$ é isomorfo a um subgrupo de $\GL(k, \C)$ (onde $k$ é independente de $n$). Pelo Teorema~\ref{thm:Jordan} de Jordan, $C$ contém subgrupos $H_n \geq B_n$ tais que $H_n/B_n$ é abeliano e o índice $|C : H_n|$ é limitado. Um grupo finitamente gerado tem apenas um número finito de subgrupos de um dado índice finito. Portanto, para infinitos valores de $n$, os subgrupos $H_n$ coincidem com um mesmo subgrupo, que denotaremos por $H$. Então $H/B_n$ é abeliano para infinitos valores de $n$, e se as ordens de $H/B_n$ tendem ao infinito, então $H/H'$ é infinito, e $H$ é o subgrupo de que precisamos. Por outro lado, se as ordens $H/B_n$ permanecem limitadas, então para infinitos valores de $n$, os subgrupos $B_n$ coincidem com um subgrupo $R$, digamos, e então o mesmo argumento mostra que $R$ é o subgrupo que estamos procurando.    
\end{proof}

Finalmente, derivamos desses resultados o teorema de Gromov.

\begin{proof}[Prova do Teorema~\ref{thm:Gromov}]
Seja $G$ um grupo com crescimento polinomial, digamos de grau $d$. Se $d = 0$, então $G$ é finito e o resultado segue. Suponha agora $d > 0$. Pelo Corolário~\ref{cor:Gr2}, $G$ contém subgrupos $H \triangleright N$ tais que $|G : H|$ é finito e $H/N$ é cíclico infinito. Pelo Corolário~\ref{cor:mil1}, $N$ é finitamente gerado, e pela Proposição~\ref{prop:cresc_pol}(2) o grupo $N$ tem grau de crescimento polinomial $d - 1$ ou menos. Então, por indução, $N$ contém um subgrupo nilpotente $K$ de índice finito, o qual é invariante por todos os automorfismos de $N$ (é um subgrupo característico de $N$) e, portanto, normal em $H$.
Assim, $H/K$ contém o subgrupo normal finito $N/K$, com quociente cíclico infinito. Seja $H/N = \langle xN \rangle$ e escreva $C = \langle K, x \rangle$. Então $H/K = C/K \cdot N/K$, e portanto $|H : C|$ é finito. Como $C/K$ é cíclico infinito, $C$ é solúvel,  logo virtualmente nilpotente, pelo Corolário~\ref{cor:Milnor-Wolf}.    
\end{proof}

\begin{exercise} 
Seja $G$ um grupo  finitamente gerado  sem torção com crescimento polinomial de ordem $n^2$. Mostre que $\Z^2 \leq G$.
\end{exercise} 

\begin{corollary}\label{cor:Gromov}
    Suponha que $G_1$, $G_2$ são grupos quasi-isométricos finitamente gerados e $G_1$ é virtualmente nilpotente. Então $G_2$ é virtualmente nilpotente.
\end{corollary}

Isso decorre imediatamente do teorema de Gromov, porque o crescimento polinomial é preservado sob quasi-isometrias. 
Shalom deu uma prova do Corolário~\ref{cor:Gromov} independente do teorema de Gromov \cite{Shalom04}. 
Em contraste, sabe-se que a solubilidade virtual não é preservada por quasi-isometrias \cite{Er00}.


\chapter{Crescimento intermediário e os grupos de Grigorchuk}
\label{cap:crescimento-interm}
\section{A definição do grupo de Grigorchuk \(\Gamma\)}

Dizemos que um grupo  finitamente gerado $\Gamma$ tem \textit{crescimento intermediário} \index{crescimento intermediário} se, para algum conjunto de geradores $S$, tivermos $\gamma_S= \displaystyle\lim_{n\to \infty} \rho_S(n)^{1/n} = 1$, mas não existir $d$ tal que $\rho_S(n)\leq n^d$. Em outras palavras, $\Gamma$ tem crescimento sub-exponencial, mas acima do polinomial.

\begin{example}
Seja $\rho(n) = n^{\log(n)}$. Então, $\log \rho(n)= (\log n)^{2}$, donde $\displaystyle\lim_{n\to \infty} \rho(n)^{1/n} =1$, i.e, $\rho$  cresce sub-exponencialmente. Por sua vez, $\rho$ cresce mais que qualquer polinômio. De fato, seja $d \in \N$ qualquer e $\log(n) > d$, então $\rho(n)> n^d$.
\end{example}

Nos anos 1980, R.~Grigorchuk foi o primeiro a mostrar que existem grupos finitamente gerados de crescimento intermediário  (veja \cite{Grigorchuk+2014+705+774}). Nesta seção, vamos descrever o primeiro exemplo apresentado por ele em \cite{Gri80, Gri84}. Curiosamente, um grupo muito semelhante ao de Grigorchuk foi construído, por um método diferente, já por S.~Aleshin em  \cite{aleshin1972finite}, onde Aleshin estava interessado na conjectura de Burnside. O grupo de Aleshin é comensurável com o de Grigorchuk, portanto também é de crescimento intermediário, mas esse fato não foi notado na época.

Para dar uma definição do grupo de Grigorchuk vamos considerar  transformações definidas no intervalo $(0,1)$. Seja $E$ a transformação identidade e $P$ a transformação que permuta as duas metades $\mathrm{I}_1:=(0,1 / 2)$ e $\mathrm{I}_2:=(1 / 2,1)$ , isto é, um ponto $x$ é mapeado para $x + \frac{1}{2}$, se $x<\frac{1}{2}$, ou para $x-\frac{1}{2}$, caso contrário, como na Figura \ref{fig:grigorchuk}. Para que esteja bem-definida, essa transformação na verdade precisa ser definida em $(0,1)\backslash \{\frac{1}{2}\}$.
Para evitar ambiguidades nas definições que seguem podemos remover de $(0,1)$ todos os pontos da forma $\frac{m}{2^n},$ com $n \in \N$, $0 \leq m \leq 2^n$ e chamar de $\mathrm{I}$ o conjunto resultante (mas se necessário, a ação pode ser estendida ao intervalo $[0, 1)$).

Além disso, dado qualquer subintervalo $(a,b) \subset (0,1)$, considerando o conjunto $\mathrm{J} = (a,b) \cap \mathrm{I}$, se necessário, denotaremos por $E_{\mathrm{J}}$ e $P_{\mathrm{J}}$ os mapas identidade e  intercâmbio das duas metades de $\mathrm{J}$, respectivamente. Quando o intervalo $\mathrm{J}$ for claro a partir do contexto, então escreveremos apenas $E$ e $P$ em vez de $E_{\mathrm{J}}$ e $P_{\mathrm{J}}$.

\begin{figure}[ht!]
\centering
\begin{tikzpicture}[scale=1]
		\draw[|-|] (-2,2) -- (0,2);
  \draw[|-|] (0,2)-- (2,2);
  \draw node at (0,3) {$P$};
  \draw[<->] (-0.5,2.3) to [bend left=70] (0.5,2.3);
  \draw node at (-2.4,2) {$a$};
 \draw node at (-2,1.8) {\tiny{$0$}};
    \draw node at (0,1.7) {\tiny{$\frac{1}{2}$}};
    \draw node at (2,1.8) {\tiny{$1$}};

    \draw[|-|] (-2,0) -- (0,0);
    \draw[|-|] (0,0)-- (1,0);
    \draw[|-|] (1,0)-- (1.5,0);
    \draw[|-|] (1.5,0)-- (2,0);
    \draw node at (-1,0.4) {\small{$P$}};
    \draw node at (0.5,0.4) {\small{$P$}};
    \draw node at (1.2,0.4) {\small{$E$}};
    \draw node at (1.8,0.4) {\tiny{$PPE$}};
    \draw node at (-2,-0.2) {\tiny{$0$}};
    \draw node at (0,-0.3) {\tiny{$\frac{1}{2}$}};
    \draw node at (1,-0.3) {\tiny{$\frac{3}{4}$}};
    \draw node at (1.5,-0.3) {\tiny{$\frac{7}{8}$}};
    \draw node at (2,-0.2) {\tiny{$1$}};
    \draw node at (1.8,0.1) {\tiny{$\cdots$}};
    \draw node at (-2.4,0) {$b$};

    \draw[|-|] (-2,-2) -- (0,-2);
    \draw[|-|] (0,-2)-- (1,-2);
    \draw[|-|] (1,-2)-- (1.5,-2);
    \draw[|-|] (1.5,-2)-- (2,-2);
    \draw node at (-1,-1.6) {\small{$P$}};
    \draw node at (0.5,-1.6) {\small{$E$}};
    \draw node at (1.2,-1.6) {\small{$P$}};
    \draw node at (1.8,-1.6) {\tiny{$PEP$}};
    \draw node at (-2,-2.2) {\tiny{$0$}};
    \draw node at (0,-2.3) {\tiny{$\frac{1}{2}$}};
    \draw node at (1,-2.3) {\tiny{$\frac{3}{4}$}};
    \draw node at (1.5,-2.3) {\tiny{$\frac{7}{8}$}};
    \draw node at (2,-2.2) {\tiny{$1$}};
    \draw node at (1.8,-1.9) {\tiny{$\cdots$}};
    \draw node at (-2.4,-2) {$c$};

    \draw[|-|] (-2,-4) -- (0,-4);
    \draw[|-|] (0,-4)-- (1,-4);
    \draw[|-|] (1,-4)-- (1.5,-4);
    \draw[|-|] (1.5,-4)-- (2,-4);
    \draw node at (-1,-3.6) {\small{$E$}};
    \draw node at (0.5,-3.6) {\small{$P$}};
    \draw node at (1.2,-3.6) {\small{$P$}};
    \draw node at (1.8,-3.6) {\tiny{\tiny{$EPP$}}};
    \draw node at (-2,-4.2) {\tiny{$0$}};
    \draw node at (0,-4.3) {\tiny{$\frac{1}{2}$}};
    \draw node at (1,-4.3) {\tiny{$\frac{3}{4}$}};
    \draw node at (1.5,-4.3) {\tiny{$\frac{7}{8}$}};
    \draw node at (1.8,-3.9) {\tiny{$\cdots$}};
    \draw node at (-2.4,-4) {$d$};
\end{tikzpicture}
\caption{Grupo de Grigorchuk }
\label{fig:grigorchuk}
\end{figure}

Definiremos algumas transformações de $\mathrm{I}$ nele mesmo, através de uma sequência de aplicações de $E$ e $P$, definidas em  subintervalos. 

Note que $ I \subset \displaystyle \bigcup_{n\in \N}I_n$, onde $I_n=\left(1-\dfrac{1}{2^{n-1}},1-\dfrac{1}{2^{n}}\right)$.
Consideramos o grupo $\Gamma$ gerado por quatro transformações $a, b, c$ e $d$, onde $a$ é simplesmente o intercâmbio $P$, aplicado ao intervalo completo $I$ e os outros três geradores atuam como segue nos infinitos intervalos $I_n$: $b$ aplica $P$ a cada um dos dois primeiros subintervalos, depois $E$ ao terceiro e repete periodicamente o padrão $PPE$. O gerador $c$ aplica de forma semelhante o padrão periódico $PEP$, e $d$ aplica o padrão $EPP$. O grupo $\Gamma =\langle a,b,c,d \rangle$ é chamado de \textit{primeiro grupo de Grigorchuk}\index{grupo de Grigorchuk}.

O grupo de Grigorchuk não é livre e tem  as seguintes relações: $a^2=b^2=c^2=d^2=1$, $bc=cb=d$, $cd=dc=b$, e $db=bd=c$. Assim, $b,c,d$ geram um subgrupo de $\Gamma$ isomorfo a $C_2 \times C_2$, onde $C_2$ denota o grupo cíclico de ordem $2$. Quaisquer dois dos elementos $b,c,d$ geram o mesmo subgrupo, mas com os três a apresentação de $\Gamma$ é mais simétrica.  
Usando as relações acima podemos escrever qualquer elemento de $\Gamma$ como uma palavra escrita, de forma alternada, por entradas iguais a $a$ e alguma entrada do trio $b, c, d$. Vamos chamar de \textit{forma canônica} dos elementos de $\Gamma$  essa escrita, e assumir que seus elementos são sempre dessa forma:
$$\gamma\in \Gamma,\ \gamma=au_1au_2\ldots \text{ ou } \gamma = u_1au_2a\ldots \text{ com } u_i\in \{ b,c,d \}.$$

Além disso, o grupo de Grigorchuk não é finitamente apresentado (cf. \cite{lysenok1985system}). Mais ainda, todos os grupos  de crescimento intermediário  conhecidos não tem apresentação finita. A existência de grupos de crescimento intermediário com apresentações finitas é um dos principais problemas em aberto remanescentes na teoria de crescimento de grupos.

Voltando ao grupo $\Gamma$, observe que o elemento $a$ permuta os subintervalos $(0,1/2)$ e $(1/2,1)$ enquanto $b,c,d$ os mantém invariantes. Assim todo elemento de $\Gamma$ ou permuta ou mantém invariantes os subintervalos $(0,1/2)$ e $(1/2,1)$. Considere o conjunto $H$ de todos elementos de $\Gamma$ que estabilizam esses subintervalos. Então $H$ é um subgrupo normal de índice $2$ de $\Gamma$, consistindo de todos elementos que podem ser representados por uma palavra com um numero par de entradas iguais a $a$ na sua forma canônica.

\begin{lemma}
O subgrupo $H$ é gerado pelos elementos $b,c,d, aba$, $aca$ e $ ada.$
\end{lemma}

\begin{proof}
Seja $\gamma \in H$. Usaremos indução no comprimento de $\gamma$, escrita na sua forma canônica, para mostrar que ela pode ser escrita em termos desses geradores.
Se $\ell(\gamma) =0$, não há nada a ser demonstrado. Se $\ell(\gamma) = 1$, como $\gamma \in H$ implica que $\gamma$ possui um numero par de entradas iguais a $a$ na sua forma canônica, segue que $\gamma = b,c \mbox{ ou } d$, e portanto não há nada a demonstrar também nesse caso.

Seja $ \Tilde{H} = \langle b,c,d, aba,aca, ada \rangle $ e suponha por indução que $\gamma \in \Tilde{H}$, sempre que $\ell(\gamma) \leq n$.
Tomando $\gamma \in H$ com $\ell(\gamma) = n+1$, e sabendo que $\gamma$ possui um número par, digamos $2k$, de entradas iguais a $a$, podemos escrever $\gamma =  u_0au_1\ldots u_{2k-1}au_{2k}$, onde $ u_0, u_{2k} \in \{e,b,c,d\}$.
Escrevendo $\gamma =  u_0(au_1a)u_2\ldots u_{2k-1}au_{2k}= u_0(au_1a) \gamma'$, onde $\ell(\gamma')<n $, e portanto $\gamma' \in \Tilde{H}$, podemos concluir que $\gamma \in \Tilde{H}$. Logo, $H=\Tilde{H}$.
\end{proof}

\begin{lemma}
\label{lem:grig1}
Os seis elementos do lema acima induzem, respectivamente, em $(0, 1/2)$ as transformações $a, a, 1, c, d, b,$ , e em $(1/2, 1)$ as transformações $c, d, b, a, a, 1$.
\end{lemma}

\begin{exercise}
    Prove o Lema \ref{lem:grig1}.
\end{exercise}

\begin{corollary}
A ação de $H$ induz em cada subintervalo $(0,1/2)$ e $(1/2,1)$ grupos isomorfos a $\Gamma$.
\end{corollary}

\begin{corollary}
O grupo $\Gamma$ é infinito.
\end{corollary}
\begin{proof}
Pelo corolário acima, o subgrupo próprio $H$ contém subgrupos isomorfos a $\Gamma$, e portanto $\Gamma$ está em bijeção com um subconjunto próprio. Isso somente é possível em um grupo infinito.
\end{proof}

\section{As propriedades principais do grupo \(\Gamma\)}

Lembramos que grupos $G_1$ e $G_2$ são ditos comensuráveis se eles possuírem subgrupos $H_1 \leq G_1$ e $H_2 \leq G_2$ de índice finito que são isomorfos. 

\begin{thm}
\label{thm:gammacommgammagamma}
O grupo $\Gamma$ é comensurável com $\Gamma \times \Gamma$.
\end{thm}

\begin{proof}
Note que um elemento $x \in H$ é completamente determinado por suas ações em $(0,1/2)$ e $(1/2,1)$. Assim, denote por $x_{\ell}$ e $x_r$ as aplicações induzidas pela ação de $x$ em $(0,1/2)$ e $(1/2,1)$, respectivamente. Podemos definir um mergulho $\varphi :  H \to \Gamma \times \Gamma $ pondo $\varphi(x) = (x_{\ell}, x_r)$. Pelo lema anterior, temos

$$\arraycolsep=0pt\def\arraystretch{1.2}
\begin{array}{rcl}
\varphi :  H \*&\;\to\;&\* \Gamma \times \Gamma;  \\
b\*&\;\mapsto\;&\* (a,c);  \\
c\*&\;\mapsto\;&\*   (a,d);\\
d\*&\;\mapsto\;&\* (1,b); \\
aba\*&\;\mapsto\;&\*  (c,a); \\
aca\*&\;\mapsto\;&\*  (d,a); \\
ada\*&\;\mapsto\;&\* (b,1).

\end{array}$$

Para finalizar esta demonstração, precisaremos de uma série de lemas.

\begin{lemma}
O grupo $\langle a, d \rangle $ é o grupo diedral de ordem $8$.
\label{lemma:diedral}
\end{lemma}
\begin{proof}
Como $a^2=d^2=1$, temos que $D:=\langle a, d \rangle$ é diedral, e sua ordem é duas vezes a ordem de $ad$. Escrevendo $(ad)^2 = ada\cdot d$ vemos que $\varphi((ad)^2) = (b,b)$. Mas como $b^2=1$ e $\varphi$ é injetiva, concluímos que $(ad)^4=1$.
\end{proof}

\begin{lemma}
Seja $B:=\langle\langle b \rangle\rangle$ o fecho normal de $\langle b\rangle$ em $\Gamma.$ Então, $[\Gamma :B]\leq 8.$
\end{lemma}
\begin{proof}
Como $c=bd$,  obtemos que $\Gamma = \langle a, b, d\rangle$. Assim, o grupo quociente $\Gamma/B$ é gerado pelas imagens de $a$ e $d$ pela projeção canônica. Pelo lema anterior, temos que o índice de $B$ será menor ou igual a $8$.
\end{proof}

\begin{lemma}
O grupo $\varphi (H)$ contém o subgrupo $B\times B$ de $\Gamma \times \Gamma$.
\end{lemma}
\begin{proof}
Sabemos que $\varphi(ada)=(b,1)$ e $\varphi(d)=(1,b)$. Para cada $x\in \Gamma $, existem elementos $h_1, h_2 \in H$ que são enviados por $\varphi$ em elementos das formas $(x,y)$ e $(z,x)$ de $\Gamma \times \Gamma$. Conjugando $(b,1)$ e $(1,b)$, por $(x,y)$ e $(z,x)$ respectivamente, obtemos 
\begin{eqnarray*}
    (xbx^{-1},1) &=&(x,y)(b,1)(x,y)^{-1} \\ &=& \varphi(h_1) \varphi(ada) \varphi(h_1^{-1})\\
    &=& \varphi(h_1 ada h_1^{-1}) \in \varphi(H),
\end{eqnarray*}
\begin{eqnarray*}
    (1,xbx^{-1}) &=&(z,x)(1,b)(z,x)^{-1} \\ &=& \varphi(h_2) \varphi(d) \varphi(h_2^{-1})\\
    &=& \varphi(h_2 d h_2^{-1}) \in \varphi(H).
\end{eqnarray*}
Logo, os subgrupos $(B,1)$ e $(1,B)$ estão ambos contidos em $\varphi(H)$ e assim podemos concluir que $B\times B \leq \varphi(H)$.
\end{proof}

Finalizamos a prova do Teorema \ref{thm:gammacommgammagamma} usando os lemas anteriores para concluir que $B\times B \leq \varphi (H) \leq \Gamma \times \Gamma $, onde o índice de $B\times B$ em $\Gamma\times \Gamma $ é finito, e portanto o índice de $\varphi(H)$ em $\Gamma\times \Gamma $ é finito. Logo, $H$ é isomorfo a $\varphi(H)$, onde $H$ tem índice finito em $\Gamma$ e $\varphi(H)$ tem índice finito  em $\Gamma\times \Gamma $ e o teorema segue.
\end{proof}

\begin{remark}
    Existem grupos finitamente gerados $G$ que são até mesmo isomorfos ao produto direto $ G \times G$ (cf \cite{jones1974direct}). 
\end{remark}

\begin{corollary}
O grupo $\Gamma$ não tem crescimento polinomial.
\end{corollary}
\begin{proof}
Suponha que $\Gamma$ tivesse crescimento polinomial, digamos de grau $d$. Então, $\Gamma \times \Gamma $ teria crescimento $\rho_{\Gamma \times \Gamma}(n)\asymp n^{2d}$. Como o crescimento é um invariante por quasi-isometrias, e portanto um invariante por comensurabilidade, segue do teorema anterior que $n^d\asymp n^{2d}$. Logo teríamos $d=0$, o que implicaria que $\Gamma$ seria finito, um absurdo.
\end{proof}

Os próximos passos são necessários para provar que $\Gamma$ não tem crescimento exponencial.
\begin{lemma}
Seja $x \in H,$ e escreva $\varphi(x)= (x_{\ell},x_r)$, como antes.  Então $\ell(x_{\ell}), \ell(x_r)\leq \dfrac{1}{2}({\ell}(x)+1)$.
\end{lemma}

\begin{proof}
Escrevendo $x$ na forma canônica, temos que $x$ é o produto alternado de elementos da forma $u$ e $ava$, com $u,v \in \{b,c,d\}$. Seja $2k+i$ o número de fatores, onde $i=0$ ou $1$. Cada par de fatores $u$ e $ava$ contribui com $4$ unidades no comprimento de $x$, enquanto cada coordenada da imagem por $\varphi$ do produto $u\cdot ava$ tem comprimento $\leq 2$. Se $i=0$, então ${\ell}(x)=4k$, e ${\ell}(x_{\ell}),{\ell}(x_r)\leq 2k$ e portanto, ${\ell}(x_{\ell}),{\ell}(x_r) < \dfrac{1}{2}({\ell}(x)+1)$. Se $i=1$ então o comprimento do fator restante é $1$ ou $3,$ logo ${\ell}(x)=4k+1$ ou ${\ell}(x)=4k+3$. Por outro lado, esse fator restante só contribui com $1$ em cada coordenada  da imagem por $\varphi$, ou seja, ${\ell}(x_{\ell}),{\ell}(x_r)\leq 2k+1$. Portanto,  ${\ell}(x_{\ell}),{\ell}(x_r)\leq 2k+1\leq \dfrac{1}{2}({\ell}(x)+1)$, como queríamos. 
\end{proof}

Para verificarmos o crescimento subexponencial, basta analisarmos o crescimento de $H$, já que $[\Gamma: H]$ é finito.

Note que $[\Gamma \times \Gamma : H\times H] =4$, logo $[H: \varphi^{-1}(H\times H)] = [\varphi(H):\varphi(H) \cap H\times H]\leq 4$. 

Considere $K= \varphi^{-1}(H\times H)$ e $L= \varphi^{-1}(K\times K)$. Então valem 
\begin{align*}
    [\Gamma:K] &= [\Gamma:H][H:K] \leq 8;\\
    [\Gamma:L] &= [\Gamma:K][K:L] \leq 128.
\end{align*}

Portanto, $\rho_{L} \asymp \rho_{H}$. Analisemos o que ocorre em $L$. Note que $\varphi(L)= K \times K = \varphi^{-1}(H\times H) \times \varphi^{-1}(H\times H) \subset H\times H$ e $\varphi^2(L) = H \times H \times H \times H$ . Assim, é possível aplicar $\varphi$ a $L$ três vezes: 
$$x \in L \mapsto (x_{\ell}, x_r) \mapsto (x_{\ell \ell}, x_{\ell r}, x_{r \ell}, x_{ r r}) \mapsto (x_1, \ldots, x_8),$$
onde por simplicidade denotamos por $x_i$, $i=1,\ldots 8$ os elementos de $H$ obtidos após aplicar $\varphi$ a cada um dos elementos anteriores, na ordem em que eles aparecem. Escreveremos também  $(\tilde{x}_1, \tilde{x}_2, \tilde{x}_3 , \tilde{x}_4) =  (x_{\ell \ell}, x_{\ell r}, x_{r \ell}, x_{ r r})$.

\begin{lemma}
    Com a notação acima, tem-se $\displaystyle\sum_{i=1}^8 \ell(x_i) \leq \frac{3}{4}\ell(x) + 8$.
    \label{lemma:sum1to8}
\end{lemma}

\begin{proof}
Pelo lema anterior, temos 
\begin{eqnarray*}
    \displaystyle\sum_{i=1}^8 \ell(x_i) &\leq & \sum_{i=1}^4 \ell(\tilde{x}_i) +4\\
    &\leq & \ell(x_{\ell})+1+ {\ell}(x_r) +1 + 4\\
    &\leq & \ell(x)+7.
\end{eqnarray*}    

Denote por $\ell_u(x)$ o número de ocorrências de $u$ na escrita de $x$ em sua forma canônica, onde $u \in \{b,c,d\}$.

Cada ocorrência de $d$ em $x$, seja apenas na forma $``d"$ ou em $``ada"$, se torna $1$ em $x_{\ell}$ ou em $x_r$. Logo, ao aplicarmos $\varphi$ uma vez, $\ell_d(x)$ é subtraída. Em seguida, cada ocorrência de $c$ se torna $1$ ao aplicarmos novamente $\varphi$. Logo, $\ell_c(x)$ é subtraída ao aplicarmos $\varphi^2$ e, analogamente, $\ell_b(x)$ é subtraída ao aplicarmos $\varphi^3$.

No entanto, observe que não podemos subtrair ambos $\ell_b$ e $\ell_c$ simultaneamente, em virtude do fato de que ocorrências de $b$ e $c$ podem ocasionar ocorrências de $b$ na palavra seguinte. Por exemplo, em $abadaca \mapsto (c\cdot 1 \cdot d, a\cdot b\cdot a) = (b,aba)$. Mas sempre podemos subtrair, ao aplicar $\varphi^3$, os valores $\ell_b(x)+ \ell_d(x)$ ou $\ell_c(x)+\ell_d(x)$.

Pelo menos um deles é  $\geq \frac{1}{2}(\ell_b(x)+ \ell_c(x) + \ell_d(x))$ e além disso $\ell_b(x)+ \ell_c(x) + \ell_d(x) \geq \frac{1}{2}(\ell(x)-1)$. Portanto, obtemos:

\begin{eqnarray*}
    \displaystyle\sum_{i=1}^8 \ell(x_i) &\leq & \ell_x+7 - (\ell_{b\ \text{ou}\ c}(x)+ \ell_d(x))\\
    & \leq & \ell(x)+7 - \frac{1}{2}(\ell_{b}(x) + \ell_c(x) + \ell_d(x))\\
    & \leq & \ell(x)+7 - \frac{1}{4}(\ell(x) -1) = \frac{3}{4}\ell(x) + 8.
\end{eqnarray*}
\end{proof}

\begin{thm}
O grupo $\Gamma$ tem crescimento intermediário.
\label{thm:intermediategrowth}
\end{thm}

\begin{proof}
    Seja $\omega = \displaystyle\lim_{n\to \infty} \rho_S(n)^{{1}/{n}}$. Queremos mostrar que $\omega=1$. Considere o subgrupo $L < \Gamma$ definido acima, escreva $\Gamma/L = \{z_1L, \ldots, z_{128}L\}$ e seja $k>0$ tal que $\ell(z_i)\leq k$, para todo $i=1, \ldots, 128$. 

    Dado $x \in \Gamma$, temos $x=z_iy$, para algum $i\in \{1, \ldots, 128\}$ e algum $y\in L$. Se $\ell(x) = n$, então $\ell(y) = \ell(z_i^{-1}x) \leq n+k$. 

    Denote por $\rho_L^{\Gamma}(n)$ a cardinalidade de $\{x \in L \mid \ell _{\Gamma}(x)\leq n\}$. Note que $$\rho_{\Gamma}(n) \leq \rho_L^{\Gamma}(n+k)[\Gamma:L].$$

    Para  $x \in L$, escreva $\varphi^3(x) = (x_1, \ldots , x_8)$ como na notação do  Lema \ref{lemma:sum1to8}. Sendo $\varphi^3$ um mergulho, segue que $x$ é determinado pelos elementos $x_i$. Sejam $n= \ell(x)$ e $n_i = \ell(x_i) $. Então $n_i \leq n$, e portanto o número de possibilidades  para  $(n_1, \ldots , n_8)$  é de, no máximo,  $(n + 1)^8$. 
    Dado um octeto $(n_1, \ldots , n_8)$, o  número de possibilidades  para  $(x_1, \ldots , x_8)$  é de $\displaystyle \prod_{i=1}^8 \rho_{\Gamma}(n_i)$. Dado $\varepsilon > 0$, temos  $\omega^n \leq \rho_{\Gamma}(n) \leq (\omega + \varepsilon)^n$, sendo a segunda desigualdade válida para $n$ suficientemente grande. Então existe uma constante $A$ tal que $\rho_{\Gamma}(n) \leq A(\omega + \varepsilon)^n$ para todo $n \in \N$. Dado um octeto $(n_1, \ldots , n_8)$ como acima, isso implica que o número de possibilidades para $(x_1, \ldots , x_8)$ é, no máximo, 
    
    $$\displaystyle \prod_{i=1}^8 A(\omega + \varepsilon)^{n_i} = A^8(\omega + \varepsilon)^{\displaystyle \sum_{i=1}^8n_i} \leq  A^8(\omega + \varepsilon)^{\frac{3}{4} n+8} \leq C(\omega + \varepsilon)^{\frac{3}{4}n},$$
para alguma constante $C>0$. Assim, 
$$\omega^n \leq \rho_{\Gamma}(n) \leq [\Gamma : L]\rho_L^{\Gamma}(n + k)\leq [\Gamma : L]C(n + k + 1)^8(\omega + \varepsilon)^{\frac{3}{4}
(n+k)}.$$
Tomando raízes $n$-ésimas e fazendo $n \to \infty$ e $\varepsilon \to 0$,  chegamos a 
$\omega \leq \omega^{{3}/{4}}$. Isso implica que  $\omega = 1$.    
\end{proof}

Uma análise da prova do Teorema \ref{thm:intermediategrowth} permite provar um resultado ainda mais forte.
\begin{thm}
Existem constantes $0<\alpha, \beta <1$ e $A,B>1$ tal que $A^{n^{\alpha}}\leq \rho_{\Gamma}(n)\leq  B^{n^{\beta}}.$
\end{thm}

Os valores de $A$ e $B$ não importam muito, dado que, para $\alpha$ fixado, todas as funções $A^{n^{\alpha}}$, para $A > 1$, são equivalentes. Por outro lado, se $\alpha \neq \beta$, então $A^{n^{\alpha}}$ e $A^{n^{\beta}}$ não são equivalentes. Portanto, existe um interesse em determinar os valores das constantes $\alpha$ e $\beta$ no teorema acima.
Analisando a prova de Grigorchuk, podemos obter $\alpha= \frac{1}{2}$ e $\beta = \log _{{32}}31 = 0.991...$.
Mais tarde, em \cite{Bartholdi98} e \cite{bartholdi2001lower} L.~Bartholdi melhorou as estimativas para $\alpha=0,5157...$ e $\beta=0.7674...$, onde 
$$\beta = \frac{\log 2}{\log \lambda},\quad \lambda^3 - \lambda^2 - 2\lambda - 4 = 0.$$

Por um longo tempo, foi um problema em aberto a existência de $\alpha$ tal que o crescimento de $\Gamma$ é da ordem de
$2^{n^{\alpha}}$.
Até mesmo a existência do limite
$${\displaystyle \lim _{n\to \infty }\log _{n}\log \rho_{\Gamma}(n)}$$ demorou para ser provada. Este problema foi finalmente resolvido em 2020 por A.~Erschler e T.~Zheng \cite{EZh2020}. Elas mostraram que o limite é igual a $\beta$.

Por outro lado, M.~Kassabov e I.~Pak \cite{KP13} construíram uma família incontável de grupos finitamente gerados de crescimento intermediário, com funções de crescimento oscilando entre limites inferiores e superiores, ambos provenientes de uma ampla classe de funções. Em particular, pode-se ter crescimento oscilando entre $e^{n^{\alpha}}$ e qualquer função prescrita, crescendo tão rapidamente quanto desejado. A construção de Kassabov e Pak é baseada nos grupos de Grigorchuk e uma variação do limite do produto entrelaçado permutacional.

\medskip

\section{Algumas outras propriedades do grupo~\(\Gamma\)}

\begin{thm}
\label{thm:2group}
O grupo de Grigorchuk $\Gamma$ é um $2$-grupo, i.e., para todo elemento $x \in \Gamma $ existe $n\in \N$ tal que $x^{2^{n}}=1.$
\end{thm}

\begin{proof}

Seja $x \in \Gamma$. Provaremos, por indução em $\ell(x)$, que $x^{2^n} = 1$ para algum $n\in \N$. Se $\ell(x) = 1$, então $x^2 = 1$. Assuma que $\ell(x) > 1$. Se uma palavra de comprimento $\ell(x)$ representando $x$ começa com $b$, digamos, então $bxb $ é um conjugado de $x$ com, no máximo, o mesmo comprimento. Analogamente, o mesmo ocorre se $x$ começa com $c$ ou $d$. Como elementos
conjugados tem a mesma ordem, podemos assumir que $x$ começa com $a$. 

Se $\ell(x) = 2$, então $x= ab, ac$, ou $ad$. Pelo Lema \ref{lemma:diedral}, $ad$ tem ordem $4$. Para $ac$ temos $\varphi((ac)^2) = (da, ad)$, e portanto $ac$ tem ordem $8$, e analogamente, $ab$ tem ordem $16$. Seja $\ell(x) \geq 3$. Se $x$  também termina com $a$, então $x$ é da forma $aya$, portanto é conjugada a $y$, a qual é mais curta. Assumimos então que $x$ começa com $a$ e termina com $b, c$, ou $d$, e assim possui comprimento par, digamos $2k$.

Se $k$ também é par, então $x \in  H$,  e escrevendo $\varphi(x) = (x_{\ell}, x_r)$, o comprimento de ambos $x_{\ell}$ e $x_r $ é no máximo $\frac{1}{2}\ell(x)$, e como a ordem de $x$ é o mínimo múltiplo comum das ordens de $x_{\ell}$ e $x_r$, a hipótese de indução se aplica. Resta o caso $\ell(x) = 4r+2$, para algum $r\in \N$. Nesse caso, $x^2 \in H$ e, escrevendo $\varphi(x^2) = (y_{\ell}, y_r)$, cada um dos elementos $y_{\ell}, y_r$ tem comprimento no máximo $\ell(x)$. 

Suponha inicialmente que $x$ contenha a letra $d$. Como $x^2$ (da maneira que é escrito) tem comprimento $2\ell (x)$ e é periódico com período $\ell(x)$, a letra $d$ ocorre ali no mínimo duas vezes, em posições que diferem de $\ell(x) = 4r + 2$. Isto significa que podemos escrever $x^2$ como um produto dos geradores $b, c, d, aba, aca, ada$ de $H$, onde ambos $d$ e $ada$ ocorrem. Então em $y_{\ell}$ o gerador $d$ se torna $1$, e em  $y_r$ o gerador $ada$ se torna $1$, e portanto ambos $y_{\ell}$ e $y_r$ tem comprimento menor do que $x$, e por indução $x^2$ tem ordem igual a uma potência de $2$, implicando que a ordem de $x$ também é dessa forma. 

Suponha agora que $x$ não envolve $d$, mas sim $c$. Então, com a mesma notação, $y_{\ell}$ e $y_r$ ambos envolvem $d$, ou esse $d$ desaparece com cancelamentos, e então a parte relevante de $y$ tem comprimento menor. Portanto ou o caso anterior ou a hipótese de indução se aplica. Finalmente, se nenhum dos termos $d$ ou $c$ ocorre, então $x$ é uma potência de $ab$, e tem ordem dividindo $16$. Assim, o teorema está provado.
\end{proof}

Provamos que $\Gamma$ é um grupo infinito, finitamente gerado e de torção. A existência de grupos com essas propriedades foi um problema em aberto por muitos anos, e ficou conhecido como \textit{problema de Burnside}. \index{problema de Burnside} Tal problema foi resolvido pela primeira vez por E.~S.~Golod \cite{golod1964nil}. 
O grupo de Grigorchuk foi primeiro construído por S.~V.~Aleshin \cite{aleshin1972finite}, como um outro exemplo de grupo infinito tipo Burnside. 

Um problema bem mais difícil é construir um grupo infinito do tipo Burnside com torção uniformemente limitada. Em 1968, P.~S.~Novikov e S.~I.~Adian \cite{adian1979burnside} deram um exemplo de um tal grupo. O grupo de Grigorchuk não é um exemplo para esse problema. De fato, em $\Gamma$ existem elementos de ordem arbitrariamente grande.

Ainda não se sabe se existem grupos de crescimento intermediário do tipo Burnside com torção uniformemente limitada. Observe que a prova do Teorema \ref{thm:2group} fornece uma estimativa para as ordens em termos de $\ell(x)$, mas essa estimativa não é ótima.

Seja $X$ uma classe de grupos. Chamaremos de \textit{$X$-grupo} todo grupo nessa classe. 
\begin{definition}
Um grupo $\Gamma$ é denominado \textit{residualmente $X$} \index{grupo residualmente $X$} se os homeomorfismos de $\Gamma$ para um $X$-grupo separam pontos, isto é: dados dois elementos distintos $x, y \in \Gamma $, existe um homomorfismo $\psi : \Gamma \to H$, onde $H \in X$, tal que $\psi(x)\neq \psi(y)$.
\end{definition}
Observamos que, como $x\neq y$ se, e somente se, $xy^{-1} \neq 1$, é suficiente separar $1$ de qualquer elemento de $\Gamma$.

\begin{lemma}\label{lemr}
Seja $X$ uma classe de grupos. Os homeomorfismos sobrejetivos de $\Gamma$ para um $X$-grupo separam pontos se, e somente se, $\bigcap N=\{1\}$ onde $N$ percorre todos subgrupos normais de $\Gamma$ tais que $\Gamma/N$ é um $X$-grupo.
\end{lemma}
\begin{proof}
Suponha que os homeomorfismos sobrejetivos de $\Gamma$ para $X$-grupos separam pontos. Assim, podemos supor que para todo $x \in \Gamma \setminus \{1\}$, existe um homomorfismo sobrejetivo $\varphi: \Gamma \to H$ tal que $x \notin \ker(\varphi)$. Como $\Gamma/\ker(\varphi) \cong H \in X$ temos que $x \notin \bigcap N$, logo $\bigcap N=\{1\}$. Reciprocamente, suponha que $\bigcap N=\{1\}$ onde $N$ percorre todos subgrupos normais de $\Gamma$ tais que $\Gamma/N$ é um $X$-grupo. Seja $x \in \Gamma $, $x\neq 1$ então $x \notin \bigcap N$ e existe $N_0$ tal que $x \not\in N_0$. 
Seja $\pi: \Gamma \to \Gamma/N_0$ a projeção canônica. Então, $\pi$ é um homomorfismo sobrejetivo de $\Gamma$ para $\Gamma/N_0\in X$ e $\pi(x)\neq 1.$ Portanto, $\Gamma$ separa pontos.
\end{proof}

Dizemos que $\Gamma$ é {\it residualmente finito} \index{grupo residualmente finito} se $X$ é a classe de todos os grupos finitos, e que $\Gamma$ é {\it residualmente-$p$} \index{grupo residualmente-$p$}se $X$ é a classe de todos os $p$-grupos finitos.

\begin{thm}
O grupo de Grigorchuk $\Gamma$ é residualmente-$2$.
\end{thm}

\begin{proof}
Considere os $2^n$ subintervalos $\left(\dfrac{k}{2^n},\dfrac{k+1}{2^n}\right) \subset (0,1)$, $k=0,\ldots, 2^{n-1}$. Denote por $S_n$ o conjunto desses intervalos. Podemos ver que $\Gamma$ permuta esses intervalos entre eles, e assim temos uma ação de $\Gamma$ em $S_n.$ Seja $H_n$ o núcleo dessa ação, ou seja, o subgrupo de $\Gamma$ que fixa os intervalos em $S_n$. Por exemplo, $H_0 = \Gamma $, $H_1=H$, etc. Então cada $\Gamma/H_n$ é finito. Todo ponto de $(0,1)$ é a interseção de todos intervalos do tipo acima que o contenham, e portanto $\bigcap H_n=1$. Assim, pelo Lema~\ref{lemr}, $\Gamma$ é residualmente finito. Aplicando o Teorema \ref{thm:2group}, cada fator finito de $\Gamma$ tem ordem $2^m$, o que implica que $\Gamma$ é residualmente-$2$.
\end{proof}

\begin{thm}
O problema da palavra é solúvel em $\Gamma$.
\end{thm}

\begin{proof}
Seja $x$ qualquer palavra nos geradores $\{a, b, c, d\}$. Contamos o número de ocorrências de $a$ nessa palavra. Se esse número for ímpar, então $x\notin H$, e certamente $x \neq 1$. Se $x \in H$, podemos avaliar $\varphi (x) = (x_\ell, x_r)$. Pelo Lema $10.10$, se $\ell(x)> 1$, então $x_\ell$ e $x_r$ são palavras mais curtas, e por indução nós podemos determinar se elas são a identidade ou não, enquanto se $\ell (x) = 1$, então $x$ é um dos elementos $b, c, d,$ e portanto não é a identidade. Uma vez que $x = 1$ se, e somente se, $x_\ell = x_r = 1$ dado que $\varphi$ é injetiva, o resultado segue.
\end{proof}

Mais algumas propriedades interessantes de $\Gamma$ estão listadas a seguir:
\begin{enumerate}
\item O grupo $\Gamma$ é ``apenas infinito''. Isto significa que $\Gamma/N$ é finito para cada subgrupo normal $N \neq \{1\}$.
\item O grupo $\Gamma$ contém elementos de ordem arbitrariamente grande.
\item O grupo $\Gamma$ não é finitamente apresentado (mas é recursivamente apresentado.)
\end{enumerate}

Também listamos alguns problemas ainda em aberto, os quais são possivelmente muito difíceis:

\noindent
{\bf Problema 1.} \textit{Existe um grupo finitamente apresentado de crescimento intermediário?}

\noindent
{\bf Problema 2.} \textit{Existe um grupo de crescimento intermediário com todos os elementos de ordem limitada por uma mesma constante?}

O terceiro problema não menciona crescimento, mas o incluímos pois discutimos o problema de Burnside.

\noindent
{\bf Problema 3.} \textit{Existe um grupo infinito, finitamente apresentado, de torção?}

\chapter{Grupos amenáveis}
\label{cap11}
\section{Definição e exemplos}

\begin{definition}
Um grupo $G$ é dito  \textit{amenável} \index{grupo amenável} se for possível definir sobre $G$ um medida não trivial, finitamente aditiva e invariante à direita.
\end{definition}

Isto significa que podemos definir uma função $\mu: 2^G \to \R_{\geq 0},$ que associa a cada subconjunto $A$ de $G$ um número não negativo $\mu(A)$, e satisfaz:

\begin{enumerate}[(1)]
\item ($\mu$ é finitamente aditiva) Se $A,B \subset G$ e $A\cap B=\emptyset$, então $\mu(A\cup B)=\mu(A)+\mu(B);$

\item$ (\mu$ é invariante à direita) Se $A \subset G$ e $x\in G$, então $\mu(Ax)=\mu(A);$

\item ($\mu$ é não trivial) $\mu(G)>0.$
\end{enumerate}

Diferentemente de medidas de Lebesgue ou de Haar, a medida $\mu$ acima é definida para todos os subconjuntos de $G$, e ela somente precisa ser  aditiva  para uniões finitas, podendo não ser aditiva para uniões enumeráveis, por exemplo.

Ao multiplicar $\mu$ por qualquer constante positiva, obtém-se outra medida com as mesmas propriedades, portanto podemos sempre assumir que $\mu(G) =
1$.

Esta noção de amenabilidade foi concebida por von Neumann \cite{Neu29} na tentativa de compreender as decomposições da esfera $\mathbb{S}^2$. 

\begin{example}
 Seja $G$ um grupo finito. Então $\mu (A) =\dfrac{| A |}{| G |}$ é uma tal medida, e na verdade é a única possível (basta usar o item $1$). Logo, todo grupo finito é amenável.
\end{example}
 
Por outro lado, se $G$ é amenável e infinito, então todo subconjunto finito de $G$ tem medida nula. De fato, suponha que existe $A \subset G$ finito com $\mu(A)=\varepsilon >0$. Seja $n\in \N$ tal que $n\varepsilon> \mu(G)$ e escolha elementos $g_0 = 1, g_1, \ldots, g_{n-1}\in G$ tais que $A, Ag_1, \ldots, Ag_{n-1}$ são disjuntos (os quais existem, pois $G$ é infinito). Usando os itens $1$ e $2$, obtemos: $$\mu(G)\geq \mu\left(\bigcup_{i=0}^{n-1}Ag_i\right)=n\mu(A)>\mu(G),$$  um absurdo. Logo $\mu(A)=0$ para cada $A$ finito. 

\begin{example}
\label{ex:Zisamenable}
    Considere o grupo cíclico infinito $\Z$. Fixe algum ultrafiltro não principal $\mathcal{F}$ em $\N$ (isso requer uma forma do Axioma da Escolha, veja a Seção \ref{sec:filtrosultrafiltros}). Para cada subconjunto $A$ de $\Z$, defina 
    $$A_n = A \cap [-n, n],\ \mu_n(A) = \dfrac{|A_n|}{2n+1},\text{ e } \mu(A) = \mathcal{F}\lim \mu_n(A).$$ 
    Temos:
    \begin{enumerate}[(1)]
    \item Se $A,B \subset G$ e $A\cap B=\emptyset$, então 
    \begin{eqnarray*}
        \mu(A\cup B) &=& \mathcal{F}\lim \mu_n(A\cup B) \\ 
        & = & \mathcal{F}\lim (\mu_n(A) + \mu_n(B)) = \mu(A)+\mu(B).
    \end{eqnarray*}
    \item Vamos mostrar que $\mu$ é invariante sob a translação $g: z \to z + 1$. Observe que
    $$| |A_n| - |Ag\cap [-n, n] | \leq 2.$$
    Então 
    $$|\mu(A) - \mu(Ag)| \leq \mathcal{F}\lim \dfrac{2}{2n+1} = 0,$$
    o que implica que $\mu$ é $\Z$-invariante.
    \item Claramente, $\mu$ é não negativa e $\mu(\Z) = 1$.
    \end{enumerate}
    Concluímos que o grupo $\Z$ é amenável.
\end{example}

\begin{remark}
Este resultado foi provado pela primeira vez por Banach \cite{Ban23} usando o teorema de Hahn--Banach.
\end{remark}

\begin{example}[von Neumann \cite{Neu29}]\label{ex F_2 nao amenavel}
O grupo livre de posto dois $F_2$ não é amenável.  De fato, suponha que exista alguma medida $\mu$ (não trivial, finitamente aditiva e invariante à direita) em $G$. Sejam $x$ e $y$ geradores de $G$. Então 
$$\mu(G \setminus \{1\}) = 1,\ G \setminus \{1\} = A\cup B \cup C \cup D,$$ 
onde $A$, $B$, $C$, $D$ são os conjuntos das palavras reduzidas terminadas em $x, x^{-1}, y, y^{-1},$ respectivamente. Então 
$$A \setminus \{x\} = Ax \cup Cx \cup Dx,$$ 
uma união disjunta. Como $\mu (A\setminus \{x\}) = \mu (A) = \mu (Ax)$, segue-se que $\mu (C) = \mu (D) = 0$. Um argumento semelhante mostra que $\mu (A) = \mu (B) = 0$, e assim $\mu(G) = 0$, uma contradição. Assim $G = F_2$ não é amenável.
\end{example}

Em geral, temos a seguinte definição:
\begin{definition}
Um grupo \(G\) é chamado \textit{paradoxal} se existirem subconjuntos disjuntos \(A_1, ..., A_n\) e \(B_1, ..., B_m\) em \(G\), juntamente com elementos do grupo \(g_i, h_j\), tais que \(G = \bigcup_{i=1}^n g_iA_i = \bigcup_{j=1}^m h_jB_j\).
\end{definition}

O Exemplo~\ref{ex F_2 nao amenavel} mostra que $F_2$ é paradoxal. De fato, se um grupo $\Gamma$ contém um subgrupo isomorfo a $F_2$, então $\Gamma$ também é paradoxal e, em particular, não é amenável.
Com efeito, uma decomposição paradoxal de $\Gamma$ pode ser obtida a partir de uma decomposição paradoxal de $F_2$. Para isso, consideramos a decomposição de $\Gamma$ em classes laterais à esquerda em relação a $F_2$. Escolhendo um conjunto de representantes para essas classes (o que, em geral, requer o axioma da escolha), podemos transportar a decomposição paradoxal de $F_2$ para todo $\Gamma$ por meio de multiplicação à esquerda pelos representantes das classes laterais.
\medskip

A noção de amenabilidade foi originalmente desenvolvida para compreender as \textit{decomposições paradoxais}\index{decomposição paradoxal} da esfera $\mathbb{S}^2$. De fato, existem $8$ subconjuntos $A_1$, \ldots , $A_4$ e $B_1$, \ldots, $B_4$ da esfera $\mathbb{S}^2$ e elementos $g_i$, $h_i$ do grupo ortogonal $\mathrm{SO}(3)$ tais que:
$$\mathbb{S}^2 = A_1\cup \ldots \cup A_4 \cup B_1\cup \ldots \cup B_4;$$
$$\mathbb{S}^2 = g_1(A_1) \cup \ldots \cup g_4(A_4),\ \mathbb{S}^2 = h_1(B_1) \cup \ldots \cup h_4(B_4),$$
onde todas as uniões são disjuntas.

Este resultado é válido em grande generalidade e é frequentemente popularizado como prova de que uma maçã pode ser cortada em pedaços finitos que podem ser reunidos para obter duas cópias idênticas da maçã inicial. Foi considerada uma das conquistas mais espetaculares da matemática pura na primeira metade do século XX. Por exemplo, Feynman escreve em sua autobiografia que este resultado foi apresentado como um argumento para estudar matemática em vez de física.

Para a esfera de dimensão $1$, que é o círculo, Banach mostrou que não há decomposições paradoxais \cite{Ban23}. 


\medskip

Uma pergunta natural é se a recíproca implicação e verdadeira: 
$$\text{Um grupo } G \text{ não é amenável} \implies F_2 \subseteq G?$$ 
Esta questão, também conhecida como Conjectura de von Neumann,\index{conjectura de von Neumann} foi por muito tempo um problema em aberto. Um teorema importante sobre essa conjectura é a Alternativa de Tits, que estabelece a conjectura nos casos em que se aplica. Em geral, foi provado na década de 1980 que a conjectura não é verdadeira. Em \cite{Ol80}, A. Yu.~Ol'shanskii mostrou que o grupo chamado Monstro de Tarski, no qual vê-se facilmente que não há um subgrupo livre de posto dois, é não-amenável.  Um outro contra-exemplo é o seguinte grupo de Burnside\index{grupos de Burnside}:
$$B(2,4381) = \langle a,b \mid \omega^{4381}(a,b)\rangle,$$
onde $\omega^{n}(a,b)$ denota qualquer palavra $\omega(a,b)$ elevada à potência $n$. Em 1968, P. S.~Novikov e S. I.~Adian provaram que este grupo é infinito. Mais tarde, voltando à prova do teorema, Adian melhorou o limite do exponente ímpar para $n \geq 665$ e mostrou que esses grupos não são amenáveis \cite{Ad82} 
(\textit{nota:} Sabemos que os grupos $B(2, n)$ para $n = 2$, $3$, $4$ e $6$ são finitos. Para $n=5$ e de  $n=7$ a $100$ é desconhecido, e para maior ímpar $n$ os grupos são infinitos, com a última melhoria devido a Adian \cite{adian2015new}).
Nenhum desses contraexemplos é finitamente apresentado, e por alguns anos foi considerado possível que a conjectura fosse válida para grupos finitamente apresentados. Entretanto, em \cite{OlS03}, A. Yu.~Ol'shanskii e M. V.~Sapir exibiram uma coleção de grupos finitamente apresentados que não satisfazem a conjectura.

\medskip

Num livro recente de Cohen e Gelander \cite{CG24},  leitores interessados podem encontrar mais detalhes, exemplos e outras propriedades de amenabilidade e tópicos relacionados, com ênfase na abordagem analítica.

\section{Medida e integração}

Em qualquer espaço com uma medida, podemos desenvolver a noção de integral definida por essa medida. Isto é particularmente simples no caso de um grupo amenável $G$, porque todos os subconjuntos de $G$ são mensuráveis e, portanto, todas as funções são mensuráveis. 

Seja $f : G \to \R$ uma função limitada, tal que $a \leq f(x) \leq b$ para todo $x \in G$. Tome $\Delta = \{a_0=a<a_1<\ldots < a_n=b\}$ uma partição do intervalo $[a,b]$. Denotando por $A_i$ o conjunto $\{x\in G \mid a_{i-1}\leq f(x)< a_i\}$, para cada $i=1,\ldots,n$, note que os conjuntos $A_i$ formam uma partição de $G$. Logo, 
$$\displaystyle \sum_{i=1}^{n}\mu(A_i)=\mu(G)=1.$$ 

Escreva $S_{\Delta} = \displaystyle\sum_{i=1}^n\mu(A_i)a_i$. Então $S_{\Delta}\geq  a\displaystyle\sum_{i=1}^n\mu(A_i)=a$. 
De modo análogo, definimos $s_{\Delta}=\displaystyle\sum_{i=1}^n \mu(A_i)a_{i-1}$ e então $s_{\Delta}\leq b$.

Se $\mathcal{Q}$ é um refinamento de $\Delta$, então $S_{\mathcal{Q}} \leq S_{\Delta}$ e $s_{\mathcal{Q}} \geq s_{\Delta}$. Quaisquer duas partições têm um refinamento comum, portanto, para quaisquer duas partições $\mathcal{Q}$ e $\Delta$, temos $s_{\mathcal{Q}} \leq S_{\Delta}$. Além disso, $0 \leq S_{\Delta} - s_{\Delta} \leq \max (a_i - a_{i-1})$. Segue que o ínfimo dos números $S_{\Delta}$, tomado sobre todas as partições, é igual ao supremo dos números $s_{\Delta}$. Este valor comum é definido como sendo a \textit{integral} de $f$ relativa a $\mu$, denotado como de costume por $\int f d\mu$.

A integral definida acima tem as seguintes propriedades básicas:
\begin{itemize}
\item Aditividade: $\int(f+g)d\mu = \int fd\mu+\int gd\mu$;
\item Invariância à direita: Dada $f$ limitada e $y \in G$ defina $f_y(x)=f(xy),$ então $\int f_y d\mu = \int f d\mu$.
\end{itemize}

\begin{proposition}
Se $G$ é um grupo amenável, então existe uma medida em $G$ que é invariante à esquerda e à direita.
\end{proposition}

\begin{proof}
Seja $\mu$ uma medida sobre $G$, e defina uma nova medida $\nu(A)=\int \mu(xA)d\mu$, a média da medida $\mu$ sobre o grupo. Como $\mu$ é invariante à esquerda segue que $\nu$ é invariante à esquerda. Pela propriedade de invariância da integral, $\nu$ é invariante à direita. 
\end{proof}

\begin{definition}\index{grupo localmente $\mathcal{G}$}
Seja $\mathcal{G}$ uma propriedade de grupos. Um grupo $G$ é dito \textit{localmente $\mathcal{G}$} se cada subgrupo finitamente gerado de $G$ tem a propriedade $\mathcal{G}$. 
\end{definition}

\begin{thm}
\label{thm:propamenaveis}
\begin{enumerate}
\item[(i)] Grupos finitos ou  abelianos são amenáveis;
\item[(ii)] Subgrupos ou quocientes de grupos amenáveis são amenáveis;
\item[(iii)] Uma extensão de grupos amenáveis é amenável. Lembramos que um grupo $G$ é dito uma extensão dos grupos $R$ e $N$ se existir uma sequência exata $1\to N\to G\to R\to 1$;
\item[(iv)] Um grupo localmente amenável é amenável.

\end{enumerate}
\end{thm}

\begin{proof}
\begin{enumerate}
\item[(i)] Já vimos que grupos finitos são amenáveis. Vimos também que $\Z$ é amenável. Agora, pelo item $(iii)$ segue que grupos abelianos finitamente gerados são amenáveis, devido ao teorema de classificação dos grupos abelianos finitamente gerados. Usando o item $(iv)$, concluímos que todo grupo abeliano é amenável.
\item[(ii)] Sejam $G$ um grupo amenável, $H\leq G$ e $N \triangleleft  G.$ Um subconjunto $A$ em $G/N$ é uma coleção de classes laterais $x_{\alpha}N$ de $N$. Defina então a seguinte medida em $G/N$: $\mu_{G/N}(A)=\mu_G(\bigcup x_{\alpha}N)$. Por outro lado, considere $R:=\{x_{\alpha}\in G \mid G=\bigcup x_{\alpha}H\}$, o conjunto de representantes de classes laterais à esquerda de $H$ em $G$. Dado $A \subset H$, defina $\mu_H(A)=\mu_G(RA)$. Deixamos como exercício a verificação de que as medidas $\mu_{G/N}$ e $\mu_{H}$ cumprem as propriedades de amenabilidade.

\item[(iii)] Sejam $\nu$ uma medida sobre $N$ e $\sigma$ uma medida sobre $G/N$, ambas invariantes à esquerda. Para uma classe lateral $Nx \in G/N$ defina $f_A(Nx)=\nu(N\cap Ax)$. Isso não depende da escolha do representante $x$, pois se $Nw=Nx$, então $wx^{-1} \in N$. Logo $w=xy$, com $y \in N$, e assim obtemos $f_A(Nw) = \nu(N\cap Axy) = \nu((N\cap Ax)y)=\nu (N\cap Ax) =f_A(Nx)$. Portanto, $f_A$ é uma função bem-definida em $G/N.$ Para $A \subset G,$ defina $\mu(A)=\int f_A d\sigma$. Assim, $\mu$ é aditiva. Note que, como $N$ é normal, temos $zN=Nz$. Logo, $f_{Az}(Nx) = f_A(Nzx)=f_A(zNx)$ e portanto, usando que $\int fd\sigma $ é invariante à direita, concluímos que  $\mu(Az) = \int f_{Az}d\sigma = \int f_A (zNx)d\sigma = \int f_A d\sigma=\mu(A).$ Note que $f_G$ é constante igual a $1$ e portanto $\mu(G)=1$, o que permite concluirmos que $G$ é amenável.
\item[(iv)] Considere o conjunto das funções $f : 2^G\to [0,1]$.
Esse conjunto pode ser visto como o produto cartesiano $[0, 1]^{2^G}$, o qual é compacto com a 
topologia produto. O conjunto de todas as medidas finitamente aditivas e invariantes $\mu$, satisfazendo $\mu(G) = 1$, é definido por várias igualdades entre os valores de $\mu$, e portanto é um conjunto fechado (possivelmente vazio) em $[0, 1]^{2^G}$. Sejam $H$ um subgrupo finitamente gerado de $G$ e $\mu$ uma medida em $H$. Estenda essa medida a $G$ pondo $\mu(A) = \mu(A \cap H)$. A  medida resultante em $G$ é finitamente aditiva e vale   $\mu(G) = 1$, mas ela somente é invariante com respeito à multiplicação por elementos de $H$. O conjunto $\mathcal{M}_H$ de todas as medidas em $G$ que são finitamente aditivas e $H$-invariantes é também fechado, e acabamos de ver que ele é não vazio. Se $H_1, \ldots , H_r$ são subgrupos de $G$ finitamente gerados, e pondo $K = \langle H_1,\ldots , H_r\rangle,$ os conjuntos $\mathcal{M}_{H_i}$
se intersectam em $\mathcal{M}_K$, e portanto essa interseção é não vazia. Como os conjuntos $\mathcal{M}_H$ são fechados num espaço compacto, sua interseção é não vazia, e toda função nessa interseção é uma medida invariante
em $G$.\end{enumerate}
\end{proof}

\begin{exercise}
Mostre que o grupo acendedor de lâmpadas (ver Seção~\ref{sec:sol-nilp-pol}) é amenável. \index{grupo acendedor de lâmpadas}
\end{exercise}

\begin{exercise}
Qualquer grupo $G$ contém um subgrupo normal amenável que contém todos os subgrupos normais amenáveis de $G$ (este subgrupo é chamado \textit{radical amenável} \index{radical amenável de um grupo} de $G$).
\end{exercise}

\begin{definition}
    A classe minimal $\mathcal{E}$ dos grupos satisfazendo os itens (i)--(iv) do Teorema~\ref{thm:propamenaveis} é denominada
 classe dos \textit{grupos amenáveis elementares}. \index{grupos amenáveis elementares}
\end{definition}

Assim, $\mathcal{E}$ é a menor classe de grupos que contém todos os grupos finitos e todos os grupos abelianos, é fechada com respeito à formação de subgrupos, quocientes e extensões. Além disso, ela contém $G$ sempre que contiver todos os  subgrupos finitamente gerados de $G$. Pelo Teorema \ref{thm:propamenaveis}, todo grupo na classe $\mathcal{E}$ é amenável.
Por muito tempo, foi uma questão em aberto a existência ou não de grupos amenáveis que não estejam na classe $\mathcal{E}$. Veremos que grupos de crescimento intermediário são exemplos neste sentido.

\section{Condição de Følner e crescimento}

\begin{definition}
Seja $G$ um grupo gerado por um conjunto $S$, e considere $A \subset G$.
A \textit{fronteira} $\partial A$ de $A$ é o conjunto dos elementos de $G$ cuja distância até $A$ é igual a $1$, isto é, dos elementos $x\in G$ tais que $x \notin A$, mas existe um elemento $y \in A$ e um  gerador $s \in S$ tal que $x = ys$ ou $x = ys^{-1}$.
\end{definition}

\begin{definition}
Dizemos que um grupo finitamente gerado $G$ satisfaz a  \textit{condição de Følner} \index{condição de Følner}
se $\inf \dfrac{|\partial X|}{|X|} = 0$, sendo o ínfimo tomado sobre todos os subconjuntos finitos $X \subset G$. Equivalentemente, podemos pedir que exista uma  \textit{sequência de Følner}\index{sequência de Følner} em $G$, isto é, uma sequência $X_n$ de subconjuntos finitos de $G$ tais que $\displaystyle \lim_{n\to \infty} \dfrac{|\partial X_n|}{|X_n|} = 0$.
\end{definition}

\begin{thm}
   Um grupo satisfaz a condição de Følner se, e somente se, é amenável.
\end{thm}

\begin{proof}
Suponhamos inicialmente que $G$ satisfaz a condição de Følner. Seja $X_n$ uma sequência de Følner em $G$. Como no Exemplo \ref{ex:Zisamenable}, fixe um ultrafiltro não principal $\mathcal{F}$ e, para cada subconjunto $A$ de $G$, defina $A_n = A \cap X_n$,
$\mu_n(A) = \dfrac{|A_n|}{|X_n|}$ e $\mu(A) = \mathcal{F}\lim \mu_n(A)$. 

A aditividade de $\mu$ é clara. Verifiquemos que $\mu$ é invariante à direita. Seja $s \in S$ um dos geradores de $G$. Então $|As \cap X_n| =
|A \cap X_ns^{-1}|$ e $A \cap X_ns^{-1} \subset A_n \cup \partial X_n$. Logo, 
$$\dfrac{|As \cap X_n|}{|X_n|} \leq
\mu_n(A)+\dfrac{|\partial X_n|}{|X_n|}.$$ Após passagem ao limite, obtemos $\mu(As) \leq \mu(A)$.
Da mesma forma,  $\mu(A) = \mu(As\cdot s^{-1}) \leq \mu(As)$, e assim $\mu(As) = \mu(A)$. Portanto, o grupo $G$ é amenável.

Reciprocamente, suponhamos que $G$ é amenável. Seguiremos a prova de Namioka \cite{Nam64}. 
Seja $\Phi $ o conjunto
$$ \{ f \in \ell^1(G) \mid f \geq 0 \mbox{ tem suporte finito e }\lVert f\rVert_{\ell^1(G)} = \displaystyle\sum_{g \in G}|f(g)|=1 \}.$$
Claramente, $\Phi$ é um subconjunto convexo de $\ell^1(G)$.

Podemos mostrar que se $A\subset G$ for finito e $\varepsilon>0$, então existe $f \in \Phi$ tal que $\lVert f - f_a\rVert_{\ell^1(G)} \leq \varepsilon$, para todo $a\in A$ (onde $f_g(x) = g(xg^{-1})$). De fato, suponha que isto não acontece. Então existem  $A\subset G$ finito e $\varepsilon>0$ tais que, para todo $f \in \Phi$, há algum $a \in A$ com $\lVert f - f_a\rVert_{\ell^1(G)} > \varepsilon$. Assim, $\{ f - f_a \mid f \in \Phi \}$ é um subconjunto convexo de $\ell^1(G)$ limitado a partir de 0, e pelo Teorema da Separação de Hahn--Banach (veja \cite[pp.~5--7]{Zal02}, existem $\beta \in \ell^1(G)^*$ e $t\in \R$ tais que $\beta(f - f_a) \ge t > 0$ para todo $f \in \Phi$. Como $\ell(G)^* = \ell^\infty(G)$, o espaço de funções limitadas em $G$, existe $m\in \ell^\infty(G)$ tal que 
$$\langle f - f*\delta_a, m\rangle = \sum_{x\in G}(f - f*\delta_a)(x)m(x) \geq t,\ \forall f \in \Phi,$$
onde $\delta_a$  e $*$  denotam a função delta de Dirac e a convolução de duas funções, respectivamente.

Tomando $f = \delta_y$ para $y\in G$, temos 
\begin{eqnarray*}
    \langle \delta_y - \delta_y*\delta_a, m\rangle& =& \sum_{x\in G}\delta_y(x)m(x) - \delta_y(xa^{-1})m(x)\\ & =& m(y) - m_{a^{-1}}(y)\geq t.
\end{eqnarray*}

Então $m(y) - m_{a^{-1}}\geq t$ para todo $y\in G$. Como $G$ é amenável, existe $M: \ell^\infty(G) \to \R$ invariante à direita. Aplicando $M$ à desigualdade acima obtemos $M(m(y) - m_{a^{-1}}(y))\geq t >0$, contradizendo a invariância à direita de $M$.

Em particular, dados $A=S \subset G$ e $\varepsilon>0$, existe $f \in \Phi$ tal que $\lVert f - f_a \rVert_{\ell^1(G)} \leq \varepsilon/|A|$, para todo $a \in A$. A função $f$ tem suporte finito, então $f=\displaystyle\sum_{i=1}^{n}c_i \chi_{F_i}$, para $c_i>0$ e certos conjuntos finitos $F_1 \supset F_2 \supset \ldots \supset F_n$.

Além disso, $\displaystyle\sum_{i=1}^{n}c_i|F_i| = 1$, pois $f\in \Phi$. Agora, $|f(g) - f_a(g)| \geq c_i$, para todo $g \in \partial_aF_i$, onde 
$$ \partial_aF_i = \{x \notin F_i \mid \mbox{ existe } y \in F_i  \mbox{ tal que  }x=ya \mbox{ ou }x=ya^{-1}\}.$$ 
Então
$$\displaystyle\sum_{i=1}^{n}c_i|\partial_a F_i|\leq \lVert f-f_a\rVert_{\ell^1(G)}\leq \dfrac{\varepsilon}{|A|}\displaystyle\sum_{i=1}^{n}c_i|F_i|, \, \forall a \in A.$$
Assim, $$\displaystyle\sum_{i=1}^{n}\sum_{a\in A} c_i|\partial_a F_i|\leq \varepsilon \displaystyle\sum_{i=1}^{n}c_i| F_i|.$$
Concluímos que  existe $i$ tal que $\displaystyle\sum_{a\in A} c_i|\partial_a F_i|\leq \varepsilon |F_i|$, e portanto 
$$\dfrac{|\partial F_i|}{|F_i|}\leq \varepsilon,\ \forall a \in A.$$
Logo, podemos colocar $F_i$ em uma sequência de Følner.
\end{proof}

\begin{corollary}
    Um grupo de crescimento subexponencial é sempre amenável. Em particular, os grupos virtualmente abelianos ou virtualmente nilpotentes são  amenáveis.  
\end{corollary}

\begin{proof}
Seja $G$ um grupo de crescimento subexponencial, e considere $B_n = \{x \in  G \mid
\ell(x) \leq n\}$ e escreva $|B_n| = s_n$. 

Então  $\partial B_n$ é o conjunto de elementos de comprimento
$n + 1$, e $\liminf \dfrac{s_{n+1}}{s_n} \leq \lim (s_{n})^{1/n} = 1$. Mas $s_{n+1} \geq s_n$, e portanto $\liminf \dfrac{s_{n+1}}{s_n} = 1$. Como $s_{n+1} = |B_n| + |\partial B_n|$, temos
$\liminf \frac{|\partial B_n|}{|B_n|} = 0$,
e deve existir uma subsequência de ${B_n}$ a qual é uma sequência de Følner em $G$.
\end{proof}


\begin{proposition}
    Sejam $G$ e $G'$  grupos finitamente gerados quasi-isométricos. Então $G$ é amenável se, e somente se, $G'$ é amenável.
\end{proposition}

\begin{proof}
    Basta considerar uma sequência de Følner $\Omega_n \subset G$. Uma quasi-isometria $f:G \to G'$ fornece uma sequência de Følner $f(\Omega_n) \subset G'$.
\end{proof}

Em \cite{chou1980elementary}, Chou descreve propriedades importantes da classe de grupos amenáveis elementares. Ele também estendeu o trabalho de Milnor e Wolf sobre crescimento de grupos para a classe de grupos amenáveis elementares:

\begin{thm}[Chou, 1980]
    Um grupo finitamente gerado elementar amenável ou
tem crescimento exponencial ou é virtualmente nilpotente.
\end{thm}

Além disso, ele provou que as operações de tomar subgrupos e quocientes são 
redundantes na definição de grupos amenáveis elementares. Isto permitiu  que fossem encontrados exemplos de grupos amenáveis que não são amenáveis elementares. O grupo de Grigorchuk foi o primeiro exemplo nesse sentido.

Os resultados anteriores implicam o seguinte:
\begin{corollary}
    Os grupos de crescimento intermediário são amenáveis, mas não amenáveis elementares.
\end{corollary}

\medskip

Em \cite{JM13} e \cite{Nekr18}, os autores sugeriram novos métodos para construir grupos amenáveis finitamente gerados que não são elementares. Isso produziu os primeiros exemplos de \textit{grupos simples} amenáveis, mas não amenáveis elementares. É interessante notar que o método de Nekrashevych em \cite{Nekr18} é baseado na análise dos  grafos de Schreier associados e sua estrutura linear.

\chapter{Alguns problemas abertos}
\label{cap12}
Recolhemos aqui alguns problemas ainda abertos. A maioria deles foi mencionada em algum lugar neste texto. Em \cite[Seção 18]{Mann}, há também uma coleção de problemas em aberto. Uma lista interessante de exercícios avançados  pode ser encontrada em \cite[páginas 295-298]{de2000topics}.

\begin{probl}[Conjectura de Andrews--Curtis]\index{conjectura de Andrews--Curtis} Qualquer apresentação balanceada do grupo trivial pode ser reduzida a uma apresentação trivial por uma sequência de conjugações dos relatores ou transformações de Nielsen?
\end{probl}

\begin{probl}[Thurston]
    Mostrar que os volumes de 3-variedades hiperbólicas não são todos racionalmente relacionados.
\end{probl}

\begin{probl}
Existe um grupo Gromov-hiperbólico que não é $\CAT(-1)$?
\end{probl}

\begin{probl}
Suponha que $G$ é um grupo $\CAT(0)$. É verdade que $G$ contém um subgrupo isomorfo a $\Z^2$?    
\end{probl}

\begin{probl}
Determine se todos os grupos hiperbólicos são residualmente finitos.
\end{probl}

\begin{probl}
Existe um grupo finitamente apresentado de crescimento intermediário?
\end{probl}

\begin{probl}
Existe um grupo de crescimento intermediário com todos os elementos de ordens limitadas por uma mesma constante?
\end{probl}

\begin{probl}
Existe um grupo infinito, finitamente apresentado, de torção?
\end{probl}

\begin{probl}\index{grupo de Thompson}
O grupo de Thompson 
$$F = \langle A,B \mid\ [AB^{-1},A^{-1}BA],\ [AB^{-1},A^{-2}BA^{2}]\rangle$$
é amenável ou não?
\end{probl}

\printindex
\bibliographystyle{alpha}
\bibliography{refs}

\end{document}